\pdfoutput=1
\documentclass[11pt]{article}

\usepackage[margin=1.05in]{geometry}
\usepackage{amsmath,amssymb,amsthm}
\usepackage{graphicx}
\usepackage{booktabs}
\usepackage{float}
\usepackage{microtype}
\usepackage{xcolor}
\usepackage{tikz}
\usepackage[shortlabels]{enumitem}
\usepackage[colorlinks=true,linkcolor=blue!45!black,citecolor=blue!45!black,urlcolor=blue!45!black]{hyperref}
\hypersetup{hypertexnames=false,pdftitle={Occupation counts of rarely switching Markov chains},pdfauthor={Arjun Pemmasani}}

\newtheorem{theorem}{Theorem}[section]
\newtheorem{lemma}[theorem]{Lemma}
\newtheorem{proposition}[theorem]{Proposition}
\newtheorem{corollary}[theorem]{Corollary}

\newtheorem{introtheorem}{Theorem}

\theoremstyle{definition}
\newtheorem{definition}[theorem]{Definition}
\newtheorem{example}[theorem]{Example}
\theoremstyle{remark}
\newtheorem{remark}[theorem]{Remark}

\newcommand{\Z}{\mathbb{Z}}
\newcommand{\R}{\mathbb{R}}
\newcommand{\E}{\mathbb{E}}
\newcommand{\Prob}{\mathbb{P}}
\newcommand{\Var}{\operatorname{Var}}
\newcommand{\Cov}{\operatorname{Cov}}
\newcommand{\supp}{\operatorname{supp}}
\newcommand{\ones}{\mathbf{1}}
\newcommand{\tree}{\operatorname{tree}}

\graphicspath{{figures/}}
\title{Occupation counts of rarely switching Markov chains:\\ stratified local laws, support selection and the boundary layer}
\author{Arjun Pemmasani\\[2pt]
\small Department of Mathematics, Harvey Mudd College\\
\small Claremont, CA 91711, USA\\
\small \texttt{apemmasani@hmc.edu}\\
\small ORCID \href{https://orcid.org/0009-0001-2776-6298}{0009-0001-2776-6298}}
\date{}
\newif\ifarxivbuild
\arxivbuildtrue
\newcommand{\Bversion}{\ifarxivbuild the first version\else the preprint\fi}

\begin{document}
\maketitle

\begin{abstract}
A Markov chain with exit rates $q_i$ runs for $N$ steps with transition matrix
$I+(T/N)Q$. On each irreducible face that carries start
mass, we prove a local limit theorem for the occupation counts with an
explicit constant and relative error $O(1/T)$, uniformly for $1\le
T\le\varepsilon N$ with $\varepsilon\max_iq_i<1$ and on compact subsets of the
open face. Up to $O(1)$, minus the log-probability of counts $n$ with support $F$ is
$|F|-1$ times $\log(N/\sqrt T)$ plus $N$ times the Donsker and Varadhan functional of $I+(T/N)Q$ at $n/N$. Every admissible face has an explicit power of $T$ in its largest
probability for large $T\le N^{1/3}$, and a constant given by a
finite-dimensional variational formula. For
$T=\tau\log N$ the face of the mode minimizes $\tau\lambda_F+|F|-1$ at first
order, with $\lambda_F$ the exit rate. If an admissible face has $\Delta\ge1$ more states than
another and an exit rate smaller by $\delta>0$, we locate every crossing of their largest
probabilities for large $T\le N^{1/3}$ up to $O(1/\log N)$. For two irreducible faces the crossing is
unique for large $N$ in $1\le T\le\varepsilon N$, with coefficient $1/2$ of $\log\log N$. Near a facet, the
continuous-time occupation density is the facet density times an explicit kernel, with relative
error $O(1/T)$ at times $x\le cT$ off the facet, $c$ small. The counts obey a lattice form with
extra error $T^2/N$ for $T\le\varepsilon_1\sqrt N$. The kernel is Bessel for $x=o(T)$ at facet profiles where the removed state has
equal entry and exit laws, and at a profile where the rate has a first correction, of order
$m\ge3$, if and only if $x=o(T^{1-2/m})$.
On the complete graph with the
uniform start, each face of at least 2 states crosses a vertex exactly once.
\end{abstract}



\section{Introduction}\label{sec:intro}

Let $Q$ be a generator on $d$ states, with exit rates $q_i=\sum_{j\ne i}q_{ij}$, and run the chain for $N$ steps with
transition matrix $I+(T/N)Q$. Its path follows the continuous-time chain with
generator $Q$ up to time $T$, while the transition matrix tends to the identity
as $N$ grows. Hence the chain degenerates. Its spectral gap is $T/N$ times that of
$Q$, and
for fixed $T$ the visit fractions converge in law to the occupation fractions of
the continuous-time chain at time $T$, which charge the faces of the simplex. We
let $T$ range from $1$ to $h_*N$, for a fixed $h_*$ with $h_*\max_iq_i<1$, and
describe the law of the occupation counts $n_i=\#\{t\le N:X_t=i\}$. The support
$F=\{i:n_i>0\}$ of the count vector $n$, which we also call a composition, is a face of the
simplex.

Our main result is a local limit theorem for the count vector on each face that satisfies
the condition (H) below, with
an explicit constant and uniformly in $T$ (Theorem~\ref{thm:A}). On compact subsets of
the open face, up to $O(1)$,
$-\log P_N(n)$ is $(|F|-1)\log(N/\sqrt T)$ plus $N$ times the Donsker and
Varadhan functional of the discrete chain at $n/N$. Hence each dimension of the
face costs a factor $N$ up to a power of $\log N$ when $T$ is of order $\log N$,
and a factor of order $\sqrt N$ when $T$ is of order $N$. The functional is one functional of the whole chain, and on each
face it is that of the chain killed on leaving the face. Two consequences follow. First, the law selects the support of the most
likely count vector, the mode, and it locates the values of $T$ at which two
faces have equal largest probabilities (Theorem~\ref{thm:B}). Second, near a
facet the occupation density of the continuous-time chain on a face is the
density on the facet times an explicit kernel, which is Bessel in a range of
the time off the facet that we determine, and the counts follow the same law
with a lattice kernel. \ifarxivbuild The complete graph serves as a worked
example (Theorem~\ref{thm:C}).\fi

To state the law, let $J_V$ be the Donsker and Varadhan rate function of the
empirical measure of the continuous-time chain with generator $Q$, given
in~\eqref{eq:JV}. For $h=T/N$ let $hJ^{(h)}_V$ be the same rate for the chain with
transition matrix $I+hQ$ (Section~\ref{sec:discrete-rate}). It satisfies
$J^{(h)}_V\ge J_V$ and tends to $J_V$ as $h\to0$. We say that a face $F$ satisfies (H)
if the chain can move
between any two states of $F$ without leaving $F$ and the start gives $F$
positive mass. Write $\lambda_F$ for the exit rate of $F$ and $M_F$ for the
largest probability of a count vector with support $F$
(Definition~\ref{def:faces}).

\begin{introtheorem}[the law on a face]\label{thm:A}
Fix the rates $q_{ij}$, the law $\mu$ of $X_1$ and $h_*$ with $h_*\max_iq_i<1$. Let $F$
satisfy \textup{(H)}, with $k=\lvert F\rvert$ states. If $k\ge2$, let $\mathcal K$ be a
compact subset of the open face. Uniformly for $\supp n=F$, $n/N\in\mathcal K$ and
$1\le T\le h_*N$,
\[
 -\log P_N(n)=(k-1)\log N+T\,J^{(h)}_V(n/N)-\frac{k-1}2\log T-\log C^{(h)}_F(n/N)
 +O\Bigl(\frac1T\Bigr),
\]
where $C^{(h)}_F>0$ is an explicit function that tends to the function $C_F$ of
Definition~\ref{def:tilt} as $h\to0$. Uniformly for $1\le T\le h_*N$,
\[
 M_F=N^{-(k-1)}\,T^{(k-1)/2}K^{(h)}_F\,(1-h\lambda_F)^N\bigl(1+O(1/T)\bigr),
\]
where $K^{(h)}_F=C^{(h)}_F(\alpha^*_F)$ and $\alpha^*_F$ is the quasi-ergodic law of
$F$. For $k=1$ the second display holds exactly, with $K^{(h)}_F=\mu_i/(1-hq_i)$ for
$F=\{i\}$. For $1\le T\le N^{1/3}$ both displays hold with $J_V$, $C_F$,
$K_F=C_F(\alpha^*_F)$ (with $K_{\{i\}}=\mu_i$) and $e^{-T\lambda_F}$ in place of $J^{(h)}_V$, $C^{(h)}_F$,
$K^{(h)}_F$ and $(1-h\lambda_F)^N$, and with the error $O(1/T+T^2/N)$. The
second display then holds for every admissible face $F$, one that a path of positive
probability can fill, with $k=|F|$, for $T_1\le T\le N^{1/3}$, with $T_1$ a large constant
(Theorem~\ref{thm:reducible}), and with $T^{(k-1)/2}$ and $K_F$
replaced by an explicit power $T^{p_F}$ and a constant $K^{\rm red}_F>0$ given by a
finite-dimensional variational formula.
\end{introtheorem}

Theorem~\ref{thm:A} is Theorems~\ref{thm:stratified}, \ref{thm:facemax}
and~\ref{thm:reducible}.
Good's formula for the counts is an integral over a torus (Lemma~\ref{lem:torus}).
In the variable of the torus the integrand is analytic in $h$ up to $h=0$, where it
is the integrand of the continuum density of the chain killed on leaving $F$. Hence
one Laplace argument with $h$ as a parameter gives the law for $1\le T\le h_*N$
(Theorem~\ref{thm:local}). Only for bounded $T$ do we compare with the continuum,
to keep the ratio bounded below. The rate of the killed chain is the restriction of
$J^{(h)}_V$ to the face (Proposition~\ref{prop:discrete-rate}). Its least value is
$-\frac1h\log(1-h\lambda_F)$, attained only at $\alpha^*_F$ for every $h$
(Proposition~\ref{prop:discrete-variational}). With the continuum rate the error
$T^2/N$ cannot be avoided, since $TJ^{(h)}_V-TJ_V$ is $T^2/N$ times a function
of $n/N$ that vanishes at most at one point, up to $O(T^3/N^2)$. On faces with (H)
the continuum form follows from the discrete one for $T\le N^{1/2}$. We state it for
$T\le N^{1/3}$, the range of a second proof through the continuum density, which needs
$T^2/N\le1/T$ to locate the most likely counts and which we make uniform over a class of
chains. The factor $T^{(k-1)/2}$ appears because
the visit fractions fluctuate on the scale $T^{-1/2}$ in each of the $k-1$ free
directions. For reversible faces $K_F$ involves the square root of a weighted
count of spanning trees (Proposition~\ref{prop:constant}). The face maxima hold
with a field $v$ that weights a composition by $e^{Tv\cdot n/N}$, uniformly for
$v$ in a compact set, and uniformly over a compact class of reversible chains
(Proposition~\ref{prop:uniform-facemax}). If the states of $F$ do not all
communicate, the time goes to the components of least exit rate and least size $k^*$,
and to one of them when $k^*\ge2$. Every other component $C$ gets time of order one
and costs $(T/N)^{|C|-1}$. Each switch
from one component to the next costs $T/N$ too. Hence the power is $p_F=(k^*-1)/2+|F|-k^*$.

\subsection{The mode and its crossings}\label{sec:intro-faces}

When $T=\tau\log N$ the cost $e^{-T\lambda_F}$ of staying in $F$ and the cost
$N^{-(|F|-1)}$ of spreading over the count vectors of $F$ are both powers of $N$.
Put $\Phi_\tau(F)=\tau\lambda_F+|F|-1$. A face is admissible if a path of
positive probability has exactly that support.

\begin{introtheorem}[the support of the mode]\label{thm:B}
Fix the rates $q_{ij}$ and the law $\mu$ of $X_1$.
\begin{enumerate}[(i)]
\item Let $T=\tau\log N$ with $\tau>0$ fixed. Every admissible face $F$ has
$M_F=N^{-\Phi_\tau(F)+o(1)}$. For large $N$ the support of every most likely
composition is an admissible face that minimizes $\Phi_\tau$ among the
admissible faces. If each such minimizing face satisfies \textup{(H)}, then the
visit fractions of every most likely composition are within $O(1/\log N)$ of
the quasi-ergodic law of its support.
\item Let $F$ and $G$ satisfy \textup{(H)}, and suppose that $F$ has
$\Delta\ge1$ more states than $G$ and an exit rate smaller by $\delta>0$. For
large $N$ the equation $M_F=M_G$ has exactly one root $T$ in $[1,h_*N]$, for any
fixed $h_*$ with $h_*\max_iq_i<1$, and it satisfies
\[
 \delta T+\frac\Delta2\log T=\Delta\log N-\log\frac{K_F}{K_G}
 +O\Bigl(\frac1{\log N}\Bigr).
\]
If $F$ and $G$ are only admissible, every root in $[T_1,N^{1/3}]$ satisfies this
equation with $(p_F-p_G)\log T$ in place of $\frac\Delta2\log T$ and $K^{\rm red}$
in place of $K$.
\end{enumerate}
\end{introtheorem}

Part (i) is Theorem~\ref{thm:first}(ii)--(iv) and part (ii) is
Theorem~\ref{thm:crossing} and Corollary~\ref{cor:reducible-crossing}. The face maxima of Theorem~\ref{thm:A} keep every root
in a window of order $\log N$, and there the uniqueness comes from a derivative of
the local law in $T$ (Lemma~\ref{lem:derivative}). Theorem~\ref{thm:hull} reads
the winners in (i) off the lower-left boundary of the convex hull of the points
$(\lambda_F,|F|-1)$. Hence as $\tau$ grows the winning size can only grow. Solving
the equation in (ii) gives $T=\frac\Delta\delta[\log N-\frac12\log\log N+O(1)]$.
The coefficient $\frac12$ of $\log\log N$ holds for every such pair, whatever
the graph, the rates, the sizes or the start, and without reversibility. For
admissible faces it is $(p_F-p_G)/\Delta$, which is $1$ and then $0$ in
Example~\ref{ex:threecycle}. If $Q$ has
symmetric support or $\mu$ full support, and at each switch of (i) only two
values of $(\lambda_F,|F|)$ compete, Corollary~\ref{cor:mode-changes} locates
every change of the mode for $T$ of order $\log N$ up to $O(1/\log N)$. When
more compete, Theorem~\ref{thm:window} describes the window of $T$ in which they
do.

\subsection{The boundary layer}\label{sec:intro-layer}

Section~\ref{sec:boundary-layer} follows the occupation density of the
continuous-time chain from the inside of a face $F$ up to a facet
$G=F\setminus\{i\}$. Let $x$ be the time spent in $i$, $S=T-x$, and $\beta$ the
profile of the time on $G$, a point of the open facet. Let $\pi^{\rm in}_\beta$ be
the law of the state from which the chain tilted to $\beta$ enters $i$, and
$\pi^{\rm out}_\beta$ the law of the state into which it leaves $i$. The Bessel
kernel is $e^{-\gamma x}\sqrt{\omega S/x}\,I_1(2\sqrt{\omega Sx})$, where
$\omega$ and $\gamma$ are explicit and $\gamma-q_i$ is $\omega/2$ times the
effective resistance between $\pi^{\rm in}_\beta$ and $\pi^{\rm out}_\beta$
(Proposition~\ref{prop:layer-constants}). For every irreducible face whose
facet satisfies (H), Theorem~\ref{thm:layer} writes the density on $F$ as the
density on $G$ times an explicit generalized kernel, with relative error
$O(1/T)$ for $x$ up to a small multiple of $T$. Proposition~\ref{prop:layer-cubic}
gives the range of the Bessel form at each $\beta$. At the $\beta$ with
$\pi^{\rm in}_\beta=\pi^{\rm out}_\beta$, which are all $\beta$ when $Q$ is
reversible and some $\beta$ for some chains that are not, the Bessel kernel holds
with relative error $O(1/T+\sqrt{x/T})$, so for $x=o(T)$. At a $\beta$ where the
rate has a term beyond the Bessel exponent, in powers of $\sqrt{x/S}$, let $m\ge3$
be the order of the first one. There the relative Bessel form holds if and only
if $x=o(T^{1-2/m})$. We do not know whether such a term exists whenever
$\pi^{\rm in}_\beta\ne\pi^{\rm out}_\beta$. For the cycle on three states with rates
$2$ in one direction and $1$ in the other, at the uniform profile on a facet of
$2$ states, $m=3$, and the logarithm of the ratio
to the Bessel form is $\frac{49}{64\sqrt\omega}\,x^{3/2}S^{-1/2}$ to leading
order. For a pair face the kernel is the facet form of Pedler's formula, which
we prove with an explicit error for all $x$ (Proposition~\ref{prop:pair-face}).
The proof rests on a local law for a weighted chain on the facet with a complex
parameter (Theorem~\ref{thm:L-facet}). We prove it on a torus through a complex
critical point, with a Laplace lemma for a complex Hessian. The counts follow
the same law. With $n_i$ visits to $i$, $P_N(n)$ is the probability of the
counts on $G$ in $N-n_i$ steps times a lattice kernel, the generalized kernel with
Laguerre polynomials\ifarxivbuild{} (Section~\ref{sec:vertex})\fi{} in place of the Bessel
function, with the extra relative error $T^2/N$ for $T$ up to a small multiple of $N^{1/2}$
(Theorem~\ref{thm:layer-counts}). The lattice kernel is $T/N$ times the
generalized kernel up to $O(T^2/N)$, so the Bessel forms hold for the counts with $T^2/N$ added to the error.
For a pair face the discrete Pedler formula is exact, and we compare it with the
lattice kernel at every $T$ (Theorem~\ref{thm:pair-counts}). For three or more
states the counts are open when $T^2/N$ is of order one.

\ifarxivbuild
\subsection{An example, the complete graph}\label{sec:intro-complete}

On the complete graph with the uniform start, $Q=J-dI$ with $J$ the all-ones
matrix, every face ties at first order and the constants decide. A face of $k$
states has $K_F=\frac1dk^{k+1/2}(4\pi)^{-(k-1)/2}$, and the $k^{k-2}$ spanning
trees of $K_k$ enter through it. Put
$a_k=\frac12\log(4\pi)-\frac{k+1/2}{k-1}\log k$, write $V_N$ for the
probability of a vertex, and call a composition balanced if its positive
coordinates differ by at most $1$. The next theorem collects results of our
preprint~\cite{paperB}, which this paper \ifarxivbuild replaces\else
supersedes\fi. Here they come out of the general theory where it applies, and
Table~\ref{tab:overlap} says which parts use the exact arguments of the
preprint.

\begin{introtheorem}[the complete graph~\cite{paperB}]\label{thm:C}
Let $Q=J-dI$ and let the start be uniform.
\begin{enumerate}[(i)]
\item Let $F$ be a face of $2\le k\le d$ states. For every $N\ge k$ the equation
$M_F=V_N$ has exactly one root $T_{N,k}$ in $(0,N/(d-1))$, and the probability
$T_{N,k}/N$ of a switch to a given state decreases strictly in $N$. As
$N\to\infty$,
\[
 T_{N,k}+\tfrac12\log T_{N,k}=\log N+a_k+O\Bigl(\frac1{\log N}\Bigr).
\]
\item For every $N$ and every $0<T\le N/d$, the most likely compositions with a
given support are exactly the balanced ones.
\item Since $a_2>a_3>\dots>a_d$, for large $N$ and $T$ up to any fixed multiple
of $\log N$ the mode changes exactly once, from the $d$ vertices straight to the
balanced compositions with $d$ positive coordinates.
\end{enumerate}
\end{introtheorem}

Part (i) is Theorem~\ref{thm:simplex-crossing}, part (ii) is
Theorem~\ref{thm:simplex-balance} and part (iii) is
Corollary~\ref{cor:simplex-modes}. The uniform start matters in (ii). For $d=2$,
$N=4$ and the start in state $1$, the most likely composition with support
$\{1,2\}$ is $(3,1)$ at $T=\frac12$, $1$ and $\frac32$, since the three
compositions $(3,1)$, $(2,2)$ and $(1,3)$ have the probabilities
$63,57,49$ over $512$ at $T=\frac12$, $15,13,9$ over $64$ at $T=1$ and
$165,147,75$ over $512$ at $T=\frac32$. Figure~\ref{fig:complete-envelope}
shows the window of (iii).

\begin{figure}[t]
\centering
\includegraphics[width=0.52\textwidth]{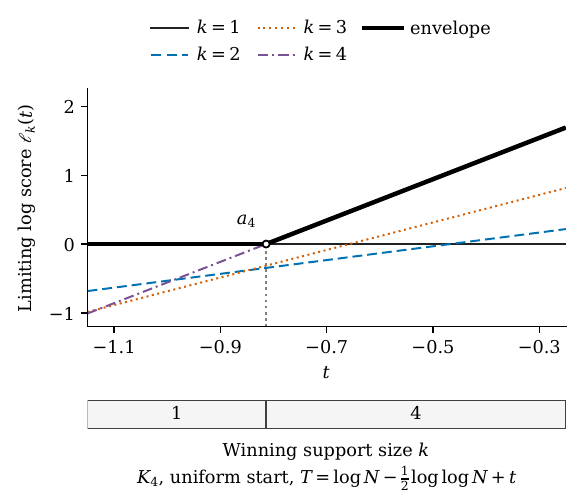}
\caption{Support selection on the complete graph with four states, unit rates
and uniform start, in the window $T=\log N-\frac12\log\log N+t$.
Thin lines are limiting log face maxima relative to a vertex,
and the bold curve is their upper envelope. The winner jumps from one state
to all four at $t=a_4$. The scores are derived in Section~\ref{sec:complete},
and the marked tie concerns the limit, not a finite-$N$ equality}
\label{fig:complete-envelope}
\end{figure}
\fi

\subsection{Related work}\label{sec:related}

The exact law of the count vector goes back to Whittle~\cite{whittle1955},
Dawson and Good~\cite{dawsongood1957}, Goodman~\cite{goodman1958} and
Good~\cite{good1961}. For two states it is the Markov binomial law, whose exact form Gabriel~\cite{gabriel1959} gave. Dobrushin~\cite{dobrushin1953} found the limit laws of the number of visits of a two-state chain whose switch probabilities
scale with the length, and for switch probabilities of order $1/N$ the limit is
not Gaussian. This is the case $d=2$ of our regime. \v{C}ekanavi\v{c}ius and
Vellaisamy~\cite{cv2010} give compound Poisson approximations to the Markov
binomial law and cite the work on it since Dobrushin. Stadje~\cite{stadje1999}
finds the joint laws of the numbers of visits and their limits when one state
becomes nearly absorbing, with the transition probabilities out of it of order
$1/N$. Here every row of the transition matrix degenerates at the same rate.
Good proposed a saddle-point analysis of his formula (10) and did not carry it
out~\cite[pp.~42--43]{good1961}, and Theorem~\ref{thm:local}(a) is that analysis,
uniform in $T$ and with its constant. Part (b) does the same for his discrete
formula (3), uniformly in $T/N\le h_*$. Polettini~\cite{polettini2015} and
Carugno, Vivo and Coghi~\cite{cvc2022} derive the exact law from the BEST and
matrix-tree theorems.

For a fixed chain the occupation counts form a Markov additive process, and its
large deviations are classical. Miller~\cite{miller1961} proves a local large deviation theorem from the Perron root of the transform matrix, Ney and Nummelin~\cite{nn1987a,nn1987b} obtain the rates from its convergence parameter on general state spaces,
Kontoyiannis and Meyn~\cite[Theorem~6.3]{km2003} give exact large deviation
asymptotics for a real functional, and Herv\'e and Ledoux~\cite{herveledoux2013}
prove a local limit theorem for densities of the additive component. These
results concern a fixed transition matrix, and their constants depend on it. For
instance, Theorems~6.3 and~6.5 of~\cite{km2003} give the exact asymptotics of
$\Prob\{S_n\ge nc\}$ as $n\to\infty$, where $S_n$ is the partial sum of a real
functional of one geometrically ergodic chain, in the non-lattice and the lattice case, with no rate in $n$ and no
uniformity over kernels. Here the matrix $I+hQ$ tends to the identity and its
spectral gap $T/N$ tends to $0$, so these results give no bound uniform in $N$.
Theorem~\ref{thm:local}(b) is an exact local large deviation estimate for the whole
count vector of the chain killed on leaving $F$, with an explicit constant, and the
new point is that it is uniform as $I+hQ\to I$. At $h=0$ it joins
Theorem~\ref{thm:local}(a), the same estimate for the occupation times of the
continuous-time chain. We found no statement in these sources that covers either,
and none that is uniform as the chain degenerates.

In continuous time, Brydges, van der Hofstad and K\"onig~\cite{bhk2007} give an
explicit formula for the joint density of the local times on the event that the
range is a given set, and use it for large deviation upper bounds beyond the
exponential scale. On the event that the range up to time $T$ is $F$,
Lemma~\ref{lem:continuum} gives this density, mixed over the start and written
as Good's run expansion, and Theorem~\ref{thm:local}(a) gives its exact asymptotics
on compact subsets of the open face. Huang, Kious, Sidoravicius and
Tarr\`es~\cite{hkst2018} write the joint density of the local times, the last
exit tree and the cycling numbers of a chain with symmetric rates as a product
of Bessel functions, one for each edge. Pedler's occupation-time density for
two-state chains~\cite{pedler1971} and the Bessel forms of Kovchegov, Meredith
and Nir~\cite{kmn2010} for the occupation time of one state are exact formulas
for a fixed chain too. In the boundary layer the Bessel kernel arises
differently. It is the large-$T$ form of the contribution of the visits to one
state, after the facet is integrated out by a local law, and without
$\pi^{\rm in}=\pi^{\rm out}$ it can fail beyond $x$ of order $T^{1/3}$.

The rate inside a face is the Donsker and Varadhan functional of the killed
chain~\cite{dv1975,dupuisliu2015,chendeng2013}. Guillin, Nectoux and
Wu~\cite{gnw2024} prove a large deviation principle for the occupation measure
of a process conditioned to stay in a set. The quasi-ergodic law goes back to
Darroch and Seneta~\cite{darrochseneta1965,darrochseneta1967}, and Collet,
Mart\'inez and San Mart\'in~\cite{cmsm2013} is the standard reference. Dietz and
Sethuraman~\cite{dietz2007,dietz2005}, Bouguet and
Cloez~\cite{bouguetcloez2018} and Bena\"im, Bravo and Faure~\cite{bbf2026}
study chains whose switching decays along the trajectory, while here the rate is
fixed during each path and changes with $N$. Bertini, Gabrielli and
Landim~\cite{bgl2024}, Landim~\cite{landim2023} and Kim and Landim~\cite{kl2024}
expand the rate functions of a fixed chain in a parameter, the last for the
joint rate of the empirical measure and the flow~\cite{bfg2014,bfg2015}. Here
the functional enters through its restriction to each face, and faces of
different dimension compete.

The two-state walk is treated at every $N$ in~\cite{paperA}. Our
preprint~\cite{paperB} on the complete graph is \ifarxivbuild the first version
of this arXiv entry, and this version replaces it\fi. Table~\ref{tab:overlap} lists the results of \Bversion\ that we
reprove, where, which part, and how.{} The other results of~\cite{paperB} are not
reproved here, namely Proposition~2.1, Proposition~3.7, Sections~4 and~5 on a
start and a field, Theorem~6.2, Remark~6.3 and Appendix~C. A companion
paper\ifarxivbuild~\cite{paperC2}\fi{} asks what the Donsker and Varadhan functional and the modal
supports determine about the generator.

\begin{table}[t]
\centering
\small
\caption{Results of \Bversion~\cite{paperB} reproved here. ``General'' means
that the proof is a case of the general theory of Sections~\ref{sec:model}
to~\ref{sec:secondorder}\ifarxivbuild, and ``as in~\cite{paperB}'' that it is the exact
argument of the preprint\fi}
\label{tab:overlap}
\begin{tabular}{@{}p{0.17\textwidth}p{0.25\textwidth}p{0.27\textwidth}p{0.2\textwidth}@{}}
\toprule
Item of~\cite{paperB} & Result here & Part covered & Method\\
\midrule
\ifarxivbuild
Prop.~2.2 & Lemma~\ref{lem:refresh}, \eqref{eq:runform}, \eqref{eq:SnVN} & all & as in~\cite{paperB}, with the run form from Lemma~\ref{lem:runs}\\
Thm.~3.1 & Cor.~\ref{cor:paperB}, Thm.~\ref{thm:simplex-crossing}(ii), (iv) & (ii) to (iv), and (i) without the term $-(k-1)/(8T)$ & general, and (iv) as in~\cite{paperB}\\
Thm.~3.2 & Thm.~\ref{thm:simplex-crossing}(i), (iii) & all but the term $1/(8T_{N,k})$ and the third-order expansion & (i) as in~\cite{paperB}, and (iii) general\\
Thm.~3.3 & Thm.~\ref{thm:stratified}, Ex.~\ref{ex:complete-face} & the local law (7), not the root (8) & general\\
Lemmas~3.4 and~3.5, Prop.~3.6 & Lemma~\ref{lem:weighted-cayley}, Prop.~\ref{prop:constant}(b), Ex.~\ref{ex:complete-face} & all but Prop.~3.6(5) & general, and Lemma~3.5 as in~\cite{paperB}\\
Thm.~3.8 & Thm.~\ref{thm:simplex-balance} & all & as in~\cite{paperB}\\
Cor.~3.9 & Cor.~\ref{cor:simplex-modes} & all & as in~\cite{paperB}, with Thm.~\ref{thm:simplex-crossing}(ii)\\
Ex.~3.10 & Rem.~\ref{rem:simplex-edge-mode} & all & as in~\cite{paperB}\\
Lemma~6.1 & Lemma~\ref{lem:uniform-laguerre-bessel} & all & as in~\cite{paperB}\\
Thm.~6.4 & Thm.~\ref{thm:vertex-uniform-root} & all & as in~\cite{paperB}\\
\fi
\bottomrule
\end{tabular}
\end{table}

\emph{Organization.} Section~\ref{sec:model} sets up the exact law and
Section~\ref{sec:dv} the killed functional, in continuous and in discrete time. Section~\ref{sec:firstorder} proves
first-order face selection and gives the face maxima of reducible faces. It uses the local law and the face maxima of
Section~\ref{sec:secondorder}, whose proofs do not use
Section~\ref{sec:firstorder}. Section~\ref{sec:boundary-layer} treats the
boundary layer, \ifarxivbuild Sections~\ref{sec:complete} and~\ref{sec:vertex} the complete
graph, \fi and Section~\ref{sec:verify} the checks and the open questions.

\section{Setting, notation and the exact law}\label{sec:model}

\subsection{Setting and notation}\label{sec:notation}

\emph{The chain.} Fix $d\ge2$ and a generator $Q$ on $[d]=\{1,\dots,d\}$. Its
rates are $q_{ij}\ge0$ for $i\ne j$, the exit rates are $q_i=\sum_{j\ne i}q_{ij}$,
$Q_{ii}=-q_i$, and $q_{\max}=\max_iq_i$. We also write $q$ for the vector
$(q_i)$. The start $\mu$ is a probability law on $[d]$, and
$\mu(F)=\sum_{i\in F}\mu_i$. We write $V=[d]$ for the face of all states. For
$N\ge2$ and $T>0$ put
\[
 h=\frac TN,\qquad L=\log N,\qquad \tau=\frac TL .
\]
The chain $X_1,\dots,X_N$ has $X_1\sim\mu$ and transition matrix $I+hQ$. So at
each step it jumps from $i$ to $j$ with probability $hq_{ij}$, and $N$ steps
look like the continuous-time chain with generator $Q$ run for time $T$. We
assume $hq_{\max}\le\frac12$ unless we say otherwise. The \emph{composition} is
the vector of visit counts $n=(n_1,\dots,n_d)$, with
$n_i=\#\{t\le N:X_t=i\}$, so that $n_1+\dots+n_d=N$. We write $P_N(n)$ for the
probability that the chain has composition $n$, $\supp n=\{i:n_i>0\}$ for its
support, and $\alpha=n/N$ for the visit fractions.

\emph{Faces.} A face is a nonempty set $F\subseteq[d]$, of size
$k=k_F=\lvert F\rvert$. Let $\Delta_F$ be the set of probability vectors on
$[d]$ that vanish off $F$, and $\Delta_F^\circ$ the set of those that are
positive on $F$. We number the states of $F$ as $1,\dots,k$ and use the
coordinates $(\alpha_2,\dots,\alpha_k)$ on $\Delta_F$ and $(t_2,\dots,t_k)$ on
$\{t\in\R^F:t_1+\dots+t_k=T\}$. We write $x\cdot y=\langle x,y\rangle
=\sum_{i\in F}x_iy_i$ and $x\circ y=(x_iy_i)_{i\in F}$. Let
$A_F=(q_{ij})_{i,j\in F}$ with zero diagonal, and $Q_F=A_F-\operatorname{diag}
(q_i)_{i\in F}$. Since $q_i$ includes the rates of leaving $F$, $Q_F$ generates
the chain killed when it leaves $F$. The switching graph $G_F$ has an edge
$i\to j$ whenever $i,j\in F$, $i\ne j$ and $q_{ij}>0$.

\emph{Perron roots.} A square matrix $B$ is Metzler if its off-diagonal entries
are nonnegative, and then its Perron root $\varrho(B)$ is its largest real
eigenvalue.
If $B$ is irreducible, $\varrho(B)$ is simple and has positive left and right
eigenvectors~\cite{bermanplemmons1994}.

\begin{definition}[faces]\label{def:faces}
Let $F$ be a face. We call $F$ \emph{irreducible} if $G_F$ is strongly connected.
In particular every singleton is irreducible. We call $F$ \emph{admissible} if for some
$N$ some path $x_1\cdots x_N$ of positive probability has support exactly $F$
(for $0<hq_{\max}<1$ this does not depend on $h$).
We say that $F$ satisfies \textup{(H)} if $F$ is irreducible and $\mu(F)>0$.
The \emph{exit rate} of $F$ is $\lambda_F=-\varrho(Q_F)$, so
$\lambda_{\{i\}}=q_i$. If $F$ is irreducible, $r_F$ and $l_F$ are right and left
eigenvectors of $Q_F$ for $-\lambda_F$ with positive entries and
$\langle l_F,r_F\rangle=1$, and $\alpha^*_F=l_F\circ r_F$ is the
\emph{quasi-ergodic law} of $F$~\cite{darrochseneta1965,darrochseneta1967}. The
\emph{face maximum} is $M_F=\max\{P_N(n):\supp n=F\}$, with $\max\emptyset=0$.
\end{definition}

Note that $\lambda_F$ is a true exit rate. If $F$ is irreducible, then
$e^{TQ_F}r_F=e^{-T\lambda_F}r_F$ with $r_F$ positive. Hence the chance that the
continuous-time chain started in $F$ stays in $F$ up to time $T$ lies between
$e^{-T\lambda_F}\min r_F/\max r_F$ and $e^{-T\lambda_F}\max r_F/\min r_F$. If
$F$ is reducible, we order the strongly connected components $C_1,\dots,C_r$ of
$G_F$ so that every edge between components goes forward. Then $Q_F$ is block
triangular with diagonal blocks $Q_{C_j}$, and
$\lambda_F=\min_j\lambda_{C_j}$. In all cases $0\le\lambda_F\le\min_{i\in
F}q_i$, since the rows of $Q_F$ have sums at most $0$ and a Perron root is at
least every diagonal entry. For a face $F$ satisfying (H), $\alpha^*_F$ is the
limit as $T\to\infty$ of the law of the visit fractions of the continuous-time chain
conditioned to stay in $F$ up to time $T$.

\emph{Conventions.} We fix $Q$, $\mu$ and the faces we discuss. Constants,
implied constants in $O(\cdot)$ and thresholds $N_0$ depend only on these and
on the compact set $\mathcal K$, if there is one, unless we say otherwise. By a
bound ``uniformly for $1\le T\le N^{1/3}$'' we mean one implied constant for
all $N\ge N_0$ and all $T$ in that range. Some results fix $h_*$ with $h_*q_{\max}<1$ and hold uniformly for
$1\le T\le h_*N$. There $h\le h_*$ replaces $hq_{\max}\le\frac12$, and the constants
also depend on $h_*$. Subscripts as in $O_{\mathcal K}$ stress a dependence.

\subsection{Admissible faces}

\begin{lemma}[admissible faces]\label{lem:admissible}
Let $F$ be a face and $C_1,\dots,C_r$ the strongly connected components of
$G_F$. \textup{(a)} $F$ is admissible if and only if the components can be
numbered so that $\mu(C_1)>0$ and, for each $j<r$, $q_{xy}>0$ for some $x\in
C_j$ and $y\in C_{j+1}$. \textup{(b)} An irreducible face is admissible if and
only if $\mu(F)>0$. So \textup{(H)} holds exactly for the irreducible admissible
faces. \textup{(c)} If $Q$ has symmetric support, that is $q_{ij}>0$ exactly
when $q_{ji}>0$, then every admissible face is irreducible.
\end{lemma}

\begin{proof}
We prove this by translating paths into walks in $G_F$. Take a path of positive probability
with support $F$ and list the states of its maximal runs. This list is a walk
in $G_F$ that starts at a state of $\supp\mu$ and visits every state of $F$.
Conversely, every such walk $v_1,\dots,v_J$ is the list of runs of a path of
positive probability with support $F$, as soon as $N\ge J$. The path follows
the walk and then stays in its last state, and a stay has positive probability.
Hence $F$ is admissible exactly when $G_F$ has a walk
that starts in $\supp\mu$ and visits every state of $F$.

(a) Suppose that there is such a walk. Once it has left a component $C$ for a
state $y\notin C$, it cannot come back to $C$. Otherwise $y$ could be reached
from $C$ and $C$ from $y$, and $y$ would belong to $C$. Hence the walk visits
each component in one block of consecutive times, and it passes from one block
to the next along a single edge. We number the components in the order of these
blocks. The walk starts in $C_1$, so $\mu(C_1)>0$, and the edge from block $j$
to block $j+1$ is an edge $x\to y$ with $x\in C_j$ and $y\in C_{j+1}$. This is
the condition in (a).

Conversely, suppose that the components are numbered as in (a), with edges
$x_j\to y_{j+1}$ from $C_j$ to $C_{j+1}$ for $j<r$. Take $y_1\in C_1$ with
$\mu_{y_1}>0$ and any $x_r\in C_r$. Since $C_j$ is strongly connected, there is
a walk inside $C_j$ from $y_j$ to $x_j$ that visits every state of $C_j$. We
join these $r$ walks by the edges $x_j\to y_{j+1}$. The result is a walk that
starts in $\supp\mu$ and visits every state of $F$.

(b) This is (a) with $r=1$.

(c) If $Q$ has symmetric support, an edge from one component to another would
come with an edge back, and the two components would be the same. Hence no edge
joins two components, and the condition in (a) forces $r=1$.
\end{proof}

Hence $M_F>0$ for large $N$ if $F$ is admissible, and $M_F=0$ otherwise.

\subsection{The exact law}

For $m\in\Z_{\ge1}^F$ put $M=\sum_{i\in F}m_i$ and
$W(m)=\sum_c\mu_{c_1}\prod_{j<M}q_{c_jc_{j+1}}$, the sum over the words
$c=c_1\cdots c_M$ in $F$ with $c_j\ne c_{j+1}$ in which each $i\in F$ occurs
$m_i$ times. Such a word lists the states of the runs of a path, and $m_i$ is
the number of runs in state $i$.

\begin{lemma}[run expansion]\label{lem:runs}
Let $hq_{\max}\le1$ and let $n$ have support $F$. Then, with
$\binom{n_i-1}{m_i-1}=0$ for $m_i>n_i$ and $0^0=1$,
\[
P_N(n)=\sum_{m\in\Z_{\ge1}^F}W(m)\,h^{M-1}\prod_{i\in F}\binom{n_i-1}{m_i-1}
(1-hq_i)^{n_i-m_i}.
\]
\end{lemma}

This is Good's coefficient identity~\cite[(3)]{good1961}, with the exponents
$n_i-1$, written for the chain killed on leaving $F$.

\begin{proof}
We prove this by counting words. The states $c_1\cdots c_M$ of the maximal runs of a path with composition $n$
form a word as in $W(m)$. The path starts at $c_1$, of probability
$\mu_{c_1}$, it has $M-1$ switches, of probabilities $hq_{c_jc_{j+1}}$, and
$n_i-m_i$ stays in state $i$, of probability $1-hq_i$ each. The number of paths
with the word $c$ and composition $n$ is $\prod_i\binom{n_i-1}{m_i-1}$, the number
of ways to write each $n_i$ as an ordered sum of $m_i$ positive run lengths.
Summing over the words $c$ with given $m$ gives $W(m)$, and summing over $m$ gives
the claim.
\end{proof}

For $F=\{i\}$ the only path is $i^N$, so uniformly for $1\le T\le N^{1/2}$
\begin{equation}\label{eq:vertexmax}
M_{\{i\}}=\mu_i(1-hq_i)^{N-1}=\mu_ie^{-q_iT}\bigl(1+O(T^2/N)\bigr).
\end{equation}
For $t\in(0,\infty)^F$ let $H(t)=\sum_mW(m)\prod_{i\in F}t_i^{m_i-1}/(m_i-1)!$,
which is finite since $W(m)\le k^Mq_{\max}^{M-1}$. In the continuum limit the
binomial coefficients of Lemma~\ref{lem:runs} become the powers
$t_i^{m_i-1}/(m_i-1)!$, and Lemma~\ref{lem:discrete} below shows that the discrete
law is close to the following density for $T\le N^{1/2}$.

\begin{lemma}[the continuum density]\label{lem:continuum}
Let $(Y_s)_{s\ge0}$ be the continuous-time chain with generator $Q$ and
$Y_0\sim\mu$, let $\mathcal O_i=\int_0^T\mathbf 1\{Y_s=i\}\,ds$, and let $E_F$
be the event that $Y$ stays in $F$ during $[0,T]$ and visits every state of $F$.
Then $\Prob(E_F,\ \mathcal O\in\cdot\,)$ has the density
$e^{-\langle q,t\rangle}H(t)$ in $(t_2,\dots,t_k)$ on
$\{t\in(0,\infty)^F:\sum_it_i=T\}$, and $\Prob(E_F,\ \mathcal O/T\in\cdot\,)$
has the density
\[
f_T(\alpha)=T^{k-1}e^{-T\langle q,\alpha\rangle}H(T\alpha)
\]
in $(\alpha_2,\dots,\alpha_k)$. For $n$ with support $F$ and $t=hn$,
$N^{-(k-1)}f_T(n/N)=h^{k-1}e^{-\langle q,t\rangle}H(t)$. For $F=\{i\}$,
$f_T=\Prob(E_F)=\mu_ie^{-q_iT}$.
\end{lemma}

The density is Good's formula~\cite[(10)]{good1961} for the killed chain,
whose diagonal keeps the full exit rates $q_i$. The proof checks the exponents
$m_i-1$.

\begin{proof}
Up to a null set, $E_F$ is the disjoint union of the events $E_c$ that $Y$
visits exactly $c_1,\dots,c_M$ in this order during $[0,T]$, over the words $c$
in the sums $W(m)$. Fix $c$, and let $s_j$ be the time that $Y$ spends in its
$j$-th run, so that $s_1+\dots+s_M=T$. For $j<M$ the run has the density
$q_{c_j}e^{-q_{c_j}s_j}$ and ends with a jump to $c_{j+1}$ of probability
$q_{c_jc_{j+1}}/q_{c_j}$, and the last run survives with probability
$e^{-q_{c_M}s_M}$. So on $\{s_j>0\ \forall j\}$ the event $E_c$ has the density
$\mu_{c_1}\prod_{j<M}q_{c_jc_{j+1}}\,e^{-\langle q,\mathcal O\rangle}$ in
$(s_1,\dots,s_{M-1})$, where $\mathcal O_i$ is the sum of the $s_j$ with $c_j=i$.
For each $i\ne c_M$ we replace the last holding time in state $i$ by
$\mathcal O_i$, a triangular change of variables with unit diagonal. With
$\mathcal O=t$ fixed, the $m_i-1$ other holding times in state $i$ are positive
with sum less than $t_i$, so they range over a simplex of volume
$t_i^{m_i-1}/(m_i-1)!$. Integrating them out gives the density
$\mu_{c_1}\prod_{j<M}q_{c_jc_{j+1}}\,e^{-\langle q,t\rangle}\prod_it_i^{m_i-1}/(m_i-1)!$
in the coordinates $t_i$ with $i\ne c_M$, and any $k-1$ of the coordinates give
the same Lebesgue measure on $\{\sum_it_i=T\}$. Summing over $c$ and $m$ gives
$e^{-\langle q,t\rangle}H(t)$, and $t=T\alpha$ has Jacobian $T^{k-1}$. For
$F=\{i\}$ the event $E_F$ says that $Y$ starts at $i$ and does not jump before
time $T$.
\end{proof}

\begin{example}[two models]\label{ex:models}
(i) For $d=2$, $q_{12}=q_{21}=1$ and uniform $\mu$ the chain gives the
directions of the persistent random walk on $\Z$ of~\cite{paperA}, which
repeats its previous step with probability $p=1-h$. So $T=N(1-p)$, and the
faces $\{1\}$ and $\{2\}$ are the endpoints. (ii) For $Q=J-dI$, with $J$ the all-ones matrix, the chain
switches to each other state with probability $h$ and keeps its state with
probability $p=1-(d-1)h$. With uniform $\mu$ it is the model of~\cite{paperB},
whose switch probability $r$ is our $h$. A word $x_1\cdots x_N$ then has
probability $\frac1dh^{N-1}(p/h)^{\#\{s<N:\,x_s=x_{s+1}\}}$, so the path law is the
one-dimensional $d$-state Potts chain with free boundaries~\cite{wu1982}. Here $q_i=d-1$ for
every $i$, so Lemma~\ref{lem:runs} reads
\begin{equation}\label{eq:runform}
 P_N(n)=p^{N-1}\sum_{a\in F}\mu_aS^{(a)}_n(u),\qquad u=\frac hp,
\end{equation}
for $p>0$ and $\supp n=F$, where $S^{(a)}_n(u)$ is the sum of $u^j$ over the words with
the counts $n$ that start with $a$ and have $j$ switches. Indeed such a word has
probability $\mu_ap^{N-1-j}h^j=\mu_ap^{N-1}u^j$. In both models $Q$ has symmetric
support and $\mu$ full support, so all admissible faces satisfy \textup{(H)} by
Lemma~\ref{lem:admissible}.
\end{example}
Figure~\ref{fig:simplex} shows the exact support competition for the three-state path.
There $\lambda_F$ is $1$ for an end, $(3-\sqrt5)/2$ for a pair of neighbors and
$0$ for all three states, so to first order the mode moves at
$\tau=(1+\sqrt5)/2$ and at $\tau=(3+\sqrt5)/2$, that is, at $T=7.3$ and
$T=11.8$ for $N=90$. At $N=90$ the mode actually moves near $T=4.9$ and $T=8.5$. Corrections of
this size are the subject of Section~\ref{sec:secondorder}.

\begin{figure}[t]
\centering
\includegraphics[width=0.9\textwidth]{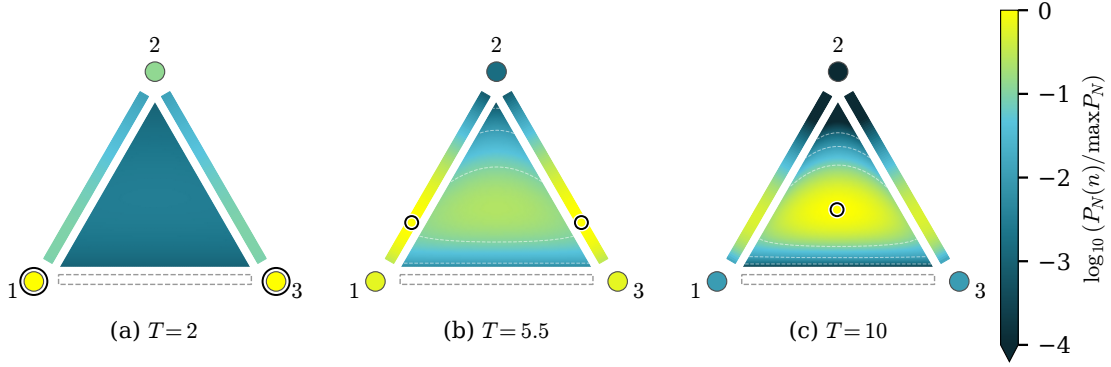}
\caption{The exact law of the composition for the path $1$--$2$--$3$ (rate $1$ on
each edge, uniform start, $N=90$) on the simplex of visit fractions, with the
faces drawn apart. The triangle holds the compositions with three states, each
band those with two, and each disk a vertex. The color is the probability
relative to the mode, the thin lines are spaced by a factor of ten, and the
rings mark the modes. The mode is (a) at an end of the path, (b) on a pair of
neighbors and (c) inside. The dashed band $\{1,3\}$ has probability zero}
\label{fig:simplex}
\end{figure}

\section{The killed functional}\label{sec:dv}

Every rate in this paper is the restriction of one functional, the Donsker and
Varadhan functional of the chain, or of its discrete version of
Section~\ref{sec:discrete-rate}.

\subsection{The restriction identity}\label{sec:restriction}

For a probability vector $\alpha$ on $[d]$ let
\begin{equation}\label{eq:JV}
 J_V(\alpha)=\sup_{u\in(0,\infty)^d}\sum_{i\in[d]}\alpha_i\Bigl(-\frac{(Qu)_i}{u_i}\Bigr)
 =\sup_{u\in(0,\infty)^d}\sum_{i\in[d]}\alpha_i\Bigl(q_i-\sum_{j\ne i}q_{ij}\frac{u_j}{u_i}\Bigr),
\end{equation}
the Donsker and Varadhan rate of the empirical measure of the continuous-time
chain with generator $Q$~\cite{dv1975}. For a face $F$ and $\alpha\in\Delta_F$ let
\[
 J_F(\alpha)=\sup_{u\in(0,\infty)^F}\sum_{i\in F}\alpha_i\Bigl(-\frac{(Q_Fu)_i}{u_i}\Bigr)
 =\sup_{u\in(0,\infty)^F}\sum_{i\in F}\alpha_i\Bigl(q_i-\sum_{j\in F\setminus\{i\}}q_{ij}\frac{u_j}{u_i}\Bigr),
\]
the same functional for the chain killed on leaving $F$. Let $q_i^{\rm out}(F)=\sum_{j\notin F}q_{ij}$. Taking $u=\ones$
gives $J_F(\alpha)\ge\sum_i\alpha_iq_i^{\rm out}(F)$, and dropping the terms
with $q_{ij}$ gives $J_F(\alpha)\le\sum_i\alpha_iq_i$.

\begin{proposition}[the restriction identity]\label{prop:restriction}
For every face $F$ and every $\alpha\in\Delta_F$, $J_F(\alpha)=J_V(\alpha)$. Hence
$J_F=J_{F'}$ on $\Delta_{F'}$ for every nonempty $F'\subseteq F$.
\end{proposition}

\begin{proof}
Fix $\alpha\in\Delta_F$ and $u\in(0,\infty)^d$. The terms of~\eqref{eq:JV} with
$i\notin F$ vanish, since $\alpha_i=0$ there. For $i\in F$ the terms with
$j\notin F$ are nonnegative and are subtracted, so
\[
 \sum_{i\in F}\alpha_i\Bigl(q_i-\sum_{j\ne i}q_{ij}\frac{u_j}{u_i}\Bigr)
 \le\sum_{i\in F}\alpha_i\Bigl(q_i-\sum_{j\in F\setminus\{i\}}q_{ij}\frac{u_j}{u_i}\Bigr)
 \le J_F(\alpha).
\]
Hence $J_V(\alpha)\le J_F(\alpha)$. Conversely, let $w\in(0,\infty)^F$ and let $u$
equal $w$ on $F$ and $\varepsilon$ off $F$. The terms with $j\notin F$ are then
$-q_{ij}\varepsilon/w_i$, which tend to zero as $\varepsilon\downarrow0$, and
the supremum over $w$ gives $J_V(\alpha)\ge J_F(\alpha)$. The second claim is the
first for $F$ and for $F'$.
\end{proof}

\subsection{The variational formula}\label{sec:spectral-geometry}

For a face $F$ and a field $v\in\R^F$ let
\begin{equation}\label{eq:lambdaFv}
 \lambda_F(v)=-\varrho\bigl(Q_F+\operatorname{diag}v\bigr),
\end{equation}
so that $\lambda_F(0)=\lambda_F$. If $F$ is irreducible, let $r$ and $l$ be
positive right and left Perron vectors of $Q_F+\operatorname{diag}v$ with
$l\cdot r=1$, and put $\alpha^*_F(v)=l\circ r$, so that $\alpha^*_F(0)=\alpha^*_F$.
Since $\sum_i\alpha_i=1$,
\begin{equation}\label{eq:Jfield}
 J_F(\alpha)-v\cdot\alpha=\sup_{u\in(0,\infty)^F}\sum_{i\in F}\alpha_i
 \Bigl(-\frac{((Q_F+\operatorname{diag}v)u)_i}{u_i}\Bigr).
\end{equation}
Hence a field replaces the diagonal $q_i$ by $q_i-v_i$ and changes nothing else.

\begin{lemma}[monotonicity in the face]\label{lem:face-monotone}
If $F\subsetneq G$ are faces and $Q_G$ is irreducible, then
$\lambda_G(v_G)<\lambda_F(v_F)$ for every field $v\in\R^G$. Now let $Q$ be
irreducible. Then $\lambda_V(v)<\lambda_F(v_F)$ for every proper face $F$, and
for every fixed $v\in\R^d$ there is $\gamma(v)>0$ with
$\lambda_F(v_F)\ge\lambda_V(v)+\gamma(v)$ for all proper $F$. So the score
$\Phi_{\tau,v}(F)=|F|-1+\tau\lambda_F(v_F)$ is least at $V$ for
$\tau>(d-1)/\gamma(v)$.
\end{lemma}

Note that the last two claims need $Q$ irreducible. For $d=2$, $q_{12}=1$ and $q_{21}=0$
we have $\lambda_V=0=\lambda_{\{2\}}$.

\begin{proof}
Choose $c$ with $B=Q_G+\operatorname{diag}v_G+cI\ge0$. Its principal submatrix on
$F$ is $Q_F+\operatorname{diag}v_F+cI$, since the diagonal of a killed generator
keeps the full exit rates. Let $\widetilde B$ be this submatrix padded with zeros
to $G\times G$. Then $0\le\widetilde B\le B$, and $\widetilde B\ne B$ because the
irreducible matrix $B$ has a positive entry in a row of $F$ and a column off $F$.
So $\varrho(\widetilde B)<\varrho(B)$~\cite[Ch.~2]{bermanplemmons1994},
and $\varrho(\widetilde B)$ is the Perron root of the submatrix on $F$, since the
padding adds only the eigenvalue $0$. This holds whether or not $F$ is
irreducible, and subtracting $c$ gives the first claim. If $Q$ is irreducible,
the first claim with $G=V$ gives the second. For fixed $v$ there are
finitely many proper faces, so $\gamma(v)$, the least of the numbers
$\lambda_F(v_F)-\lambda_V(v)$, is positive. Finally
$\Phi_{\tau,v}(F)-\Phi_{\tau,v}(V)=|F|-d+\tau(\lambda_F(v_F)-\lambda_V(v))\ge-(d-1)+\tau\gamma(v)$.
\end{proof}

\begin{proposition}[the rate on a face]\label{prop:rate}
Let $F$ be irreducible with $k=\lvert F\rvert\ge2$, and let $v\in\R^F$.
\begin{enumerate}[(a)]
\item On $\Delta_F$, $J_F(\alpha)-v\cdot\alpha\ge\lambda_F(v)$, with equality exactly at
$\alpha=\alpha^*_F(v)$. In particular $J_F\ge\lambda_F$, with equality exactly at
$\alpha^*_F=l_F\circ r_F$.
\item The rate $J_F$ is convex on $\Delta_F$ and strictly convex on $\Delta_F^\circ$. For
$\alpha\in\Delta_F^\circ$ the supremum defining $J_F(\alpha)$ is attained at a
vector $u=r_\alpha$, which is unique up to a positive factor, and $\alpha$ is
the stationary law of the generator on $F$ with off-diagonal entries
$q_{ij}r_{\alpha,j}/r_{\alpha,i}$.
\item If there is $\nu\in(0,\infty)^F$ with $\nu_iq_{ij}=\nu_jq_{ji}$ for all
$i,j\in F$, then for every $\alpha\in\Delta_F$
\[
 J_F(\alpha)=\sum_{\{i,j\}\subseteq F,\ i\ne j}\bigl(\sqrt{\alpha_iq_{ij}}-\sqrt{\alpha_jq_{ji}}\bigr)^2
 +\sum_{i\in F}\alpha_iq_i^{\rm out}(F).
\]
\item For every nonempty face $F'\subsetneq F$ and every
$\alpha\in\Delta_{F'}$, $J_F(\alpha)-v\cdot\alpha=J_{F'}(\alpha)-v\cdot\alpha\ge\lambda_{F'}(v_{F'})>\lambda_F(v)$.
\end{enumerate}
\end{proposition}

For symmetric $Q_F$, $l_F=r_F$ and $\alpha^*_F$ is the entrywise square of the
Perron vector normalized in $\ell^2$. Without reversibility the proof of (c)
still gives the formula as an upper bound.

\begin{proof}
We prove this through the matrix $B$. Put $B=Q_F+\operatorname{diag}v$, an irreducible Metzler matrix with the switching
graph $G_F$ and $\lambda_F(v)=-\varrho(B)$.

\emph{Setup.} With $u=e^\psi$, \eqref{eq:Jfield} reads
$J_F(\alpha)-v\cdot\alpha=\sup_\psi\Psi_\alpha(\psi)$ with
\[
\Psi_\alpha(\psi)=\sum_i\alpha_i(q_i-v_i)-\sum_{i\ne j}\alpha_iq_{ij}e^{\psi_j-\psi_i},
\]
where $i,j$ run over $F$. It is concave in $\psi$, affine in $\alpha$, and
unchanged when a constant is added to $\psi$. Let $G^\psi$ be the conservative
generator on $F$ with off-diagonal entries $q_{ij}e^{\psi_j-\psi_i}$. It has the
graph $G_F$, so it has exactly one stationary law, and it does not involve $v$.
Since $\nabla_\psi\Psi_\alpha=-\alpha^\top G^\psi$ and a differentiable concave
function attains its maximum exactly at its critical points, $\psi$ attains the
supremum if and only if $\alpha$ is the stationary law of $G^\psi$. We call this
the criterion.

\emph{A lower bound on every face.} Let $F'$ be any face, irreducible or not, and
$B'=Q_{F'}+\operatorname{diag}v_{F'}$. We show
$J_{F'}(\alpha)-v\cdot\alpha\ge-\varrho(B')$ on $\Delta_{F'}$. Let
$s>\varrho(B')$ and write $B'=B_0-cI$ with $B_0\ge0$. Then $s+c>\varrho(B_0)$, so
$u=(sI-B')^{-1}\ones=\sum_{\ell\ge0}(s+c)^{-\ell-1}B_0^\ell\ones$ converges and is
positive. Since $(sI-B')u=\ones$, we get $-(B'u)_i/u_i=-s+1/u_i>-s$ at every
state, and letting $s\downarrow\varrho(B')$ gives the claim.

\emph{The value at $\alpha^*_F(v)$.} Take $\psi=\log r$ and $R=\operatorname{diag}r$.
Then $G^\psi=R^{-1}(B-\varrho(B)I)R$, since the off-diagonal entries agree and
$Br=\varrho(B)r$. So $(l\circ r)^\top G^\psi=l^\top(B-\varrho(B)I)R=0$, and by the
criterion $\psi$ attains the supremum at $\alpha^*_F(v)$. Every ratio
$-(Br)_i/r_i$ equals $\lambda_F(v)$, so the value is $\lambda_F(v)$.

(b) Convexity holds because $J_F$ is a supremum of affine functions. Let
$\alpha\in\Delta_F^\circ$ and take $v=0$.

\emph{The supremum is attained.} Normalize $\min_i\psi_i=0$ and put
$\bar\psi=\max_i\psi_i$. Let $a$ be the least positive rate $q_{ij}$ in $F$. A
shortest path in $G_F$ from a state where $\psi=0$ to a state where
$\psi=\bar\psi$ has at most $k-1$ edges, so one of its edges $i\to j$ has
$\psi_j-\psi_i\ge\bar\psi/(k-1)$. Keeping only this edge gives
$\Psi_\alpha(\psi)\le\sum_i\alpha_iq_i-a(\min_i\alpha_i)e^{\bar\psi/(k-1)}$, which
tends to $-\infty$ as $\bar\psi\to\infty$. Hence the continuous function
$\Psi_\alpha$ attains its supremum on a bounded set. Write $r_\alpha=e^\psi$ for a
maximizer. The last claim of (b) is the criterion.

\emph{The maximizer is unique up to a factor.} If $\psi$ and $\psi+\delta$ both
attain the supremum, $\Psi_\alpha$ is constant on the segment between them. There
it is a constant minus a sum of the convex functions
$\alpha_iq_{ij}e^{\psi_j-\psi_i+s(\delta_j-\delta_i)}$, so each of them is affine
in $s$, and $\delta_j=\delta_i$ on every edge of $G_F$. As $G_F$ is strongly
connected, $\delta$ is constant.

\emph{Strict convexity.} Let $\alpha^0\ne\alpha^1$ lie in $\Delta_F^\circ$, put
$\alpha^s=(1-s)\alpha^0+s\alpha^1$ with $0<s<1$, and let $\psi$ attain
$J_F(\alpha^s)$. Since $\Psi_\alpha(\psi)$ is affine in $\alpha$,
$J_F(\alpha^s)=(1-s)\Psi_{\alpha^0}(\psi)+s\Psi_{\alpha^1}(\psi)\le(1-s)J_F(\alpha^0)+sJ_F(\alpha^1)$.
Equality would make $\alpha^0$ and $\alpha^1$ both stationary for $G^\psi$ by the
criterion, which is impossible.

(d) The equality is Proposition~\ref{prop:restriction}. The lower bound on $F'$
gives $J_{F'}(\alpha)-v\cdot\alpha\ge\lambda_{F'}(v_{F'})$, and
Lemma~\ref{lem:face-monotone} with $G=F$ gives $\lambda_{F'}(v_{F'})>\lambda_F(v)$.

(a) The lower bound and the value at $\alpha^*_F(v)$ are proved above. Let
$\alpha\ne\alpha^*_F(v)$. If $\alpha$ lies on the boundary of $\Delta_F$, then (d)
with $F'=\supp\alpha$ gives $J_F(\alpha)-v\cdot\alpha>\lambda_F(v)$. If
$\alpha\in\Delta_F^\circ$ and equality held, strict convexity would give a value
below $\lambda_F(v)$ at $\frac12(\alpha+\alpha^*_F(v))$.

(c) Since $q_i=\sum_{j\in F\setminus\{i\}}q_{ij}+q_i^{\rm out}(F)$, the function
inside the supremum is $\sum_i\alpha_iq_i^{\rm out}(F)$ plus, for each pair
$\{i,j\}$, the term $\alpha_iq_{ij}(1-u_j/u_i)+\alpha_jq_{ji}(1-u_i/u_j)$. By the
AM--GM inequality this term is at most
$(\sqrt{\alpha_iq_{ij}}-\sqrt{\alpha_jq_{ji}})^2$, which gives the upper bound. For
$\alpha\in\Delta_F^\circ$ the vector $u_i=\sqrt{\alpha_i/\nu_i}$ gives
$\alpha_iq_{ij}u_j/u_i=\alpha_jq_{ji}u_i/u_j$ by $\nu_iq_{ij}=\nu_jq_{ji}$, so there is
equality for every pair at once. If some $\alpha_i$ vanish, take $u_i=\varepsilon$
for these $i$ and let $\varepsilon\to0$. A pair with $\alpha_i=0<\alpha_j$ then has
the term $\alpha_jq_{ji}(1-\varepsilon/u_j)\to\alpha_jq_{ji}$, which is its bound, and
a pair with $\alpha_i=\alpha_j=0$ has the term $0$.
\end{proof}

\begin{proposition}[the variational formula]\label{prop:variational}
Let $F$ be irreducible. For every $v\in\R^F$,
\[
 \lambda_F(v)=\min_{\alpha\in\Delta_F}\bigl(J_F(\alpha)-v\cdot\alpha\bigr),
\]
and the minimum is attained only at $\alpha^*_F(v)$. Equivalently, $-\lambda_F(v)$
is the Legendre transform of $J_F$ at $v$, and $J_F$ is the Legendre transform of
$v\mapsto-\lambda_F(v)$ on $\Delta_F^\circ$. The map $v\mapsto\lambda_F(v)$ is
concave and real-analytic, and $\lambda_F(v+c\ones)=\lambda_F(v)-c$. If $Q_F$ is
reversible with respect to $\nu$, let $S_F=D_{\sqrt\nu}(-Q_F)D_{\sqrt\nu}^{-1}$, where $D_x=\operatorname{diag}x$.
Then with $\psi_i=\sqrt{\alpha_i}$,
\[
 J_F(\alpha)-v\cdot\alpha=\sum_{i\in F}(q_i-v_i)\alpha_i-\sum_{i\ne j\in F}\sqrt{q_{ij}q_{ji}}\,\psi_i\psi_j
\]
is the Rayleigh quotient of the symmetric matrix $S_F-\operatorname{diag}v$ at the
unit vector $\psi$, and $\lambda_F(v)=\lambda_{\min}(S_F-\operatorname{diag}v)$.
\end{proposition}

The formula is the finite form of the Donsker and Varadhan formula for the
principal eigenvalue~\cite{dv1975}, close to the characterization of Collatz and
Wielandt~\cite[Ch.~2]{bermanplemmons1994}.

\begin{proof}
The formula and the unique minimizer are Proposition~\ref{prop:rate}(a). As a
minimum of affine functions of $v$, $\lambda_F(v)$ is concave. It is
real-analytic because the Perron root of an irreducible Metzler matrix is a simple
eigenvalue, and the shift is clear.

\emph{Legendre duality.} The formula says
$-\lambda_F(v)=\max_\alpha(v\cdot\alpha-J_F(\alpha))$. Conversely
$J_F(\alpha)\ge v\cdot\alpha+\lambda_F(v)$ for all $\alpha$ and $v$. Let
$\alpha\in\Delta_F^\circ$, take $r_\alpha$ from Proposition~\ref{prop:rate}(b), and
put $v=q-A_Fr_\alpha/r_\alpha$, with the quotient taken entrywise. Then $Q_F+\operatorname{diag}v$ has the
off-diagonal entries of $Q_F$ and the diagonal $-A_Fr_\alpha/r_\alpha$, so
$r_\alpha>0$ is a null vector of it. Only the Perron root has a positive
eigenvector, so $\lambda_F(v)=0$. Also $J_F(\alpha)=\langle q-A_Fr_\alpha/r_\alpha,\alpha\rangle
=v\cdot\alpha$, so equality holds at this $v$.

\emph{The reversible form.} Reversibility gives
$\sqrt{\nu_i}\,q_{ij}/\sqrt{\nu_j}=\sqrt{q_{ij}q_{ji}}$, so $S_F$ is symmetric with
diagonal $q_i$ and off-diagonal entries $-\sqrt{q_{ij}q_{ji}}$. Expanding the
squares in Proposition~\ref{prop:rate}(c) gives
$J_F(\alpha)=\langle\psi,S_F\psi\rangle$, and subtracting
$v\cdot\alpha=\langle\psi,\operatorname{diag}(v)\psi\rangle$ gives the display. Every
unit vector $\psi\ge0$ arises from $\alpha=\psi\circ\psi$, and the off-diagonal
entries of $S_F-\operatorname{diag}v$ are nonpositive, so passing to the entrywise
modulus does not increase the quadratic form. Hence the minimum over $\Delta_F$ is
$\lambda_{\min}(S_F-\operatorname{diag}v)$.
\end{proof}

\subsection{The functional of the discrete chain}\label{sec:discrete-rate}

When $T$ is a fixed fraction of $N$ the counts follow the functional of the
discrete chain itself. Let $0<h<1/q_{\max}$ and $p_i=1-hq_i$, so that $I+hQ$ is
a stochastic matrix with positive diagonal. For a face $F$ and
$\alpha\in\Delta_F$ let
\[
 J^{(h)}_F(\alpha)=\frac1h\sup_{u\in(0,\infty)^F}\sum_{i\in F}\alpha_i
 \log\frac{u_i}{((I+hQ_F)u)_i},
\]
and put $J^{(0)}_F=J_F$. For $F=V$, $hJ^{(h)}_V$ is the Donsker and Varadhan rate
of the empirical measure of the chain with transition matrix
$I+hQ$~\cite{dv1975}. For $\psi\in\R^F$ put
$\theta_i(\psi)=\sum_{j\in F\setminus\{i\}}q_{ij}e^{\psi_j-\psi_i}$, and let
$L_h(y)=\int_0^1y\,ds/(1-shy)$, which is $-\log(1-hy)/h$ for $h>0$ and $y$ for
$h=0$. With $u=e^\psi$ we have $((I+hQ_F)u)_i/u_i=p_i+h\theta_i(\psi)$, so
\[
 J^{(h)}_F(\alpha)=\sup_{\psi\in\R^F}G_h(\alpha,\psi),\qquad
 G_h(\alpha,\psi)=\sum_{i\in F}\alpha_iL_h\bigl(q_i-\theta_i(\psi)\bigr),
\]
and at $h=0$ this is the supremum that defines $J_F$. Since $L_h(y)\ge y$, we
have $J^{(h)}_F\ge J_F$. We say that a function on $[0,h_*]\times U$ is
real-analytic up to $h=0$ if near each point it is the restriction of a
real-analytic function on an open set.

\begin{proposition}[the discrete functional]\label{prop:discrete-rate}
Fix $h_*$ with $h_*q_{\max}<1$.
\begin{enumerate}[(a)]
\item For every face $F$, every $0<h<1/q_{\max}$ and every $\alpha\in\Delta_F$,
$J^{(h)}_F(\alpha)=J^{(h)}_V(\alpha)$. Hence $J^{(h)}_F=J^{(h)}_{F'}$ on
$\Delta_{F'}$ for every nonempty $F'\subseteq F$.
\item Let $F$ be irreducible with $k\ge2$, and let
$(h,\alpha)\in[0,h_*]\times\Delta_F^\circ$. The supremum defining
$J^{(h)}_F(\alpha)$ is attained at exactly one $\psi=\psi^{(h)}(\alpha)$ with
$\psi_1=0$, and there the Hessian of $G_h(\alpha,\cdot)$ in
$(\psi_2,\dots,\psi_k)$ is negative definite. The maps $\psi^{(h)}(\alpha)$,
$\theta^{(h)}(\alpha)=\theta(\psi^{(h)}(\alpha))$ and $J^{(h)}_F(\alpha)$ are
real-analytic on $[0,h_*]\times\Delta_F^\circ$ up to $h=0$, and
$\theta^{(0)}(\alpha)=A_Fr_\alpha/r_\alpha$ with $r_\alpha$ from
Proposition~\ref{prop:rate}(b). The function $J^{(h)}_F$ is convex on $\Delta_F$
and strictly convex on $\Delta_F^\circ$.
\item Let $F$ be irreducible with $k\ge2$, $0\le h\le h_*$,
$\alpha\in\Delta_F^\circ$ and $\psi\in\R^F$, and put $\theta=\theta(\psi)$. Then
$\psi$ attains the supremum if and only if the vector
$(\alpha_i/(p_i+h\theta_i))_{i\in F}$ is proportional to $l\circ e^\psi$, where
$l$ is a positive left null vector of $A_F-\operatorname{diag}\theta$.
\end{enumerate}
\end{proposition}

\begin{proof}
We prove (a), then (b), then (c). (a) As in the proof of Proposition~\ref{prop:restriction}, the terms
$hq_{ij}u_j$ with $j\notin F$ only increase $((I+hQ)u)_i$, which gives
$J^{(h)}_V\le J^{(h)}_F$ on $\Delta_F$, and $u=\varepsilon$ off $F$ with
$\varepsilon\downarrow0$ gives the converse, since the diagonal of $I+hQ_F$ keeps
the full exit rates. The second claim is the first for $F$ and for $F'$.

(c) We have $\partial\theta_i/\partial\psi_m=q_{im}e^{\psi_m-\psi_i}$ for
$m\ne i$ and $\partial\theta_i/\partial\psi_i=-\theta_i$, and
$L_h'(y)=1/(1-hy)$. So with $w_i=\alpha_i/(p_i+h\theta_i)$ the gradient of
$G_h(\alpha,\cdot)$ is $-w^\top G^\psi$, where $G^\psi$ is the conservative
generator of the proof of Proposition~\ref{prop:rate}, with off-diagonal entries
$q_{ij}e^{\psi_j-\psi_i}$. The function $G_h(\alpha,\cdot)$ is concave by the
proof of (b) below. Hence $\psi$ attains the supremum if and only if $w$ is
proportional to the stationary law of $G^\psi$. With
$U=\operatorname{diag}(e^\psi)$ we have
$G^\psi=U^{-1}(A_F-\operatorname{diag}\theta)U$, whose stationary law is
proportional to $l\circ e^\psi$.

(b) \emph{Concavity.} At $h=0$ the Hessian of $-G_0$ in a direction $\delta$ is
$\sum_i\alpha_i\sum_{j\ne i}q_{ij}e^{\psi_j-\psi_i}(\delta_j-\delta_i)^2$. For
$h>0$,
\[
 -hG_h(\alpha,\psi)=\sum_{i\in F}\alpha_i
 \log\Bigl(p_i+\sum_{j\in F\setminus\{i\}}hq_{ij}e^{\psi_j-\psi_i}\Bigr).
\]
The $i$-th term is $\alpha_i$ times the logarithm of a sum of exponentials of
the affine functions $0$ and $\psi_j-\psi_i$, with the positive weights $p_i$ and
$hq_{ij}$. Its second derivative in a direction $\delta$ is $\alpha_i$ times the
variance of the values $0$ and $\delta_j-\delta_i$ under the weights
proportional to $p_i$ and $hq_{ij}e^{\psi_j-\psi_i}$. As $p_i>0$, it vanishes
only if $\alpha_i=0$ or $\delta_j=\delta_i$ on every edge $i\to j$ of $G_F$. So
for every $h\in[0,h_*]$ the function $G_h(\alpha,\cdot)$ is concave, and for
$\alpha\in\Delta_F^\circ$ its Hessian is negative definite on $\{\delta_1=0\}$,
since $G_F$ is strongly connected. This is where $h_*q_{\max}<1$ is used.

\emph{The supremum is attained.} Let $\psi_1=0$ and let $\bar\psi$ be the
largest difference of two entries of $\psi$. As in the proof of
Proposition~\ref{prop:rate}(b), some edge $i\to j$ has
$\psi_j-\psi_i\ge\bar\psi/(k-1)$, so $\theta_i(\psi)\ge ae^{\bar\psi/(k-1)}$ with
$a$ the least positive rate in $F$. Since $L_h$ is increasing and tends to
$-\infty$ at $-\infty$, the term of $i$ tends to $-\infty$ as $\bar\psi\to\infty$,
while every term is at most $\alpha_iL_h(q_i)\le q_{\max}/p_{\min}$, with
$p_{\min}=1-h_*q_{\max}$. Hence the supremum is attained, and by strict concavity at
exactly one $\psi$ with $\psi_1=0$.

\emph{Analyticity.} The function $L_h(y)$ is real-analytic in $(h,y)$ where
$1-hy>0$, including $h=0$ and small negative $h$, and $1-h(q_i-\theta_i)\ge p_i>0$
for $h\in[0,h_*]$. So $G$ is real-analytic near every point $(h,\alpha,\psi)$ with
$h\in[0,h_*]$. At the maximizer the gradient in $(\psi_2,\dots,\psi_k)$ vanishes
and the Hessian is invertible. The implicit function theorem gives a real-analytic
solution near $(h,\alpha)$, which is $\psi^{(h)}(\alpha)$ for $h\ge0$ by
uniqueness. Hence $\theta^{(h)}(\alpha)$ and
$J^{(h)}_F(\alpha)=G_h(\alpha,\psi^{(h)}(\alpha))$ are real-analytic up to
$h=0$. At $h=0$ the maximizer is $\log r_\alpha$, normalized.

\emph{Convexity.} As a supremum of affine functions, $J^{(h)}_F$ is convex. Let
$\alpha^0\ne\alpha^1$ lie in $\Delta_F^\circ$, let $0<s<1$, and let $\psi$ attain
$J^{(h)}_F(\alpha^s)$ at $\alpha^s=(1-s)\alpha^0+s\alpha^1$. Equality in
$J^{(h)}_F(\alpha^s)\le(1-s)J^{(h)}_F(\alpha^0)+sJ^{(h)}_F(\alpha^1)$ would make
$\psi$ attain the supremum at $\alpha^0$ and at $\alpha^1$. By (c) both would then
be proportional to $((p_i+h\theta_i)l_ie^{\psi_i})_{i\in F}$, and as probability
vectors they would be equal.
\end{proof}

For a field $v\in\R^F$, $J^{(h)}_F(\alpha)-v\cdot\alpha$ is $1/h$ times the
supremum of $\sum_i\alpha_i\log(u_i/(Mu)_i)$ with
$M=e^{h\operatorname{diag}v}(I+hQ_F)$, since $-hv_i\alpha_i=\alpha_i\log e^{-hv_i}$.
For any face $F'$, any $v\in\R^{F'}$ and $0<h<1/q_{\max}$ we put
$\lambda^{(h)}_{F'}(v)=-\frac1h\log\varrho\bigl(e^{h\operatorname{diag}v}(I+hQ_{F'})\bigr)$,
which is defined since the matrix is nonnegative with a positive diagonal.

\begin{proposition}[the discrete exit rate]\label{prop:discrete-variational}
Let $F$ be irreducible with $k\ge2$, let $0<h<1/q_{\max}$ and $v\in\R^F$, and
put $M=e^{h\operatorname{diag}v}(I+hQ_F)$. Let $r$ and $l$ be positive right and
left Perron vectors of $M$ with $l\cdot r=1$, and put
$\alpha^{*(h)}_F(v)=l\circ r$.
\begin{enumerate}[(a)]
\item On $\Delta_F$, $J^{(h)}_F(\alpha)-v\cdot\alpha\ge\lambda^{(h)}_F(v)$, with
equality only at $\alpha^{*(h)}_F(v)$, which lies in $\Delta_F^\circ$. For every
nonempty face $F'\subsetneq F$ and every $\alpha\in\Delta_{F'}$,
$J^{(h)}_F(\alpha)-v\cdot\alpha\ge\lambda^{(h)}_{F'}(v_{F'})>\lambda^{(h)}_F(v)$.
\item At $v=0$, $\varrho(M)=1-h\lambda_F$ and $\alpha^{*(h)}_F(0)=\alpha^*_F$, so
$\lambda^{(h)}_F(0)=-\frac1h\log(1-h\lambda_F)$ for every $h$.
\item Fix $h_*$ with $h_*q_{\max}<1$, and put $\lambda^{(0)}_F(v)=\lambda_F(v)$ and
$\alpha^{*(0)}_F(v)=\alpha^*_F(v)$. Then $\lambda^{(h)}_F(v)$ and
$\alpha^{*(h)}_F(v)$ are real-analytic on $[0,h_*]\times\R^F$ up to $h=0$.
\end{enumerate}
\end{proposition}

\begin{proof}
We prove this through the matrix $B$. Put $B=(M-I)/h$. Its off-diagonal entries $e^{hv_i}q_{ij}$ are nonnegative and
positive exactly on the edges of $G_F$, so $B$ is an irreducible Metzler matrix.
Its entries are real-analytic in $(h,v)$ up to $h=0$, and $B=Q_F+\operatorname{diag}v$
at $h=0$. Let $\kappa=\varrho(B)$. Then $\varrho(M)=1+h\kappa$, and $M$ and $B$
have the same Perron vectors. As $\kappa$ is a simple eigenvalue of $B$, it and the
normalized Perron vectors are real-analytic in $(h,v)$ up to $h=0$, and so is
$\lambda^{(h)}_F(v)=-\kappa\int_0^1ds/(1+sh\kappa)$. This gives (c). At $v=0$ we
have $B=Q_F$ for every $h$, which gives (b).

(a) \emph{A lower bound on every face.} Let $F'\subseteq F$ be nonempty,
irreducible or not, and let $M'$ be the principal submatrix of $M$ on $F'$, which
is $e^{h\operatorname{diag}v_{F'}}(I+hQ_{F'})$ since the diagonal keeps the full
exit rates. For $\alpha\in\Delta_{F'}$, part (a) of
Proposition~\ref{prop:discrete-rate} gives
$J^{(h)}_F(\alpha)=J^{(h)}_{F'}(\alpha)$, so $J^{(h)}_F(\alpha)-v\cdot\alpha$ is
$1/h$ times the supremum of $\sum_{i\in F'}\alpha_i\log(u_i/(M'u)_i)$ over
$u\in(0,\infty)^{F'}$. Let $s>\varrho(M')$. As $M'\ge0$, the vector
$u=(sI-M')^{-1}\ones=\sum_{j\ge0}s^{-j-1}M'^j\ones$ is positive, and $M'u=su-\ones$.
Also $M'u>0$, since $M'$ has a positive diagonal. So $u_i/(M'u)_i>1/s$ for every
$i$, and letting $s\downarrow\varrho(M')$ gives
$J^{(h)}_F(\alpha)-v\cdot\alpha\ge\lambda^{(h)}_{F'}(v_{F'})$. If $F'\ne F$, then
$M'$ padded with zeros is at most $M$ and differs from it, since the irreducible
$M$ has a positive entry in a row of $F'$ and a column off $F'$. So
$\varrho(M')<\varrho(M)$~\cite[Ch.~2]{bermanplemmons1994}, which gives the second
claim. The case $F'=F$ is the lower bound.

\emph{Equality.} With $u=r$ every ratio $r_i/(Mr)_i$ is $1/\varrho(M)$. The
function $\psi\mapsto\sum_i\alpha_i(\psi_i-\log(Me^\psi)_i)$ is concave, and at
$\psi=\log r$ its derivative in $\psi_m$ is
$\alpha_m-\sum_i\alpha_iM_{im}r_m/(\varrho(M)r_i)$. For $\alpha=l\circ r$ this is
$\alpha_m-(l^\top M)_mr_m/\varrho(M)=0$. So $u=r$ attains the supremum at
$\alpha^{*(h)}_F(v)$, and equality holds there. This point lies in $\Delta_F^\circ$
since $r,l>0$.

\emph{Uniqueness.} On the boundary of $\Delta_F$ the inequality is strict by the
second claim. In $\Delta_F^\circ$, $J^{(h)}_F(\alpha)-v\cdot\alpha$ is strictly
convex by Proposition~\ref{prop:discrete-rate}(b), so it has only one minimizer.
\end{proof}

\section{First-order face selection}\label{sec:firstorder}

\subsection{An example with four states on a cycle}\label{sec:worked}

Take $d=4$ states on a cycle, with rate $1$ between neighbors and the uniform
start. Up to symmetry the admissible faces are a vertex, an adjacent pair, an
arc of three states and the whole cycle, with exit rates $2$, $1$, $2-\sqrt2$
and $0$. For the arc, $-Q_F=2I-A$ with $A$ the adjacency matrix of the path
$1$--$2$--$3$, whose largest eigenvalue is $\sqrt2$. The chain stays in a face $F$
up to time $T$ with probability about $e^{-T\lambda_F}$, and this mass is shared
by about $(N/\sqrt T)^{|F|-1}$ compositions. Hence
$M_F\approx N^{-(|F|-1)}T^{(|F|-1)/2}e^{-T\lambda_F}$, and with $T=\tau\log N$
this is $N^{-\Phi_\tau(F)}$ up to powers of $\log N$.
Figure~\ref{fig:C4worked}(a) plots the four exponents. The vertices win for
$\tau<1$, the pairs for $1<\tau<2$ and the whole cycle for $\tau>2$. The arcs
never win, since they beat the pairs only for $\tau>1+\sqrt2$, where the
cycle beats them. The whole cycle carries almost all the mass for every
$\tau>0$, and it carries the mode only for $\tau>2$, because it spreads its
mass over about $N^3$ compositions. Theorems~\ref{thm:first}
and~\ref{thm:facemax} make this rule exact, and for a pair
Corollary~\ref{cor:symface} gives the constant $1/\sqrt{2\pi}$ in front.

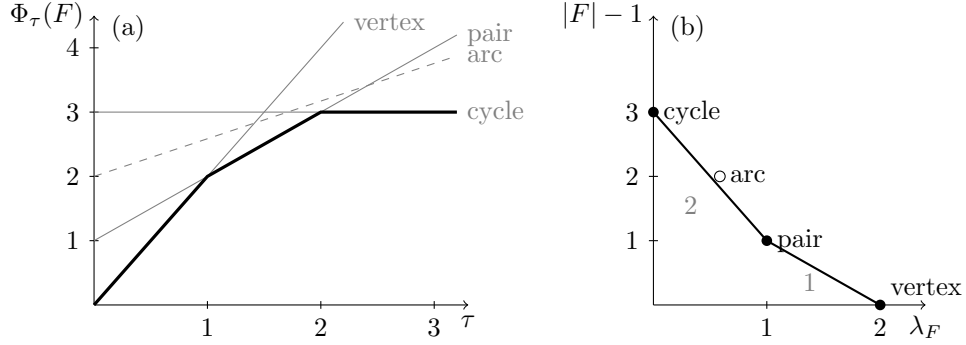
\begin{figure}[t]
\centering
\begin{tikzpicture}[x=1.5cm,y=0.85cm,font=\small]
 \draw[->] (0,0) -- (3.3,0) node[below] {$\tau$};
 \draw[->] (0,0) -- (0,4.5) node[left] {$\Phi_\tau(F)$};
 \foreach \x in {1,2,3} \draw (\x,0.06) -- (\x,-0.06) node[below] {$\x$};
 \foreach \y in {1,2,3,4} \draw (0.04,\y) -- (-0.04,\y) node[left] {$\y$};
 \draw[gray] (0,0) -- (2.2,4.4) node[anchor=west] {vertex};
 \draw[gray] (0,1) -- (3.2,4.2) node[anchor=west] {pair};
 \draw[gray,dashed] (0,2) -- (3.2,{2+0.5858*3.2}) node[anchor=west] {arc};
 \draw[gray] (0,3) -- (3.2,3) node[anchor=west] {cycle};
 \draw[very thick] (0,0) -- (1,2) -- (2,3) -- (3.2,3);
 \node at (0.3,4.3) {(a)};
 \begin{scope}[xshift=7.4cm]
 \draw[->] (0,0) -- (2.4,0) node[below] {$\lambda_F$};
 \draw[->] (0,0) -- (0,4.5) node[left] {$\lvert F\rvert-1$};
 \foreach \x in {1,2} \draw (\x,0.06) -- (\x,-0.06) node[below] {$\x$};
 \foreach \y in {1,2,3} \draw (0.04,\y) -- (-0.04,\y) node[left] {$\y$};
 \draw[thick] (2,0) -- (1,1) -- (0,3);
 \fill (2,0) circle (2pt) node[above right] {vertex};
 \fill (1,1) circle (2pt) node[right] {pair};
 \draw[fill=white] (0.5858,2) circle (2pt) node[right] {arc};
 \fill (0,3) circle (2pt) node[right] {cycle};
 \node[gray] at (1.38,0.35) {$1$};
 \node[gray] at (0.33,1.55) {$2$};
 \node at (0.3,4.3) {(b)};
 \end{scope}
\end{tikzpicture}
\caption{The 4-cycle. (a) The exponents $\Phi_\tau(F)=\tau\lambda_F+\lvert
F\rvert-1$ of the four kinds of faces against $\tau$, and their lower envelope
(bold). The mode sits on the face whose line is lowest, namely a vertex for
$\tau<1$, an adjacent pair for $1<\tau<2$ and the whole cycle for $\tau>2$. The
dashed line of the arc of three states is never lowest. (b) The same faces as points
$(\lambda_F,\lvert F\rvert-1)$, with the lower-left boundary of their convex
hull. The gray numbers are minus the slopes of its edges, the values of $\tau$
where the winner changes (Theorem~\ref{thm:hull}). The arc (open) lies above
the boundary}
\label{fig:C4worked}
\end{figure}

\subsection{Face selection}

We also need an upper bound on a whole face. This includes the region near its boundary, where the local law does not apply. For an irreducible face $F$ put
$V_F=\{\theta\in\R^F:\varrho(A_F-\operatorname{diag}\theta)=0\}$. Each
$\theta\in V_F$ gives an exponential bound with the linear rate
$\alpha\mapsto\langle q-\theta,\alpha\rangle$, and $J_F$ is the supremum of these
linear rates. Indeed, for $u\in(0,\infty)^F$ the entrywise quotient $\theta=A_Fu/u$
lies in $V_F$, since $u$ is a positive null vector of the irreducible
Metzler matrix $A_F-\operatorname{diag}\theta$ and only its Perron root has a
positive eigenvector. Then $\langle q-\theta,\alpha\rangle$ is the function of $u$
inside the supremum that defines $J_F(\alpha)$. Conversely every $\theta\in V_F$
arises in this way from a positive Perron vector of
$A_F-\operatorname{diag}\theta$. For a field $v\in\R^F$ the vector
$\theta=q-v-\lambda_F(v)\ones$ lies in $V_F$, since
$A_F-\operatorname{diag}\theta=Q_F+\operatorname{diag}v+\lambda_F(v)I$.

\begin{lemma}[a uniform upper bound]\label{lem:upper}
Let $F$ be irreducible with $k=\lvert F\rvert\ge2$, and let $\theta\in V_F$.
There is $C_\theta$, depending on $\theta$, $Q$ and $\mu$ and bounded for $\theta$
in compact subsets of $V_F$, such that for every $n\in\Z_{\ge0}^F$ with
$\supp n=F$ and every $h$ with $hq_{\max}\le\frac12$, the probability that
$X_1,\dots,X_{\lvert n\rvert}$ have composition $n$ is at most
$C_\theta h^{k-1}\exp\bigl(-h\langle n,q-\theta\rangle\bigr)$.
\end{lemma}

The proof is in Appendix~\ref{app:torus}. It writes the probability as an integral
over a torus and bounds the integrand by its value at the real point. The lemma
holds for a path of any length $\lvert n\rvert$, so it also applies to the blocks of
a path. With $\theta=q-v-\lambda_F(v)\ones$ and $\lvert n\rvert=N$ it gives
\begin{equation}\label{eq:upper-field}
 P_N(n)e^{Tv\cdot n/N}\le C_\theta h^{k-1}e^{-T\lambda_F(v)}
\end{equation}
for every $n$ with support $F$. The next theorem uses the local law and the
face maxima of Section~\ref{sec:secondorder}, which are proved there without
it. Its part (i) is a weak form of Theorem~\ref{thm:stratified}(b), and its part
(ii) is the only part that bounds the face maxima of reducible faces.

\begin{theorem}[first-order face selection]\label{thm:first}
Fix $Q$ and $\mu$.
\begin{enumerate}[(i)]
\item Let $F$ satisfy \textup{(H)}, let $k=\lvert F\rvert$, and let
$\mathcal K\subset\Delta_F^\circ$ be compact. There is $N_0$ such that,
uniformly for $N\ge N_0$, $1\le T\le N^{1/3}$ and $\supp n=F$ with
$n/N\in\mathcal K$,
\[
 \log P_N(n)=-(k-1)L-T\,J_F(n/N)+O_{\mathcal K}(1+\log T).
\]
\item Let $F$ be admissible with components $C_1,\dots,C_r$, and let $k^*$ be
the least size of a component $C_j$ with $\lambda_{C_j}=\lambda_F$. There are
$\kappa$, $T_0$ and $N_0$ such that for $N\ge N_0$ and $T_0\le T\le N^{1/3}$
\begin{align*}
 \log M_F&\ge-(k_F-1)L+(k_F-k^*)\log T-T\lambda_F-\kappa,\\
 \log M_F&\le-(k_F-1)L+(k_F-1)\log T-T\lambda_F+\kappa.
\end{align*}
If $F$ is not admissible, $M_F=0$.
\item Let $T=\tau L$ with $\tau>0$ fixed, and let $\mathcal F_\tau$ be the set
of admissible faces that minimize $\Phi_\tau$. Then
$\log\max_nP_N(n)=-L\min_F\Phi_\tau(F)+O(\log L)$, with the minimum over
admissible faces, and for large $N$ every global mode $\hat n$ has
$\supp\hat n\in\mathcal F_\tau$. This holds for every start $\mu$.
\item If in (iii) every face in $\mathcal F_\tau$ is irreducible, then every
global mode $\hat n$ satisfies $\hat n/N=\alpha^*_F+O(1/L)$ with
$F=\supp\hat n$. This is the case when $Q$ has symmetric support or $\mu$ has
full support (Theorem~\ref{thm:hull}(c)).
\end{enumerate}
\end{theorem}

\begin{proof}
(i) This is Theorem~\ref{thm:stratified}(b), since $J_V=J_F$ on $\Delta_F$, $\log C_F$
is bounded on $\mathcal K$ and $J_{\{i\}}=q_i$.

(ii) This follows from Theorem~\ref{thm:reducible} below, whose proof in
Appendix~\ref{app:counts} does not use the present theorem. After $T_1$ and $N_1$ are
enlarged, that theorem gives $\log M_F=-(k_F-1)L+p_F\log T-T\lambda_F+O(1)$ for $N\ge N_1$ and
$T_1\le T\le N^{1/3}$. Since $T\ge1$ and $k_F-k^*\le p_F\le k_F-1$, both bounds hold with
$T_0=T_1$ and $N_0=N_1$. If $F$ is not admissible, no path has support $F$, so $M_F=0$.
Appendix~\ref{app:first} gives a direct proof by explicit paths, which does not
use the local law.

(iii) For large $N$ we have $T_0\le T=\tau L\le N^{1/3}$, and both bounds in (ii)
read $\log M_F=-L\,\Phi_\tau(F)+O(\log L)$ for every admissible face, while
$M_F=0$ for the others. Since $\max_nP_N(n)=\max_FM_F$ over finitely many faces,
the first claim follows. Let $m_\tau=\min_F\Phi_\tau(F)$. An admissible
$F\notin\mathcal F_\tau$ has $\Phi_\tau(F)\ge m_\tau+\varepsilon_0$ for a fixed
$\varepsilon_0>0$, so $M_F\le N^{-m_\tau-\varepsilon_0+o(1)}$, which is smaller
than $\max_nP_N(n)=N^{-m_\tau+o(1)}$ for large $N$. A global mode has positive
probability, so its support is admissible, and hence it lies in $\mathcal F_\tau$.

(iv) By (iii), for large $N$ a global mode $\hat n$ maximizes $P_N$ on its face
$F\in\mathcal F_\tau$, which satisfies (H) by Lemma~\ref{lem:admissible}(b). So
Theorem~\ref{thm:facemax}(b) with $T=\tau L\le N^{1/3}$ gives
$\hat n/N=\alpha^*_F+O(1/L)$.
\end{proof}

On a reducible face the power of $T$ can differ from $(k_F-1)/2$, as in
Example~\ref{ex:threecycle}. Theorem~\ref{thm:reducible} below gives the power and the constant
for every admissible face.

\begin{remark}[a field]\label{rem:first-field}
Let $v\in\R^d$, put $M_F(T,v)=\max\{P_N(n)e^{Tv\cdot n/N}:\supp n=F\}$ and
$\Phi_{\tau,v}(F)=|F|-1+\tau\lambda_F(v_F)$. On an admissible face,
$Q_F+\operatorname{diag}v_F$ is block triangular in the order of the
components, so $\lambda_F(v_F)=\min_j\lambda_{C_j}(v_{C_j})$. The upper bound in
(ii) holds for $M_F(T,v)$ with $\lambda_F(v_F)$ in place of $\lambda_F$. The proof
in Appendix~\ref{app:first} applies with~\eqref{eq:upper-field} on each block and
the exact weight $e^{hv_iN_j}$ on a block $\{i\}$ of length $N_j$. For $F$ satisfying (H),
Theorem~\ref{thm:facemax} gives the matching lower bound. If $\mu$ has full
support, a reducible admissible face $F$ has a component $C$ with
$\lambda_C(v_C)=\lambda_F(v_F)$, which is admissible with fewer states, so
$\Phi_{\tau,v}(C)\le\Phi_{\tau,v}(F)-1$. Hence for
fixed $\tau$ and $v$, (iii)
and (iv) hold for the weighted law with $\Phi_{\tau,v}$ and $\alpha^*_F(v)$, every
face that minimizes $\Phi_{\tau,v}$ is irreducible, and a strict minimizer carries
every maximizer of the weighted law for large $N$.
\end{remark}

\subsection{Reading the winners off a convex hull}

For an admissible face $F$ put $P_F=(\lambda_F,\lvert F\rvert-1)$, and let $\mathcal S$
be the set of these points. Then $\Phi_\tau(F)=\tau x+y$ at $(x,y)=P_F$. Picture
a line of slope $-\tau$ moving up from far below. The first points of $\mathcal S$ that
it touches are the winners at $\tau$. For each $\tau>0$ the points of
$\operatorname{conv}\mathcal S$ where $\tau x+y$ is least form a vertex or an edge of
$\operatorname{conv}\mathcal S$, and we call the union of these over all $\tau>0$ the
lower-left boundary of $\operatorname{conv}\mathcal S$. Figure~\ref{fig:C4worked}(b) shows
it for the 4-cycle.

\begin{theorem}[geometry of the winners]\label{thm:hull}
Let $\mathcal F_\tau$ be the set of admissible faces that minimize
$\Phi_\tau$.
\begin{enumerate}[(a)]
\item The lower-left boundary of $\operatorname{conv}\mathcal S$ is a path whose
vertices lie in $\mathcal S$. Read from right to left, they are
$v_i=(\lambda^{(i)},y^{(i)})$, $0\le i\le m$, with
$\lambda^{(0)}>\dots>\lambda^{(m)}$ and $y^{(0)}<\dots<y^{(m)}$. Put
$\tau_i=(y^{(i)}-y^{(i-1)})/(\lambda^{(i-1)}-\lambda^{(i)})$ for $1\le i\le m$,
$\tau_0=0$ and $\tau_{m+1}=\infty$. Then $\tau_1<\dots<\tau_m$. For
$\tau_i<\tau<\tau_{i+1}$, $\mathcal F_\tau$ is the set of faces $F$ with
$P_F=v_i$. At $\tau=\tau_i$ it is the set of faces whose point lies on the
closed edge $[v_{i-1},v_i]$.
\item If $0<\tau<\tau'$, $F\in\mathcal F_\tau$ and $G\in\mathcal F_{\tau'}$,
then $\lambda_F\ge\lambda_G$ and $\lvert F\rvert\le\lvert G\rvert$. Both
inequalities are strict when $\mathcal F_\tau\cap\mathcal F_{\tau'}=\emptyset$.
\item If $Q$ has symmetric support, or $\mu$ has full support, then every face
in $\mathcal F_\tau$ is irreducible, for every $\tau>0$.
\end{enumerate}
\end{theorem}

\begin{proof}
(a) We prove this through the least value $g(\tau)=\min_{(x,y)\in\mathcal S}(\tau x+y)$ for $\tau>0$. Each point
$(x,y)$ of $\mathcal S$ gives the affine function $\tau\mapsto\tau x+y$, whose slope is $x$, and different
points give different functions. Hence $g$ is a minimum of finitely many affine functions, so it is
concave and piecewise affine. Let $\tau_1<\dots<\tau_m$ be the points of $(0,\infty)$ where its
slope changes, and put $\tau_0=0$ and $\tau_{m+1}=\infty$.

\emph{Between two kinks.} On $(\tau_i,\tau_{i+1})$ the function $g$ is affine. The affine
function of a point of $\mathcal S$ is nowhere below $g$. So if it equals $g$ at one $\tau$ inside this
interval, it equals $g$ on the whole interval. Hence exactly one point of $\mathcal S$ attains $g$ there.
We call it $v_i=(\lambda^{(i)},y^{(i)})$, and for $\tau_i<\tau<\tau_{i+1}$ the set
$\mathcal F_\tau$ consists of the faces $F$ with $P_F=v_i$. The slope of $g$ on this interval is
$\lambda^{(i)}$. The slope of a concave function decreases, and it changes at each $\tau_i$, so
$\lambda^{(0)}>\dots>\lambda^{(m)}$. By continuity both $v_{i-1}$ and $v_i$ attain $g(\tau_i)$,
namely
\[
\tau_i\lambda^{(i-1)}+y^{(i-1)}=\tau_i\lambda^{(i)}+y^{(i)} .
\]
This is the formula for $\tau_i$ in the statement, and since $\tau_i>0$ it gives
$y^{(i)}>y^{(i-1)}$.

\emph{At a kink.} The points $(x,y)$ of $\mathcal S$ that attain $g(\tau_i)$ are those on the line
$\tau_ix+y=g(\tau_i)$, which passes through $v_{i-1}$ and $v_i$. Let $(x,y)$ be one of them. For
$\tau$ slightly below $\tau_i$ we have $\tau x+y\ge g(\tau)=\tau\lambda^{(i-1)}+y^{(i-1)}$, with
equality at $\tau=\tau_i$. Subtracting gives $(\tau-\tau_i)(x-\lambda^{(i-1)})\ge0$, so
$x\le\lambda^{(i-1)}$. In the same way, $\tau$ slightly above $\tau_i$ gives
$x\ge\lambda^{(i)}$. So $(x,y)$ lies on the closed edge $[v_{i-1},v_i]$. Conversely every point of
$\mathcal S$ on this edge lies on the line and so attains $g(\tau_i)$.

\emph{The boundary.} On $\operatorname{conv}\mathcal S$ the function $\tau x+y$ is least exactly on the
convex hull of the points of $\mathcal S$ where it is least. By the two steps above this hull is the
vertex $v_i$ for $\tau_i<\tau<\tau_{i+1}$ and the edge $[v_{i-1},v_i]$ for $\tau=\tau_i$. So the
lower-left boundary is the path $v_0\cdots v_m$, and its vertices lie in $\mathcal S$.

(b) We prove this by adding two inequalities. Since $F$ wins at $\tau$ and $G$ wins at $\tau'$, we have
$\Phi_\tau(F)\le\Phi_\tau(G)$ and $\Phi_{\tau'}(G)\le\Phi_{\tau'}(F)$. Adding them gives
$(\tau'-\tau)(\lambda_F-\lambda_G)\ge0$, so $\lambda_F\ge\lambda_G$. Then the first inequality
gives $\lvert F\rvert-\lvert G\rvert\le\tau(\lambda_G-\lambda_F)\le0$. If
$\lambda_F=\lambda_G$, the two inequalities give $\lvert F\rvert=\lvert G\rvert$, so
$\Phi_{\tau'}(F)=\Phi_{\tau'}(G)$ and $F\in\mathcal F_{\tau'}$. So if $\mathcal F_\tau$ and
$\mathcal F_{\tau'}$ are disjoint, then $\lambda_F>\lambda_G$, and the first inequality gives
$\lvert F\rvert<\lvert G\rvert$.

(c) With symmetric support every admissible face is irreducible by
Lemma~\ref{lem:admissible}(c). Now let $\mu$ have full support, and let $F$ be admissible and
reducible. Take a component $C$ of $F$ with $\lambda_C=\lambda_F$. It is irreducible with
$\mu(C)>0$, so it is admissible by Lemma~\ref{lem:admissible}(b). It has the same exit rate as
$F$ and fewer states, so $\Phi_\tau(C)\le\Phi_\tau(F)-1$. Hence $F\notin\mathcal F_\tau$.
\end{proof}

Note that a face with small $\lambda_F$ for its size is a set of states that the chain
rarely leaves. As $\tau$ grows the exit rate counts for more, so the mode moves
to larger such sets. It need not grow around its first vertex and it can move
to a disjoint set. Faces of one size share the term $\lvert F\rvert-1$. Hence only
the faces with the least exit rate of each size can win.

\subsection{Reducible faces}\label{sec:reducible-max}

Let $F$ be admissible, with components $C_1,\dots,C_r$ numbered as in
Lemma~\ref{lem:admissible}(a). A path of positive probability with support $F$ visits them in
this order, each in one block of consecutive times. The order is unique, since consecutive
blocks are joined by an edge and the components form an acyclic graph, in which at most one path
passes through every vertex. Put $k_j=|C_j|$ and $\lambda_j=\lambda_{C_j}$, so
$\lambda_F=\min_j\lambda_j$. Let $k^*$ be the least $k_j$ with $\lambda_j=\lambda_F$, let
$S^*$ be the set of $j$ with $\lambda_j=\lambda_F$ and $k_j=k^*$, and put
\[
a_F=\frac{k^*-1}2,\qquad p_F=a_F+k_F-k^*,\qquad R_j=(q_{xy})_{x\in C_j,\,y\in C_{j+1}} .
\]
For an irreducible $C$ with $|C|\ge2$, states $y,x\in C$ and $t\in(0,\infty)^C$ let
$\Phi_C(t)_{yx}=e^{-\langle q,t\rangle}H_{A_C,e_y,e_x}(t)$, in the notation before
Lemma~\ref{lem:torus-complex}. As in Lemma~\ref{lem:continuum}, this is the density of the
occupation times of the chain started at $y$ on the event that it stays in $C$ up to time
$|t|$, visits every state of $C$ and ends at $x$. For $C=\{i\}$ put $\Phi_C(t)=e^{-q_it}$. For
$j\in S^*$ let $r$ and $l$ be the vectors $r_{C_j}$ and $l_{C_j}$ of Definition~\ref{def:faces},
with $l\cdot r=1$, and let
\begin{align*}
A_j&=\sup\,e^{\lambda_F|t^<|}\,\mu_{C_1}^\top\Phi_{C_1}(t^{(1)})R_1\cdots\Phi_{C_{j-1}}(t^{(j-1)})R_{j-1}\,r,\\
B_j&=\sup\,e^{\lambda_F|t^>|}\,l^\top R_j\Phi_{C_{j+1}}(t^{(j+1)})R_{j+1}\cdots\Phi_{C_r}(t^{(r)})\,\ones,
\end{align*}
over all $t^{(m)}\in(0,\infty)^{C_m}$ for $m<j$, respectively $m>j$, where $|t^<|$ and $|t^>|$
are their total times, $A_1=\mu_{C_1}\cdot r$ and $B_r=l\cdot\ones$. Finally let
\[
K^{\rm red}_F=\max_{j\in S^*}\frac{A_jB_j}{(2\pi)^{a_F}\sqrt{\det\Sigma_{C_j}}},
\]
where $\Sigma_{C_j}$ is $\Sigma_\alpha$ of Definition~\ref{def:tilt} for the face $C_j$ at
$\alpha=\alpha^*_{C_j}$, and the denominator is $1$ if $k^*=1$. If $F$ is irreducible, then
$p_F=(k_F-1)/2$ and $K^{\rm red}_F=K_F$.

\begin{theorem}[face maxima on reducible faces]\label{thm:reducible}
Let $F$ be admissible. Then $0<K^{\rm red}_F<\infty$, and there are $T_1$, $N_1$ and $C$ such that
for $N\ge N_1$ and $T_1\le T\le N^{1/3}$
\[
\Bigl|\frac{M_F}{N^{-(k_F-1)}T^{p_F}e^{-T\lambda_F}}-K^{\rm red}_F\Bigr|\le C\Bigl(\frac1T+\frac{T^2}N\Bigr).
\]
In particular the ratio lies between $\frac12K^{\rm red}_F$ and $2K^{\rm red}_F$ after $T_1$ and
$N_1$ are enlarged.
\end{theorem}

Hence the power of $T$ is $p_F$, between the two powers of Theorem~\ref{thm:first}(ii). The time
goes to the components of least exit rate and least size $k^*$, and to one of them when
$k^*\ge2$. Such a component contributes $h^{k^*-1}T^{-(k^*-1)/2}$ through its local law, and each
other component $C$ gets time of order one and contributes $h^{|C|-1}$. Each switch from one component to the next contributes
one more factor $h$. The constants $A_j$ and $B_j$ are maxima over these short blocks.
For irreducible $F$ the theorem is Theorem~\ref{thm:facemax}(b) at $v=0$. The proof is in
Appendix~\ref{app:counts}.

\begin{corollary}[crossings of admissible faces]\label{cor:reducible-crossing}
Let $F$ and $G$ be admissible, with $\Delta=k_F-k_G\ge1$ and $\delta=\lambda_G-\lambda_F>0$. Put
$\pi=p_F-p_G$, $\kappa=\log(K^{\rm red}_F/K^{\rm red}_G)$ and $D_N=\log M_F-\log M_G$, and let
$T_1$ serve for both faces in Theorem~\ref{thm:reducible}.
\begin{enumerate}[(i)]
\item For large $N$, $D_N$ has a zero in $[\frac\Delta{2\delta}L,\frac{2\Delta}\delta L]$.
\item For large $N$, every zero $T_N\in[T_1,N^{1/3}]$ of $D_N$ satisfies
$\delta T_N+\pi\log T_N=\Delta L-\kappa+O(1/L)$.
\item All such zeros lie in one interval of length $O(1/L)$, and
\[
T_N=\frac\Delta\delta\Bigl[L-\frac\pi\Delta\log L+b_{FG}\Bigr]+O\Bigl(\frac{\log L}L\Bigr),\qquad
b_{FG}=-\frac\pi\Delta\log\frac\Delta\delta-\frac\kappa\Delta .
\]
\end{enumerate}
\end{corollary}

The coefficient of $\log\log N$ is $\pi/\Delta$. It is $\frac12$ when $F$ and $G$ are
irreducible, which gives Theorem~\ref{thm:crossing}(i) to (iii) for zeros in $[T_1,N^{1/3}]$. In
Example~\ref{ex:threecycle}, where $q_{12}=b$, we have $p_{\{1\}}=0$, $p_{\{1,2\}}=1$ and $p_{\{1,2,3\}}=1$, and
$K^{\rm red}_{\{1,2\}}=b$ from $A_2=\sup_tbe^{(1-b)t}=b$ and $B_2=1$. So the corollary gives the
coefficients $1$ and $0$ of parts (ii) and (iii) there, with their constants. We do not claim
that the zero is unique.

\section{The second-order theory}\label{sec:secondorder}

To locate a crossing up to $o(1)$ we need the power of $T$ and the constant in front
of $N^{-(k-1)}e^{-T\lambda_F}$ for every face that satisfies (H). Throughout this section
gradients and Hessians in $\alpha$ are taken in the coordinates
$\alpha'=(\alpha_2,\dots,\alpha_k)$, $|\cdot|$ is the Euclidean norm, and
$q=(q_i)_{i\in F}$.

\subsection{The tilted chain and the constant}\label{sec:tilt}

To find the probability of a composition with fractions $\alpha$ we change the
rates so that $\alpha$ becomes typical. The optimal vector $r_\alpha$ in the
definition of $J_F(\alpha)$ gives a new generator $\widehat Q_\alpha$ on $F$, with
the same switching graph, whose stationary law is exactly $\alpha$. Under it the
visit fractions satisfy a central limit theorem with covariance
$\Gamma_\alpha/T$, and the price of the change of rates is $e^{-TJ_F(\alpha)}$
times boundary factors for the start and the end of the path. The constant
$C_F(\alpha)$ collects these factors and the Gaussian normalization.

\begin{definition}[tilt and constants]\label{def:tilt}
Let $F$ be irreducible with $k\ge2$. For
$\alpha\in\Delta_F^\circ$ let $r_\alpha$ attain the supremum defining
$J_F(\alpha)$ (Proposition~\ref{prop:rate}(b)), and put
$\theta(\alpha)=A_Fr_\alpha/r_\alpha$ and $l_\alpha=\alpha/r_\alpha$
(entrywise). The \emph{tilted generator} is
$\widehat Q_\alpha=R^{-1}(A_F-\operatorname{diag}\theta(\alpha))R$ with
$R=\operatorname{diag}r_\alpha$. With $D=\operatorname{diag}\alpha$ let
\[
Z_\alpha=(\ones\alpha^\top-\widehat Q_\alpha)^{-1}-\ones\alpha^\top,
\qquad\Gamma_\alpha=DZ_\alpha+Z_\alpha^\top D,
\]
and let $\Sigma_\alpha$ be $\Gamma_\alpha$ with the row and column of state
$1$ deleted. Finally let
\[
C_F(\alpha)=\frac{(\mu\cdot r_\alpha)(l_\alpha\cdot\ones)}
{(2\pi)^{(k-1)/2}\sqrt{\det\Sigma_\alpha}},\qquad K_F=C_F(\alpha^*_F).
\]
\end{definition}

For a single state we put $K_{\{i\}}=\mu_i$.
Nothing in Definition~\ref{def:tilt} changes when $r_\alpha$ is scaled, and nothing
in it involves the diagonal of $Q_F$. Hence the objects depend on $A_F$ and $\mu$ only, and a
field does not change them. For $v\in\R^F$ we put $K_F(v)=C_F(\alpha^*_F(v))$, so
that $K_F(0)=K_F$, and $K_{\{i\}}(v)=\mu_i$. Also
$l_\alpha\cdot r_\alpha=1$ and $J_F(\alpha)=\langle q-\theta(\alpha),\alpha\rangle$. Since
$r_\alpha$ minimizes $u\mapsto\sum_{i\ne j}\alpha_iq_{ij}u_j/u_i$, its derivative in $u_j$, which
is $\sum_il_{\alpha,i}q_{ij}-l_{\alpha,j}\theta_j(\alpha)$, vanishes at $r_\alpha$. Hence
\begin{equation}\label{eq:stationary}
l_\alpha^\top\bigl(A_F-\operatorname{diag}\theta(\alpha)\bigr)=0 .
\end{equation}
Also $(A_F-\operatorname{diag}\theta(\alpha))r_\alpha=0$ with $r_\alpha>0$, so $\theta(\alpha)\in
V_F$~\cite{bermanplemmons1994}. Hence $\widehat Q_\alpha$ is an irreducible generator with the
same off-diagonal zero pattern as $A_F$, and by~\eqref{eq:stationary} its stationary law is
$\alpha$. If $(\ones\alpha^\top-\widehat Q_\alpha)v=0$, multiplying by $\alpha^\top$ gives
$\alpha^\top v=0$, so $\widehat Q_\alpha v=0$. Hence $v$ is a multiple of $\ones$, and
$\alpha^\top v=0$ gives $v=0$. Thus $Z_\alpha$ exists, and
\begin{equation}\label{eq:Zprops}
Z_\alpha\ones=0,\qquad\alpha^\top Z_\alpha=0,\qquad
-\widehat Q_\alpha Z_\alpha=I-\ones\alpha^\top .
\end{equation}
Hence $\Gamma_\alpha\ones=0$. As $\Gamma_\alpha$ is symmetric, its adjugate is then a multiple of
$\ones\ones^\top$, so deleting another state instead of state $1$ gives the same
$\det\Sigma_\alpha$. At $\alpha^*_F$ the vector $u=r_F$ gives the value $\lambda_F=J_F(\alpha^*_F)$,
so $r_{\alpha^*_F}$ is a multiple of $r_F$, $\theta(\alpha^*_F)=q-\lambda_F\ones$, and
$l_{\alpha^*_F}$ is the matching multiple of $l_F$.

For conductances $c_{ij}=c_{ji}$ on $F$ let
$\tree(c)=\sum_{\mathcal T}\prod_{\{i,j\}\in\mathcal T}c_{ij}$, the sum over
the spanning trees $\mathcal T$ of the complete graph on $F$.

\begin{proposition}[the constant]\label{prop:constant}
Let $F$ be irreducible with $k\ge2$.
\begin{enumerate}[(a)]
\item The maps $\alpha\mapsto\theta(\alpha)$, $J_F(\alpha)$, $\Sigma_\alpha$,
$r_\alpha$ and $l_\alpha$ (with $\sum_ir_{\alpha,i}=1$) are real-analytic on
$\Delta_F^\circ$, $\Sigma_\alpha$ is positive definite, and
$\nabla^2J_F(\alpha)=\Sigma_\alpha^{-1}$. Hence
$\det\Sigma_\alpha\cdot\det\nabla^2J_F(\alpha)=1$ and
$C_F(\alpha)=(\mu\cdot r_\alpha)(l_\alpha\cdot\ones)
\sqrt{\det\nabla^2J_F(\alpha)/(2\pi)^{k-1}}$.
\item Suppose that $\nu_iq_{ij}=\nu_jq_{ji}$ for $i,j\in F$ and some
$\nu\in(0,\infty)^F$. Then $r_\alpha$ is a multiple of $\sqrt{\alpha/\nu}$
and $l_\alpha$ of $\sqrt{\alpha\nu}$, and $\widehat Q_\alpha$ is reversible
with conductances $c_{ij}(\alpha)=\sqrt{\alpha_i\alpha_j}\,q_{ij}\sqrt{\nu_i/\nu_j}$.
Also $J_F(\alpha)=\sum_iq_i\alpha_i-\sum_{i\ne j}c_{ij}(\alpha)$, and
\[
\det\Sigma_\alpha=\frac{2^{k-1}(\prod_i\alpha_i)^2}{\tree(c(\alpha))},\qquad
C_F(\alpha)=\frac{(\mu\cdot r_\alpha)(l_\alpha\cdot\ones)}{(4\pi)^{(k-1)/2}}
\,\frac{\sqrt{\tree(c(\alpha))}}{\prod_i\alpha_i}.
\]
\end{enumerate}
\end{proposition}

\begin{proof}
We prove (a) in Appendix~\ref{app:torus}, using Lemma~\ref{lem:perron}.

(b) We guess the vector $r_\alpha$ and check that it is a critical point. Put
$u=\sqrt{\alpha/\nu}$ and $l=\alpha/u=\sqrt{\alpha\nu}$ entrywise. As we saw
before~\eqref{eq:stationary}, the derivative of $\sum_{i\ne j}\alpha_iq_{ij}u_j/u_i$ in $u_j$ is
$\sum_il_iq_{ij}-l_j(A_Fu)_j/u_j$. Here
\[
\sum_il_iq_{ij}=\sum_i\sqrt{\alpha_i\nu_i}\,q_{ij}\qquad\text{and}\qquad
\frac{l_j(A_Fu)_j}{u_j}=\nu_j\sum_iq_{ji}\sqrt{\alpha_i/\nu_i}
=\sum_i\sqrt{\alpha_i\nu_i}\,q_{ij},
\]
where the last step uses $\nu_jq_{ji}=\nu_iq_{ij}$. Hence the derivative vanishes at $u$. The
function is convex in $\log u$, so its critical point $u$ is a minimizer, and the minimizer is
unique up to a factor (Proposition~\ref{prop:rate}(b)). Hence $r_\alpha$ is a multiple of $u$ and
$l_\alpha$ is the matching multiple of $l$. Then
$\alpha_i(\widehat Q_\alpha)_{ij}=l_iq_{ij}u_j=c_{ij}(\alpha)$ for $i\ne j$, and
$c_{ij}(\alpha)=c_{ji}(\alpha)$ because $q_{ij}\sqrt{\nu_i/\nu_j}=q_{ji}\sqrt{\nu_j/\nu_i}$.
Hence $\widehat Q_\alpha$ is reversible with these conductances. Also
$\alpha_i\theta_i(\alpha)=l_i(A_Fu)_i=\sum_{j\ne i}c_{ij}(\alpha)$, and
$J_F(\alpha)=\langle q-\theta(\alpha),\alpha\rangle$ gives the formula for $J_F$.

For the determinant we express the quadratic form of $\Sigma_\alpha$ through the Laplacian of the
conductances, namely $\mathcal L=-D\widehat Q_\alpha$. It is symmetric, its off-diagonal entries
are $-c_{ij}(\alpha)$, and $\mathcal L\ones=0$. Fix $\eta\in\R^F$ with $\eta_1=0$, and put
$w=Z_\alpha\eta$ and $\zeta=D\eta-(\alpha\cdot\eta)\alpha$. By~\eqref{eq:Zprops},
$\mathcal Lw=D(I-\ones\alpha^\top)\eta=\zeta$ and $\alpha\cdot w=0$. So
$\eta^\top\Gamma_\alpha\eta=2(D\eta)^\top w=2\zeta^\top w$, where the second step uses
$\alpha\cdot w=0$. Next we remove state $1$. Let $v=w-w_1\ones$,
so that $v_1=0$. Since $\zeta\cdot\ones=0$ and $\mathcal L\ones=0$, we have
$\zeta^\top w=\zeta^\top v$ and $\mathcal Lv=\zeta$. We write primes for the coordinates
$2,\dots,k$, $D'=\operatorname{diag}\alpha'$, and $\mathcal L_{(1)}$ for $\mathcal L$ without row
and column $1$. As $v_1=0$, the rows $2,\dots,k$ of $\mathcal Lv=\zeta$ say that
$\mathcal L_{(1)}v'=\zeta'$, and as $\eta_1=0$ we have $\zeta'=(D'-\alpha'\alpha'^\top)\eta'$. The
matrix-tree theorem~\cite{bollobas1998} gives $\det\mathcal L_{(1)}=\tree(c(\alpha))$, which is
positive because $F$ is irreducible. Hence $\mathcal L_{(1)}$ is invertible, and
\[
\eta'^\top\Sigma_\alpha\eta'=2\zeta'^\top v'
=2\eta'^\top(D'-\alpha'\alpha'^\top)\mathcal L_{(1)}^{-1}(D'-\alpha'\alpha'^\top)\eta'
\]
for every $\eta'\in\R^{k-1}$. Both matrices are symmetric, hence
$\Sigma_\alpha=2(D'-\alpha'\alpha'^\top)\mathcal L_{(1)}^{-1}(D'-\alpha'\alpha'^\top)$. We take
determinants. The matrix determinant lemma gives
$\det(D'-\alpha'\alpha'^\top)=\prod_{i\ge2}\alpha_i\,(1-\sum_{i\ge2}\alpha_i)=\prod_i\alpha_i$.
Hence $\det\Sigma_\alpha=2^{k-1}(\prod_i\alpha_i)^2/\tree(c(\alpha))$, and putting this into
Definition~\ref{def:tilt} gives the formula for $C_F(\alpha)$.
\end{proof}

\begin{corollary}[symmetric faces]\label{cor:symface}
Let $F$ satisfy \textup{(H)} with $k\ge2$, and let $Q_F$ be symmetric with constant row sums.
Then $\lambda_F=q_i-\sum_{j\in F\setminus\{i\}}q_{ij}$ for every $i\in F$, $\alpha^*_F$ is uniform
on $F$, and
\[
K_F=\mu(F)\,k^{(k+1)/2}\sqrt{\tree(q_F)}\,(4\pi)^{-(k-1)/2},
\]
where $\tree(q_F)$ is $\tree(c)$ for the conductances $c_{ij}=q_{ij}$, $i,j\in F$.
\end{corollary}

\begin{proof}
The positive vector $\ones$ is a right and left eigenvector of $Q_F$, so it belongs to the
Perron root. Hence $\lambda_F$ is minus the row sum, $r_F$ and $l_F$ are multiples of $\ones$, and
$\alpha^*_F=\ones/k$.
At $\alpha=\ones/k$, Proposition~\ref{prop:constant}(b) with $\nu=\ones$ gives
$r_\alpha=\ones$, $l_\alpha=\alpha$ and $(\mu\cdot r_\alpha)(l_\alpha\cdot\ones)=\mu(F)$.
Also $c_{ij}(\alpha)=q_{ij}/k$, so $\tree(c(\alpha))=k^{-(k-1)}\tree(q_F)$, and
$\prod_i\alpha_i=k^{-k}$.
\end{proof}

\subsection{The local law and the face maxima}\label{sec:locallaw}

A field $v\in\R^F$ enters the continuum density as an exact factor. By
Lemma~\ref{lem:continuum}, $f_T(\alpha)e^{Tv\cdot\alpha}$ is the density $f_T$ of
the killed matrix $Q_F+\operatorname{diag}v$, which has the same $A_F$ and $\mu$ and
the diagonal $q-v$ in place of $q$. On the lattice the factor is
$e^{Tv\cdot n/N}=e^{h\langle v,n\rangle}$.

\emph{The discrete constant.} Let $F$ be irreducible with $k\ge2$, fix $h_*$ with
$h_*q_{\max}<1$, and let
$(h,\alpha)\in[0,h_*]\times\Delta_F^\circ$. Let $\theta=\theta^{(h)}(\alpha)$ be the
vector of Proposition~\ref{prop:discrete-rate}(b), let $r$ and $l$ be positive right
and left null vectors of $A_F-\operatorname{diag}\theta$ with $l\cdot r=1$, put
$c_h(\alpha)=\sum_i\alpha_i/(1-hq_i+h\theta_i)$, and let
\[
 C^{(h)}_F(\alpha)=c_h(\alpha)\,(\mu\cdot r)(l\cdot\ones)
 \sqrt{\det\nabla^2J^{(h)}_F(\alpha)/(2\pi)^{k-1}} .
\]
Lemma~\ref{lem:discrete-hessian} shows that $\nabla^2J^{(h)}_F$ is positive definite
and that $C^{(h)}_F$ is positive and real-analytic up to $h=0$. At $h=0$ we have
$c_0=1$ and $\theta^{(0)}(\alpha)=\theta(\alpha)$, so $C^{(0)}_F=C_F$ by
Proposition~\ref{prop:constant}(a). For $v\in\R^F$ put
$K^{(h)}_F(v)=C^{(h)}_F(\alpha^{*(h)}_F(v))$ and $K^{(h)}_F=K^{(h)}_F(0)$, and for a
single state $K^{(h)}_{\{i\}}(v)=\mu_i/(1-hq_i)$. At $\alpha^*_F$ the maximizer of
Proposition~\ref{prop:discrete-rate}(b) is $\log r_F$ up to a constant, for every
$h$, by part (c) of that proposition. Hence
$\theta^{(h)}(\alpha^*_F)=q-\lambda_F\ones$ and
$c_h(\alpha^*_F)=1/(1-h\lambda_F)$.

\begin{theorem}[local law]\label{thm:local}
Let $F$ satisfy \textup{(H)} with $k\ge2$, and let $\mathcal K\subset\Delta_F^\circ$ be
compact.
\begin{enumerate}[(a)]
\item Uniformly for $\alpha\in\mathcal K$, $T\ge1$ and $v\in\R^F$,
\[
f_T(\alpha)e^{Tv\cdot\alpha}=T^{(k-1)/2}C_F(\alpha)e^{-T(J_F(\alpha)-v\cdot\alpha)}\bigl(1+O_{\mathcal K}(1/T)\bigr).
\]
\item Fix $h_*$ with $h_*q_{\max}<1$. Uniformly for $1\le T\le h_*N$,
$\supp n=F$ with $\alpha=n/N\in\mathcal K$, and $v\in\R^F$,
\[
P_N(n)e^{Tv\cdot\alpha}=N^{-(k-1)}T^{(k-1)/2}C^{(h)}_F(\alpha)
e^{-T(J^{(h)}_F(\alpha)-v\cdot\alpha)}\bigl(1+O_{\mathcal K}(1/T)\bigr).
\]
\end{enumerate}
For $F=\{i\}$, $f_T=\mu_ie^{-q_iT}$ and $P_N(Ne_i)=\mu_i(1-hq_i)^{N-1}$.
\end{theorem}

The field cancels from the two sides, so the case $v=0$ is the whole content, and the
error does not depend on $v$. In (b) the ratio of the two sides lies between two
positive constants for all $T$ in the range, so the error is relative.

\emph{Idea of the proof.} We give the proof in Appendix~\ref{app:torus}. The
generating function of the words in $H$ is a matrix resolvent, so
Cauchy's formula writes $H$ as an integral over a torus with the radii
$\theta_2(\alpha),\dots,\theta_k(\alpha)$ (Lemma~\ref{lem:torus}).
All the coefficients are nonnegative, so the integrand is largest at the real
point, and a Laplace lemma with a parameter (Lemma~\ref{lem:laplace}) gives the
expansion. The Hessian at the real point comes from the Perron root of a
perturbed generator (Lemma~\ref{lem:perron}), which brings in $\Sigma_\alpha$.
Note that $f_T(\alpha)$ is continuous and positive in $(T,\alpha)$, as
$H$ is a convergent power series with some $W(m)>0$ by (H). For (b) the same torus
carries Good's discrete formula, with $e^{t_iz_i}$ replaced by $(1-hq_i+hz_i)^{n_i-1}$
(Lemma~\ref{lem:torus}). In the variable $z$ the integrand is analytic in $h$ up to
$h=0$, where it is the integrand of (a), so one Laplace argument with $(h,\alpha)$ in
the compact set $[0,h_*]\times\mathcal K$ covers $1\le T\le h_*N$. Only for
bounded $T$ does the proof use the comparison of Lemma~\ref{lem:discrete}, to keep
the ratio bounded below.

\begin{lemma}[discrete and continuum]\label{lem:discrete}
Let $F$ satisfy \textup{(H)} with $k\ge2$, and let $\mathcal K\subset\Delta_F^\circ$ be
compact. Uniformly for $\operatorname{supp}n=F$ with $n/N\in\mathcal K$ and
$1\le T\le N^{1/2}$,
$P_N(n)=N^{-(k-1)}f_T(n/N)(1+O_{\mathcal K}(T^2/N))$.
\end{lemma}

We prove this in Appendix~\ref{app:torus}. A field
multiplies both sides by the same factor $e^{Tv\cdot n/N}$. Theorem~\ref{thm:local}(b)
makes this transfer unnecessary for the local law, but the proofs of
Theorem~\ref{thm:local}(b) for bounded $T$ and of Lemma~\ref{lem:derivative} use it.
Hence, together with the restriction identities, the local law gives one display for all
faces.

\begin{theorem}[the stratified local law]\label{thm:stratified}
Fix $Q$, a start $\mu$ and $h_*$ with $h_*q_{\max}<1$. Let $F$ satisfy \textup{(H)} with
$|F|=k\ge2$, and let $\mathcal K\subset\Delta_F^\circ$ be compact.
\begin{enumerate}[(a)]
\item Uniformly for $\supp n=F$, $n/N\in\mathcal K$ and $1\le T\le h_*N$,
\[
 -\log P_N(n)=(k-1)L+T\,J^{(h)}_V(n/N)-\frac{k-1}2\log T-\log C^{(h)}_F(n/N)
 +O_{\mathcal K}\Bigl(\frac1T\Bigr).
\]
\item Uniformly for $\supp n=F$, $n/N\in\mathcal K$ and $1\le T\le N^{1/3}$,
\[
 -\log P_N(n)=(k-1)L+T\,J_V(n/N)-\frac{k-1}2\log T-\log C_F(n/N)+O_{\mathcal K}\Bigl(\frac1T+\frac{T^2}N\Bigr).
\]
\end{enumerate}
For a singleton, $-\log P_N(Ne_i)=-\log\mu_i-(N-1)\log(1-hq_i)$, which is
$Tq_i-\log\mu_i+O(T^2/N)$. With $T=\tau L$ the coefficient of $L$ in (b) is
$(k-1)+\tau J_V(n/N)$, the lattice entropy of the face plus $\tau$ times the
Donsker and Varadhan cost of the point.
\end{theorem}

\begin{proof}
(a) This is Theorem~\ref{thm:local}(b) with $v=0$ and
Proposition~\ref{prop:discrete-rate}(a), which replaces $J^{(h)}_F$ by $J^{(h)}_V$
on $\Delta_F$. (b) By Proposition~\ref{prop:discrete-rate}(b) and
Lemma~\ref{lem:discrete-hessian}, $J^{(h)}_F$ and $\log C^{(h)}_F$ are
real-analytic in $(h,\alpha)$ up to $h=0$, so
$TJ^{(h)}_F=TJ_F+O(Th)=TJ_F+O(T^2/N)$ and $\log C^{(h)}_F=\log C_F+O(T/N)$
uniformly on $\mathcal K$. Put this into (a) and replace $J_F$ by $J_V$ on
$\Delta_F$ by Proposition~\ref{prop:restriction}. Part (b) also follows from
Theorem~\ref{thm:local}(a) and Lemma~\ref{lem:discrete}. The singleton case is the
path $i^N$ and~\eqref{eq:vertexmax}.
\end{proof}

For a field $v\in\R^d$ we write $\lambda_F(v)$, $\alpha^*_F(v)$, $K_F(v)$,
$\lambda^{(h)}_F(v)$, $\alpha^{*(h)}_F(v)$ and $K^{(h)}_F(v)$ for the quantities of
$v_F$, and
\[
M_F(T,v)=\max\{P_N(n)e^{Tv\cdot n/N}:\supp n=F\},
\]
so that $M_F(T,0)=M_F$.

\begin{theorem}[face maxima]\label{thm:facemax}
Let $F$ satisfy \textup{(H)}, let $\mathcal V\subset\R^d$ be compact, and fix $h_*$ with
$h_*q_{\max}<1$.
\begin{enumerate}[(a)]
\item If $k\ge2$, then uniformly for $v\in\mathcal V$ and $1\le T\le h_*N$,
\[
M_F(T,v)=N^{-(k-1)}T^{(k-1)/2}K^{(h)}_F(v)e^{-T\lambda^{(h)}_F(v)}\bigl(1+O_{\mathcal V}(1/T)\bigr),
\]
and every maximizer $\hat n$ satisfies $\hat n/N=\alpha^{*(h)}_F(v)+O_{\mathcal V}(1/T)$.
At $v=0$ this reads
$M_F=N^{-(k-1)}T^{(k-1)/2}K^{(h)}_F(1-h\lambda_F)^N(1+O(1/T))$ with
$\hat n/N=\alpha^*_F+O(1/T)$. For $k=1$,
$M_{\{i\}}(T,v)=e^{Tv_i}\mu_i(1-hq_i)^{N-1}=K^{(h)}_{\{i\}}(v)e^{-T\lambda^{(h)}_{\{i\}}(v)}$
exactly.
\item If $k\ge2$, then uniformly for $v\in\mathcal V$ and $1\le T\le N^{1/3}$,
\[
M_F(T,v)=N^{-(k-1)}T^{(k-1)/2}K_F(v)e^{-T\lambda_F(v)}\bigl(1+O_{\mathcal V}(1/T+T^2/N)\bigr),
\]
and every maximizer $\hat n$ satisfies $\hat n/N=\alpha^*_F(v)+O_{\mathcal V}(1/T)$. For
$k=1$, $M_{\{i\}}(T,v)=e^{Tv_i}M_{\{i\}}$ with $M_{\{i\}}$ given by~\eqref{eq:vertexmax}.
In both cases, uniformly for $v\in\mathcal V$ and $1\le T\le N^{1/3}$,
\begin{equation}\label{eq:logmax}
\log M_F(T,v)=-(k-1)L+\tfrac{k-1}2\log T+\log K_F(v)-T\lambda_F(v)+O_{\mathcal V}(1/T+T^2/N).
\end{equation}
\end{enumerate}
\end{theorem}

The constants depend on $\mathcal V$ only through the compact set
$\{\alpha^{*(h)}_F(v):0\le h\le h_*,\ v\in\mathcal V\}$, the bounds of
$\lambda^{(h)}_F$ and $K^{(h)}_F$ on $[0,h_*]\times\mathcal V$, and finitely many
vectors of $V_F$.

\begin{proof}
(a) For $k=1$ the only composition is $Ne_i$. Let $k\ge2$. For
$(h,v)\in[0,h_*]\times\mathcal V$ put
$J^{h,v}(\alpha)=J^{(h)}_F(\alpha)-v\cdot\alpha-\lambda^{(h)}_F(v)$, which is
nonnegative on $\Delta_F$ and vanishes only at $\alpha^{*(h)}_F(v)$, by
Proposition~\ref{prop:discrete-variational}(a) for $h>0$ and
Proposition~\ref{prop:rate}(a) for $h=0$. Put
$X=N^{-(k-1)}T^{(k-1)/2}K^{(h)}_F(v)e^{-T\lambda^{(h)}_F(v)}$, the value that (a)
claims, and write $x=\alpha'-\alpha^{*(h)}_F(v)'$ for the displacement of $\alpha$
from $\alpha^{*(h)}_F(v)$. Below, constants depend only on $Q$, $\mu$, $h_*$ and
$\mathcal V$.

\emph{The balls.} By Proposition~\ref{prop:discrete-variational}(c) the set
$\mathcal A=\{\alpha^{*(h)}_F(v):0\le h\le h_*,\ v\in\mathcal V\}$ is a compact subset
of $\Delta_F^\circ$, and $\lambda^{(h)}_F(v)$ and $\log K^{(h)}_F(v)$ are bounded. By
Proposition~\ref{prop:discrete-rate}(b) and Lemma~\ref{lem:discrete-hessian},
$J^{(h)}_F$ and $\log C^{(h)}_F$ are real-analytic in $(h,\alpha)$ up to $h=0$, with
$\nabla^2J^{(h)}_F$ positive definite, and $J^{h,v}$ has a critical point at
$\alpha^{*(h)}_F(v)$. Taylor's theorem with bounds on the derivatives of order up to
three on $[0,h_*]$ times a compact neighborhood of $\mathcal A$ gives $\rho_0>0$ and
positive constants $c_U$, $C_U$, $C_1$ and $C_3$ with the following properties. The
set $\mathcal K_{\mathcal V}$ of points within $\rho_0$ of $\mathcal A$ lies in
$\Delta_F^\circ$, and for $(h,v)\in[0,h_*]\times\mathcal V$ the ball
$U(h,v)=\{|x|\le\rho_0\}$ satisfies
\[
c_U|x|^2\le J^{h,v}\le C_U|x|^2,\qquad
\bigl|\log(C^{(h)}_F/K^{(h)}_F(v))-g\cdot x\bigr|\le C_1|x|^2
\]
on $U(h,v)$, where $g=\nabla\log C^{(h)}_F(\alpha^{*(h)}_F(v))$ has $|g|\le C_1$.
Also $\log C^{(h)}_F$ is $C_3$-Lipschitz on $\mathcal K_{\mathcal V}$.

\emph{Outside $U(h,v)$.} For $\beta\in\Delta_F$ the definition of $J^{(h)}_F$ and
$((I+hQ_F)u)_i/u_i=1-hq_i+h\theta_i$ with $\theta=A_Fu/u$ give
$J^{(h)}_F(\beta)=\sup_{\theta\in V_F}\sum_i\beta_iL_h(q_i-\theta_i)$, and for fixed
$\theta\in V_F$ the sum is continuous in $(h,\beta)\in[0,h_*]\times\Delta_F$, since
$1-h(q_i-\theta_i)\ge1-h_*q_{\max}$. The set $\mathcal E$ of triples
$(h,v,\beta)\in[0,h_*]\times\mathcal V\times\Delta_F$ with $\beta\notin U(h,v)^\circ$ is
compact, and $J^{h,v}(\beta)$ is positive and lower semicontinuous on it. So
$J^{h,v}\ge3\eta$ on $\mathcal E$ for some $\eta>0$. For each point of $\mathcal E$
the supremum gives $\theta\in V_F$ with
$\sum_i\beta_iL_h(q_i-\theta_i)-v\cdot\beta>\lambda^{(h)}_F(v)+2\eta$, and the same
$\theta$ works near that point. By compactness finitely many vectors
$\theta_j\in V_F$ suffice. Now let $\supp n=F$ with $(h,v,n/N)\in\mathcal E$, and take $j$ that
works there. The bound~\eqref{eq:A-chernoff} with $\theta=\theta_j$, multiplied by
$e^{h\langle v,n\rangle}$, gives, since
$(1-hq_i+h\theta_{j,i})^{n_i}=e^{-hn_iL_h(q_i-\theta_{j,i})}$ with $hn_i=T\beta_i$ for
$\beta=n/N$,
\[
P_N(n)e^{Tv\cdot n/N}\le C(T/N)^{k-1}e^{-T(\lambda^{(h)}_F(v)+2\eta)}=X\cdot O\bigl(T^{(k-1)/2}e^{-2\eta T}\bigr).
\]

\emph{Inside $U(h,v)$.} By Theorem~\ref{thm:local}(b) with
$\mathcal K=\mathcal K_{\mathcal V}$, for $\alpha=n/N\in U(h,v)$,
\begin{equation}\label{eq:insideU}
P_N(n)e^{Tv\cdot n/N}=X\,\tfrac{C^{(h)}_F(\alpha)}{K^{(h)}_F(v)}\,e^{-TJ^{h,v}(\alpha)}(1+E),
\qquad|E|\le C_2/T,
\end{equation}
and $1+E$ lies between two positive constants. For $T\ge2C_1/c_U$ the logarithm of
the two middle factors is at most
\[
g\cdot x+C_1|x|^2-c_UT|x|^2\le g\cdot x-\tfrac{c_UT}2|x|^2\le\frac{|g|^2}{2c_UT},
\]
the maximum of a concave quadratic. Hence $P_N(n)e^{Tv\cdot n/N}\le X(1+O(1/T))$ on
$U(h,v)$. The bound outside $U(h,v)$ is smaller for large $T$, and hence
$M_F(T,v)\le X(1+O(1/T))$. For the lower bound let $n^0$ have support $F$ and
$|n^0_i-N\alpha^{*(h)}_F(v)_i|<1$ for all $i$, which exists for $N\ge N_0$. Then
$|x|\le k/N$, $C^{(h)}_F(n^0/N)\ge K^{(h)}_F(v)e^{-kC_3/N}$ and
$TJ^{h,v}(n^0/N)\le C_Uk^2T/N^2\le C_Uk^2h_*/N$. Since $1/N\le h_*/T$,
\eqref{eq:insideU} gives $P_N(n^0)e^{Tv\cdot n^0/N}\ge X(1-O(1/T))$.

\emph{Bounded $T$.} Fix $T_2$ such that the two bounds give
$M_F(T,v)/X\in[\frac12,\frac32]$ for $T\ge T_2$. For $1\le T\le T_2$ the same three
steps, with $1+E$ between two positive constants, give $M_F(T,v)/X\in[c,C]$ with
$0<c<C$. This proves the first display, and Proposition~\ref{prop:discrete-variational}(b)
gives the form at $v=0$.

\emph{The maximizer.} Let $\hat n$ be a maximizer with displacement $\hat x$. For
$T\ge T_3$ the bound outside $U(h,v)$ is below $X/2\le P_N(n^0)e^{Tv\cdot n^0/N}$, so
$\hat n/N\in U(h,v)$, and $|E|\le\frac12$ in~\eqref{eq:insideU}. Taking logarithms in
the comparison of $\hat n$ with $n^0$, and using
$\log(C^{(h)}_F/K^{(h)}_F(v))\le C_3|\hat x|$ and $TJ^{h,v}\ge c_UT|\hat x|^2$, we get
\[
c_UT|\hat x|^2\le C_3|\hat x|+\frac{C_4}T
\]
for all $T_3\le T\le h_*N$. Solving the quadratic inequality gives
$|\hat x|\le\bigl(C_3+\sqrt{C_3^2+4c_UC_4}\bigr)/(2c_UT)$. For $1\le T<T_3$ the claim
holds since $|\hat x|$ is bounded.

(b) By Proposition~\ref{prop:discrete-variational}(c) and
Lemma~\ref{lem:discrete-hessian}, uniformly for $v\in\mathcal V$,
$K^{(h)}_F(v)=K_F(v)(1+O(h))$, $\lambda^{(h)}_F(v)=\lambda_F(v)+O(h)$ and
$\alpha^{*(h)}_F(v)=\alpha^*_F(v)+O(h)$. For $T\le N^{1/3}$ we have
$h\le T^2/N\le1/T$ and $Th=T^2/N\le1$. Hence the value $X$ in (a) is
$N^{-(k-1)}T^{(k-1)/2}K_F(v)e^{-T\lambda_F(v)}(1+O(T^2/N))$, and every maximizer is
within $O(1/T)$ of $\alpha^*_F(v)$. For $k=1$ the formula is the path $i^N$,
and~\eqref{eq:logmax} is the logarithm of the two displays.

\emph{The continuum route to (b).} Part (b) also follows from the proof of (a)
with $J_F$, $C_F$, $K_F(v)$, $\lambda_F(v)$, Theorem~\ref{thm:local}(a) together
with Lemma~\ref{lem:discrete}, and Lemma~\ref{lem:upper} in place of $J^{(h)}_F$,
$C^{(h)}_F$, $K^{(h)}_F(v)$, $\lambda^{(h)}_F(v)$, Theorem~\ref{thm:local}(b)
and~\eqref{eq:A-chernoff}. In~\eqref{eq:insideU} the error is then
$|E|\le C_2(1/T+T^2/N)$, which is at most $\frac12$ for $T_2\le T\le N^{1/3}$ and large $N$.
For bounded $T$ we use that $f_T$ is continuous and positive. In the step that locates the
maximizer the comparison gives $c_UT|\hat x|^2\le C_3|\hat x|+C_4(1/T+T^2/N)$, and
$T^2/N\le1/T$ gives $|\hat x|=O(1/T)$. Proposition~\ref{prop:uniform-facemax} makes this
route uniform over a class of chains.
\end{proof}

\emph{The range of part (b).} The proof of Theorem~\ref{thm:facemax}(b) from part (a)
uses only $Th\le1$ and $h\le1/T$, so it gives the display and the location of the
maximizer for $1\le T\le N^{1/2}$. The same lines give Theorem~\ref{thm:stratified}(b)
for $T\le N^{1/2}$. We state both for $T\le N^{1/3}$, the range of the continuum route,
which needs $T^2/N\le1/T$ to locate the maximizer and is the route that
Proposition~\ref{prop:uniform-facemax} makes uniform. The error $T^2/N$ is not an
artifact of the proof. By the envelope theorem and $L_h(y)=y+hy^2/2+O(h^2)$ for bounded $y$,
uniformly on $\mathcal K$,
\[
 TJ^{(h)}_F(\alpha)=TJ_F(\alpha)+\frac{T^2}{2N}\sum_{i\in F}\alpha_i\bigl(q_i-\theta_i(\alpha)\bigr)^2
 +O\Bigl(\frac{T^3}{N^2}\Bigr),
\]
and $-N\log(1-h\lambda_F)=T\lambda_F+T^2\lambda_F^2/(2N)+O(T^3/N^2)$. The sum
vanishes only where $\theta(\alpha)=q$. Since $\theta$ determines $\alpha$, as the
stationary law of $R^{-1}(A_F-\operatorname{diag}\theta)R$ with $R$ the diagonal of
the Perron vector, this happens only when $\lambda_F=0$ and $\alpha=\alpha^*_F$. Hence
the continuum exponent $TJ_V$ fails when $T$ is of order $N^{1/2}$, while the
discrete functional absorbs every power of $T/N$ and has no threshold below $h_*N$.
Theorems~\ref{thm:first}(ii) and~\ref{thm:reducible}, which cover reducible faces, and
\ifarxivbuild Theorem~\ref{thm:simplex-crossing}(ii) on the complete graph\fi{} are stated for
$T\le N^{1/3}$, and the tie windows of Theorem~\ref{thm:window} for $T$ of order
$L$. We have not extended them.

\subsection{Crossings and tie windows}\label{sec:crossings}

\begin{theorem}[the crossing equation]\label{thm:crossing}
Let $F$ and $G$ satisfy \textup{(H)}, with $\Delta=k_F-k_G\ge1$ and
$\delta=\lambda_G-\lambda_F>0$, and let $D_N(T)=\log M_F-\log M_G$.
\begin{enumerate}[(i)]
\item For large $N$, $D_N$ has a zero in
$[\frac{\Delta}{2\delta}L,\frac{2\Delta}{\delta}L]$.
\item Fix $h_*$ with $h_*q_{\max}<1$. For large $N$, every zero $T_N\in[1,h_*N]$ of
$D_N$ satisfies
\begin{equation}\label{eq:crossing}
\delta T_N+\frac\Delta2\log T_N=\Delta L-\log\frac{K_F}{K_G}
+O\Bigl(\frac1L\Bigr).
\end{equation}
\item All such zeros lie in one interval of length $O(1/L)$, and
\[
T_N=\frac\Delta\delta\Bigl[L-\frac12\log L+b_{FG}\Bigr]
+O\Bigl(\frac{\log L}L\Bigr),\qquad
b_{FG}=-\frac12\log\frac\Delta\delta-\frac1\Delta\log\frac{K_F}{K_G}.
\]
\item For large $N$, $D_N$ has exactly one zero in $[1,h_*N]$.
\end{enumerate}
\end{theorem}

\begin{proof}
\emph{Where the zeros lie.} We show that for large $N$ every zero of $D_N$ in
$[1,h_*N]$ lies in $W=[\frac\Delta{2\delta}L,\frac{2\Delta}\delta L]$. By
Theorem~\ref{thm:facemax}(a) with $v=0$ for $F$, and for $G$ too if $k_G\ge2$,
uniformly for $1\le T\le h_*N$,
\[
D_N(T)=-\Delta L+\Xi_N(T)+\log\frac{K^{(h)}_F}{K^{(h)}_G}+O\Bigl(\frac1T\Bigr),
\qquad\Xi_N(T)=N\log\frac{1-h\lambda_F}{1-h\lambda_G}+\frac\Delta2\log T .
\]
If $G=\{j\}$ this holds by the exact formula $M_G=K^{(h)}_G(1-h\lambda_G)^N$ of
Theorem~\ref{thm:facemax}(a). Since $\lambda_F<\lambda_G\le q_{\max}<1/h$,
\[
\Xi_N'(T)=\frac{\lambda_G}{1-h\lambda_G}-\frac{\lambda_F}{1-h\lambda_F}+\frac\Delta{2T}
=\frac\delta{(1-h\lambda_F)(1-h\lambda_G)}+\frac\Delta{2T}\ge\delta,
\]
and the first term of $\Xi_N$ vanishes at $T=0$, so $\Xi_N(T)\ge\delta T$ for $T\ge1$.
The constants $K^{(h)}_F$ and $K^{(h)}_G$ are continuous and positive in
$h\in[0,h_*]$, so the logarithm of their ratio is bounded. Hence a zero $T_N$ has
$\Xi_N(T_N)=\Delta L+O(1)$, hence $\delta T_N\le\Delta L+O(1)$ and
$T_N<\frac{2\Delta}\delta L$. For $T\le\frac{2\Delta}\delta L$ we have $Th=O(L^2/N)$,
so $\Xi_N(T)=\delta T+\frac\Delta2\log T+O(L^2/N)$, and
$\Xi_N(\frac\Delta{2\delta}L)=\frac\Delta2L+O(\log L)$. As $\Xi_N$ is increasing,
$T_N>\frac\Delta{2\delta}L$ for large $N$. Also $W\subset[1,N^{1/3}]$ for large $N$.

\emph{On $W$.} We subtract~\eqref{eq:logmax} for $G$ from~\eqref{eq:logmax} for $F$.
Uniformly on $[1,N^{1/3}]$,
\[
D_N(T)=-\Delta L+\Xi(T)+\log\frac{K_F}{K_G}+O\Bigl(\frac1T+\frac{T^2}N\Bigr),
\qquad\Xi(T)=\delta T+\frac\Delta2\log T .
\]
The function $\Xi$ is increasing, with $\Xi'\ge\delta$.

(i) At $T=\frac{\Delta}{2\delta}L$ the display gives $D_N=-\frac\Delta2L+O(\log L)<0$, and at
$T=\frac{2\Delta}\delta L$ it gives $D_N=\Delta L+O(\log L)>0$. For fixed $N$, $M_F$ is a maximum
of finitely many polynomials in $h$, and it is positive because $F$ is admissible. The same holds
for $G$. Hence $D_N$ is continuous and has a zero in between.

(ii) Let $T_N\in[1,h_*N]$ be a zero. By the first step $T_N\in W$, so $T_N$ is of
order $L$, the error term of the display is $O(1/T_N+T_N^2/N)=O(1/L)$, and
$D_N(T_N)=0$ is~\eqref{eq:crossing}.

(iii) Let $T^\circ$ be the solution of $\Xi(T^\circ)=\Delta L-\log(K_F/K_G)$. By (ii) every zero
has $\Xi(T_N)=\Xi(T^\circ)+O(1/L)$, and $\Xi'\ge\delta$, so every zero is within $O(1/L)$ of
$T^\circ$. To solve for $T^\circ$, we use $T^\circ=\frac\Delta\delta L+O(\log L)$, so
$\log T^\circ=\log L+\log\frac\Delta\delta+O(\log L/L)$. Putting this into
$\delta T^\circ=\Delta L-\frac\Delta2\log T^\circ-\log(K_F/K_G)$ and dividing by $\delta$ gives
the explicit form.

(iv) By the first step, for large $N$ every zero in $[1,h_*N]$ lies in $W$. We show that $D_N$ is strictly increasing
on $W$. For fixed $N$, $M_F$ is the maximum of the finitely many polynomials $T\mapsto P_N(n)$
with $\supp n=F$ that are not identically zero. By Lemma~\ref{lem:runs} each term of such a
polynomial is either zero for all $h$ or positive for all $0<h<1/q_{\max}$, so these polynomials
are positive on $W$. Hence $\log M_F$ has at each $T\in W$ a right
derivative, the largest $\partial_T\log P_N(n)$ over the maximizers $n$ at $T$, and a left
derivative, the least of them. By Theorem~\ref{thm:facemax} every maximizer has
$n/N=\alpha^*_F+O(1/L)$ on $W$. Hence for large $N$ it lies in a closed ball
$\mathcal K\subset\Delta_F^\circ$ around $\alpha^*_F$, and $J_F(n/N)=\lambda_F+O(1/L)$, since
$J_F$ is smooth on $\mathcal K$ with $J_F(\alpha^*_F)=\lambda_F$. Lemma~\ref{lem:derivative}
then gives both one-sided derivatives of $\log M_F$ as $-\lambda_F+O(1/L)$ on $W$. The same
holds for $G$ if $k_G\ge2$. If $G=\{j\}$, then $\log M_G=\log\mu_j+(N-1)\log(1-hq_j)$ has the
derivative $-(N-1)q_j/(N(1-hq_j))=-\lambda_G+O(L/N)$. So both one-sided derivatives of $D_N$
are $\delta+O(1/L)\ge\delta/2$ on $W$ for large $N$. A continuous function whose right
derivative is positive at every point of an interval is strictly increasing there. Hence $D_N$
has at most one zero in $W$, and by (i) it has one.
\end{proof}

On the $4$-cycle at $N=10^{14}$ the explicit form in (iii) differs from the root
of~\eqref{eq:crossing} without its error term by $0.035$ for a vertex against a pair and by
$0.087$ for a pair against the cycle, the two crossings that carry the mode. Hence computations
should solve~\eqref{eq:crossing}.

At a switch value $\tau^*=\tau_i$ of Theorem~\ref{thm:hull}(a) several faces minimize
$\Phi_{\tau^*}$, sometimes more than two, as all faces do at $\tau^*=1$ for the
complete graph $Q=J-dI$, where $\lambda_F=d-|F|$\ifarxivbuild{} (Corollary~\ref{cor:paperB})\fi. Then the competition plays out in a window $T=\tau^*(L-\frac12\log L+t)$ with $t$ of
order $1$. In this window each tied face has a line $\ell_F(t)$, and the mode
follows the highest line.

\begin{theorem}[tie windows]\label{thm:window}
Fix $\tau^*>0$. Let $m^*$ be the least value of $\Phi_{\tau^*}$ over the
admissible faces and $\mathcal F^*$ the set of minimizers. Assume that every
face in $\mathcal F^*$ is irreducible. By Theorem~\ref{thm:hull}(c) this holds
if $Q$ has symmetric support or $\mu$ has full support. Put $T=\tau^*(L-\frac12\log L+t)$, and for
$F\in\mathcal F^*$ let
$\ell_F(t)=\log K_F+\frac{k_F-1}2\log\tau^*-\tau^*\lambda_Ft$ and
$E(t)=\max_{G\in\mathcal F^*}\ell_G(t)$. The function $E$ is convex and piecewise linear. We call
the finitely many points where its slope changes the \emph{kinks} of $E$, and we write $t_j$ for
them.
\begin{enumerate}[(i)]
\item Uniformly for $t$ in compact sets, for $F\in\mathcal F^*$,
$\log M_F=-m^*L+\frac{m^*}2\log L+\ell_F(t)+O(\log L/L)$.
\item There are $\varepsilon_1,\varepsilon_0>0$ such that for large $N$ and
$|T/L-\tau^*|\le\varepsilon_1$, every admissible $F\notin\mathcal F^*$ has
$M_F\le N^{-\varepsilon_0}\max_{G\in\mathcal F^*}M_G$.
\item There are $\varepsilon,t_0>0$ such that the following holds for large $N$. If
$t_0\le t\le\varepsilon L$, then every global mode lies on a face $F\in\mathcal F^*$ with
$\ell_F=E$ on $[t_0,\infty)$. If $-\varepsilon L\le t\le-t_0$, then every global mode lies on a
face $F\in\mathcal F^*$ with $\ell_F=E$ on $(-\infty,-t_0]$. A face of the first kind has the
least $\lambda_F$ in $\mathcal F^*$, and a face of the second kind has the largest.
\end{enumerate}
The numbers $\varepsilon_0$, $\varepsilon_1$, $\varepsilon$ and $t_0$ depend only on $Q$,
$\mu$ and $\tau^*$.
Consequently the following holds for every $\eta>0$ and large $N$. If $|t|\le\varepsilon L$ and
$t$ is at distance at least $\eta$ from every kink of $E$, then every global mode lies on a face
whose line equals $E$ near $t$. So for $|t|\le\varepsilon L$ the line of the face that carries
the global mode changes only at $T=\tau^*(L-\frac12\log L+t_j)+o(1)$, where $t_j$ runs over the
kinks of $E$, and it changes near each kink. Coincident lines do not distinguish
their faces at this order.
\end{theorem}

\begin{proof}
Every $F\in\mathcal F^*$ is admissible and irreducible, so it satisfies (H) by
Lemma~\ref{lem:admissible}(b). We first put $T$ into~\eqref{eq:logmax}. Write
$T=\tau^*L(1+y)$ with $y=(t-\frac12\log L)/L$, so that $|y|\le\frac12$ for $|t|\le L/4$ and
large $N$. Then $\log T=\log\tau^*+\log L+\log(1+y)$ and
$T\lambda_F=\tau^*\lambda_F(L-\frac12\log L+t)$, and the error term of~\eqref{eq:logmax} is
$O(1/L)$. Since $\tau^*\lambda_F+k_F-1=m^*$, we get, uniformly for $|t|\le L/4$,
\begin{equation}\label{eq:windowline}
\log M_F=-m^*L+\tfrac{m^*}2\log L+\ell_F(t)+\tfrac{k_F-1}2\log(1+y)+O(1/L)
\qquad(F\in\mathcal F^*).
\end{equation}

(i) For $t$ in a compact set, $y=O(\log L/L)$, so $\log(1+y)=O(\log L/L)$
in~\eqref{eq:windowline}.

(ii) Let $g>0$ be the least value of $\Phi_{\tau^*}(F)-m^*$ over the admissible
$F\notin\mathcal F^*$. If there is no such face, (ii) says nothing, and we put $g=1$. Put
$\varepsilon_1=\min\bigl(\tau^*/2,\,g/(4q_{\max}+4)\bigr)$. If $|T/L-\tau^*|\le\varepsilon_1$, then
$\frac12\tau^*L\le T\le\frac32\tau^*L$, so for large $N$ we are in the range of
Theorem~\ref{thm:first}(ii) and $\log T=O(\log L)$. Take an admissible $F\notin\mathcal F^*$ and
any $G\in\mathcal F^*$. Since $0\le\lambda_F,\lambda_G\le q_{\max}$
(Section~\ref{sec:notation}), $\Phi_\tau(F)-\Phi_\tau(G)$ has slope at most $q_{\max}$ in
absolute value, and it is at least $g$ at $\tau^*$. Hence it is at least $g/2$ for
$|\tau-\tau^*|\le\varepsilon_1$. With $\tau=T/L$ the
upper bound of Theorem~\ref{thm:first}(ii) for $F$ and the lower bound for $G$ give
$\log M_F-\log M_G\le-(\Phi_\tau(F)-\Phi_\tau(G))L+O(\log L)\le-\frac g4L$ for large $N$. Hence (ii)
holds with $\varepsilon_0=g/4$.

(iii) Take $\varepsilon=\min(\frac14,\varepsilon_1/(2\tau^*))$. Then for $|t|\le\varepsilon L$
and large $N$ we have $|T/L-\tau^*|=\tau^*|y|\le\varepsilon_1$, so (ii) applies and no global
mode lies on a face outside $\mathcal F^*$. It remains to compare the tied faces. Let $t_0\ge1$
be larger than every $|t_j|$, so that $E$ is a single line on $[t_0,\infty)$. Let
$F\in\mathcal F^*$ have this line, and let $G\in\mathcal F^*$ have another line. We show that
$M_F>M_G$ for $t_0\le t\le\varepsilon L$ if $t_0$ and $N$ are large. Since $\ell_F\ge\ell_G$
on $[t_0,\infty)$, the slopes satisfy $\lambda_F\le\lambda_G$. As both faces are tied,
\[
\Delta=k_F-k_G=\tau^*(\lambda_G-\lambda_F)\ge0 .
\]
There are two cases.

First let $\Delta=0$. Then the two lines are parallel, and $\ell_F-\ell_G=\log(K_F/K_G)>0$. The
terms with $\log(1+y)$ in~\eqref{eq:windowline} are the same for $F$ and $G$, so
$\log(M_F/M_G)=\log(K_F/K_G)+O(1/L)>0$ for large $N$.

Now let $\Delta>0$. Then $\ell_F(t)-\ell_G(t)=\Delta t+\ell_F(0)-\ell_G(0)$,
and~\eqref{eq:windowline} gives
\[
\log\frac{M_F}{M_G}=\Delta t+\ell_F(0)-\ell_G(0)+\frac\Delta2\log(1+y)+O\Bigl(\frac1L\Bigr).
\]
For large $N$ and $t\ge1$ we have $|\log(1+y)|\le2|y|\le(2t+\log L)/L\le t/2$. Hence the right
side is at least $\frac34\Delta t-C$, which is positive for $t\ge t_0$ once $t_0$ is large.

Hence for $t_0\le t\le\varepsilon L$ every global mode lies on a face whose line is $E$ on
$[t_0,\infty)$. The slope of $\ell_F$ is $-\tau^*\lambda_F$, and $E$ follows the line of largest
slope on the far right, so such a face has the least $\lambda_F$ in $\mathcal F^*$. The case
$t\le-t_0$ is the same with the signs reversed.

\emph{The consequence.} Fix $\eta>0$ and let
$S_\eta=[-t_0,t_0]\setminus\bigcup_j(t_j-\eta,t_j+\eta)$. For $G\in\mathcal F^*$ the function
$E-\ell_G$ is convex, piecewise linear and nonnegative, with kinks only at kinks of $E$. So on
each interval of $S_\eta$ either $\ell_G=E$ or $E-\ell_G\ge g_\eta$ for some $g_\eta>0$. Let
$t\in S_\eta$ and $\ell_F(t)=E(t)$. By (i), every $G\in\mathcal F^*$ whose line is not $E$ near
$t$ has $\log(M_F/M_G)\ge g_\eta-O(\log L/L)>0$, and by (ii) the faces outside $\mathcal F^*$
are smaller still. With (iii) for $t_0\le|t|\le\varepsilon L$ this gives the claim, and the line
does change near each kink, since $E$ is a different line on its two sides.
\end{proof}

\begin{corollary}[where the mode changes]\label{cor:mode-changes}
Let $Q$ have symmetric support or $\mu$ full support. Let $v_0,\dots,v_m$ and
$\tau_1<\dots<\tau_m$ be the vertices and switch values of Theorem~\ref{thm:hull}(a), and
assume that no point of $\mathcal S$ other than $v_{i-1}$ and $v_i$ lies on the edge
$[v_{i-1},v_i]$, for $1\le i\le m$. Put $\Delta_i=y^{(i)}-y^{(i-1)}$ and
$\delta_i=\lambda^{(i-1)}-\lambda^{(i)}$, let $K^{(i)}$ be the largest $K_F$ over the faces
$F$ with $P_F=v_i$, and let $T_{N,i}$ be the unique positive root of
$\delta_iT+\frac{\Delta_i}2\log T=\Delta_iL-\log(K^{(i)}/K^{(i-1)})$. Fix $0<a<b$. There are
$C$ and $N_0$ such that for $N\ge N_0$ and every $T\in[aL,bL]$ with $|T-T_{N,j}|\ge C/L$ for
all $j$, every global mode $\hat n$ has $P_{\supp\hat n}=v_i$, where $i$ is the number of $j$
with $T_{N,j}<T$.
\end{corollary}

\begin{proof}
By Theorem~\ref{thm:hull}(a) each face whose point lies on the lower-left boundary belongs to
some $\mathcal F_\tau$, so it satisfies (H) by Theorem~\ref{thm:hull}(c) and
Lemma~\ref{lem:admissible}(b). Let $\mathbb F_i$ be the set of faces $F$ with $P_F=v_i$, and
$M^{(i)}=\max_{F\in\mathbb F_i}M_F$. These faces share $|F|-1=y^{(i)}$ and
$\lambda_F=\lambda^{(i)}$. So for $T$ of order $L$,~\eqref{eq:logmax} gives
$\log(M^{(i)}/M^{(i-1)})=\Xi_i(T)-\Xi_i(T_{N,i})+O(1/L)$ with
$\Xi_i(T)=\delta_iT+\frac{\Delta_i}2\log T$, and $\Xi_i'\ge\delta_i$. Also
$T_{N,i}=\tau_iL+O(\log L)$. Put $\tau_0=0$ and $\tau_{m+1}=\infty$.

\emph{Near a switch value.} By Theorem~\ref{thm:hull}(a) and the assumption,
$\mathcal F_{\tau_i}=\mathbb F_{i-1}\cup\mathbb F_i$. Theorem~\ref{thm:window}(ii) with
$\tau^*=\tau_i$ gives $\varepsilon_1>0$ such that for large $N$ and $|T/L-\tau_i|\le\varepsilon_1$
no face outside $\mathbb F_{i-1}\cup\mathbb F_i$ carries a global mode. We shrink
$\varepsilon_1$ below a third of every gap $\tau_{j+1}-\tau_j$. Then this window contains
$T_{N,i}$ and no other $T_{N,j}$, and by the display there is $C$ such that the global modes lie
in $\mathbb F_i$ for $T\ge T_{N,i}+C/L$ and in $\mathbb F_{i-1}$ for $T\le T_{N,i}-C/L$.

\emph{Between switch values.} Let $J\subset(\tau_i,\tau_{i+1})$ be compact. For $\tau\in J$ only
$v_i$ attains the least value of $\Phi_\tau$ on $\mathcal S$. So $\Phi_\tau(F)-\Phi_\tau(G)\ge g_J>0$
for $\tau\in J$, $G\in\mathbb F_i$ and every admissible $F$ with $P_F\ne v_i$. The bounds of
Theorem~\ref{thm:first}(ii), with $k^*=k_G$ for the irreducible face $G$, give
$\log M_F-\log M_G\le-g_JL+(d-1)\log T+2\kappa<0$ for $T/L\in J$ and large $N$. Hence every
global mode lies in $\mathbb F_i$, and no $T_{N,j}$ lies in $LJ$. Finitely many windows and
intervals $J$ cover $[a,b]$.
\end{proof}

The equation for $T_{N,i}$ is~\eqref{eq:crossing} for a face of largest constant on each side,
so Theorem~\ref{thm:crossing}(iii) gives $T_{N,i}$ explicitly. If three or more points of
$\mathcal S$ lie on one edge, Theorem~\ref{thm:window} locates the changes only up to $o(1)$. At a
kink of $E$ where the line of $G$ gives way to that of $F$, $\Delta=\tau^*\delta$ and
$t_j=b_{FG}$, so the kink is where Theorem~\ref{thm:crossing}(iii) puts the crossing of $F$ and
$G$. The coefficient $\frac12$ of $\log\log N$ comes from the power $T^{(k-1)/2}$ in
Theorem~\ref{thm:facemax}, so it needs only (H).

\subsection{A reducible face}\label{sec:reducible}

The best composition of a reducible face can visit a state only once, with no
Gaussian window in that direction. Hence the coefficient $\frac12$ of
Theorem~\ref{thm:crossing} can fail, and in the next example it is $1$ and then $0$, as
Corollary~\ref{cor:reducible-crossing} predicts.

\begin{example}[a directed $3$-cycle]\label{ex:threecycle}
Let $d=3$, $q_{12}=b>2$, $q_{23}=q_{31}=1$, all other rates $0$, and
$\mu_1=1$. From state $1$ the chain moves only along $1\to2\to3\to1$, so the
admissible faces are $\{1\}$, $\{1,2\}$ (reducible, $\lambda=1$) and
$\{1,2,3\}$ (irreducible, $\lambda=0$).
\begin{enumerate}[(i)]
\item Exactly $M_{\{1\}}=(1-bh)^{N-1}$ and $M_{\{1,2\}}=bh(1-h)^{N-2}$, and the latter is
attained only at $n=(1,N-1,0)$. For the full face,
$\alpha^*=(1,b,b)/(1+2b)$ and $K_{\{1,2,3\}}=(1+2b)^2/(2\sqrt3\,\pi b)$.
\item The equation $M_{\{1\}}=M_{\{1,2\}}$ has exactly one root $T^{(1)}_N$ with $bh<1$, and
$T^{(1)}_N=\frac1{b-1}[L-\log L+\log\frac{b-1}b]+O(\log L/L)$. The coefficient of
$\log\log N$ in the bracket is $1$.
\item Every root $T\in[L/2,2L]$ of $M_{\{1,2\}}=M_{\{1,2,3\}}$ satisfies
$T=L+\log\frac{2\sqrt3\,\pi b^2}{(1+2b)^2}+O(1/L)$, so the coefficient of
$\log\log N$ is $0$.
\end{enumerate}
\end{example}

\begin{proof}[Proof]
(i) A path with support $\{1,2\}$ is $1^{n_1}2^{n_2}$, of probability
$(1-bh)^{n_1-1}\,bh\,(1-h)^{n_2-1}$. Moving one visit from state $2$ to state $1$
multiplies it by $(1-bh)/(1-h)<1$, so the maximum is at $n_1=1$. On the full face $Q$
is a generator, so $\lambda=0$, $r=\ones$, $l=\alpha^*$ and $\alpha^*$ is the stationary law
$(1,b,b)/(1+2b)$. So $(\mu\cdot r)(l\cdot\ones)=1$ and
$K_{\{1,2,3\}}=(2\pi)^{-1}(\det\Sigma_{\alpha^*})^{-1/2}$. By the inequality of the
arithmetic and geometric means, the infimum of
$b\alpha_1u_2/u_1+\alpha_2u_3/u_2+\alpha_3u_1/u_3$ over $u>0$ is $3\phi(\alpha)$ with
$\phi(\alpha)=(b\alpha_1\alpha_2\alpha_3)^{1/3}$. So
$J_V(\alpha)=b\alpha_1+\alpha_2+\alpha_3-3\phi(\alpha)$ by~\eqref{eq:JV}, and by
Proposition~\ref{prop:constant}(a), $\det\Sigma_{\alpha^*}$ is the inverse of the determinant
of the Hessian of $J_V$ in $(\alpha_2,\alpha_3)$ at $\alpha^*$. In all three coordinates the
Hessian of $-3\phi$ is $\phi D^{-1}(I-\frac13\ones\ones^\top)D^{-1}$ with
$D=\operatorname{diag}\alpha$. At $\alpha^*=(c,bc,bc)$ with $c=1/(1+2b)$ we have $\phi=bc$, and
the substitution $\alpha_1=1-\alpha_2-\alpha_3$ gives the Hessian
\[
\frac{1+2b}{3b}\begin{pmatrix}2b^2+2b+2&2b^2+2b-1\\2b^2+2b-1&2b^2+2b+2\end{pmatrix},
\]
whose determinant is $(1+2b)^4/(3b^2)$. Hence $\det\Sigma_{\alpha^*}=3b^2/(1+2b)^4$, which
gives the value of $K_{\{1,2,3\}}$.

(ii) The function $\log(M_{\{1,2\}}/M_{\{1\}})=\log(bh)+(N-2)\log(1-h)-(N-1)\log(1-bh)$
has the derivative $1/h-(N-2)/(1-h)+(N-1)b/(1-bh)>0$ in $h$, since $b>1$ and
$1-bh<1-h$. It tends to $-\infty$ as $h\to0$ and to $+\infty$ as $bh\to1$, so it has
exactly one zero. Its expansion $\log b+\log T-L+(b-1)T+O(T^2/N)$ gives
$(b-1)T+\log T=L-\log b+O(T^2/N)$ at the zero, and with $\log T=\log L-\log(b-1)+O(\log L/L)$
this gives the formula.

(iii) For the full face $k=3$, so Theorem~\ref{thm:facemax} gives
$\log M_{\{1,2,3\}}=-2L+\log T+\log K_{\{1,2,3\}}+O(1/L)$ on $[L/2,2L]$, with the same
power $T^1$ as $\log M_{\{1,2\}}=\log b+\log T-L-T+O(T^2/N)$. Hence $\log T$ cancels, and
equating the two gives $T=L+\log b-\log K_{\{1,2,3\}}+O(1/L)$.
\end{proof}

\subsection{Face maxima uniformly over a class}\label{sec:uniform-class}

The constants of Theorem~\ref{thm:facemax} depend on $Q$ and $\mu$. We make them
uniform over a compact class of reversible chains on a fixed graph. The companion paper\ifarxivbuild~\cite{paperC2}\fi{} uses this statement to make its
finite-horizon recovery uniform over such a class.

Fix a connected graph $\mathcal G=([d],\mathcal E)$, numbers $0<w_*\le R$ and
$0<\mu_*\le1/d$, and a compact set $\mathcal V\subset\R^d$ of fields. A generator $Q$ is
\emph{on $\mathcal G$} if $q_{ij}>0$ exactly when $\{i,j\}\in\mathcal E$. For a reversible
irreducible $Q$ with stationary law $\pi$ put $w_{ij}=q_{ij}q_{ji}$ and
$S_V=D_{\sqrt\pi}(-Q)D_{\sqrt\pi}^{-1}$, with $D_x=\operatorname{diag}x$, the matrix $S_F$ of
Proposition~\ref{prop:variational} for $F=V$. It is symmetric, with diagonal $q_i$ and
the entries $-\sqrt{w_{ij}}$ on edges. Let $\mathcal C$ be the set of reversible
irreducible generators on $\mathcal G$ with $\|S_V\|_2\le R$ and $\sqrt{w_{ij}}\ge w_*$ on
every edge, and let $\mathcal M$ be the set of starts with $\min_i\mu_i\ge\mu_*$. We use
$M_F(T,v)$, $\alpha^*_F(v)$ and $K_F(v)$ as defined before Theorem~\ref{thm:facemax}.

\begin{proposition}[uniform face maxima on the class]\label{prop:uniform-facemax}
Let $\mathcal G$ induce a connected graph on the face $F$, with $k=|F|\ge2$. There are
$N_0$ and $C$, depending only on $\mathcal G$, $R$, $w_*$, $\mu_*$ and $\mathcal V$, such that for all
$Q\in\mathcal C$, $\mu\in\mathcal M$ and $v\in\mathcal V$ the following hold.
\begin{enumerate}[(a)]
\item For all $N\ge2$ and $T>0$ with $TR\le N/2$,
$M_F(T,v)\le C(T/N)^{k-1}e^{-T\lambda_F(v)+CT^2/N}$.
\item For $N\ge N_0$ and $1\le T\le N^{1/3}$, we have $C^{-1}\le K_F(v)\le C$ and
\[
 M_F(T,v)=K_F(v)\,N^{-(k-1)}T^{(k-1)/2}e^{-T\lambda_F(v)}(1+E),\qquad
 |E|\le C\Bigl(\frac1T+\frac{T^2}N\Bigr).
\]
\end{enumerate}
For a singleton, $M_{\{i\}}(T,v)=\mu_i(1-Tq_i/N)^{N-1}e^{Tv_i}$ exactly.
\end{proposition}

\begin{proof}
For fixed $(Q,\mu)$, part (b) is Theorem~\ref{thm:facemax}(b) with the compact
set $\mathcal V$ of fields, whose constants already do not depend on
$v\in\mathcal V$. We use its continuum route, given at the end of the proof of
Theorem~\ref{thm:facemax}. Part (a) is~\eqref{eq:upper-field}, that is, Lemma~\ref{lem:upper}
with the Perron tilt $\theta=q-v-\lambda_F(v)\ones\in V_F$, and the singleton is
the path $i^N$. It remains to show that no constant depends on $(Q,\mu)$ either.

\emph{The parameter set is compact.} Let $\mathcal P=\mathcal C\times\mathcal M\times\mathcal V$.
Since $S_V$ is symmetric and similar to $-Q$, it is positive semidefinite, so
$0\le q_l\le R$ for every $l$ and $\sqrt{w_{ij}}\le R$ on every edge.
On an edge $q_{ij}\ge w_*^2/R$, since $q_{ij}q_{ji}\ge w_*^2$ and $q_{ji}\le R$, and off the
edges $q_{ij}=0$. Reversibility is the closed condition of Kolmogorov on cycle products,
and $\|S_V\|_2\le R$ is closed since $\pi$ is continuous. Hence $\mathcal P$ is compact, and
every connected face is irreducible for every $Q\in\mathcal C$.

\emph{Joint regularity.} The Perron root is simple, so $\lambda_F(v)$, the normalized
Perron vectors and $\alpha^*_F(v)$ are real-analytic in $(Q,v)$. The proof of
Proposition~\ref{prop:constant}(a), run with $Q$ as an extra variable, makes
$\theta(\alpha)$, $r_\alpha$, $l_\alpha$, $\Sigma_\alpha$, $J_F$ and $C_F$ real-analytic in
$(\alpha,Q)$, and $C_F$ is linear in $\mu$. Hence $K_F(v)$ is continuous and positive on
$\mathcal P$, which gives its bounds, and there is $\varepsilon_0>0$ with
$\alpha^*_{F,i}(v)\ge2k\varepsilon_0$ on $\mathcal P$. Put
$\mathcal K_0=\{\alpha\in\Delta_F:\alpha_i\ge\varepsilon_0\ \forall i\}$.

\emph{The constants.} They are the following, and each is continuous on $\mathcal P$ or on
$\mathcal P\times\mathcal K_0$, or comes from a compactness argument run there.
(i) The constant $C_\theta$ of Lemma~\ref{lem:upper} is built from the factor $b(z')$ of
Lemma~\ref{lem:torus}. This factor is continuous where the
radius condition~\eqref{eq:A-radii} holds strictly, as it does on $V_F$. So it is continuous along the Perron tilt.
(ii) In Theorem~\ref{thm:local}(a) the Perron root in~\eqref{eq:A-radii} at $\theta'(\alpha)$ is
continuous and negative on $\mathcal P\times\mathcal K_0$, so one radius margin serves, and
Lemma~\ref{lem:laplace} applies on the compact space $\mathcal P\times\mathcal K_0$, since
the matrix $\mathcal J$ of Appendix~\ref{app:torus} is continuous and positive definite and
$\chi(\alpha,0)=(\mu\cdot r_\alpha)(l_\alpha\cdot\ones)\prod_{i\ge2}\theta_i(\alpha)/\alpha_1$ is
continuous and positive.
(iii) Step~3 of the proof of Lemma~\ref{lem:discrete} needs only $\Upsilon(1)\le CT$, with $C$
uniform on $\mathcal P\times\mathcal K_0$, and we prove this directly. The upper bound
$H_e(T\alpha)\le e^{k-1}b(\theta')\prod_{i\ge2}\theta_i\,e^{eT\langle\alpha,\theta\rangle}$ of that
step, with $\theta=\theta(\alpha)$, has a constant that is continuous on
$\mathcal P\times\mathcal K_0$ by (i). For the lower bound take a walk $c^0$ in $\mathcal G$ through
every state of $F$ with at most $2k-1$ letters, for instance along a depth-first search of a
spanning tree of the graph that $\mathcal G$ induces on $F$, and let $m^0$ count its letters.
Every start has $\mu_{c^0_1}\ge\mu_*$ and every edge has $q_{ij}\ge w_*^2/R$, so for $T\ge1$ and
$\alpha\in\mathcal K_0$ the term of $m^0$ in $H(T\alpha)$ is at least
\[
\mu_*\min\Bigl(1,\frac{w_*^2}R\Bigr)^{2k-2}\prod_{i\in F}\frac{(T\alpha_i)^{m^0_i-1}}{(m^0_i-1)!}
\ge\mu_*\min\Bigl(1,\frac{w_*^2}R\Bigr)^{2k-2}\frac{\min(1,\varepsilon_0)^{k-1}}{(k-1)!}=c_0>0,
\]
since $\sum_i(m^0_i-1)\le k-1$ and $T\alpha_i\ge\varepsilon_0$. Hence
$\Upsilon(1)=\log H_e(T\alpha)-\log H(T\alpha)\le C+eT\max\theta_i-\log c_0\le C'T$, uniformly on
$\mathcal P\times\mathcal K_0$. Step~4 uses only $q_{\max}\le R$ and
$n_{\min}\ge\varepsilon_0N$.
(iv) In the continuum route to Theorem~\ref{thm:facemax}(b) the bounds for $\Sigma_\alpha^{-1}$, for the third
derivatives of $J_F$ and for two derivatives of $\log C_F$ are continuous on
$\mathcal P\times\mathcal K_0$, which gives one radius $\rho_0\le\varepsilon_0$ and uniform
$c_U$, $C_U$, $C_1$ and $C_3$. Away from the ball, $J_F(\beta)-v\cdot\beta-\lambda_F(v)$ is
lower semicontinuous and positive on a compact subset of $\mathcal P\times\Delta_F$, hence at
least $3\eta$, and finitely many vectors $\theta_u=A_Fu/u\in V_F$ serve, with constants
bounded by (i). For bounded $T$, $f_T$ is continuous and positive on
$[1,T_2]\times\mathcal P\times\mathcal K_0$.
\end{proof}

\section{The boundary layer at a facet}\label{sec:boundary-layer}

The local law of Theorem~\ref{thm:local}(a) holds on compact subsets of the open face. Hence
it says nothing near the boundary. In this section we follow the occupation density of the
continuous-time chain from the inside of a face $F$ up to a facet $G=F\setminus\{i\}$, where the
time spent in $i$ is small. We state the main results for the density in
Section~\ref{sec:layer-main}, after the setting, and prove them in
Sections~\ref{sec:layer-exact} to~\ref{sec:layer-proofs}. Section~\ref{sec:layer-counts} gives
the same law for the counts, with a lattice kernel.

\subsection{Setting}\label{sec:layer-setting}

Let $F$ be an irreducible face with $k=|F|\ge2$, let $i\in F$, and put $G=F\setminus\{i\}$.
We assume that $G$ satisfies (H), so $G$ is irreducible and $\mu(G)>0$. The start mass
$\mu_i$ may be zero. For $j\in G$ put
\[
c^-_j=q_{ji},\qquad c^+_j=q_{ij},
\]
the rates into $i$ and out of $i$. Both vectors are nonnegative and nonzero, since $F$ is
irreducible. For $z\in\mathbb C$ let
\[
Q_G(z)=Q_G+z\,c^-(c^+)^\top,\qquad \mu^{(z)}=\mu_G+z\mu_ic^+,\qquad e^{(z)}=\ones+zc^- ,
\]
and write $Q_G(z)=A(z)-\operatorname{diag}d(z)$ with $A(z)$ the off-diagonal part and
$d(z)=q_G-z\,c^-\circ c^+$. A jump $j\to j'$ of $Q_G(z)$ of rate $zc^-_jc^+_{j'}$ stands for a
visit to $i$ that is entered from $j$ and left to $j'$.

\emph{Occupation coordinates.} Let $E_F$ be the event of Lemma~\ref{lem:continuum}. On $E_F$
let $s\in(0,\infty)^G$ be the times spent in the states of $G$, $S=\sum_{j\in G}s_j$, and
$x=T-S$ the time spent in $i$. We write
\[
f_F(s,x)=e^{-\langle q_G,s\rangle-q_ix}H_F(s,x),\qquad f_G(s)=e^{-\langle q_G,s\rangle}H_G(s),
\]
where $H_F$ and $H_G$ are the sums $H$ of Lemma~\ref{lem:continuum} for the faces $F$ and $G$.
By that lemma $f_F$ is the density of the occupation times on $E_F$ in any $k-1$ of the
coordinates, and $f_G(s)$ is the same density for $G$ at total time $S$. In this
notation Theorem~\ref{thm:local}(a) for the face $G$ reads
\begin{equation}\label{eq:layer-local-G}
f_G(S\beta)=S^{-(k-2)/2}C_G(\beta)e^{-SJ_G(\beta)}\bigl(1+O_{\mathcal K}(1/S)\bigr)
\end{equation}
uniformly for $\beta$ in a compact $\mathcal K\subset\Delta_G^\circ$ and $S\ge1$, with
$C_G(\beta)=\mu_j$ and $J_G=q_j$ when $G=\{j\}$. The profile on the facet is
$\beta=s/S\in\Delta_G^\circ$.

\emph{The weighted chain.} For real $z\ge0$ the matrix $Q_G(z)$ is Metzler and irreducible.
We let $J_G(\beta,z)$ be its rate, the supremum in Proposition~\ref{prop:rate} with $Q_G(z)$
in place of $Q_F$. Its local-law constant $C_G(\beta,z)$ is Definition~\ref{def:tilt} for the
off-diagonal part $A(z)$, with the start $\mu^{(z)}$ in place of $\mu$ and the end vector
$e^{(z)}$ in place of $\ones$. Namely
\[
C_G(\beta,z)=\frac{(\mu^{(z)}\cdot r_\beta(z))(l_\beta(z)\cdot e^{(z)})}{(2\pi)^{(k-2)/2}\sqrt{\det\Sigma_\beta(z)}} .
\]
So $J_G(\beta,0)=J_G(\beta)$ and $C_G(\beta,0)=C_G(\beta)$.

\emph{The facet constants.} Let $r_\beta=r_\beta(0)$ and $l_\beta=\beta/r_\beta$ be the vectors
of Definition~\ref{def:tilt} for the face $G$, so $l_\beta\cdot r_\beta=1$. Put
\[
a_\beta=c^-\cdot l_\beta,\qquad b_\beta=c^+\cdot r_\beta,\qquad \omega_\beta=a_\beta b_\beta,\qquad
\pi^{\rm in}_{\beta,j}=\frac{c^-_jl_{\beta,j}}{a_\beta},\qquad \pi^{\rm out}_{\beta,j}=\frac{c^+_jr_{\beta,j}}{b_\beta}.
\]
Note that only $\omega_\beta$ is intrinsic, since $a_\beta$ and $b_\beta$ move in opposite ways when
$r_\beta$ is scaled. The law $\pi^{\rm in}_\beta$ is the law of the state from which the
tilted chain enters $i$, and $\pi^{\rm out}_\beta$ the law of the state into which it leaves
$i$. Let $\mathcal L_\beta$ be the Laplacian on $G$ with the conductances
$\kappa_{jj'}=l_{\beta,j}q_{jj'}r_{\beta,j'}+l_{\beta,j'}q_{j'j}r_{\beta,j}$ for $j\ne j'$,
which are twice those of the additive reversibilization
$\frac12(\widehat Q_\beta+\widehat Q_\beta^*)$ of the tilted generator $\widehat Q_\beta$, where
$\widehat Q_\beta^*$ is its time reversal with respect to $\beta$. Its kernel is $\R\ones$, since $G$ is irreducible. Let
$\mathcal L_\beta^+$ be its pseudo-inverse and
\begin{equation}\label{eq:layer-gamma}
\gamma_\beta=q_i+\frac{\omega_\beta}2\,(\pi^{\rm out}_\beta-\pi^{\rm in}_\beta)^\top\mathcal L_\beta^+(\pi^{\rm out}_\beta-\pi^{\rm in}_\beta).
\end{equation}
The quadratic form is the effective resistance of the network $\kappa$ between the laws
$\pi^{\rm in}_\beta$ and $\pi^{\rm out}_\beta$. Hence $\gamma_\beta\ge q_i$, with equality exactly
when the two laws agree. Finally put
\[
\zeta=2\sqrt{\omega_\beta Sx},\qquad z_*=\sqrt{\frac x{\omega_\beta S}},\qquad
\mathbf B_\beta(x,S)=e^{-\gamma_\beta x}\sqrt{\frac{\omega_\beta S}x}\,I_1(\zeta),
\]
with the modified Bessel function $I_1$. We call $\mathbf B_\beta$ the Bessel kernel.

We extend the rate on $F$ to occupation times by homogeneity of degree one, and write
$\bar J_F(s,x)=TJ_F((s,x)/T)$ with $T=S+x$. So $\bar J_F(S\beta,x)=S\bar J_F(\beta,x/S)$.

\emph{The generalized kernel.} Fix a compact set $\mathcal K\subset\Delta_G^\circ$.
Theorem~\ref{thm:L-facet} below extends $J_G(\beta,\cdot)$ and $C_G(\beta,\cdot)$ to holomorphic
functions on a disc $|z|<2z_0$, for $\beta\in\mathcal K$, with $z_0>0$ given there. With these extensions we define, for
$\beta\in\mathcal K$ and $x,S>0$, the generalized kernel
\[
\mathbf K_\beta(x,S)=e^{-q_ix}\,\frac1{2\pi\mathrm i}\oint_{|z|=\rho}e^{x/z}\,\frac{C_G(\beta,z)}{C_G(\beta,0)}\,e^{S(J_G(\beta)-J_G(\beta,z))}\,\frac{dz}{z^2},
\]
with $0<\rho<2z_0$. The integrand is holomorphic on the punctured disc, so the integral does not
depend on $\rho$.

\subsection{The main results}\label{sec:layer-main}

\begin{theorem}[the facet layer]\label{thm:layer}
Let $F$ be irreducible with $k=|F|\ge2$, let $i\in F$, let $G=F\setminus\{i\}$ satisfy
\textup{(H)}, and let $\mathcal K\subset\Delta_G^\circ$ be compact. Then there
are $\delta>0$, $T_0$ and $C$ such that the following hold uniformly for $\beta\in\mathcal K$, $T\ge T_0$ and
$0<x\le\delta T$, with $S=T-x$ and $s=S\beta$.
\begin{enumerate}[(a)]
\item $f_F(s,x)=f_G(s)\,\mathbf K_\beta(x,S)\,(1+E)$ with
$|E|\le C(T^{-2}+x^{1/2}T^{-3/2})$, so $|E|\le2C/T$.
\item $f_F(s,x)=f_G(s)\,\mathbf B_\beta(x,S)\,e^{E}$ with $|E|\le C\bigl(1/T+\sqrt{x/T}\,(1+x)\bigr)$.
\item For the $\beta\in\mathcal K$ with $\pi^{\rm in}_\beta=\pi^{\rm out}_\beta$, and so for all
$\beta\in\mathcal K$ if $Q$ is reversible on $F$, we have $\gamma_\beta=q_i$, (b) holds with
$|E|\le C(1/T+\sqrt{x/T})$, and
\[
f_F(s,x)=f_G(s)\,\mathbf B_\beta(x,S)\Bigl(1+c_1(\beta)\,z_*\frac{I_0(\zeta)}{I_1(\zeta)}+E'\Bigr),\qquad|E'|\le C\,\frac{1+x}T ,
\]
where $c_1(\beta)=\partial_z\log C_G(\beta,z)$ at $z=0$.
\end{enumerate}
For $k=2$, (a) holds with $E=0$, and Proposition~\ref{prop:pair-face} gives (b) for all
$x>0$ with an explicit bound.
\end{theorem}

For reversible $Q$, \eqref{eq:layer-c1} below gives $c_1(\beta)$ explicitly. Part (b) gives
the relative Bessel form $f_F=f_G\mathbf B_\beta(1+o(1))$ for $x=o(T^{1/3})$, and beyond that a
logarithmic form. To find the exact range of the relative form we put
\[
\mathfrak r_\beta(\varepsilon)=\bar J_F(\beta,\varepsilon^2)-q_i\varepsilon^2-J_G(\beta)+2\sqrt{\omega_\beta}\,\varepsilon-(\gamma_\beta-q_i)\varepsilon^2 .
\]
Hence $S\mathfrak r_\beta(\sqrt{x/S})$ is the part of the rate $\bar J_F(s,x)$ that the exponent
$SJ_G(\beta)+\gamma_\beta x-2\sqrt{\omega_\beta Sx}$ of $f_G\mathbf B_\beta$ does not capture.

\begin{proposition}[the range of the Bessel form]\label{prop:layer-cubic}
Let $F$, $i$, $G$ and $\mathcal K$ be as in Theorem~\ref{thm:layer}.
\begin{enumerate}[(a)]
\item The function $\mathfrak r_\beta$ is real-analytic near $\varepsilon=0$, jointly with $\beta$, and
$\mathfrak r_\beta(\varepsilon)=\varkappa_\beta\varepsilon^3+O(\varepsilon^4)$ uniformly on $\mathcal K$, where
\[
\varkappa_\beta=\frac{\partial_z^3J_G(\beta,0)}{6\,\omega_\beta^{3/2}}+\frac{(\gamma_\beta-q_i)^2}{\omega_\beta^{1/2}} .
\]
\item There are $\delta>0$, $T_0$ and $C$ such that, uniformly for $\beta\in\mathcal K$,
$T\ge T_0$ and $0<x\le\delta T$, with $S=T-x$ and $s=S\beta$,
\[
\Bigl|\log\frac{f_F(s,x)}{f_G(s)\,\mathbf B_\beta(x,S)}+S\,\mathfrak r_\beta\Bigl(\sqrt{\frac xS}\Bigr)\Bigr|
\le C\Bigl(\frac1T+\sqrt{\frac xT}\Bigr).
\]
In particular the logarithm is $-\varkappa_\beta x^{3/2}S^{-1/2}+O(1/T+\sqrt{x/T}+x^2/T)$.
\item If $\pi^{\rm in}_\beta=\pi^{\rm out}_\beta$, then $\mathfrak r_\beta=0$ near $0$. This holds for every
$\beta$ if $Q$ is reversible on $F$.
\item Fix $\beta\in\mathcal K$ with $\mathfrak r_\beta\ne0$, and let $m\ge3$ be the order of the first
nonzero Taylor coefficient of $\mathfrak r_\beta$ at $0$. There is $\delta_\beta>0$ such that for
$0<x=x(T)\le\delta_\beta T$, the ratio $f_F(s,x)/(f_G(s)\mathbf B_\beta(x,S))$ tends to $1$ as
$T\to\infty$ if and only if $x/T^{1-2/m}\to0$. If $\mathfrak r_\beta=0$ near $0$, the ratio tends to $1$
whenever $x/T\to0$.
\item For the cycle with $q_{12}=q_{23}=q_{31}=2$ and $q_{21}=q_{32}=q_{13}=1$, a start with
$\mu_1+\mu_2>0$, $i=3$ and $\beta=(\frac12,\frac12)$,
\[
\omega_\beta=2+\tfrac94\sqrt2,\qquad \partial_z^3J_G(\beta,0)=-\tfrac{1323}{64}\sqrt2,\qquad
\varkappa_\beta=-\frac{49}{64\sqrt{\omega_\beta}}=-0.3363\ldots,
\]
so $m=3$ there, and the ratio tends to infinity when $x/T^{1/3}\to\infty$ with
$x\le\delta_\beta T$.
\end{enumerate}
\end{proposition}

The condition in (c) is weaker than reversibility. Let $G$ be the 3-state cycle with the rates of (e),
$q_{12}=q_{23}=q_{31}=2$ and $q_{21}=q_{32}=q_{13}=1$, and add a fourth state $i$ with
$q_{ji}=q_{ij}=c>0$ for $j\in G$, and a start with $\mu(G)>0$. The chain on $F=G\cup\{i\}$ is not reversible, since the
products of the rates around the cycle are $8$ and $1$. At the uniform $\beta$, the criterion in
the proof of Proposition~\ref{prop:rate} gives $r_\beta\propto\ones$, since the uniform law is
stationary for the cycle, and then $l_\beta\propto\beta$. Hence $\pi^{\rm in}_\beta$ and
$\pi^{\rm out}_\beta$ are both uniform, and the relative Bessel form holds for $x=o(T)$ there.
At other $\beta$ the two laws differ.

\emph{Strategy of the proofs.} We prove the results in three steps. First, pulling the visits to $i$
out of the run expansion writes $f_F$ as a contour integral of the density $f^{(z)}_G$ of a
weighted chain on $G$, in which $z$ counts the visits to $i$
(Lemma~\ref{lem:layer-exact}). Second, a local law for the weighted chain, uniform for complex
$z$ near $0$, replaces $f^{(z)}_G(S\beta)$ by $S^{-(k-2)/2}C_G(\beta,z)e^{-SJ_G(\beta,z)}$
(Theorem~\ref{thm:L-facet}). Third, a saddle point on the circle $|z|=z_*$ evaluates the
integral, and the expansion of $J_G(\beta,z)$ at $z=0$ (Proposition~\ref{prop:layer-constants})
compares the result with the Bessel kernel. Sections~\ref{sec:layer-exact}
to~\ref{sec:layer-prefactor} give the exact formulas, the constants and the matching of the
prefactor, Section~\ref{sec:layer-theorem} the second step, and Section~\ref{sec:layer-proofs}
the third.

\subsection{The exact representation}\label{sec:layer-exact}

\begin{lemma}[the runs in $i$ pulled out]\label{lem:layer-exact}
For $z\in\mathbb C$ and $s\in(0,\infty)^G$ let
\begin{gather*}
f^{(z)}_G(s)=e^{-\langle d(z),s\rangle}\sum_{m\in\Z_{\ge1}^G}W_z(m)\prod_{j\in G}\frac{s_j^{m_j-1}}{(m_j-1)!},\\
W_z(m)=\sum_w\mu^{(z)}_{w_1}\prod_{l<|w|}A(z)_{w_lw_{l+1}}\,e^{(z)}_{w_{|w|}},
\end{gather*}
the inner sum over the words $w$ in $G$ with no two equal neighbors and $m_j$ letters $j$.
Then $z\mapsto f^{(z)}_G(s)$ is entire, $f^{(0)}_G=f_G$, and for $x>0$ and every $\rho>0$
\begin{align}
f_F(s,x)&=e^{-q_ix}\sum_{K\ge1}\frac{x^{K-1}}{(K-1)!}\,[z^K]f^{(z)}_G(s)\label{eq:layer-series}\\
&=e^{-q_ix}\,\frac1{2\pi\mathrm i}\oint_{|z|=\rho}e^{x/z}f^{(z)}_G(s)\,\frac{dz}{z^2}.\label{eq:layer-contour}
\end{align}
\end{lemma}

Hence $f^{(z)}_G$ is the killed density on the facet of the chain in which each visit to $i$ is
replaced by an instantaneous jump $j\to j'$ of rate $zc^-_jc^+_{j'}$, and $z$ counts the
visits. The $K$ visits have total length $x$, with density $q_i^Kx^{K-1}e^{-q_ix}/(K-1)!$,
and each return carries a factor $c^+_{j'}/q_i$. Thus the powers of $q_i$ cancel.

\begin{proof}
Fix a word $c$ in $F$ with no equal neighbors, every letter of $F$, and $K=m_i(c)\ge1$
letters $i$. Its term in $H_F(s,x)$ is $\mu_{c_1}\prod q_{c_lc_{l+1}}$ times
$x^{K-1}/(K-1)!$ times $\prod_{j\in G}s_j^{m_j-1}/(m_j-1)!$. The letters $i$ are isolated in
$c$. We delete them and mark the places where they were. The result is a marked word $c'$ in
$G$, in which two equal neighbors occur exactly where $c$ had $j\,i\,j$. A mark between
letters $j$ and $j'$ carries $q_{ji}q_{ij'}=c^-_jc^+_{j'}$. A mark before the first letter $j'$
carries $\mu_ic^+_{j'}$, and a mark after the last letter $j$ carries $c^-_j$. Two neighbors
without a mark are different and carry $q_{jj'}$. So the weight of $c$ is the product of these
weights. Between two different letters of $c'$ there is at most one mark, and between two
equal letters exactly one.

We group the marked words by their reduced word $w$, obtained by merging each block of equal
letters of $c'$ into one letter, and by the numbers $n_1,n_2,\dots$ of marks inside the
blocks. If $j$ forms $\varrho$ blocks with $n$ inner marks in all, it has $\varrho+n$ letters
in $c'$, and there are $\binom{n+\varrho-1}{\varrho-1}$ ways to spread the marks. Since
$\binom{n+\varrho-1}{\varrho-1}/(n+\varrho-1)!=1/(n!(\varrho-1)!)$,
\[
\sum_{n\ge0}\binom{n+\varrho-1}{\varrho-1}\frac{y^ns^{n+\varrho-1}}{(n+\varrho-1)!}=e^{ys}\frac{s^{\varrho-1}}{(\varrho-1)!}.
\]
With $y=zc^-_jc^+_j$, one power of $z$ for each mark, the inner marks give the factor
$e^{zc^-_jc^+_js_j}$, which accounts for the term $-z\,c^-\circ c^+$ of $d(z)$. The marks between different
letters give $A(z)$, and the first and last marks give $\mu^{(z)}$ and $e^{(z)}$. Collecting
the powers of $z$ gives~\eqref{eq:layer-series}. The entries of $A(z)$, $\mu^{(z)}$ and
$e^{(z)}$ are at most $q_{\max}(1+|z|q_{\max})$, $1+|z|q_{\max}$ and $1+|z|q_{\max}$ in
modulus. So the bound $W(m)\le k^Mq_{\max}^{M-1}$ of Section~\ref{sec:model} holds for $W_z$
with $q_{\max}$ replaced by $q_{\max}(1+|z|q_{\max})$, up to a factor $(1+|z|q_{\max})^2$, and
the series for $f^{(z)}_G(s)$ converges locally uniformly in $z$. Hence it is entire. For
$z=0$ no word carries a mark, and $f^{(0)}_G$ is the sum of Lemma~\ref{lem:continuum} for $G$.
Finally we expand $e^{x/z}=\sum_{n\ge0}x^nz^{-n}/n!$ and integrate term by term on the circle.
This gives $\sum_{n\ge0}(x^n/n!)[z^{n+1}]f^{(z)}_G(s)$, which is~\eqref{eq:layer-series}.
\end{proof}

The discrete law has the same structure. The only change is that the visits to $i$ come in
blocks of steps, and the lengths of the blocks form a composition of $n_i$.

\begin{lemma}[the discrete form]\label{lem:layer-exact-discrete}
Let $hq_{\max}\le1$, $\supp n=F$, $N'=N-n_i$, and let $\hat n$ be $n$ restricted to $G$. For
$z\in\mathbb C$ let
\[
P^{(z)}_{N'}(\hat n)=\sum_{x}\mu^{h,(z)}_{x_1}\prod_{t<N'}\bigl(I+hQ_G+zh^2c^-(c^+)^\top\bigr)_{x_tx_{t+1}}e^{h,(z)}_{x_{N'}},
\]
the sum over the words $x_1\cdots x_{N'}$ in $G$ with composition $\hat n$, where
$\mu^{h,(z)}=\mu_G+zh\mu_ic^+$ and $e^{h,(z)}=\ones+zhc^-$. Then $P^{(0)}_{N'}(\hat n)$ is the
probability of the composition $\hat n$ for the chain run for $N'$ steps, and
\[
P_N(n)=\sum_{K\ge1}\binom{n_i-1}{K-1}(1-hq_i)^{n_i-K}\,[z^K]P^{(z)}_{N'}(\hat n).
\]
\end{lemma}

\begin{proof}
We delete the maximal blocks of $i$ from a path with composition $n$. What is left is a path
$x_1\cdots x_{N'}$ in $G$ with composition $\hat n$. A block of length $\ell$ between
$x_t=j$ and $x_{t+1}=j'$ replaces the factor $(I+hQ)_{jj'}$ by
$hq_{ji}(1-hq_i)^{\ell-1}hq_{ij'}$. A block at the start replaces $\mu_{x_1}$ by
$\mu_i(1-hq_i)^{\ell-1}hq_{ix_1}$, and a block at the end adds the factor
$hq_{x_{N'}i}(1-hq_i)^{\ell-1}$. We count the blocks by $z$. If there are $K$ blocks, the
factors $(1-hq_i)^{\ell-1}$ give $(1-hq_i)^{n_i-K}$, and the block lengths form one of the
$\binom{n_i-1}{K-1}$ compositions of $n_i$ into $K$ positive parts.
\end{proof}

\subsection{Pair faces}\label{sec:layer-pair}

For $k=2$ the facet is one state, so the weighted chain has no words to sum. Hence the contour
formula gives the law in closed form.

\begin{proposition}[pair faces]\label{prop:pair-face}
Let $F=\{j,i\}$ with $q_{ji},q_{ij}>0$ and $\mu_j>0$, and put $\omega=q_{ji}q_{ij}$ and
$\zeta=2\sqrt{\omega sx}$. For all $s,x>0$
\[
f_F(s,x)=\mu_je^{-q_js-q_ix}\Bigl[\sqrt{\frac{\omega s}x}\,I_1(\zeta)+\Bigl(q_{ji}+\frac{\mu_i}{\mu_j}q_{ij}\Bigr)I_0(\zeta)
+\frac{\mu_i}{\mu_j}\sqrt{\frac{\omega x}s}\,I_1(\zeta)\Bigr].
\]
Here $\omega_\beta=\omega$, $\gamma_\beta=q_i$ and $f_G(s)=\mu_je^{-q_js}$, and for all
$s,x>0$
\[
0\le\frac{f_F(s,x)}{f_G(s)\mathbf B_\beta(x,s)}-1\le\Bigl(q_{ji}+\frac{\mu_i}{\mu_j}q_{ij}\Bigr)\Bigl(\frac1{\omega s}+\sqrt{\frac x{\omega s}}\Bigr)+\frac{\mu_i}{\mu_j}\,\frac xs .
\]
\end{proposition}

This is the facet form of Pedler's formula for two-state chains~\cite{pedler1971}. For
$x\le T/2$ the bound is $O(1/T+\sqrt{x/T})$, and the term $\sqrt{x/T}$ is attained, since the
middle term of the formula has exactly this size.

\begin{proof}
For $G=\{j\}$ the only word is $j$, so
$f^{(z)}_G(s)=(\mu_j+z\mu_iq_{ij})(1+zq_{ji})e^{-(q_j-z\omega)s}$. With $z=z_*t$ and
$z_*=\sqrt{x/(\omega s)}$ we have $x/z+\omega sz=\frac\zeta2(t+1/t)$, and the generating
function $e^{(\zeta/2)(t+1/t)}=\sum_{n\in\Z}I_n(\zeta)t^n$~\cite[Eq.~10.35.1]{dlmf} gives
\begin{equation}\label{eq:layer-bessel-moments}
\frac1{2\pi\mathrm i}\oint e^{x/z+\omega sz}z^{n-2}\,dz=z_*^{n-1}I_{n-1}(\zeta)\qquad(n\in\Z).
\end{equation}
We put this into~\eqref{eq:layer-contour} for $n=0,1,2$ and use $I_{-1}=I_1$. This gives the
formula. Each term is nonnegative, which gives the lower bound. For the upper bound we divide
the last two terms by the first. This gives
$(q_{ji}+\frac{\mu_i}{\mu_j}q_{ij})z_*I_0(\zeta)/I_1(\zeta)+\frac{\mu_i}{\mu_j}x/s$. The lower
bound $I_1(\zeta)/I_0(\zeta)\ge\zeta/(1+\sqrt{1+\zeta^2})$ of Amos~\cite{amos1974} gives
$I_0/I_1\le(1+\sqrt{1+\zeta^2})/\zeta\le1+2/\zeta$, and $z_*(1+2/\zeta)=\sqrt{x/(\omega s)}+1/(\omega s)$.
For a pair face, $\beta=(1)$ and the vectors of the facet are $r_\beta=l_\beta=1$, so
$\omega_\beta=\omega$ and $\pi^{\rm in}_\beta=\pi^{\rm out}_\beta$, which gives
$\gamma_\beta=q_i$.
\end{proof}

\subsection{The constants for every irreducible chain}\label{sec:layer-constants}

For $u\in(0,\infty)^G$ put
\[
\Psi_\beta(u)=\sum_{j\in G}\beta_j\Bigl(q_j-\sum_{j'\in G\setminus\{j\}}q_{jj'}\frac{u_{j'}}{u_j}\Bigr),\qquad
P_\beta(u)=(c^+\cdot u)\sum_{j\in G}\frac{c^-_j\beta_j}{u_j}.
\]
Since $\sum_j\beta_j(-(Q_G(z)u)_j/u_j)=\Psi_\beta(u)-zP_\beta(u)$, the rate of the weighted
chain is $J_G(\beta,z)=\sup_u[\Psi_\beta(u)-zP_\beta(u)]$ for real $z\ge0$.

\begin{proposition}[the facet expansion of the rate]\label{prop:layer-constants}
Let $F$ be irreducible, $i\in F$, and $G=F\setminus\{i\}$ irreducible.
\begin{enumerate}[(i)]
\item For $s=S\beta$ with $\beta\in\Delta_G^\circ$ and $x>0$,
\[
\bar J_F(s,x)=q_ix+\sup_{z>0}\bigl[S\,J_G(\beta,z)-x/z\bigr].
\]
\item The map $z\mapsto J_G(\beta,z)$ is real-analytic near $z=0$, jointly with $\beta$, and
\[
J_G(\beta,0)=J_G(\beta),\qquad \partial_zJ_G(\beta,0)=-\omega_\beta,\qquad
\partial_z^2J_G(\beta,0)=2\omega_\beta(\gamma_\beta-q_i).
\]
\item There are $\delta>0$ and $C$ such that, uniformly for $\beta$ in a compact subset of
$\Delta_G^\circ$ and $0<x\le\delta S$,
\[
\Bigl|\bar J_F(s,x)-\bigl(S\,J_G(\beta)+q_ix-2\sqrt{\omega_\beta Sx}+(\gamma_\beta-q_i)x\bigr)\Bigr|\le C\,x^{3/2}S^{-1/2}.
\]
\item If $\nu_jq_{jj'}=\nu_{j'}q_{j'j}$ on $F$ for some $\nu\in(0,\infty)^F$, then
$\pi^{\rm in}_\beta=\pi^{\rm out}_\beta$, $\gamma_\beta=q_i$, $J_G(\beta,z)=J_G(\beta)-z\omega_\beta$
for all real $z\ge0$, $\omega_\beta=\bigl(\sum_{j\in G}\sqrt{q_{ij}q_{ji}\beta_j}\bigr)^2$, and
for all $s=S\beta$ and $x>0$
\[
\bar J_F(s,x)=S\,J_G(\beta)+q_ix-2\sqrt{\omega_\beta Sx}.
\]
\item The map $z\mapsto J_G(\beta,z)$ is affine on $[0,\infty)$ if and only if
$\pi^{\rm in}_\beta=\pi^{\rm out}_\beta$. Then $J_G(\beta,z)=J_G(\beta)-z\omega_\beta$ for all
real $z\ge0$, and $\bar J_F(s,x)=S\,J_G(\beta)+q_ix-2\sqrt{\omega_\beta Sx}$ for all
$s=S\beta$ and $x>0$.
\end{enumerate}
\end{proposition}

Hence in the occupation fraction $\alpha_i$ of $i$ the rate has a square-root singularity
$-2\sqrt{\omega_\beta\alpha_i}$ at the facet, and the time in $i$ is charged at the rate
$\gamma_\beta$, not $q_i$.

\begin{proof}
(i) In the supremum that defines $\bar J_F(s,x)$ we write $u=(u_G,\epsilon)$ with
$\epsilon=u_i$. The terms with $\epsilon$ are
$-\epsilon\sum_js_jc^-_j/u_j-(x/\epsilon)(c^+\cdot u_G)$, and both coefficients are positive.
Their supremum over $\epsilon>0$ is $-2\sqrt{xS\,P_\beta(u_G)}$. So
$\bar J_F(s,x)=q_ix+\sup_u[S\Psi_\beta(u)-2\sqrt{xSP_\beta(u)}]$. Now
$2\sqrt{xSP}=\inf_{z>0}[SzP+x/z]$, and we exchange the two suprema.

(ii) Write $u=e^\psi$. For real $z\ge0$ the function $\psi\mapsto\Psi_\beta(e^\psi)-zP_\beta(e^\psi)$
is concave, since $P_\beta(e^\psi)$ is a positive combination of the convex functions
$e^{\psi_m-\psi_j}$, and it does not change when a constant is added to $\psi$. At $z=0$ its
maximizer is $\psi_0=\log r_\beta$, unique up to a constant (Proposition~\ref{prop:rate}(b)).
The Hessian of $\sum_{j\ne j'}\beta_jq_{jj'}e^{\psi_{j'}-\psi_j}$ at $\psi_0$ is the Laplacian
with the conductances $\kappa_{jj'}$, because $\beta_jq_{jj'}r_{\beta,j'}/r_{\beta,j}=l_{\beta,j}q_{jj'}r_{\beta,j'}$.
So the Hessian of $\Psi_\beta$ at $\psi_0$ is $-\mathcal L_\beta$, which is negative definite on
the complement of the constants. By the implicit function theorem on that complement, the
critical point $\psi(z)$ is real-analytic in $(\beta,z)$ near $(\beta,0)$, and for real $z\ge0$
it is the maximizer by concavity. So $J_G(\beta,z)=\Psi_\beta(\psi(z))-zP_\beta(\psi(z))$ is
real-analytic. The envelope theorem gives $\partial_zJ_G=-P_\beta(\psi(z))$, so
$\partial_zJ_G(\beta,0)=-(c^+\cdot r_\beta)(c^-\cdot l_\beta)=-\omega_\beta$. Differentiating the
critical equation gives $-\mathcal L_\beta\psi'(0)=\nabla P$, with $\nabla P$ the gradient of
$P_\beta(e^\psi)$ at $\psi_0$. Its entries are
\[
\partial_{\psi_m}P=c^+_mr_{\beta,m}\,(c^-\cdot l_\beta)-(c^+\cdot r_\beta)\,c^-_ml_{\beta,m}=\omega_\beta(\pi^{\rm out}_{\beta,m}-\pi^{\rm in}_{\beta,m}),
\]
which sum to zero. So $\partial_z^2J_G(\beta,0)=-\nabla P\cdot\psi'(0)=\nabla P^\top\mathcal L_\beta^+\nabla P$,
and this is $2\omega_\beta(\gamma_\beta-q_i)$ by~\eqref{eq:layer-gamma}.

(iii) By the proof of (i), $\bar J_F(s,x)-q_ix=S\sup_\psi[\Psi_\beta(\psi)-2\varepsilon P_\beta(\psi)^{1/2}]$
with $\varepsilon=\sqrt{x/S}$. The maximizer lies in a bounded set of $\psi$ modulo constants,
uniformly for $\beta$ in a compact set, as in the proof of Proposition~\ref{prop:rate}(b).
Comparing its value with the value at $\psi_0$ gives
$\Psi_\beta(\psi_0)-\Psi_\beta(\psi^*)\le2\varepsilon\sup P_\beta^{1/2}$ on that set. We
normalize $\psi$ and $\psi_0$ to sum to zero. The Hessian $-\mathcal L_\beta$ at $\psi_0$ is
negative definite on these vectors and continuous in $\beta$, so Taylor's theorem gives $c,r>0$
with $\Psi_\beta(\psi_0)-\Psi_\beta(\psi)\ge c|\psi-\psi_0|^2$ for $|\psi-\psi_0|\le r$,
uniformly on the compact set of $\beta$. For $|\psi-\psi_0|>r$, concavity along the segment from
$\psi_0$ to $\psi$ gives $\Psi_\beta(\psi_0)-\Psi_\beta(\psi)\ge cr|\psi-\psi_0|\ge cr^2$. Hence
$\Psi_\beta(\psi_0)-\Psi_\beta(\psi)\ge c\min(|\psi-\psi_0|^2,r^2)$, which gives
$\psi^*-\psi_0=O(\varepsilon^{1/2})$. The critical
equation $\nabla\Psi_\beta(\psi^*)=\varepsilon\nabla P_\beta(\psi^*)/\sqrt{P_\beta(\psi^*)}$ has a
right side of size $O(\varepsilon)$, and $\nabla^2\Psi_\beta$ is invertible near $\psi_0$ on the
complement of the constants, so $\psi^*-\psi_0=O(\varepsilon)$. We expand to
second order around $\psi_0$. With $g=\nabla(P^{1/2})=\nabla P/(2\sqrt{\omega_\beta})$, the maximizer
is $\psi^*=\psi_0-2\varepsilon\mathcal L_\beta^+g+O(\varepsilon^2)$, and the value is
$J_G(\beta)-2\varepsilon\sqrt{\omega_\beta}+2\varepsilon^2g^\top\mathcal L_\beta^+g+O(\varepsilon^3)$.
Here $2g^\top\mathcal L_\beta^+g=\nabla P^\top\mathcal L_\beta^+\nabla P/(2\omega_\beta)=\gamma_\beta-q_i$
by (ii). Multiplying by $S$ gives (iii), since $S\varepsilon^3=x^{3/2}S^{-1/2}$.

(iv) By Proposition~\ref{prop:constant}(b) for the face $G$ we may take
$r_{\beta,j}=\sqrt{\beta_j/\nu_j}$ and $l_{\beta,j}=\sqrt{\beta_j\nu_j}$. Then
$c^-_jl_{\beta,j}=q_{ji}\sqrt{\beta_j\nu_j}=\nu_iq_{ij}\sqrt{\beta_j/\nu_j}=\nu_ic^+_jr_{\beta,j}$, so
$\pi^{\rm in}_\beta=\pi^{\rm out}_\beta$ and $\gamma_\beta=q_i$. More is true. For every real
$z\ge0$ the matrix $Q_G(z)$ is reversible with respect to $\nu_G$, since
$\nu_jc^-_jc^+_{j'}=\nu_iq_{ij}q_{ij'}$ is symmetric in $j,j'$. So by
Proposition~\ref{prop:constant}(b) its optimal vector is $r_\beta$ for every $z$, and
$J_G(\beta,z)=\Psi_\beta(r_\beta)-zP_\beta(r_\beta)=J_G(\beta)-z\omega_\beta$. Next
$\omega_\beta=\nu_i\bigl(\sum_jq_{ij}\sqrt{\beta_j/\nu_j}\bigr)^2=\bigl(\sum_j\sqrt{q_{ij}q_{ji}\beta_j}\bigr)^2$,
using $q_{ji}=q_{ij}\nu_i/\nu_j$. Finally (i) gives
$\sup_{z>0}[-Sz\omega_\beta-x/z]=-2\sqrt{\omega_\beta Sx}$.

(v) If the two laws agree, the gradient $\nabla P$ in the proof of (ii) vanishes. So
$\psi_0$ is a critical point of the concave function $\psi\mapsto\Psi_\beta(e^\psi)-zP_\beta(e^\psi)$
for every real $z\ge0$, hence a maximizer, and
$J_G(\beta,z)=\Psi_\beta(r_\beta)-zP_\beta(r_\beta)=J_G(\beta)-z\omega_\beta$. Then (i) gives
$\bar J_F$ as in (iv). Conversely, if $J_G(\beta,\cdot)$ is affine, then
$\partial_z^2J_G(\beta,0)=0$, so $\gamma_\beta=q_i$ by (ii), and the two laws agree, since the
effective resistance in~\eqref{eq:layer-gamma} vanishes only when they do.
\end{proof}

For $1/T\ll x\ll T$ the Bessel kernel is
$(4\pi)^{-1/2}(\omega_\beta S)^{1/4}x^{-3/4}\exp(2\sqrt{\omega_\beta Sx}-\gamma_\beta x)(1+O(1/\zeta))$,
by $I_1(\zeta)=e^\zeta(2\pi\zeta)^{-1/2}(1+O(1/\zeta))$~\cite[Eq.~10.40.1]{dlmf}. Hence
Proposition~\ref{prop:layer-constants}(iii) says that $f_G(s)\mathbf B_\beta(x,S)$ and the local
law formula of $F$, which in occupation coordinates is $T^{-(k-1)/2}C_F(\alpha)e^{-TJ_F(\alpha)}$
(proved only on compact subsets of the open face), have
the same exponent up to $O(x^{3/2}S^{-1/2})$. For their powers of $x$ and their constants to
agree as well we need $C_F((1-\epsilon)\beta,\epsilon)\sim(4\pi)^{-1/2}\omega_\beta^{1/4}\epsilon^{-3/4}C_G(\beta)$
as $\epsilon\to0$. We prove this for reversible $Q$ in the next proposition.

\subsection{Prefactor matching}\label{sec:layer-prefactor}

\begin{proposition}[prefactor matching]\label{prop:layer-prefactor}
Let $Q$ be reversible on $F$ with respect to $\nu$, and let $G$ satisfy \textup{(H)}, with $|G|\ge1$.
Uniformly for $\beta$ in a compact subset of $\Delta_G^\circ$, as $\epsilon\to0$,
\[
C_F((1-\epsilon)\beta,\epsilon)=(4\pi)^{-1/2}\,\omega_\beta^{1/4}\,\epsilon^{-3/4}\,C_G(\beta)\bigl(1+O(\sqrt\epsilon)\bigr),
\]
with $C_G(\beta)=\mu_j$ when $G=\{j\}$.
\end{proposition}

\begin{proof}
We use Proposition~\ref{prop:constant}(b) at $\alpha=((1-\epsilon)\beta,\epsilon)$. On $G$ the
vectors $r_\alpha=\sqrt{\alpha/\nu}$ and $l_\alpha=\sqrt{\alpha\nu}$ are $\sqrt{1-\epsilon}$ times
$r_\beta$ and $l_\beta$, and at $i$ they are $O(\sqrt\epsilon)$. So
$(\mu\cdot r_\alpha)(l_\alpha\cdot\ones)=(\mu_G\cdot r_\beta)(l_\beta\cdot\ones)(1+O(\sqrt\epsilon))$,
where $\mu_G\cdot r_\beta>0$ by (H). Next $\prod_{j\in F}\alpha_j=\epsilon(1-\epsilon)^{k-1}\prod_{j\in G}\beta_j$.
The conductances are $c_{jj'}(\alpha)=(1-\epsilon)c_{jj'}(\beta)$ on $G$, and at $i$ they are
$c_{ij}(\alpha)=\sqrt{\epsilon(1-\epsilon)}\sqrt{q_{ij}q_{ji}\beta_j}$. A spanning tree of $F$
with $m$ edges at $i$ has weight $O(\epsilon^{m/2})$. The trees with one edge at $i$ are a
spanning tree of $G$ and one edge $\{i,j\}$. So
$\tree(c(\alpha))=\sqrt\epsilon\,\sqrt{\omega_\beta}\,\tree(c_G(\beta))(1+O(\sqrt\epsilon))$, by
$\sum_j\sqrt{q_{ij}q_{ji}\beta_j}=\sqrt{\omega_\beta}$ from
Proposition~\ref{prop:layer-constants}(iv). We put these into the formula for $C_F$ and divide
by the same formula for $C_G$, which has one factor $(4\pi)^{1/2}$ fewer. The powers of
$\epsilon$ are $\epsilon^{1/4}$ from the tree and $\epsilon^{-1}$ from the product, which gives
$\epsilon^{-3/4}$. For $G=\{j\}$ we use $C_{\{j\}}=\mu_j$ and $\tree=1$ on one vertex.
\end{proof}

\emph{The first amplitude correction.} Put $c_1(\beta)=\partial_z\log C_G(\beta,z)$ at $z=0$.
For reversible $Q$,
\begin{equation}\label{eq:layer-c1}
c_1(\beta)=\frac{\mu_ib_\beta}{\mu_G\cdot r_\beta}+\frac{a_\beta}{l_\beta\cdot\ones}+\frac{\omega_\beta}2\sum_{\{j,j'\}\subseteq G}R_{jj'}\pi_{\beta,j}\pi_{\beta,j'},
\end{equation}
with $\pi_\beta=\pi^{\rm in}_\beta=\pi^{\rm out}_\beta$ and $R_{jj'}$ the effective resistance for
the conductances $l_{\beta,j}q_{jj'}r_{\beta,j'}$. To see this, we use that $Q_G(z)$ is
reversible with respect to $\nu_G$ for real $z\ge0$, by the proof of
Proposition~\ref{prop:layer-constants}(iv). So its vectors $r_\beta(z)=r_\beta$ and $l_\beta(z)=l_\beta$ do not move
with $z$. By Proposition~\ref{prop:constant}(b) and the formula for $C_G(\beta,z)$ in
Section~\ref{sec:layer-setting}, for real $z\ge0$
\[
C_G(\beta,z)=\frac{(\mu^{(z)}\cdot r_\beta)(l_\beta\cdot e^{(z)})}{(4\pi)^{(k-2)/2}}\,
\frac{\sqrt{\tree(c(z))}}{\prod_{j\in G}\beta_j},\qquad
c_{jj'}(z)=l_{\beta,j}A(z)_{jj'}r_{\beta,j'} .
\]
Since $l_{\beta,j}c^-_j=a_\beta\pi_{\beta,j}$ and $c^+_{j'}r_{\beta,j'}=b_\beta\pi_{\beta,j'}$, we
have $c_{jj'}(z)=c_{jj'}(0)+z\omega_\beta\pi_{\beta,j}\pi_{\beta,j'}$. The derivatives at $z=0$ of
the logarithms of $\mu^{(z)}\cdot r_\beta$ and $l_\beta\cdot e^{(z)}$ are the first two terms
of~\eqref{eq:layer-c1}. The third term is the derivative of $\frac12\log\tree(c(z))$, by the rule
of Kirchhoff that $\partial\log\tree(c)/\partial c_e=R_e$, which says that $c_eR_e$ is the
weighted fraction of the spanning trees that contain the edge $e$. The function $C_G(\beta,z)$
of Theorem~\ref{thm:L-facet} below is holomorphic and agrees with this one for real $z\ge0$, so
its derivative at $0$ is also given by~\eqref{eq:layer-c1}.

\subsection{The weighted facet chain}\label{sec:layer-theorem}

The layer needs the local law of the weighted chain $Q_G(z)$, uniformly for complex $z$ near
$0$. We prove it in this section.

\begin{theorem}[the weighted facet chain]\label{thm:L-facet}
Let $G$ satisfy \textup{(H)} with $|G|=k-1\ge1$, let $i\notin G$ with $F=G\cup\{i\}$
irreducible, and let $\mathcal K\subset\Delta_G^\circ$ be compact. There are a number $z_0>0$,
not given explicitly, and a constant $C$, both depending only on $Q$, $\mu$, $i$ and
$\mathcal K$, and functions $J_G(\beta,z)$ and
$C_G(\beta,z)$ on $\mathcal K\times\{|z|<2z_0\}$, holomorphic in $z$, jointly continuous, real
for real $z$, equal to the rate and the local-law constant of Section~\ref{sec:layer-setting}
for real $0\le z<2z_0$, and with $C_G(\beta,0)=C_G(\beta)>0$, such that uniformly for
$\beta\in\mathcal K$, $S\ge1$ and complex $|z|\le z_0$
\begin{equation}\label{eq:L-facet}
f^{(z)}_G(S\beta)=S^{-(k-2)/2}\,C_G(\beta,z)\,e^{-S\,J_G(\beta,z)}\,(1+\varepsilon),\qquad|\varepsilon|\le C/S .
\end{equation}
For $|G|=1$ this holds with $\varepsilon=0$.
\end{theorem}

At $z=0$ the display is the local law~\eqref{eq:layer-local-G}, that is
Theorem~\ref{thm:local}(a) for the face $G$. The proof inherits every constant from $z=0$ by
continuity, which is why $z_0$ is not explicit. The rate $J_G(\beta,z)$ is defined for all real
$z\ge0$, but for complex $z$ the proof needs a critical point of the torus phase near the real
one and a Hessian with positive definite real part, and neither need persist far from $z=0$.
The number $\delta$ of Theorem~\ref{thm:layer} is chosen so that the saddle radius
$z_*$ stays below $z_0/2$.

\begin{lemma}[one and two states]\label{lem:L-two}
Theorem~\ref{thm:L-facet} holds when $|G|\le2$, with the functions $J_G$ and $C_G$ given in the
proof.
\end{lemma}

\begin{proof}
For $G=\{j\}$ the proof of Proposition~\ref{prop:pair-face} gives
$f^{(z)}_G(s)=(\mu_j+z\mu_iq_{ij})(1+zq_{ji})e^{-(q_j-z\omega)s}$, which is~\eqref{eq:L-facet}
with $\varepsilon=0$, $J_G(\beta,z)=q_j-z\omega$ and
$C_G(\beta,z)=(\mu_j+z\mu_iq_{ij})(1+zq_{ji})$. Now let $G=\{1,2\}$, and write
$a_{12}(z)=q_{12}+zc^-_1c^+_2$ and $a_{21}(z)=q_{21}+zc^-_2c^+_1$ for
the entries of $A(z)$, $\varpi(z)=a_{12}(z)a_{21}(z)$, $p=\mu^{(z)}$ and $e=e^{(z)}$. A word in $G$
with no equal neighbors alternates. The words that start with $1$ are $1(21)^n$, with weight
$p_1\varpi^ne_1$ and $(n+1,n)$ letters, and $(12)^n$, with weight $p_1\varpi^{n-1}a_{12}e_2$ and
$(n,n)$ letters, for $n\ge1$. With $\xi=2\sqrt{\varpi s_1s_2}$ the power series of $I_1$ and $I_0$
give
\[
\sum_{n\ge1}\frac{\varpi^ns_1^ns_2^{n-1}}{n!\,(n-1)!}=\sqrt{\frac{\varpi s_1}{s_2}}\,I_1(\xi),\qquad
\sum_{n\ge1}\frac{(\varpi s_1s_2)^{n-1}}{(n-1)!^2}=I_0(\xi).
\]
We do the same for the words that start with $2$ and obtain
\[
f^{(z)}_G(s)=e^{-\langle d(z),s\rangle}\Bigl[\Bigl(p_1e_1\sqrt{\frac{\varpi s_1}{s_2}}+p_2e_2\sqrt{\frac{\varpi s_2}{s_1}}\Bigr)I_1(\xi)
+\bigl(p_1a_{12}e_2+p_2a_{21}e_1\bigr)I_0(\xi)\Bigr].
\]
Both sides are entire in $z$, so the branch of the square root does not matter. Since $G$ is
irreducible, $\varpi(0)=q_{12}q_{21}>0$. We choose $z_0$ so that $|\varpi(z)-\varpi(0)|\le\varpi(0)/2$
for $|z|\le2z_0$ and take the principal branch of $\sqrt\varpi$. Then
$\xi=2S\sqrt{\beta_1\beta_2\varpi}$ at $s=S\beta$ satisfies $|\arg\xi|\le\pi/12$ and $|\xi|\ge cS$,
uniformly on $\mathcal K$. The expansion $I_\nu(\xi)=e^\xi(2\pi\xi)^{-1/2}(1+O(1/|\xi|))$ holds
for $\nu=0,1$, uniformly in the sector $|\arg\xi|\le\pi/4$ and $|\xi|\ge1$~\cite[Eq.~10.40.1]{dlmf}.
Since $(2\pi\xi)^{-1/2}=S^{-1/2}(4\pi)^{-1/2}(\varpi\beta_1\beta_2)^{-1/4}$, this
gives~\eqref{eq:L-facet} with
\[
J_G(\beta,z)=\langle d(z),\beta\rangle-2\sqrt{\beta_1\beta_2\varpi},\qquad
C_G(\beta,z)=\frac{U+V}{\sqrt{4\pi}\,(\varpi\beta_1\beta_2)^{1/4}},
\]
where $U=p_1e_1\sqrt{\varpi\beta_1/\beta_2}+p_2e_2\sqrt{\varpi\beta_2/\beta_1}$ and
$V=p_1a_{12}e_2+p_2a_{21}e_1$, and with $\varepsilon=(U\varepsilon_1+V\varepsilon_0)/(U+V)$, where
$|\varepsilon_\nu|\le C/S$ are the errors of the two expansions. Both functions are holomorphic
in $z$, jointly continuous and real for real $z$. At $z=0$ we have $p=\mu_G$ and $e=\ones$, so
$U+V>0$ because $\mu(G)>0$. After shrinking $z_0$, $|U+V|$ is bounded below and $|U|+|V|$ above
on $\mathcal K\times\{|z|\le2z_0\}$, which gives $|\varepsilon|\le C/S$ once $cS\ge1$. For
$1\le S\le1/c$ the error $\varepsilon$ is continuous on a compact set, so it is bounded, and we
enlarge $C$. For real $z\ge0$ the
rate of the two-state matrix $Q_G(z)$ is
$\sup_u[\beta_1(d_1-a_{12}u_2/u_1)+\beta_2(d_2-a_{21}u_1/u_2)]=\langle d,\beta\rangle-2\sqrt{\beta_1\beta_2\varpi}$,
by the inequality of the arithmetic and geometric means. The constant agrees with that of
Section~\ref{sec:layer-setting}, because for real $z\ge0$ both are the limit of
$S^{1/2}e^{SJ_G(\beta,z)}f^{(z)}_G(S\beta)$. For the first this is the computation above. For the
second it is Lemma~\ref{lem:local-general} on $G$ for the weights $A(z)\ge0$, $\mu^{(z)}$ and
$e^{(z)}$, since $f^{(z)}_G(s)=e^{-\langle d(z),s\rangle}H_{A(z),\mu^{(z)},e^{(z)}}(s)$ in the
notation of Appendix~\ref{app:torus}.
\end{proof}

\emph{Facets with two or more states.} Let $|G|=k-1\ge2$. We number the states of $G$ as
$1,\dots,k-1$ and put $m=k-2$. For $z\in\mathbb C$ let $A'(z)$ be $A(z)$ without the row and
column of state $1$, and let $X_z$, $g_z$ and $b_z$ be the functions of
Lemma~\ref{lem:torus-complex} on $G$ for the weights $A(z)$, $\mu^{(z)}$ and $e^{(z)}$. Since
$c^-$, $c^+$ and $\mu_i$ are nonnegative, the weights $A(|z|)$, $\mu^{(|z|)}$ and $e^{(|z|)}$
dominate these in the sense of that lemma. For fixed $\beta\in\Delta_G^\circ$ and
$z'\in\mathbb C^m$ put
\[
\Phi_z(z')=\beta_1g_z(z')+\sum_{j\ge2}\beta_jz_j .
\]
Let $\theta(\beta)$ be the tilt of the face $G$ at $\beta$, let $\theta(\beta,z)$ be that of
$Q_G(z)$ for real $z\ge0$, and let $\mathcal H(\beta)$ be the Hessian $\mathcal H$ of the proof of
Theorem~\ref{thm:local}(a) for the face $G$ at $\alpha=\beta$. On $\mathbb C^m$ we use the norm
$\lVert w\rVert=\max_j|w_j|$ with closed balls $\bar B(w,r)$, and $we^{\mathrm i\varphi}$ is the
vector with entries $w_je^{\mathrm i\varphi_j}$. The proof of Theorem~\ref{thm:local}(a) for $G$
gives, for every $\beta\in\Delta_G^\circ$,
\begin{enumerate}[label=(P\arabic*)]
\item $\varrho(A'(0)-\operatorname{diag}\theta'(\beta))<0$ and $g_0(\theta'(\beta))=\theta_1(\beta)$,
by~\eqref{eq:A-radii},
\item $\nabla\Phi_0(\theta'(\beta))=0$, and $\nabla^2\Phi_0(\theta'(\beta))=\mathcal H(\beta)$ is
real, symmetric and positive definite,
\item $\operatorname{Re}\Phi_0(\theta'(\beta)e^{\mathrm i\varphi})\le\Phi_0(\theta'(\beta))
-\sum_{j\ge2}\beta_j\theta_j(\beta)(1-\cos\varphi_j)$ for $\varphi\in\mathbb T^m$,
\item $b_0(\theta'(\beta))>0$.
\end{enumerate}
Also $\theta$ and $\mathcal H$ are continuous, by Proposition~\ref{prop:constant}(a).

\begin{lemma}[the complex critical point]\label{lem:critical}
Let $|G|\ge2$ and let $\mathcal K\subset\Delta_G^\circ$ be compact. There are $z_1>0$,
$r_0\ge r_1>0$, $\sigma_0>0$ and $B_0$, which depend only on $Q$, $\mu$, $i$ and $\mathcal K$, such
that the following hold for every $\beta\in\mathcal K$.
\begin{enumerate}[(a)]
\item $r_0<\frac14\min_j\theta_j(\beta)$, and $\varrho(A'(|z|)-\operatorname{diag}s')<0$ for
$|z|\le z_1$ and real $s'$ with $\lVert s'-\theta'(\beta)\rVert\le3r_0$.
\item The functions $g_z(z')$ and $b_z(z')$ are holomorphic in $(z,z')$ and bounded in modulus by
$B_0$ on $\{|z|<z_1\}\times U_\beta$, where $U_\beta$ is the set of $z'$ with
$|z_j|>\theta_j(\beta)-3r_0$ for all $j$.
\item For $|z|<z_1$ there is exactly one $\zeta'(\beta,z)\in\bar B(\theta'(\beta),r_1)$ with
$\nabla\Phi_z(\zeta'(\beta,z))=0$. It is holomorphic in $z$ and jointly continuous in
$(\beta,z)$, $\zeta'(\beta,0)=\theta'(\beta)$, $\lVert\zeta'(\beta,z)-\theta'(\beta)\rVert\le B_0|z|$,
and $\zeta'(\beta,\bar z)=\overline{\zeta'(\beta,z)}$.
\item The matrix $\mathcal J(\beta,z)=Z_c\nabla^2\Phi_z(\zeta'(\beta,z))Z_c$, where
$Z_c=\operatorname{diag}\zeta'(\beta,z)$, is complex symmetric and jointly continuous, and
$\mathcal J(\beta,0)=\Theta'\mathcal H(\beta)\Theta'\ge\sigma_0I$, where
$\Theta'=\operatorname{diag}\theta'(\beta)$.
\item For real $0\le z<z_1$ we have $\zeta'(\beta,z)=\theta'(\beta,z)$,
$\Phi_z(\zeta'(\beta,z))=\langle\beta,\theta(\beta,z)\rangle$,
$\beta_1^2\det\mathcal J(\beta,z)=\prod_{j\ge2}\theta_j(\beta,z)^2\det\Sigma_\beta(z)$ and
$\beta_1b_z(\zeta'(\beta,z))=(\mu^{(z)}\cdot r_\beta(z))(l_\beta(z)\cdot e^{(z)})$.
\end{enumerate}
\end{lemma}

\begin{proof}
(a) The Perron root of a Metzler matrix is a continuous function of its entries. By (P1) the
function $(\beta,y,s')\mapsto\varrho(A'(y)-\operatorname{diag}s')$ is negative on the compact set
of the points $(\beta,0,\theta'(\beta))$ with $\beta\in\mathcal K$. So it is negative on a
neighborhood of this set, and the neighborhood contains the points with $\beta\in\mathcal K$,
$0\le y\le z_1$ and $\lVert s'-\theta'(\beta)\rVert\le3r_0$ when $z_1$ and $r_0$ are small. We also
take $r_0<\frac14\min_{\mathcal K}\min_j\theta_j$, which is positive.

(b) Put $s'_\beta=\theta'(\beta)-3r_0\ones>0$. By (a) we may use Lemma~\ref{lem:torus-complex}
with the dominating weights at $|z|$ and the radii $s'_\beta$. So $\mathcal Z'-A'(z)$ is invertible
on $\{|z|\le z_1\}\times U_\beta$, and there $|g_z|\le\bar g(s'_\beta)$ and
$|b_z|\le\bar b(s'_\beta)$ for the functions $\bar g$ and $\bar b$ of the weights at $|z|$. The
terms of their series do not decrease with $|z|$, so we may take the weights at $z_1$. Then the
bounds are continuous in $\beta$, since $(\operatorname{diag}s'-A'(z_1))^{-1}$ is continuous where
$\varrho(A'(z_1)-\operatorname{diag}s')<0$. So they are bounded on $\mathcal K$. The weights are
affine in $z$, so the last claim of Lemma~\ref{lem:torus-complex} gives the holomorphy.

(c) For $|z|\le z_1/2$ and $z'\in\bar B(\theta'(\beta),r_0)$ the polydisc of radius
$\min(z_1/2,r_0)$ about $(z,z')$ lies in $\{|z|<z_1\}\times U_\beta$, since its points have
$|z_j|\ge\theta_j-2r_0$. So by (b) and Cauchy's estimates, the derivatives of $g_z(z')$ and
$b_z(z')$ in $(z,z')$ up to order $4$ are bounded on $\{|z|\le z_1/2\}\times\bar B(\theta'(\beta),r_0)$
by a constant $B_1$ that does not depend on $\beta$. Let
$\mathcal G(z,z')=\nabla\Phi_z(z')=\beta_1\nabla g_z(z')+\beta'$. By (P2), $\mathcal G(0,\theta')=0$ and
$D_{z'}\mathcal G(0,\theta')=\mathcal H(\beta)$, whose inverse has norm at most $B_2$ on $\mathcal K$ by
continuity. Put $\mathcal T_z(z')=z'-\mathcal H(\beta)^{-1}\mathcal G(z,z')$. For $|z|\le z_1/2$, $r\le r_0$ and
$z'\in\bar B(\theta',r)$ we get
\[
\lVert D_{z'}\mathcal T_z(z')\rVert=\lVert\mathcal H^{-1}(D_{z'}\mathcal G(0,\theta')-D_{z'}\mathcal G(z,z'))\rVert\le B_3(|z|+r),\qquad
\lVert \mathcal T_z(\theta')-\theta'\rVert\le B_3|z|,
\]
with $B_3$ depending only on $m$, $B_1$ and $B_2$. Choose $r_1\le r_0$ and $z_2\le z_1/2$ with
$B_3(z_2+r_1)\le\frac12$ and $B_3z_2\le\frac12r_1$. For $|z|\le z_2$ the map $\mathcal T_z$ sends
$\bar B(\theta',r_1)$ into itself and is a contraction with constant $\frac12$ there. So $\mathcal G(z,\cdot)$
has exactly one zero $\zeta'(\beta,z)$ in this ball, the fixed point of $\mathcal T_z$, and
$\lVert\zeta'-\theta'\rVert\le2\lVert \mathcal T_z(\theta')-\theta'\rVert\le2B_3|z|$. The iterates
$\mathcal T_z^n(\theta')$ are holomorphic in $z$, continuous in $(\beta,z)$ and converge uniformly, which
gives the regularity of $\zeta'$, and $\zeta'(\beta,0)=\theta'$. The data $A_G$, $c^\pm$, $\mu$
and $\beta$ are real, so $\mathcal G(\bar z,\bar w)=\overline{\mathcal G(z,w)}$. Hence
$\overline{\zeta'(\beta,z)}$ is a zero of $\mathcal G(\bar z,\cdot)$ in the ball, and it equals
$\zeta'(\beta,\bar z)$. We rename $z_2$ as $z_1$ and enlarge $B_0$ to at least $2B_3$.

(d) The Hessian $\nabla^2\Phi_z$ is symmetric and $Z_c$ is diagonal, so $\mathcal J$ is complex
symmetric, and it is continuous by (b) and (c). At $z=0$, $Z_c=\Theta'$ and (P2) gives
$\mathcal J(\beta,0)=\Theta'\mathcal H(\beta)\Theta'$, which is positive definite and continuous in
$\beta$. So its least eigenvalue has a positive minimum $\sigma_0$ on $\mathcal K$.

(e) Let $0\le z<z_1$ be real and write $\zeta'=\zeta'(\beta,z)$. By (c), $\zeta'$ is real, and it is
positive since $\lVert\zeta'-\theta'\rVert\le r_0<\frac14\min_j\theta_j$. The matrix
$A(z)=A(|z|)$ is nonnegative, and its graph contains that of $A_G$, so it is strongly connected.
By (a) and the first fact of Appendix~\ref{app:torus}, $X_z(\zeta')\ge0$, and as in the proof of Theorem~\ref{thm:local}(a) the
vector $\tilde r=(1,X_z(\zeta')b_1(z))\ge0$, with $b_1(z)=(A(z)_{j1})_{j\ge2}$, lies in the kernel of $A(z)-\operatorname{diag}\tilde\theta$
with $\tilde\theta=(g_z(\zeta'),\zeta')$. By the second fact of Appendix~\ref{app:torus}, $\tilde r>0$ and $0$ is the Perron
root. Let $\tilde l>0$ be the left Perron vector with $\tilde l\cdot\tilde r=1$, and
$\pi=\tilde l\circ\tilde r$. With $\tilde R=\operatorname{diag}\tilde r$, the matrix
$\tilde R^{-1}(A(z)-\operatorname{diag}\tilde\theta)\tilde R$ is an irreducible generator with
stationary law $\pi$, so Lemma~\ref{lem:perron} gives $\nabla\lambda(\tilde\theta)=-\pi$ for
$\lambda(y)=\varrho(A(z)-\operatorname{diag}y)$. For real $z'$ near $\zeta'$ the same argument
gives $\lambda(g_z(z'),z')=0$, and differentiating at $\zeta'$ gives
$\pi_1\partial_jg_z(\zeta')+\pi_j=0$. The critical equation $\beta_1\partial_jg_z(\zeta')+\beta_j=0$
then gives $\pi_j/\pi_1=\beta_j/\beta_1$ for $j\ge2$, and since $\pi$ and $\beta$ both sum to $1$,
$\pi=\beta$. So $\tilde l=\beta/\tilde r$, and $\tilde l^\top(A(z)-\operatorname{diag}\tilde\theta)=0$
says that the derivative of $u\mapsto\sum_{j\ne j'}\beta_jA(z)_{jj'}u_{j'}/u_j$ vanishes at
$\tilde r$, as before~\eqref{eq:stationary}. This function is convex in $\log u$, so $\tilde r$ is
a minimizer, and Proposition~\ref{prop:rate}(b) for $Q_G(z)$ shows that $\tilde r$ is a multiple
of $r_\beta(z)$. Hence $\tilde\theta=A(z)\tilde r/\tilde r=\theta(\beta,z)$, and
$\Phi_z(\zeta')=\beta_1\theta_1+\sum_{j\ge2}\beta_j\theta_j=\langle\beta,\theta(\beta,z)\rangle$.
The last two claims are the identities of Lemma~\ref{lem:local-general} on $G$ for $A(z)$,
$\mu^{(z)}$ and $e^{(z)}$ at $\alpha=\beta$, since $\mathcal J(\beta,z)$ is the Hessian
$\mathcal H$ of that lemma multiplied on both sides by $\operatorname{diag}\theta'(\beta,z)$.
\end{proof}

\begin{proof}[Proof of Theorem~\ref{thm:L-facet}]
For $|G|=1$ this is Lemma~\ref{lem:L-two}. Let $|G|\ge2$, and let $z_1$, $r_0$, $r_1$, $\sigma_0$
and $B_0$ be as in Lemma~\ref{lem:critical}. For $\beta\in\mathcal K$ and $|z|<z_1$ write
$\zeta'=\zeta'(\beta,z)$ and $\rho_j=|\zeta_j|$, and put
\[
\Psi(\beta,z,\varphi)=\Phi_z(\zeta'e^{\mathrm i\varphi}),\qquad
\chi(\beta,z,\varphi)=b_z(\zeta'e^{\mathrm i\varphi})\prod_{j\ge2}\zeta_je^{\mathrm i\varphi_j}.
\]
They play the roles of the phase $\Psi$ and the amplitude $a$ of Lemma~\ref{lem:laplace-complex}.

\emph{Step 1. The constants.} Since $\theta_j$ is real, $|\rho_j-\theta_j|\le|\zeta_j-\theta_j|\le r_1$.
So the points $\zeta'e^{\mathrm i\varphi}$ lie in $U_\beta$, and by Lemma~\ref{lem:critical}(b),
(c) the $\varphi$-derivatives of $\Psi$ and $\chi$ are jointly continuous on the compact set
$P_1\times\mathbb T^m$, where $P_1=\mathcal K\times\{|z|\le z_1/2\}$. Let $B_\Psi$ and $B_a$ bound
those of $\Psi$ up to order $4$ and those of $\chi$ up to order $2$. Since
$\partial/\partial\varphi_j=\mathrm iw_j\,\partial/\partial w_j$ at $w=\zeta'e^{\mathrm i\varphi}$
and $\nabla\Phi_z(\zeta')=0$, we have $\nabla_\varphi\Psi(\beta,z,0)=0$ and
$-\nabla^2_\varphi\Psi(\beta,z,0)=\mathcal J(\beta,z)$, as in the proof of Theorem~\ref{thm:local}(a).
Fix $\varepsilon_0\in(0,1)$ with $m^3B_\Psi\varepsilon_0\le\frac34\sigma_0$, and put
$2\eta=\beta_{\min}\theta_{\min}(1-\cos\varepsilon_0)$, where $\beta_{\min}$ and $\theta_{\min}$ are
the least entries of $\beta$ and $\theta(\beta)$ over $\mathcal K$. By (P3),
$\operatorname{Re}\Psi(\beta,0,\varphi)-\Psi(\beta,0,0)\le-2\eta$ for
$\lVert\varphi\rVert\ge\varepsilon_0$. By uniform continuity on $P_1\times\mathbb T^m$ we can choose
$z_0\in(0,z_1/4]$ such that for $\beta\in\mathcal K$ and $|z|\le2z_0$
\begin{enumerate}[(\roman*)]
\item $\operatorname{Re}\mathcal J(\beta,z)\ge\frac12\sigma_0I$, since
$\lVert\mathcal J(\beta,z)-\mathcal J(\beta,0)\rVert\le\frac12\sigma_0$ in the operator norm and
$\mathcal J(\beta,0)$ is real,
\item $|\chi(\beta,z,0)|\ge c_\chi$, where $2c_\chi>0$ is the minimum over $\mathcal K$ of
$\chi(\beta,0,0)=b_0(\theta')\prod_{j\ge2}\theta_j$, which is positive by (P4),
\item $\operatorname{Re}\Psi(\beta,z,\varphi)-\operatorname{Re}\Psi(\beta,z,0)\le-\eta$ for
$\lVert\varphi\rVert\ge\varepsilon_0$.
\end{enumerate}
The numbers $z_0$, $\varepsilon_0$, $\eta$ and $c_\chi$ depend only on $Q$, $\mu$, $i$ and
$\mathcal K$.

\emph{Step 2. The torus through the critical point.} Let $\beta\in\mathcal K$, $|z|\le z_0$ and
$S\ge1$. Since $\lVert\rho'-\theta'\rVert\le r_0$, Lemma~\ref{lem:critical}(a) shows that the
radii $\rho'$ are admissible in Lemma~\ref{lem:torus-complex} with the dominating weights at
$|z|$. With $t=S\beta$ it gives
\[
e^{S\langle d(z),\beta\rangle}f^{(z)}_G(S\beta)=\frac1{(2\pi)^m}\int_{\mathbb T^m}
\exp\bigl(S\Phi_z(w)\bigr)\,b_z(w)\prod_{j\ge2}w_j\,d\varphi,\qquad w_j=\rho_je^{\mathrm i\varphi_j}.
\]
The translation of $\varphi$ by the arguments of the $\zeta_j$ preserves the measure and turns $w$
into $\zeta'e^{\mathrm i\varphi}$. So the integral is $\int_{\mathbb T^m}e^{S\Psi}\chi\,d\varphi$.
No contour is moved. The torus representation holds at every admissible radius, and we use it at
the radii $\rho'$ of the complex critical point.

\emph{Step 3. The Laplace lemma.} By Step~1, Lemma~\ref{lem:laplace-complex} applies on
$\mathcal P=\mathcal K\times\{|z|\le z_0\}$ with $\sigma_{\min}=\frac12\sigma_0$, the cutoff $\varepsilon_0$ in place of $\varepsilon$,
and $\eta$. Since $\Psi(\beta,z,0)=\Phi_z(\zeta')$, it gives~\eqref{eq:L-facet} with
\[
J_G(\beta,z)=\langle d(z),\beta\rangle-\Phi_z(\zeta'),\qquad
C_G(\beta,z)=(2\pi)^{-m/2}\det\mathcal J(\beta,z)^{-1/2}\,b_z(\zeta')\prod_{j\ge2}\zeta_j,
\]
and with $|\varepsilon|\le C/(c_\chi S)$, where $C$ is the constant of
Lemma~\ref{lem:laplace-complex}.

\emph{Step 4. The functions.} We define $J_G$ and $C_G$ on $\mathcal K\times\{|z|<2z_0\}$ by these
formulas. Since $\zeta'$ is holomorphic in $z$ and $g_z(z')$ and $b_z(z')$ are holomorphic in
$(z,z')$, the numbers $\Phi_z(\zeta')$, $b_z(\zeta')$, $\prod_j\zeta_j$ and the entries of
$\mathcal J(\beta,z)$ are holomorphic in $z$ and jointly continuous. By (i) the matrix
$\mathcal J(\beta,z)$ lies in the set $\mathcal S_m$ of Appendix~\ref{app:laplace}, so
$\det\mathcal J(\beta,z)^{-1/2}$ is continuous, and as a continuous square root of a holomorphic
function without zeros it is holomorphic in $z$. For real $z$, $\zeta'$ is real by
Lemma~\ref{lem:critical}(c), so $J_G$ and $C_G$ are real, and $\mathcal J(\beta,z)$ is real and
positive definite, so the branch is the positive root. For real $0\le z<2z_0$,
Lemma~\ref{lem:critical}(e) gives $J_G(\beta,z)=\langle d(z)-\theta(\beta,z),\beta\rangle$, which is
the rate of $Q_G(z)$ by the identity $J=\langle q-\theta,\alpha\rangle$ after
Definition~\ref{def:tilt}, since the diagonal of $Q_G(z)$ is $-d(z)$. It also gives
$C_G(\beta,z)=(\mu^{(z)}\cdot r_\beta(z))(l_\beta(z)\cdot e^{(z)})/((2\pi)^{m/2}\sqrt{\det\Sigma_\beta(z)})$,
the constant of Section~\ref{sec:layer-setting}. At $z=0$ it is $C_G(\beta)>0$, by (H).
\end{proof}

The proof does not use the irreducibility of $F$, which only makes $c^-$ and $c^+$ nonzero.

\subsection{Proofs of the main results}\label{sec:layer-proofs}

\begin{proof}[Proof of Theorem~\ref{thm:layer}\textup{(a)}]
Theorem~\ref{thm:L-facet} gives~\eqref{eq:L-facet}, with $\varepsilon=0$ for $k=2$. Take
$\delta\le\frac12$, so $S\ge T/2$. Write $\Lambda(z)=S^{-(k-2)/2}C_G(\beta,z)e^{-SJ_G(\beta,z)}$
and $\varepsilon(z)=f^{(z)}_G(s)/\Lambda(z)-1$. By Step~1(i), (ii) of the proof of
Theorem~\ref{thm:L-facet}, $C_G(\beta,z)\ne0$ for $|z|\le2z_0$ if $|G|\ge2$, so $\varepsilon$ is
holomorphic there (for $|G|=1$, $\varepsilon=0$), and $|\varepsilon|\le C/S$ for $|z|\le z_0$. The Schwarz lemma for
$(\varepsilon(z)-\varepsilon(0))S/(2C)$ on $|z|<z_0$ gives
$|\varepsilon(z)-\varepsilon(0)|\le2C|z|/(z_0S)$ for $|z|\le z_0$. Since
$f_G(s)=\Lambda(0)(1+\varepsilon(0))$, \eqref{eq:layer-contour} gives for every $\rho\le z_0$
\[
e^{q_ix}f_F(s,x)=(1+\varepsilon(0))\mathcal A+\mathcal D,\qquad \mathcal A=\frac1{2\pi\mathrm i}\oint_{|z|=\rho}e^{x/z}\Lambda(z)\frac{dz}{z^2},\qquad
|\mathcal D|\le\frac{2C\rho}{z_0S}\,\mathcal A^{\rm abs},
\]
where $\mathcal A^{\rm abs}$ is the integral of $|e^{x/z}\Lambda(z)||z|^{-2}$ against $|dz|/(2\pi)$. By
definition $\mathbf K_\beta(x,S)=e^{-q_ix}\mathcal A/\Lambda(0)$, and $|\varepsilon(0)|\le\frac12$
for $T_0$ large, so $f_F=f_G\mathbf K_\beta(1+E)$ with $|E|\le2|\mathcal D|/\mathcal A$ once
$\mathcal A>0$. So it suffices to show that
$\mathcal A>0$, $\mathcal A^{\rm abs}\le C'\mathcal A$ and $\rho/S\le C'(T^{-2}+x^{1/2}T^{-3/2})$. Put
\[
\Xi(z)=\frac{C_G(\beta,z)}{C_G(\beta,0)}\,e^{S(J_G(\beta)-J_G(\beta,z)-\omega_\beta z)},
\]
so that $\Lambda(z)/\Lambda(0)=e^{S\omega_\beta z}\Xi(z)$. By
Proposition~\ref{prop:layer-constants}(ii), the holomorphy of $C_G$, and Cauchy estimates on the
compact set $\mathcal K\times\{|z|\le\frac32z_0\}$, the function $R=\log\Xi$ is
holomorphic on $|z|\le z_0$ after shrinking $z_0$, and
\[
R(z)=c_1z-S\omega_\beta(\gamma_\beta-q_i)z^2+O(|z|^2+S|z|^3),\qquad |R''(z)|\le C(1+S).
\]
We take $\delta$ so small that $z_*\le\sqrt{2\delta/\omega_{\min}}\le z_0/2$, where
$\omega_{\min}=\min_{\mathcal K}\omega_\beta>0$.

\emph{Few visits, $\zeta\le Z_0$} for a fixed large $Z_0$. Take $\rho=\max(z_*,1/S)$, which is
at most $C/S$. On the circle $|x/z|\le Z_0/2$, $|S\omega_\beta z|\le C$ and
$|\Xi(z)-1|\le C/S$. Let $\mathcal A_0$ be $\mathcal A$ with $\Xi=1$. By
\eqref{eq:layer-bessel-moments} it equals $\Lambda(0)\sqrt{\omega_\beta S/x}\,I_1(\zeta)$, which
is at least $\Lambda(0)\omega_\beta S$, because
$\sqrt{\omega S/x}\,I_1(\zeta)=\omega S\sum_{n\ge0}(\zeta/2)^{2n}/(n!(n+1)!)$. Also
$|\mathcal A-\mathcal A_0|\le C\Lambda(0)\rho^{-1}/S\le C\Lambda(0)$ and
$\mathcal A^{\rm abs}\le C\Lambda(0)\rho^{-1}\le C\Lambda(0)S$. So
$\mathcal A\ge\Lambda(0)(\omega_\beta S-C)>0$ and $\mathcal A^{\rm abs}\le C'\mathcal A$ for $S$ large,
and $\rho/S\le C/S^2\le4C/T^2$.

\emph{Many visits, $\zeta\ge Z_0$.} Let $\lambda(z)=x/z+S\omega_\beta z+R(z)$, so that
$e^{x/z}\Lambda(z)=\Lambda(0)e^{\lambda(z)}$. It is real on the real axis, and
$\lambda'(z)=-x/z^2+S\omega_\beta+O(1+S|z|)$. By the implicit function theorem it has a real
critical point $\rho_s=z_*(1+O(\sqrt{x/S}+1/\zeta))$, and we take $\rho=\rho_s$. For
$z=\rho e^{\mathrm i\phi}$,
\[
\operatorname{Re}\lambda(z)-\lambda(\rho)=(x/\rho+S\omega_\beta\rho)(\cos\phi-1)+\operatorname{Re}[R(z)-R(\rho)] .
\]
Since $R$ is real on the real axis, $\operatorname{Re}[R'(\rho)(z-\rho)]=R'(\rho)\rho(\cos\phi-1)$,
and $|R''|\le C(1+S)$ bounds the rest by $2C(1+S)\rho^2(1-\cos\phi)$. Also
$|R'(\rho)|\rho\le C(\rho+S\rho^2)$. Now $S\rho^2\le Cx$, $x/\rho+S\omega_\beta\rho\ge\zeta$,
and $x/\zeta\le C\sqrt\delta$. So for $\delta$ small
$\operatorname{Re}\lambda(z)-\lambda(\rho)\le-\frac\zeta2(1-\cos\phi)$, and
\[
\mathcal A^{\rm abs}\le\Lambda(0)\rho^{-1}e^{\lambda(\rho)}\frac1{2\pi}\int_{-\pi}^\pi e^{-\frac\zeta2(1-\cos\phi)}d\phi
\le C\Lambda(0)\rho^{-1}e^{\lambda(\rho)}\zeta^{-1/2}.
\]
For the lower bound we write
$\mathcal A=\Lambda(0)\rho^{-1}\frac1{2\pi}\int e^{\zeta\ell(\phi)}e^{-\mathrm i\phi}d\phi$ with
$\ell(\phi)=\lambda(\rho e^{\mathrm i\phi})/\zeta$. Here $\ell(0)$ is real, $\ell'(0)=0$,
$-\ell''(0)=\rho^2\lambda''(\rho)/\zeta$ is real and lies in $[\frac12,2]$, the derivatives of
$\ell$ up to order four are bounded uniformly in all parameters, and the gap above holds. The
parameters $(\beta,x,S)$ do not range over a compact set, so we use the uniform form of
Lemma~\ref{lem:laplace} in Remark~\ref{rem:laplace-uniform}, with $\sigma_{\min}=\frac12$, a fixed
$\varepsilon_0\in(0,1)$ with $B_\Psi\varepsilon_0\le\frac34$, where $B_\Psi$ bounds the derivatives of
$\ell$ up to order four, the cutoff $\varepsilon_0$ in place of $\varepsilon$, and the gap
$\eta=\frac12(1-\cos\varepsilon_0)$. With $m=1$, $\Psi=\ell$, $a=e^{-\mathrm i\phi}$ and $\zeta$ in place
of $T$ it gives
$\mathcal A=\Lambda(0)\rho^{-1}e^{\lambda(\rho)}(2\pi\zeta(-\ell''(0)))^{-1/2}(1+O(1/\zeta))$,
which is positive and at least $c\,\mathcal A^{\rm abs}$. Finally $\rho_s\le2z_*$ for $\delta$ small
and $Z_0$ large, so $\rho/S\le2x^{1/2}(\omega_{\min}S^3)^{-1/2}\le Cx^{1/2}T^{-3/2}$.
\end{proof}

\begin{proof}[Proof of Theorem~\ref{thm:layer}\textup{(b)} and \textup{(c)}]
We use two facts on Bessel functions besides~\eqref{eq:layer-bessel-moments}. For integers
$n\ge1$, $I_n(\zeta)=\frac1\pi\int_0^\pi e^{\zeta\cos t}\cos(nt)\,dt$~\cite[Eq.~10.32.3]{dlmf}.
Also $I_n\le I_1$ for $n\ge1$ and $\zeta>0$, since $I_\nu(\zeta)$ decreases in $\nu>0$~\cite[\S10.37]{dlmf}. The bound $|1-\cos nt-(1-\cos t)|\le n^2t^2/2$ and a Laplace
estimate give $|I_1(\zeta)-I_n(\zeta)|\le Cn^2I_1(\zeta)/\zeta$ for $\zeta\ge1$. Also
$I_0/I_1\le1+2/\zeta$, as in the proof of Proposition~\ref{prop:pair-face}.

\emph{The case $\pi^{\rm in}_\beta=\pi^{\rm out}_\beta$.} By
Proposition~\ref{prop:layer-constants}(v), $J_G(\beta,z)=J_G(\beta)-z\omega_\beta$ for real
$0\le z<2z_0$, hence on the whole disc by holomorphy, so
$\Xi(z)=C_G(\beta,z)/C_G(\beta,0)=1+c_1z+\sum_{n\ge2}c_nz^n$ with $|c_n|\le Cz_0^{-n}$. By
\eqref{eq:layer-bessel-moments} the term $c_nz^n$ contributes $c_nz_*^nI_{n-1}(\zeta)/I_1(\zeta)$
relative to the main term $z_*^{-1}I_1(\zeta)$. For $n\ge2$ this is at most $|c_n|z_*^n$, and the
sum over $n\ge2$ is $O(z_*^2)=O(x/T)$, since $z_*\le z_0/2$. The term $n=1$ is
$c_1z_*I_0(\zeta)/I_1(\zeta)=O(z_*+1/(\omega_\beta S))=O(\sqrt{x/T}+1/T)$. With (a) this proves
both claims of (c), for all $0<x\le\delta T$.

\emph{The general case, $x\le1$.} Write $\Xi(z)=e^{-S\omega_\beta(\gamma_\beta-q_i)z^2}(1+c_1z+R_2(z))$
with $|R_2(z)|\le C(|z|^2+S|z|^3)$. On $|z|=z_*$ the Gaussian factor has modulus at most
$e^{(\gamma_\beta-q_i)x}\le C$. We expand it in powers of $z^2$ and use
\eqref{eq:layer-bessel-moments}. With $y=(\gamma_\beta-q_i)x$ the term $1$ gives the relative
factor $\sum_{m\ge0}(-y)^mI_{2m-1}(\zeta)/(m!\,I_1(\zeta))$. For $\zeta\ge1$ it differs from
$e^{-y}$ by at most $C\sum_{m\ge1}y^mm^2/(m!\,\zeta)\le Cx/\zeta=O(\sqrt{x/S})$. For $\zeta<1$
we have $x<1/(4\omega_\beta S)$, and $|I_{2m-1}/I_1-1|\le1$ bounds the difference by
$e^y-1=O(1/S)$. The term $c_1z$ times the Gaussian factor gives, by the same expansion, the
relative factor $c_1z_*\sum_{m\ge0}(-y)^mI_{2m}(\zeta)/(m!\,I_1(\zeta))$, since
$S\omega_\beta z_*^2=x$. As $I_{2m}\le I_0$, its modulus is at most
$|c_1|z_*e^yI_0(\zeta)/I_1(\zeta)\le Cz_*(1+2/\zeta)=O(\sqrt{x/S}+1/S)$.
The term $R_2$ is bounded through the modulus of the integrand, and
$\frac1{2\pi}\int e^{\zeta\cos\phi}d\phi=I_0(\zeta)$ gives
$O\bigl((z_*^2+Sz_*^3)I_0(\zeta)/I_1(\zeta)\bigr)=O(1/S+\sqrt{x/S})$. So
$\mathbf K_\beta=\mathbf B_\beta(1+O(1/T+\sqrt{x/T}))$ for $x\le1$.

\emph{The general case, $1\le x\le\delta T$.} Here $\zeta\ge2\sqrt{\omega_{\min}S}\ge Z_0$. We
use the evaluation of $\mathcal A$ in the proof of (a), and the same evaluation with
$\Xi=1$, whose integral is exactly $\Lambda(0)z_*^{-1}I_1(\zeta)$. For $\Xi=1$ the
critical point is $z_*$, the critical value is $\zeta$, and $-\ell''(0)=1$. The ratio of the two
evaluations is $e^{\lambda(\rho_s)-\zeta}(1+O(\sqrt{x/S}+1/\zeta))$, since
$\rho_s/z_*-1$ and the relative change of $-\ell''(0)$ are $O(\sqrt{x/S}+1/\zeta)$. Now
$x/z+S\omega_\beta z$ takes the value $\zeta$ at its critical point $z_*$, and the shift from
$z_*$ to $\rho_s$ changes $\lambda$ only by
$O(\lambda''(\rho_s)(\rho_s-z_*)^2)=O(\zeta x/S+1/\zeta)$. So
\[
\lambda(\rho_s)-\zeta=R(z_*)+O(\zeta x/S+1/\zeta)=-(\gamma_\beta-q_i)x+O\bigl(\sqrt{x/S}\,(1+x)\bigr),
\]
because $c_1z_*=O(\sqrt{x/S})$, $S\omega_\beta z_*^2=x$, $Sz_*^3=O(x^{3/2}S^{-1/2})$,
$\zeta x/S=O(x^{3/2}S^{-1/2})$ and $1/\zeta\le C\sqrt{x/S}$ for $x\ge1$. Hence
$\mathbf K_\beta/\mathbf B_\beta=e^{(\gamma_\beta-q_i)x}\mathcal A/\mathcal A_0=\exp(O(\sqrt{x/T}(1+x)))$,
where the factor $e^{(\gamma_\beta-q_i)x}$ comes from $e^{-q_ix}$ in $\mathbf K_\beta$ against $e^{-\gamma_\beta x}$ in
$\mathbf B_\beta$, and $\mathcal A_0$ is $\mathcal A$ with $\Xi=1$. With (a) this proves (b).
\end{proof}

\begin{proof}[Proof of Proposition~\ref{prop:layer-cubic}]
(a) We prove this first. By the proof of Proposition~\ref{prop:layer-constants}(iii) with $S=1$,
$\mathfrak r_\beta(\varepsilon)=j(\varepsilon)-J_G(\beta)+2\sqrt{\omega_\beta}\,\varepsilon-(\gamma_\beta-q_i)\varepsilon^2$
with $j(\varepsilon)=\sup_\psi[\Psi_\beta(\psi)-2\varepsilon P_\beta(\psi)^{1/2}]$, and the maximizer
is within $O(\varepsilon)$ of $\psi_0$ modulo constants. The critical equation
$\nabla\Psi_\beta(\psi)=\varepsilon\nabla P_\beta(\psi)/\sqrt{P_\beta(\psi)}$ is real-analytic in
$(\psi,\varepsilon,\beta)$ near $(\psi_0,0,\beta)$, and its derivative in $\psi$ at $\varepsilon=0$
is $-\mathcal L_\beta$, which is invertible on the vectors that sum to zero. So the implicit
function theorem gives an analytic solution $\psi(\varepsilon)$, which is the maximizer for small
$\varepsilon\ge0$. Hence $j$ and $\mathfrak r_\beta$ are real-analytic near $0$, jointly in $\beta$, and by
Proposition~\ref{prop:layer-constants}(iii)
\[
j(\varepsilon)=J_G(\beta)-2\sqrt{\omega_\beta}\,\varepsilon+(\gamma_\beta-q_i)\varepsilon^2+j_3\varepsilon^3+O(\varepsilon^4)
\]
uniformly on $\mathcal K$. So $\mathfrak r_\beta(\varepsilon)=j_3\varepsilon^3+O(\varepsilon^4)$, and it
remains to show $j_3=\varkappa_\beta$. For each $\psi$ the supremum of $-zP_\beta(\psi)-\varepsilon^2/z$
over $z>0$ is $-2\varepsilon P_\beta(\psi)^{1/2}$, attained at $z=\varepsilon/\sqrt{P_\beta(\psi)}$.
So $z(\varepsilon)=\varepsilon/\sqrt{P_\beta(\psi(\varepsilon))}=z_*(1+O(\varepsilon))$ maximizes
$J_G(\beta,z)-\varepsilon^2/z$ over $z>0$, with the value $j(\varepsilon)$. Write $z=\varepsilon w$,
$a=\frac12\partial_z^2J_G(\beta,0)=\omega_\beta(\gamma_\beta-q_i)$ and
$c=\frac16\partial_z^3J_G(\beta,0)$. Then $(j(\varepsilon)-J_G(\beta))/\varepsilon$ is the maximum
over $w$ near $w_0=\omega_\beta^{-1/2}$ of $\phi(w)+p(w)+O(\varepsilon^3)$, where
$\phi(w)=-\omega_\beta w-1/w$ and $p(w)=a\varepsilon w^2+c\varepsilon^2w^3$. Here
$\phi(w_0)=-2\sqrt{\omega_\beta}$, $\phi'(w_0)=0$ and $\phi''(w_0)=-2\omega_\beta^{3/2}$, so the
maximum is $\phi(w_0)+p(w_0)+p'(w_0)^2/(4\omega_\beta^{3/2})+O(\varepsilon^3)$. This gives
$j_3=c\,\omega_\beta^{-3/2}+a^2\omega_\beta^{-5/2}=\varkappa_\beta$.

(b) Take $\delta$ at most the $\delta$ of Theorem~\ref{thm:layer}, and so small that
$\sqrt{2\delta}$ lies in the disc where (a) holds uniformly. Then $\varepsilon=\sqrt{x/S}$ lies
there, since $S\ge T/2$. First let $x\le1$. Theorem~\ref{thm:layer}(b) bounds the logarithm by
$C(1/T+\sqrt{x/T})$, and by (a) $S|\mathfrak r_\beta(\varepsilon)|\le CS\varepsilon^3=Cx^{3/2}S^{-1/2}\le
C\sqrt{x/S}$. Now let $1\le x\le\delta T$. The proof of Theorem~\ref{thm:layer}(b) gives
$\log(\mathbf K_\beta/\mathbf B_\beta)=(\gamma_\beta-q_i)x+\lambda(\rho_s)-\zeta+O(\sqrt{x/S})$,
since $1/\zeta\le C\sqrt{x/S}$ for $x\ge1$. On the real interval $I=[z_*/2,2z_*]$ we have
$\lambda=\lambda_0+g$ with $\lambda_0(z)=x/z+S(J_G(\beta)-J_G(\beta,z))$ and
$g(z)=\log(C_G(\beta,z)/C_G(\beta,0))=O(z_*)$. Also
$\lambda''(z)=2x/z^3+R''(z)\ge\omega_\beta S/(4z_*)-C(1+S)>0$ on $I$ for $\delta$ small, so
$\lambda$ is convex on $I$, and $\rho_s\in I$. Hence $\lambda(\rho_s)=\min_I\lambda$, which differs
from $\min_I\lambda_0$ by $O(z_*)$. By the proof of (a) the point $z(\varepsilon)\in I$ maximizes
$SJ_G(\beta,z)-x/z$ over $z>0$, so
$\min_I\lambda_0=SJ_G(\beta)+q_ix-\bar J_F(s,x)$ by Proposition~\ref{prop:layer-constants}(i).
Since $\bar J_F(s,x)=q_ix+Sj(\varepsilon)$, $\zeta=2\sqrt{\omega_\beta}\,S\varepsilon$ and
$(\gamma_\beta-q_i)x=(\gamma_\beta-q_i)S\varepsilon^2$, this gives
\[
\log\frac{\mathbf K_\beta}{\mathbf B_\beta}=-S\bigl(j(\varepsilon)-J_G(\beta)+2\sqrt{\omega_\beta}\,\varepsilon-(\gamma_\beta-q_i)\varepsilon^2\bigr)+O\Bigl(\sqrt{\frac xS}\Bigr)
=-S\,\mathfrak r_\beta(\varepsilon)+O\Bigl(\sqrt{\frac xS}\Bigr).
\]
With Theorem~\ref{thm:layer}(a) and $S\ge T/2$ this proves the bound. The last claim follows
from (a), since $S\varepsilon^4=x^2/S$.

(c) If the two laws agree, then $\gamma_\beta=q_i$, and
Proposition~\ref{prop:layer-constants}(v) gives
$\bar J_F(\beta,\varepsilon^2)=J_G(\beta)+q_i\varepsilon^2-2\sqrt{\omega_\beta}\,\varepsilon$. So
$\mathfrak r_\beta=0$. For reversible $Q$ the laws agree by
Proposition~\ref{prop:layer-constants}(iv).

(d) Let $\varrho_m\ne0$ be the coefficient of $\varepsilon^m$ in $\mathfrak r_\beta$. Choose $\delta_\beta\le\delta$
so small that $|\mathfrak r_\beta(\varepsilon)-\varrho_m\varepsilon^m|\le\frac12|\varrho_m|\varepsilon^m$ for
$\varepsilon\le\sqrt{2\delta_\beta}$. Put $X=x^{m/2}S^{1-m/2}=S\varepsilon^m$. By (b) the logarithm
$\ell$ of the ratio satisfies $\bigl||\ell|-|\varrho_m|X\bigr|\le\frac12|\varrho_m|X+C(1/T+\sqrt{x/T})$.
Since $T/2\le S\le T$, $X$ tends to $0$ if and only if $x/T^{1-2/m}$ does. If it does, then
$x/T\to0$ too, and $\ell\to0$. If it does not, there are $\eta>0$ and arbitrarily large $T$ with
$X\ge\eta$. For these $T$, with $u=x/S$, we have $X=Su^{m/2}\ge\eta$, so $u\ge(\eta/S)^{2/m}$ and
$\sqrt u/X=u^{-(m-1)/2}/S\le\eta^{-(m-1)/m}S^{-1/m}$, which tends to $0$. So
$|\ell|\ge\frac12|\varrho_m|X-C(1/T+\sqrt{2u})\ge\frac14|\varrho_m|\eta$ for these $T$ when $T$ is
large, and the ratio does not tend to $1$. If $\mathfrak r_\beta=0$, then $|\ell|\le C(1/T+\sqrt{x/T})$ by
(b).

(e) For the cycle, $c^-=(q_{13},q_{23})=(1,2)$, $c^+=(q_{31},q_{32})=(2,1)$ and $q_1=q_2=q_3=3$.
So $d(z)=(3-2z)\ones$, $a_{12}(z)=2+z$ and $a_{21}(z)=1+4z$, and Lemma~\ref{lem:L-two} gives
$J_G(\beta,z)=3-2z-\sqrt{(2+z)(1+4z)}$. Its derivatives at $0$ are $-\omega_\beta=-2-\frac94\sqrt2$,
$\frac{49}{16}\sqrt2$ and $-\frac{1323}{64}\sqrt2$. Then
$c+a^2/\omega_\beta=-\frac{441}{128}\sqrt2+\frac{2401}{128(8+9\sqrt2)}=-\frac{49}{64}\,\omega_\beta$,
and dividing by $\omega_\beta^{3/2}$ gives $\varkappa_\beta$. Since $\varkappa_\beta<0$, (b) and the
choice of $\delta_\beta$ as in the proof of (d) give $\ell\ge\frac12|\varkappa_\beta|X-C(1/T+\sqrt{x/T})$, which tends to
infinity when $x/T^{1/3}\to\infty$.
\end{proof}

\subsection{The counts near a facet}\label{sec:layer-counts}

Now let $n$ have support $F$. Put $N'=N-n_i$, let $\hat n$ be $n$ restricted to $G$, and put
$\beta=\hat n/N'$, $S=hN'$, $x=hn_i$ and $p_i=1-hq_i$. So $T=S+x$. Let
$\widetilde\Lambda_\beta(z)=C_G(\beta,z)C_G(\beta,0)^{-1}e^{S(J_G(\beta)-J_G(\beta,z))}$ be the
integrand of $\mathbf K_\beta$ without the factors $e^{-q_ix}$, $e^{x/z}$ and $z^{-2}$, and define the
lattice kernel
\[
\mathbf K^h_\beta(n_i,S)=\frac h{2\pi\mathrm i}\oint_{|z|=\rho}\Bigl(p_i+\frac hz\Bigr)^{n_i-1}\widetilde\Lambda_\beta(z)\,\frac{dz}{z^2}
=\sum_{K=1}^{n_i}\binom{n_i-1}{K-1}p_i^{\,n_i-K}h^K\,[z^K]\widetilde\Lambda_\beta(z),
\]
with $0<\rho<2z_0$. It is $h\mathbf K_\beta$ with $e^{x/z-q_ix}$ replaced by $(p_i+h/z)^{n_i-1}$, as
Lemma~\ref{lem:layer-exact-discrete} replaces Lemma~\ref{lem:layer-exact}. If
$\widetilde\Lambda_\beta(z)=e^{S\omega_\beta z}\sum_{r\ge0}\xi_rz^r$, then
$\mathbf K^h_\beta(n_i,S)=p_i^{n_i}\sum_r\xi_r(h/p_i)^rB^{(r)}_{n_i}(u)$ with $u=hS\omega_\beta/p_i$
and \ifarxivbuild $B_n$ of~\eqref{eq:refresh-polynomial}, so the Laguerre polynomials of
Section~\ref{sec:vertex} replace the Bessel kernel\fi. The ratio of the two laws below carries one
power of $h$ for every $k$. Indeed $P_{N'}(\hat n)$ is close to $h^{k-2}f_G(s)$, by
Lemma~\ref{lem:discrete} for the face $G$ if $k\ge3$ and by the path $j^{N'}$ if $G=\{j\}$,
and Theorems~\ref{thm:layer}(a)
and~\ref{thm:layer-counts}(b) show that $P_N(n)$ is then within a factor $1+O(1/T+T^2/N)$ of
$h^{k-1}f_F(s,x)$, also when $n_i$ is small.

\begin{theorem}[the counts on a pair face]\label{thm:pair-counts}
Let $F=\{j,i\}$ with $q_{ji},q_{ij}>0$ and $\mu_j>0$. Put $\omega=q_{ji}q_{ij}$, $c_0=1$,
$c_1=q_{ji}+\frac{\mu_i}{\mu_j}q_{ij}$, $c_2=\frac{\mu_i}{\mu_j}\omega$ and $p_l=1-hq_l$, and let
$hq_{\max}\le\frac12$ and $N',n_i\ge1$.
\begin{enumerate}[(a)]
\item \textup{(The discrete Pedler formula.)} Exactly, $P_{N'}(\hat n)=\mu_jp_j^{N'-1}$ and
\[
\frac{P_N(n)}{P_{N'}(\hat n)}=R:=\sum_{K=1}^{n_i}\binom{n_i-1}{K-1}p_i^{\,n_i-K}\sum_{e=0}^2c_e\,h^e\binom{N'-1}{K-e}\Bigl(\frac{h^2\omega}{p_j}\Bigr)^{K-e}.
\]
\item $\mathbf K^h(n_i,S)=p_i^{n_i}\bigl[B_{n_i}(u)+c_1\frac h{p_i}B'_{n_i}(u)+c_2\frac{h^2}{p_i^2}B''_{n_i}(u)\bigr]$
with $u=hS\omega/p_i$.
\item With $U=n_iu$,
\[
1-\frac{U+2\sqrt U+2}{2N'}\le\frac R{\mathbf K^h(n_i,S)}\le e^{2hq_j(1+2\sqrt U)} .
\]
So if $n_i\le N'$ and $1\le T\le N^{1/2}$, then $R=\mathbf K^h(n_i,S)(1+O(h+h^2\sqrt{N'n_i}))$.
\item If $n_i\le N'$ and $hS\le1$, then $\mathbf K^h(n_i,S)=h\,\mathbf K(x,S)(1+O(h+hS))$, where
$\mathbf K(x,S)=f_F(S,x)/f_G(S)$ is given by Proposition~\ref{prop:pair-face}.
\end{enumerate}
\end{theorem}

Part (c) has no term $1/T$ and holds at every $T$, also when $T^2/N$ is of order one, where (d)
does not apply and \ifarxivbuild Section~\ref{sec:vertex}\fi{} treats the complete graph.

\begin{theorem}[the counts near a facet]\label{thm:layer-counts}
Let $F$, $i$, $G$ and $\mathcal K$ be as in Theorem~\ref{thm:layer}. There are $\delta>0$,
$\varepsilon_1>0$, $T_0$ and $C$ such that the following hold uniformly for $\supp n=F$ with
$\beta\in\mathcal K$, $1\le n_i\le\delta N$ and $T_0\le T\le\varepsilon_1N^{1/2}$.
\begin{enumerate}[(a)]
\item $P_N(n)=P_{N'}(\hat n)\,\mathbf K^h_\beta(n_i,S)\,(1+E)$ with
$|E|\le C(T^{-2}+x^{1/2}T^{-3/2}+T^2/N)\le\frac12$.
\item $\mathbf K^h_\beta(n_i,S)=h\,\mathbf K_\beta(x',S)(1+O(T^2/N))$ for $x'=x$ and for $x'=x-h$, where
$\mathbf K_\beta(0,S)=S\omega_\beta+c_1(\beta)$ by the same formula. Hence
$P_N(n)=hP_{N'}(\hat n)\,\mathbf K_\beta(x,S)(1+O(1/T+T^2/N))$.
\item Parts \textup{(b)} and \textup{(c)} of Theorem~\ref{thm:layer} hold with
$P_N(n)/(hP_{N'}(\hat n))$ in place of $f_F(s,x)/f_G(s)$ and with $T^2/N$ added to each error bound.
\end{enumerate}
\end{theorem}

For $k=2$, Theorem~\ref{thm:pair-counts}(a) gives the ratio $P_N(n)/P_{N'}(\hat n)$ exactly. By
Theorem~\ref{thm:layer-counts}(b) the lattice kernel matters only when $T^2/N$ is not small.
For $k\ge3$ that regime is open, since Theorem~\ref{thm:layer-counts}(a) then has an error of order
one. We prove both theorems in Appendix~\ref{app:counts}. The second one discretizes the
states of $G$ as in the proof of Lemma~\ref{lem:discrete}, keeps the state $i$ exact, and evaluates
the contour integral in $z$ as in the proof of Theorem~\ref{thm:layer}(a).

\section{The complete graph}\label{sec:complete}\label{sec:completegraph}

On the complete graph every face ties at first order, and a few exact facts resolve the
tie. The results of this section and the next are those of our preprint~\cite{paperB}, and
Table~\ref{tab:overlap} says which of them follow from the general theory and which use the
exact arguments of the preprint. On a complete face the tree sum of Proposition~\ref{prop:constant}(b) has a closed form. It is
classical, and we prove it here.

\begin{lemma}[weighted Cayley formula]\label{lem:weighted-cayley}
For $x_1,\ldots,x_k>0$ and the conductances $c_{ij}=x_ix_j$,
\[
\tree(c)=x_1x_2\cdots x_k\,(x_1+\cdots+x_k)^{k-2}.
\]
With every $x_i=1$ this is Cayley's count $k^{k-2}$ of the spanning trees of
the complete graph $K_k$.
\end{lemma}

\begin{proof}
Put $X=\sum_ix_i$. The Laplacian of the conductances is $\operatorname{diag}(Xx_i)-xx^\top$,
so without its first row and column it is
$\operatorname{diag}(Xx_2,\ldots,Xx_k)-x'x'^\top$ with $x'=(x_2,\ldots,x_k)$. The
matrix determinant lemma gives the determinant
$X^{k-1}\prod_{i\ge2}x_i\,(1-\sum_{i\ge2}x_i/X)=X^{k-2}\prod_ix_i$, and this is
$\tree(c)$ by the matrix-tree theorem~\cite{bollobas1998}. That theorem holds for real
conductances, since both sides are polynomials in the $c_{ij}$ that agree for integer
conductances, the multigraph case.
\end{proof}

\begin{example}[the complete face]\label{ex:complete-face}
Let $F$ be a face of $k\ge2$ states with $q_{ij}=1$ for all $i\ne j$ in $F$, and
any exit rates, for example a face of $Q=J-dI$, where $q_i=d-1$. Then $F$ is
reversible with $\nu=\ones$, so by Proposition~\ref{prop:constant}(b) we may take
$r_\alpha=l_\alpha=\sqrt\alpha$ entrywise, and the conductances
$c_{ij}(\alpha)=\sqrt{\alpha_i\alpha_j}$ are products. Put
$A=\sum_{i\in F}\sqrt{\alpha_i}$. Lemma~\ref{lem:weighted-cayley} gives
$\tree(c(\alpha))=(\prod_i\alpha_i)^{1/2}A^{k-2}$, so for every
$\alpha\in\Delta_F^\circ$
\[
C_F(\alpha)=\frac{(\mu\cdot\sqrt\alpha)\,A^{k/2}}{(4\pi)^{(k-1)/2}\,(\prod_i\alpha_i)^{3/4}}
=\frac{(\mu\cdot\sqrt\alpha)\,A}{\sqrt{\det(2\pi\Sigma_\alpha)}},\qquad
J_F(\alpha)=\sum_{i\in F}q_i\alpha_i-(A^2-1).
\]
The second form of $C_F(\alpha)$ uses $\det\Sigma_\alpha=2^{k-1}(\prod_i\alpha_i)^2/\tree(c(\alpha))$.
For $Q=J-dI$ the rate is $J_F(\alpha)=d-A^2$, the rate of Dupuis and
Liu~\cite[Eq.~(3.2)]{dupuisliu2015}. At the center $\alpha=\ones/k$ we have $A^2=k$
and $C_F=\mu(F)k^{k-1/2}(4\pi)^{-(k-1)/2}$, which is Corollary~\ref{cor:symface}
with $\tree(q_F)=k^{k-2}$.
\end{example}

From now on in this section $Q=J-dI$. The chain switches to each other state with
probability $h$ and keeps its state with probability $p=1-(d-1)h$. The
following symbols of~\cite{paperB} are local to this section and to
Appendix~\ref{app:complete}. Put
\[
\sigma=p-h,\qquad u=\frac hp,\qquad v=\frac h\sigma=\frac u{1-u},\qquad
V_N=\frac{p^{N-1}}d .
\]
Thus $u$ is the odds of a switch to one given state against no switch, and $V_N$ is the
probability of a vertex $Ne_i$ for the uniform start, that is, $M_{\{i\}}$. For a composition
$n$ let $W_n(j)$ be the number of words with the counts $n$ and $j$ switches, and
$S_n(u)=\sum_jW_n(j)u^j$. For a
composition with some zero coordinates, $S_n$ is the polynomial of its positive coordinates,
which does not depend on $d$. By~\eqref{eq:runform}, the uniform start has
\begin{equation}\label{eq:SnVN}
P_N(n)=V_NS_n(u)\qquad(0<p\le1),
\end{equation}
and so does every start that is constant on $\supp n$, with $V_N$ replaced by
$\mu_ap^{N-1}$ for $a\in\supp n$. A composition is \emph{balanced} if its positive coordinates
differ by at most $1$. For the uniform start all faces of $k$ states have the same face maximum
$M_F$, and all balanced compositions with $k$ positive coordinates have the same probability. For
$p\ge1/d$ this probability is $M_F$, by Theorem~\ref{thm:simplex-balance} below. Put
\[
a_k=\frac12\log(4\pi)-\frac{k+\frac12}{k-1}\log k,\qquad
D_k=e^{-(k-1)a_k}=\frac{k^{k+1/2}}{(4\pi)^{(k-1)/2}}\qquad(k\ge2),
\]
and $D_1=1$. Numerically $a_2=-0.4674$, $a_3=-0.6571$ and $a_4=-0.8139$.

\begin{corollary}[the complete graph]\label{cor:paperB}
A face $F$ of $k$ states has $\lambda_F=d-k$ and $\alpha^*_F$ uniform, and if $\mu(F)>0$ then
$K_F=\frac{\mu(F)}kD_k$. So $\Phi_1(F)=d-1$ for every face, and all faces tie at $\tau=1$. With
the uniform start, the crossing of a face of $k\ge2$ states with a vertex has
$\Delta=\delta=k-1$ and $b_{FG}=a_k$.
\end{corollary}

\begin{proof}
Each $Q_F$ is symmetric with row sums $-(d-k)$, and the complete graph on $F$ has $k^{k-2}$
spanning trees by Lemma~\ref{lem:weighted-cayley}. Hence Corollary~\ref{cor:symface} gives
$\lambda_F$, $\alpha^*_F$ and
\[
K_F=\mu(F)k^{(k+1)/2}k^{(k-2)/2}(4\pi)^{-(k-1)/2}=\frac{\mu(F)}kD_k.
\]
A vertex has $K=\frac1d$ and $\lambda=d-1$ for the uniform start, so
$b_{FG}=-\frac1{k-1}\log(dK_F)=-\frac1{k-1}\log D_k=a_k$.
\end{proof}

By Example~\ref{ex:complete-face} at the center, for the uniform start, $D_k=dK_F=k/\sqrt{\det(2\pi\Sigma_k)}$ with
$\Sigma_k=\Sigma_{\ones/k}=(2/k^2)(I-J/k)$ in the coordinates $\alpha'$. Hence $D_k$ is the mass $k$
of a face, in units of a vertex, over the Gaussian volume of the fluctuations of its
fractions.

\begin{theorem}[balanced crossings]\label{thm:simplex-crossing}
Let $\mu$ be uniform and let $F$ be a face of $2\le k\le d$ states.
\begin{enumerate}[(i)]
\item For every $N\ge k$ there is exactly one $p=p_{N,k}\in(0,1)$ with $M_F=V_N$. The face
maximum $M_F$ is larger than $V_N$ below $p_{N,k}$ and smaller above it, and so is the
probability of a balanced composition of $F$. The crossing $p_{N,k}$ is strictly increasing in $N$,
$p_{k,k}=1/(1+(d-1)(k!)^{-1/(k-1)})$, and every crossing has $u\le(k!)^{-1/(k-1)}<1$, so
$p_{N,k}>1/d$.
\item For large $N$, uniformly for $1\le T\le N^{1/3}$,
\[
\log\frac{M_F}{V_N}=(k-1)\Bigl[T-L+\frac12\log T-a_k\Bigr]
+O\Bigl(\frac1T+\frac{T^2}N\Bigr).
\]
With $T=L-\frac12\log L+t$, the ratio $M_F/V_N$ tends to $e^{(k-1)(t-a_k)}$, uniformly for
$t$ in compact sets.
\item As $N\to\infty$, the value $T_{N,k}=N(1-p_{N,k})/(d-1)$ of $T$ at the crossing satisfies
\[
T_{N,k}+\frac12\log T_{N,k}=L+a_k+O(1/L),
\]
and
\begin{equation}\label{eq:simplex-root-expansion}
T_{N,k}=L-\frac12\log L+a_k+\frac{\log L}{4L}+O\Bigl(\frac1L\Bigr).
\end{equation}
\item $a_2>a_3>a_4>\cdots$.
\end{enumerate}
\end{theorem}

For $d=2$ the chain gives the directions of the persistent walk, and $T_{N,2}$ is the
crossing of its middle compositions and its endpoints, with $a_2=\frac12\log(\pi/8)$. We prove the theorem
in Appendix~\ref{app:complete}.

\begin{theorem}[exact balancing]\label{thm:simplex-balance}
Let $1/d\le p<1$, and let $\mu$ be constant and positive on a face $F$. Moving one unit from a
coordinate $b$ of a composition with support $F$ to another coordinate $a$, where $n_b\ge n_a+2$ and
$n_a\ge1$, strictly increases the probability of the composition. So the most likely compositions with support
$F$ are exactly the balanced ones, and they all have the same probability.
\end{theorem}

For an integer $n\ge1$ put
\begin{equation}\label{eq:refresh-polynomial}
B_n(w)=\sum_{m=1}^n\binom{n-1}{m-1}\frac{w^m}{m!}
=\frac wn L^{(1)}_{n-1}(-w),
\end{equation}
where $L^{(1)}_{n-1}$ is the Laguerre polynomial~\cite[Eq.~18.5.12]{dlmf}. The proof
in Appendix~\ref{app:complete} sums over the numbers of blocks between refreshes first. This writes the law as an
integral of a product of one polynomial $B_{n_i}$ per state, and $B_n(w)$ is log-concave in $n$
for $w>0$. The condition
$p\ge1/d$ holds for $T\le N/d$, and so at every crossing by
Theorem~\ref{thm:simplex-crossing}(i).

\begin{corollary}[exactly one mode change]\label{cor:simplex-modes}
Let $\mu$ be uniform, and fix $d\ge2$ and $C>1$. There is $N_0$ such that for $N\ge N_0$ and
every $p$ with $0<T\le CL$ the following hold.
\begin{enumerate}[(i)]
\item If $p>p_{N,d}$, the global modes are exactly the $d$ vertices.
\item If $p<p_{N,d}$, they are exactly the balanced compositions with $d$ positive
coordinates.
\item If $p=p_{N,d}$, they are these two classes and nothing else.
\end{enumerate}
Moreover $p_{N,2}<p_{N,3}<\cdots<p_{N,d}$ and
$p_{N,d}-p_{N,k}=(d-1)(a_k-a_d+o(1))/N$. So all the crossings of balanced compositions with the
vertices lie in a $p$-interval of length $(d-1)(a_2-a_d+o(1))/N$, and as $p$ decreases the
balanced compositions with $d$ positive coordinates cross first.
\end{corollary}

On the triangle the middle composition
$(\frac N3,\frac N3,\frac N3)$ overtakes the vertices while the edge centers are still
below them by the factor $e^{-(a_2-a_3)}=0.827$.
In the tie window at $\tau^*=1$ the corollary is the case of Theorem~\ref{thm:window} whose envelope
$E$ has a single kink, at $t=a_d$. We prove the corollary in Appendix~\ref{app:complete}.

\begin{remark}[an intermediate face can be a mode at finite $N$]\label{rem:simplex-edge-mode}
Note that Corollary~\ref{cor:simplex-modes} does not hold for every $N$. Take $d=3$, $N=4$ and
$p=10/17$, so $u=7/20$ and $0<T=14/17<\log4$. Counting words gives
$P_4(2,2,0)/V_4=2u+2u^2+2u^3=4123/4000>1$ and $P_4(2,1,1)/V_4=6u^2+6u^3=3969/4000<1$.
Since $p>1/3$, Theorem~\ref{thm:simplex-balance} shows that every global mode is a vertex or
a balanced composition, so the three edge centers are the only global modes. Hence $N_0>4$
when $d=3$, for every $C>1$.
\end{remark}

\section{Near a vertex}\label{sec:vertex}

This section treats the boundary layer of Section~\ref{sec:boundary-layer} for the counts
themselves, on the complete graph near a vertex, in the regime where $T$ is of order $\sqrt N$. There the layer is discrete,
and its Laguerre kernel cannot be replaced by the Bessel kernel.

We keep the notation of Section~\ref{sec:completegraph}, namely $Q=J-dI$, the uniform start,
$p$, $u$, $v$, and $S_n(u)=P_N(n)/V_N$ by~\eqref{eq:SnVN}. Let one count $b$ be large and let
the other positive counts $a_1,\ldots,a_\ell$ be small compared with $b$, with $\ell\ge1$ fixed. We
put $|a|=a_1+\cdots+a_\ell$ and write $S_{(b,a)}$ for $S_n$ with $n=(b,a_1,\ldots,a_\ell)$. These
rare counts are not the constants $a_k$ of Section~\ref{sec:completegraph}. Note that $S_n$ does not
depend on the start. Hence $P_N(n)\le(\max_i\mu_i)p^{N-1}S_n(u)$, and upper bounds on $S_n$ bound $P_N(n)$
for every start.

A rare run away from the ends needs two switches, which gives a factor $u^2$, and it can go in
about $b$ places. Hence the rare counts enter through $bu^2$, and for fixed rare counts the
crossing moves to odds $u$ of order $b^{-1/2}$, that is, to $T$ of order $\sqrt N$. We use $B_a$
from~\eqref{eq:refresh-polynomial} and put
\[
\mathcal I(x)=\sum_{j\ge0}\frac{x^j}{j!\,(j+1)!}=\frac{I_1(2\sqrt x)}{\sqrt x},
\]
where $I_1$ is the modified Bessel function~\cite[Eq.~10.25.2]{dlmf}, and $\mathcal I(0)=1$.
By the Laguerre form in~\eqref{eq:refresh-polynomial}, the next lemma is a quantitative form
of the convergence of Laguerre polynomials to Bessel functions.

\begin{lemma}[a uniform Laguerre--Bessel inequality]\label{lem:uniform-laguerre-bessel}
For every integer $a\ge1$ and real $w>0$,
\begin{equation}\label{eq:uniform-laguerre-bessel}
0\le1-\frac{B_a(w)}{w\,\mathcal I(aw)}\le\frac w2 .
\end{equation}
\end{lemma}

We prove the lemma in Appendix~\ref{app:vertex}. Put
\[
\mathcal B(y)=\prod_{i=1}^{\ell}B_{a_i}(y),\qquad
\kappa(y)=\frac{y\,\mathcal B'(y)}{\mathcal B(y)} .
\]
The polynomial $\mathcal B$ has nonnegative coefficients, $\mathcal B(0)=0$ and
$\mathcal B(1)\ge1$, so $\mathcal B(y_a)=1$ has exactly one positive root $y_a$, and $y_a\le1$.
The polynomial $S_{(b,a)}$ has nonnegative coefficients, zero constant term and a positive
coefficient at $u^\ell$, so $S_{(b,a)}(u)=1$ has exactly one positive root $u_{b,a}$, the
crossing odds of the composition $(b,a_1,\ldots,a_\ell)$ with a vertex.

\begin{theorem}[the vertex law]\label{thm:vertex-uniform-root}
Fix $\ell\ge1$ and put
\[
 c_\ell=\ell+\sqrt\ell+1,\qquad
 \varepsilon=\sqrt{\frac{|a|+1}{b}},\qquad \delta=16c_\ell\varepsilon.
\]
For all positive counts $b,a_1,\ldots,a_\ell$ with $\delta\le1/2$,
\begin{equation}\label{eq:vertex-uniform-bracket}
 (1-\delta)\sqrt{y_a/b}\le u_{b,a}\le
 (1+\delta)\sqrt{y_a/b},\qquad \mathcal B(y_a)=1.
\end{equation}
In particular
\begin{equation}\label{eq:vertex-uniform-root}
 \biggl|\frac{u_{b,a}}{\sqrt{y_a/b}}-1\biggr|\le16c_\ell\sqrt{\frac{|a|+1}{b}}
 \qquad\text{whenever }\sqrt{\frac{|a|+1}{b}}\le\frac1{32c_\ell},
\end{equation}
so $u_{b,a}=\sqrt{y_a/b}\,(1+o(1))$ as $(|a|+1)/b\to0$. The constants are not
optimized.
\end{theorem}

For fixed rare counts Lemma~\ref{lem:vertex-randomization} implies
$S_{(b,a)}(\sqrt{y/b})\to\mathcal B(y)$ as $b\to\infty$. The theorem shows that the crossing
stays near $\sqrt{y_a/b}$ uniformly as long as $|a|=o(b)$. For $d=2$ and $\ell=1$ its leading
term is that of the fixed-edge expansion of~\cite[Theorem~4.10]{paperA}. We prove the theorem in
Appendix~\ref{app:vertex} and the proof rests on two facts. First, $\log\mathcal B$ is concave, and for
$0<y\le1$ the elasticity $\kappa(y)$ of $\mathcal B$ lies between $\max\{\ell,\frac12\sqrt{|a|y}\}$ and $\ell+\sqrt{\ell|a|y}$
(Lemma~\ref{lem:vertex-elasticity}). Second, $S_{(b,a)}(u)=(1-u)^{|a|}\E[\mathcal B(v\mathcal X)]$
exactly, where $\mathcal X$, a gamma variable plus a compound binomial sum, has mean
$2+(b-1)u$ (Lemma~\ref{lem:vertex-randomization}). Jensen's inequality then gives a lower bound
that exceeds $1$ at $u=(1+\delta)\sqrt{y_a/b}$, and the tangent line of $\log\mathcal B$ at
$bu^2$ with the moment generating function of $\mathcal X$ gives an upper bound that is below
$1$ at $u=(1-\delta)\sqrt{y_a/b}$.

With one rare count the vertex law is the discrete pair-face case of
Section~\ref{sec:boundary-layer}, with $F=\{1,2\}$, $G=\{1\}$, $i=2$ and $n=(b,a)$. By
Theorem~\ref{thm:pair-counts}(a) the main part of $P_N(b,a)$ is
$\mu_1p^{N-1}\sum_K\binom{a-1}{K-1}\binom{b-1}Ku^{2K}$, which becomes $B_a(bu^2)$ when
$\binom{b-1}K$ is replaced by $b^K/K!$. This is the layer with $\omega=1$, $S=hb$ and $x=ha$,
and $w=bu^2$ is $\omega Sh$ up to the factor $p^{-2}=1+O(h)$. In
Theorem~\ref{thm:vertex-uniform-root} with fixed rare counts, $w$ is of order one, Lemma~\ref{lem:uniform-laguerre-bessel}
pins $B_a(w)/(w\mathcal I(aw))$ between $1-w/2$ and $1$, and the continuum layer does not apply.
When $T^2/N\to0$, $w$ tends to $0$ and the lattice adds no separate regime.

\section{Verification and open questions}\label{sec:verify}

Every theorem, proposition, lemma and corollary of this paper is proved in the
text or in the appendices, apart from the results of~\cite{paperA}
and~\cite{paperB} that we quote where we say so. No proof rests on a
computation. Hence the formal proofs and the scripts below are independent checks of parts of the paper. All of them are in the repository named under Code availability.

\subsection*{Formal proofs}

The folders \nolinkurl{lean/PaperC} and \nolinkurl{lean/PaperB} prove many finite
and algebraic steps in Lean~4 (v4.33.1, with Mathlib), with no \texttt{sorry} or
\texttt{admit}. The file \nolinkurl{lean/PaperC/MAP.md} gives for each statement the Lean theorems
that prove it or its finite core, and lists what is not in Lean. Mathlib has no Perron and Frobenius theorem for irreducible Metzler
matrices, so a positive Perron vector is a hypothesis where one is
needed, and the asymptotic results are not in Lean.

\subsection*{Computations}

The command \texttt{py -I verify\_all.py} in the folder \nolinkurl{paperC} runs the
9 scripts that check this paper, in about four minutes, and writes each output,
headed by the sha256 of its script, to \nolinkurl{paperC/data}. Each check is an
assertion, and the checks in floating point use a tolerance without interval
certification. The scripts \texttt{verify\_first\_order.py} and
\texttt{verify\_second\_order.py} compute the law in two independent ways, by the run
expansion and by a dynamic program or a discrete torus integral, and check
Sections~\ref{sec:model} to~\ref{sec:secondorder}, including the gaps quoted after
Theorem~\ref{thm:crossing}. The script \texttt{verify\_full\_range.py} checks
Lemma~\ref{lem:torus}, Lemma~\ref{lem:discrete-hessian} and Theorems~\ref{thm:local}(b)
and~\ref{thm:facemax}(a). For $q_{12}=1$, $q_{21}=2$ and $N=4000$ the ratio $R$ of the exact
law to Theorem~\ref{thm:local}(b) has $T|R-1|\le0.25$ at points of $[0.45,0.85]$ for $T=16$, $63$, $200$, $400$
and $1000=N/4$, while the continuum form has $T|R-1|$ of order $T$ from $T=200$ on. The script \texttt{verify\_examples.py} finds the
first-order winners on complete graphs, cycles and small graphs in two ways, and the
finite-$N$ modes from two exact codes. The script \texttt{verify\_rate\_flows.py}
computes no law and checks the rate of Proposition~\ref{prop:rate} against a form
through flows. The script \texttt{verify\_complete.py} checks
\ifarxivbuild Theorem~\ref{thm:simplex-balance}, Theorem~\ref{thm:simplex-crossing}(i) and
Remark~\ref{rem:simplex-edge-mode} in exact arithmetic for small $N$, and the failure of
balancing for a start in one state quoted after Theorem~\ref{thm:C}. No script checks
Theorem~\ref{thm:simplex-crossing}(ii) and (iii) or Corollary~\ref{cor:simplex-modes}
at large $N$.\fi{} The script \texttt{verify\_layer.py} checks the constants and kernels of
Section~\ref{sec:boundary-layer} on three chains with three states. One of them
is the cycle of Proposition~\ref{prop:layer-cubic}(e), with $\omega_\beta=5.1820$
and $\gamma_\beta=3.4179$ against $q_3=3$, and the script recomputes the
derivatives of $J_G(\beta,z)$ and $\varkappa_\beta$ for it exactly. The
generalized kernel matches $f_F/f_G$ to $O(1/T)$ for $x$ from $0.01/T$ to $T/4$. The
script also checks the example after Proposition~\ref{prop:layer-cubic}, where
$J_G(\beta,\cdot)$ is affine at the uniform $\beta$ to $10^{-15}$.
The script \texttt{verify\_layer\_general.py} checks Theorem~\ref{thm:L-facet}
on two facets with four states, namely a facet of a non-reversible face with five
states and a path whose jump chain is periodic. It computes the word sum exactly
and prints $S$ times the relative error for $S$ from $4$ to $32$. These products
settle near $0.438$ and $0.545$ at $z=0$, as~\eqref{eq:L-facet} predicts. The script
\texttt{verify\_counts.py} checks Section~\ref{sec:reducible-max} and
Section~\ref{sec:layer-counts}. It checks Theorem~\ref{thm:pair-counts}(a) against an
enumeration of all paths in exact arithmetic, for every $N\le12$ and two values of $h$, and
the bounds of Theorem~\ref{thm:pair-counts}(c) on a grid up to $N'=10^4$. For the cycle of
Proposition~\ref{prop:layer-cubic}(e) with $i=3$ and $\beta=(\frac12,\frac12)$ it computes
the exact law by a dynamic program and the lattice kernel by a contour integral, and the error
$E$ of Theorem~\ref{thm:layer-counts}(a) satisfies $0<E\le0.62\,T^2/N$ and falls by a factor
of about $2.5$ each time $N'$ doubles, at $T=8$, $n_3=10$ and $200\le N'\le1600$. On the face $\{1\}\to\{2,3\}$ with
$q_{12}=q_{23}=q_{32}=1$ and the start in state $1$, where $p_F=\frac32$ and
$K^{\rm red}_F=\sqrt{2/\pi}$, it computes $M_F$ exactly at $N=40000$, and the slope of
$\log M_F$ in $\log T$ is within $10^{-3}$ of $\frac32$ for $8\le T\le20$, far from the powers $1$
and $2$ of Theorem~\ref{thm:first}(ii), while $M_F$ divided by $N^{-2}T^{3/2}$ is within
$0.05\%$ of $\sqrt{2/\pi}$. The
figures only illustrate the theorems. Hence no statement is deduced from them or from
these numbers.

\subsection*{Open questions}

A local law up to the lower-dimensional corners of a face is open. The results on reducible
faces (Theorems~\ref{thm:first}(ii) and~\ref{thm:reducible} and
Corollary~\ref{cor:reducible-crossing})\ifarxivbuild{} and Theorem~\ref{thm:simplex-crossing}(ii)\fi{}
are proved for $T\le N^{1/3}$, and we have not extended them to $T\le h_*N$. We do not know whether the
zero in Corollary~\ref{cor:reducible-crossing} is unique, or where the most likely counts of a
reducible face lie. For a facet of three or more states, the law of the counts near the facet
when $T^2/N$ is of order one is open. Theorem~\ref{thm:layer-counts} has an error of order one
there, and a statement would need a discrete form of Theorem~\ref{thm:L-facet}. We do not know
whether $\mathfrak r_\beta=0$ forces $\pi^{\rm in}_\beta=\pi^{\rm out}_\beta$ in
Proposition~\ref{prop:layer-cubic}.

\paragraph{Acknowledgments.} The question grew out of Problem~131 of Peter
Kagey's Open Problem Collection~\cite{kagey131}, and I thank him for posing it.
I thank Chase Hameetman and Aidan Deshong for conversations about some of the
proofs of \Bversion~\cite{paperB}.

\section*{A note on the use of AI}
I wrote all of the proofs and text of this version. Apart from the addition
described below, I developed its general theory myself and used Claude Opus 5.5
(Anthropic) only to check it and inform me of any mistakes.

The first version of this arXiv entry~\cite{paperB} describes how AI was used for
that version. I wrote its Sections~2 and~3 by hand. The model then checked the
manuscript and drafted rewordings, which I reviewed and edited, and I also checked
its proofs with GPT-6 Astra (OpenAI). Of the results that version credits to the
model, one is reproved here: the tree form of the constants in its
Proposition~3.6, which is Lemma~\ref{lem:weighted-cayley} and
Proposition~\ref{prop:constant}(b).

The model reviewed the manuscript in several referee-style
passes and a proofreading pass, ran numerical experiments that aided the
revision, and searched the literature.

I wrote the verification scripts and the Lean files described in
Section~\ref{sec:verify}. The model checked the scripts and fixed errors in the
Lean code. I take responsibility for all claims, proofs and computations.

\paragraph{Statements and declarations.} No funding was received for this work.
The author has no competing interests to declare. Ethics approval, consent to
participate, consent for publication and materials availability are not
applicable. The use of AI is described in the note above. I am the sole author and am
responsible for the whole paper.

\paragraph{Data availability.} The outputs of the computations described in
Section~\ref{sec:verify} are in the folder
\nolinkurl{kagey-problems/problem131/paperC/data} of the repository named under
Code availability, each headed by the sha256 of the script that wrote it.

\paragraph{Code availability.} All code, data and Lean files supporting this
paper are in the folder \nolinkurl{kagey-problems/problem131} of
\url{https://github.com/apemm/Kagey-Problems}, with the scripts in the
subfolder \nolinkurl{paperC} and the Lean files in \nolinkurl{lean}.
Section~\ref{sec:verify} says how to run them. The scripts need Python~3 with
\texttt{numpy}, \texttt{scipy}, \texttt{sympy} and \texttt{mpmath}, and the
option \texttt{-{}-full} adds the largest cases of four of them.

\appendix
\section{Proofs of the local law and of the first-order bounds}\label{app:torus}

We prove Proposition~\ref{prop:constant}(a),
Lemma~\ref{lem:upper}, Theorem~\ref{thm:local}, Lemma~\ref{lem:discrete},
Lemma~\ref{lem:derivative}, Theorem~\ref{thm:first}(ii) and
Lemma~\ref{lem:discrete-hessian}. Except in
Appendix~\ref{app:first}, $F$ is a face with $k\ge2$ states, numbered
$1,\dots,k$. For $x\in\mathbb C^F$ we write $x'=(x_2,\dots,x_k)$. We write
$\mathrm i$ for the imaginary unit, since $i$ denotes a state, and
$\mathbb T^m=(\R/2\pi\Z)^m$. We use two standard facts~\cite{bermanplemmons1994}.
First, let $B\ge0$ be square and $S$ diagonal with positive diagonal. Then
$\varrho(B-S)<0$ if and only if $\varrho(S^{-1}B)<1$, and also if and only if
$(S-B)v>0$ for some $v>0$. In that case
$(S-B)^{-1}=\sum_{j\ge0}(S^{-1}B)^jS^{-1}\ge0$. Second, if $M$ is an
irreducible Metzler matrix, then every nonnegative eigenvector of $M$ is positive
and belongs to $\varrho(M)$, and $\varrho(M)$ is larger than the Perron root of each
proper principal submatrix of $M$ (Lemma~\ref{lem:face-monotone}).

\subsection{The Perron root under a diagonal perturbation}

An irreducible generator on $F$ is a Metzler matrix $\widetilde Q$ with $\widetilde Q\ones=0$ whose
graph is strongly connected. It has a unique stationary law $\pi>0$. The
argument after Definition~\ref{def:tilt}, with $(\widetilde Q,\pi)$ in place of
$(\widehat Q_\alpha,\alpha)$, shows that
$Z=(\ones\pi^\top-\widetilde Q)^{-1}-\ones\pi^\top$ exists and
satisfies~\eqref{eq:Zprops}. Put $D=\operatorname{diag}\pi$ and
$\Gamma=DZ+Z^\top D$.

\begin{lemma}[the Perron root]\label{lem:perron}
Let $\widetilde Q$, $\pi$, $Z$ and $\Gamma$ be as above. Then
$\kappa(\eta)=\varrho(\widetilde Q+\operatorname{diag}\eta)$ is real-analytic near $0$,
with $\kappa(0)=0$, $\nabla\kappa(0)=\pi$ and $\nabla^2\kappa(0)=\Gamma$. For
$\eta\in\R^F$ and $u=Z\eta$,
\begin{equation}\label{eq:A-dirichlet}
\eta^\top\Gamma\eta=\sum_{i\ne j}\pi_i\widetilde Q_{ij}(u_i-u_j)^2 .
\end{equation}
So $\Gamma$ is positive semidefinite with kernel $\R\ones$, and the matrix $\Sigma$
obtained from $\Gamma$ by deleting the row and column of state $1$ is positive
definite.
\end{lemma}

\begin{proof}
We differentiate the eigenvalue equation twice at $\eta=0$. Since
$\widetilde Q+\operatorname{diag}\eta$ is irreducible, its Perron root is a simple
eigenvalue. So by the implicit function theorem, $\kappa$ and the right
eigenvector $v(\eta)$ with $\pi^\top v=1$ are real-analytic near $0$, with
$\kappa(0)=0$ and $v(0)=\ones$. We write subscripts for derivatives at
$\eta=0$. Differentiating $(\widetilde Q+\operatorname{diag}\eta)v=\kappa v$ in $\eta_j$
gives $\widetilde Qv_j+e_j=\kappa_j\ones$, and multiplying by $\pi^\top$ gives
$\kappa_j=\pi_j$, since $\pi^\top \widetilde Q=0$. Then $v_j=Ze_j$, because
by~\eqref{eq:Zprops} both vectors solve $\widetilde Qx=\pi_j\ones-e_j$ with
$\pi^\top x=0$, and the kernel of $\widetilde Q$ is $\R\ones$. Differentiating again, in
$\eta_i$, gives
$\widetilde Qv_{ij}+\operatorname{diag}(e_i)v_j+\operatorname{diag}(e_j)v_i
=\kappa_{ij}\ones+\kappa_iv_j+\kappa_jv_i$. We multiply by $\pi^\top$ again.
The terms with $\widetilde Q$, $\kappa_i$ and $\kappa_j$ vanish, and we get
$\kappa_{ij}=\pi_iZ_{ij}+\pi_jZ_{ji}=\Gamma_{ij}$.

For~\eqref{eq:A-dirichlet} let $\bar\eta=\eta-(\pi\cdot\eta)\ones$.
By~\eqref{eq:Zprops}, $-\widetilde Qu=\bar\eta$ and $\pi\cdot u=0$, so
$\eta^\top\Gamma\eta=2\sum_i\pi_i\eta_iu_i=2\sum_i\pi_i\bar\eta_iu_i
=-2\sum_i\pi_iu_i(\widetilde Qu)_i$. Now expand the squares in~\eqref{eq:A-dirichlet}.
By $\widetilde Q\ones=0$ and $\pi^\top \widetilde Q=0$ the terms with $u_i^2$ and with $u_j^2$ each
sum to $-\sum_i\pi_i\widetilde Q_{ii}u_i^2$, and the cross terms sum to
$-2\sum_i\pi_iu_i(\widetilde Qu)_i+2\sum_i\pi_i\widetilde Q_{ii}u_i^2$. So both sides
of~\eqref{eq:A-dirichlet} agree. Also $\Gamma\ones=0$ by~\eqref{eq:Zprops}.
If $\eta^\top\Gamma\eta=0$, then $u$ is constant along every edge of the
strongly connected graph of $\widetilde Q$, so $u$ is constant. Then $u=0$ because
$\pi\cdot u=0$, so $\bar\eta=0$ and $\eta\in\R\ones$. A nonzero vector that vanishes
at state $1$ is not in $\R\ones$, which gives the claim about $\Sigma$.
\end{proof}

\begin{proof}[Proof of Proposition~\ref{prop:constant}(a)]
\emph{Step 1. The derivatives of the Perron root.} For $\vartheta\in\R^F$ let
$\Lambda(\vartheta)=\varrho(Q_F+\operatorname{diag}\vartheta)$, so that
$\Lambda(\vartheta+c\ones)=\Lambda(\vartheta)+c$, which we call the shift property. The Perron root is
a simple eigenvalue with positive right and left eigenvectors
$\hat r,\hat l$, normalized by $\sum_i\hat r_i=1$ and $\hat l\cdot\hat r=1$,
and by the implicit function theorem $\Lambda$, $\hat r$ and $\hat l$ are
real-analytic in $\vartheta$. Put $\pi=\hat l\circ\hat r$ and
$\widetilde Q=\hat R^{-1}(Q_F+\operatorname{diag}\vartheta-\Lambda(\vartheta)I)\hat R$
with $\hat R=\operatorname{diag}\hat r$. Then $\widetilde Q\ones=0$ and $\pi^\top \widetilde Q=0$
by the two eigenvector equations, so $\widetilde Q$ is an irreducible generator with
stationary law $\pi$. Since $\hat R$ commutes with diagonal matrices,
$\Lambda(\vartheta+\eta)=\Lambda(\vartheta)+\varrho(\widetilde Q+\operatorname{diag}\eta)$.
So by Lemma~\ref{lem:perron}, $\nabla\Lambda(\vartheta)=\pi$ and
$\nabla^2\Lambda(\vartheta)=\Gamma$, with $\Gamma$ built from $(\widetilde Q,\pi)$.

\emph{Step 2. The gradient can be inverted.} Let
$\Lambda_0(\beta)=\Lambda(0,\beta)$ for $\beta\in\R^{k-1}$. Its Hessian is
$\Gamma$ without the row and column of state $1$, which is positive definite
by Lemma~\ref{lem:perron}. Hence $\Lambda_0$ is strictly convex,
$\nabla\Lambda_0$ is injective, and by the inverse function theorem it has a
real-analytic inverse on its open image.

\emph{Step 3. The tilt that belongs to $\alpha$.} Fix
$\alpha\in\Delta_F^\circ$ and put $\vartheta(\alpha)=q-\theta(\alpha)$. Then
$Q_F+\operatorname{diag}\vartheta(\alpha)=A_F-\operatorname{diag}\theta(\alpha)$
has the positive right and left eigenvectors $r_\alpha$ and $l_\alpha$ for
the eigenvalue $0$, by Definition~\ref{def:tilt} and~\eqref{eq:stationary}.
So its Perron root is $0$, and at $\vartheta(\alpha)$ we have $\pi=\alpha$,
$\widetilde Q=\widehat Q_\alpha$ and $\Gamma=\Gamma_\alpha$. Let $\beta(\alpha)$ be the
last $k-1$ coordinates of $\vartheta(\alpha)-\vartheta_1(\alpha)\ones$. A
shift along $\ones$ does not change the derivatives of $\Lambda$, so
$\nabla\Lambda_0(\beta(\alpha))=\alpha'$ and
$\nabla^2\Lambda_0(\beta(\alpha))=\Sigma_\alpha$. Hence $\Sigma_\alpha$ is
positive definite, and $\beta(\alpha)=(\nabla\Lambda_0)^{-1}(\alpha')$ is
real-analytic. The shift property and $\Lambda(\vartheta(\alpha))=0$ give
$\vartheta_1(\alpha)=-\Lambda_0(\beta(\alpha))$. So $\theta(\alpha)$ is
real-analytic, and then so are its Perron vectors and $\Sigma_\alpha$.

\emph{Step 4. The Hessian of $J_F$.} Since
$\vartheta(\alpha)=\vartheta_1(\alpha)\ones+(0,\beta(\alpha))$ and
$\sum_i\alpha_i=1$, we have $J_F(\alpha)=\langle\vartheta(\alpha),\alpha\rangle
=\langle\beta(\alpha),\alpha'\rangle-\Lambda_0(\beta(\alpha))$. We
differentiate in $\alpha'$. The terms with the derivative of $\beta$ cancel
because $\nabla\Lambda_0(\beta(\alpha))=\alpha'$, and $\nabla J_F=\beta$
remains. So $\nabla^2J_F$ is the Jacobian of $(\nabla\Lambda_0)^{-1}$ at
$\alpha'$, which is $\Sigma_\alpha^{-1}$. The formula for $C_F$ follows from
Definition~\ref{def:tilt}.
\end{proof}

\subsection{The torus representation}

The generating function of the words in $H$ is Good's resolvent
formula~\cite[(2)]{good1961}, and the torus integral below is his constant-term
formula, here for the switching rates $A_F$ of the killed chain and with the circle of one state
integrated out. We split
\[
A_F=\begin{pmatrix}0&a_1^\top\\ b_1&A'\end{pmatrix},\qquad
\mu_F=(\mu_1,\mu'),
\]
so $a_1=(q_{1j})_{j\ge2}$, $b_1=(q_{i1})_{i\ge2}$ and
$A'=(q_{ij})_{i,j\ge2}$. For $z'\in\mathbb C^{k-1}$ let
$\mathcal Z'=\operatorname{diag}(z')$. Where $\mathcal Z'-A'$ is invertible
we put $X(z')=(\mathcal Z'-A')^{-1}$, $g(z')=a_1^\top X(z')b_1$ and
\[
b(z')=\bigl(\mu_1+\mu'^\top X(z')b_1\bigr)\bigl(1+a_1^\top X(z')\ones\bigr).
\]

\begin{lemma}[torus representation]\label{lem:torus}
Let $s'\in(0,\infty)^{k-1}$ satisfy
$\varrho(A'-\operatorname{diag}s')<0$. On the set of $z'\in\mathbb C^{k-1}$
with $|z_i|\ge s_i$ for every $i\ge2$, the matrix
$\mathcal Z'-A'$ is invertible, and $X$, $g$ and $b$ are given by absolutely
convergent power series in $1/z_2,\dots,1/z_k$ with nonnegative coefficients.
Hence $|g(z')|\le g(s')$ and
$|b(z')|\le b(s')$ there. For every $t\in(0,\infty)^F$,
\[
H(t)=\frac1{(2\pi)^{k-1}}\int_{\mathbb T^{k-1}}
\exp\Bigl(t_1g(z')+\sum_{i\ge2}t_iz_i\Bigr)\,b(z')\prod_{i\ge2}z_i\,d\varphi,
\qquad z_i=s_ie^{\mathrm i\varphi_i}.
\]
Let moreover $hq_{\max}\le1$, put $p_i=1-hq_i$, and let $n\in\Z_{\ge1}^F$. Then the
probability that $X_1,\dots,X_{|n|}$ have composition $n$ is
\[
P_{|n|}(n)=\frac{h^{k-1}}{(2\pi)^{k-1}}\int_{\mathbb T^{k-1}}
\bigl(p_1+hg(z')\bigr)^{n_1-1}\prod_{i\ge2}(p_i+hz_i)^{n_i-1}\,b(z')\prod_{i\ge2}z_i\,d\varphi .
\]
\end{lemma}

\begin{proof}
\emph{Step 1. The series.} By the first fact,
$\varrho(\operatorname{diag}(s')^{-1}A')<1$. Hence the series
$X(z')=\sum_{j\ge0}(\mathcal Z'^{-1}A')^j\mathcal Z'^{-1}$ converges
absolutely for $|z_i|\ge s_i$, since its terms are bounded
entrywise in modulus by those at $z'=s'$. As $a_1$, $b_1$ and $\mu$ are
nonnegative, this gives the claims about $X$, $g$ and $b$.

\emph{Step 2. A radius for state $1$.} Choose $s_1>g(s')$ and put
$S=\operatorname{diag}(s_1,s')$. We show that $\varrho(S^{-1}A_F)<1$. For
$\varepsilon>0$ let $v=(1,v')$ with
$v'=X(s')(b_1+\varepsilon\ones)>0$. Then $((S-A_F)v)_i=\varepsilon$ for
$i\ge2$, and $((S-A_F)v)_1=s_1-g(s')-\varepsilon a_1^\top X(s')\ones$. For
small $\varepsilon$ both are positive, and the first fact gives the claim.

\emph{Step 3. The $k$ circles.} Let $\mathcal Z=\operatorname{diag}(z)$ with
$|z_i|=s_i$ for every $i\in F$. The matrices $A_FS^{-1}$ and $S^{-1}A_F$ are similar, so
$\varrho(A_FS^{-1})=\varrho(S^{-1}A_F)<1$, and expanding
$(\mathcal Z-A_F)^{-1}=\mathcal Z^{-1}(I-A_F\mathcal Z^{-1})^{-1}$ gives
\[
\mu_F^\top(\mathcal Z-A_F)^{-1}\ones=\sum_c\mu_{c_1}\prod_{j<M}q_{c_jc_{j+1}}
\prod_{i\in F}z_i^{-m_i(c)},
\]
absolutely and uniformly on the torus. Here $c=c_1\cdots c_M$ runs over the
words in $F$ with $M\ge1$ and $c_j\ne c_{j+1}$, and $m_i(c)$ is the number of
letters $i$ in $c$. For an integer $m$ and $t_i>0$, Cauchy's formula gives
$\frac1{2\pi}\int_0^{2\pi}e^{t_iz_i}z_i^{1-m}\,d\varphi_i=t_i^{m-1}/(m-1)!$ if
$m\ge1$ and $0$ if $m\le0$. We multiply the series by
$e^{\langle t,z\rangle}\prod_iz_i$ and integrate term by term. The words that
miss a state of $F$ drop out, and grouping the others by $m(c)$ gives
\begin{equation}\label{eq:A-full}
H(t)=\frac1{(2\pi)^k}\int_{\mathbb T^k}e^{\langle t,z\rangle}\,
\mu_F^\top(\mathcal Z-A_F)^{-1}\ones\prod_{i\in F}z_i\,d\varphi .
\end{equation}

\emph{Step 4. The first circle.} The Schur complement $z_1-g(z')$ of
$\mathcal Z'-A'$ in $\mathcal Z-A_F$ is not $0$, since
$|g(z')|\le g(s')<s_1$. Block inversion gives
\[
\mu_F^\top(\mathcal Z-A_F)^{-1}\ones=\frac{b(z')}{z_1-g(z')}
+\mu'^\top X(z')\ones .
\]
We integrate over $\varphi_1$ in~\eqref{eq:A-full} and divide by $2\pi$. The
second term does not depend on $z_1$, so it contributes
$\frac1{2\pi}\int_0^{2\pi}e^{t_1z_1}z_1\,d\varphi_1=0$. In the first term
$g(z')$ lies inside the circle $|z_1|=s_1$, so Cauchy's formula gives
$e^{t_1g(z')}b(z')$. The factors without $z_1$ are unchanged, and this proves
the formula for $H$.

\emph{The discrete formula.} We repeat Steps~3 and~4 with $(p_i+hz_i)^{n_i-1}$ in
place of $e^{t_iz_i}$. The coefficient of $z^{m-1}$ in the polynomial
$(p+hz)^{n-1}$ is $h^{m-1}\binom{n-1}{m-1}p^{n-m}$, which is $0$ for $m>n$. So
$\frac1{2\pi}\int_0^{2\pi}(p_i+hz_i)^{n_i-1}z_i^{1-m}\,d\varphi_i$ is this
coefficient if $m\ge1$ and $0$ if $m\le0$. The polynomial is bounded on the
circle, so the term-by-term integration of Step~3 applies, and it gives
\[
\frac1{(2\pi)^k}\int_{\mathbb T^k}\mu_F^\top(\mathcal Z-A_F)^{-1}\ones
\prod_{i\in F}(p_i+hz_i)^{n_i-1}z_i\,d\varphi
=\sum_mW(m)\prod_{i\in F}h^{m_i-1}\binom{n_i-1}{m_i-1}p_i^{n_i-m_i}.
\]
Since $h^{M-1}=h^{k-1}\prod_ih^{m_i-1}$, this is $h^{-(k-1)}P_{|n|}(n)$ by
Lemma~\ref{lem:runs} with $|n|$ in place of $N$. In Step~4 the second term gives
$\frac1{2\pi}\int_0^{2\pi}(p_1+hz_1)^{n_1-1}z_1\,d\varphi_1=0$, and Cauchy's formula
evaluates the polynomial $(p_1+hz_1)^{n_1-1}$ at $z_1=g(z')$.
\end{proof}

Section~\ref{sec:boundary-layer} needs the same formula for complex weights and an end vector
other than $\ones$. Let $A$ be a complex matrix on $F$ with zero diagonal and let
$\nu,e\in\mathbb C^F$. Let $W_{A,\nu,e}(m)$ be the sum over the words $c$ of $W(m)$ of
$\nu_{c_1}\prod_{j<M}A_{c_jc_{j+1}}e_{c_M}$, and $H_{A,\nu,e}(t)$ the sum $H(t)$ with $W_{A,\nu,e}$
in place of $W$. So $H=H_{A_F,\mu,\ones}$. We split $A$ and $\nu$ as above and define $X$ and $g$
by the same formulas, and
\[
b(z')=\bigl(\nu_1+\nu'^\top X(z')b_1\bigr)\bigl(e_1+a_1^\top X(z')e'\bigr).
\]

\begin{lemma}[complex weights]\label{lem:torus-complex}
Let $\bar A\ge0$ be a matrix on $F$ with zero diagonal and $\bar\nu,\bar e\in[0,\infty)^F$ with
$|A_{ij}|\le\bar A_{ij}$, $|\nu_i|\le\bar\nu_i$ and $|e_i|\le\bar e_i$ for all $i,j$, and let
$\bar X$, $\bar g$ and $\bar b$ be the functions above for $(\bar A,\bar\nu,\bar e)$. Let
$s'\in(0,\infty)^{k-1}$ satisfy $\varrho(\bar A'-\operatorname{diag}s')<0$. Then
the formula for $H$ in Lemma~\ref{lem:torus} holds for $H_{A,\nu,e}$ with these $X$, $g$
and $b$, and their
series no longer have nonnegative coefficients, but they are dominated term by term in modulus by those of $\bar X$, $\bar g$ and $\bar b$ at $s'$.
In particular $|g(z')|\le\bar g(s')$ and $|b(z')|\le\bar b(s')$ when $|z_i|\ge s_i$ for every
$i\ge2$. Where $\mathcal Z'-A'$ is invertible, $X$, $g$ and $b$ are rational functions of $z'$ and
of the entries of $A$, $\nu$ and $e$.
\end{lemma}

\begin{proof}
We follow the four steps of the proof of Lemma~\ref{lem:torus}. In Step~1 the first fact gives
$\varrho(\operatorname{diag}(s')^{-1}\bar A')<1$, and for $|z_i|\ge s_i$ the entries of
$\mathcal Z'^{-1}A'$ are bounded in modulus by those of $\operatorname{diag}(s')^{-1}\bar A'$. So
the series of $X$ converges absolutely, and its terms are dominated by those of
$\bar X(s')\ge0$. Since $a_1$, $b_1$, $\nu$ and $e$ are dominated by their barred versions, the
same holds for $g$ and for the two factors of $b$. Step~2 applies to $\bar A$ as it stands and
gives $s_1>\bar g(s')$ with $\varrho(S^{-1}\bar A)<1$. In Step~3 the entries of $A\mathcal Z^{-1}$
are bounded in modulus by those of $\bar AS^{-1}$, so the word expansion of
$\nu^\top(\mathcal Z-A)^{-1}e$ converges absolutely and uniformly on the torus, and the same
integration gives~\eqref{eq:A-full} with $A$, $\nu$ and $e$ in place of $A_F$, $\mu_F$ and
$\ones$. In Step~4 the Schur complement $z_1-g(z')$ is not $0$, since $|g(z')|\le\bar g(s')<s_1$,
and block inversion gives
$\nu^\top(\mathcal Z-A)^{-1}e=b(z')/(z_1-g(z'))+\nu'^\top X(z')e'$. The rest of Step~4 is
unchanged. The last claim is Cramer's rule.
\end{proof}

\subsection{Proof of \texorpdfstring{Lemma~\ref{lem:upper}}{the upper bound}}

We bound the integrand of the discrete formula of Lemma~\ref{lem:torus} by its
modulus at the real point. Step~2 bounds the continuum sums in the same way, for
later use.

\emph{Step 1. The radii.} Let $F$ be irreducible and $\theta\in V_F$. Then
$A_F-\operatorname{diag}\theta$ is an irreducible Metzler matrix with Perron
root $0$. Hence it has a right eigenvector $u>0$ for $0$, and
$\theta_i=(A_Fu)_i/u_i$. Each row of $A_F$ has a positive entry, since $G_F$
is strongly connected and $k\ge2$, so $\theta>0$. By the second fact,
$\varrho(A'-\operatorname{diag}\theta')<0$. The rows $i\ge2$ of
$(A_F-\operatorname{diag}\theta)u=0$ give $u'=u_1X(\theta')b_1$, and then row
$1$, which reads $a_1^\top u'=\theta_1u_1$, gives $\theta_1=g(\theta')$. Hence
\begin{equation}\label{eq:A-radii}
\theta>0,\qquad \varrho(A'-\operatorname{diag}\theta')<0,\qquad
g(\theta')=\theta_1\qquad\text{for every }\theta\in V_F .
\end{equation}

\emph{Step 2. Continuum sums.} For $\sigma>0$ let $H_\sigma$ be $H$ for $\sigma A_F$
in place of $A_F$. We use Lemma~\ref{lem:torus} for
$\sigma A_F$ with the radii $\sigma\theta'$, which is allowed since
$\varrho(\sigma A'-\operatorname{diag}\sigma\theta')
=\sigma\varrho(A'-\operatorname{diag}\theta')<0$. The functions of
Lemma~\ref{lem:torus} for $\sigma A_F$ satisfy
$X_\sigma(\sigma z')=\sigma^{-1}X(z')$, so $g_\sigma(\sigma z')=\sigma g(z')$
and $b_\sigma(\sigma z')=b(z')$. So on the torus with radii $\sigma\theta'$
we have $|g_\sigma|\le\sigma g(\theta')=\sigma\theta_1$
by~\eqref{eq:A-radii}, $|b_\sigma|\le b(\theta')$ and
$|z_i|=\sigma\theta_i$. Hence the integrand has modulus at most
$e^{\sigma\langle t,\theta\rangle}b(\theta')\sigma^{k-1}
\prod_{i\ge2}\theta_i$, and for every $\sigma>0$ and $t\in(0,\infty)^F$
\begin{equation}\label{eq:A-Hsigma}
H_\sigma(t)\le\sigma^{k-1}b(\theta')\prod_{i\ge2}\theta_i\,
e^{\sigma\langle t,\theta\rangle}.
\end{equation}

\emph{Step 3. The bound.} Let $n$ have support $F$, let $hq_{\max}\le1$ and
$p_i=1-hq_i$. We use the discrete formula of Lemma~\ref{lem:torus} with the radii
$\theta'$. On that torus $|p_i+hz_i|\le p_i+h\theta_i$ for $i\ge2$,
$|p_1+hg(z')|\le p_1+h\theta_1$ since $|g(z')|\le g(\theta')=\theta_1$, and
$|b(z')|\le b(\theta')$. So
\begin{equation}\label{eq:A-chernoff}
P_{|n|}(n)\le h^{k-1}b(\theta')\prod_{i\ge2}\theta_i
\prod_{i\in F}(p_i+h\theta_i)^{n_i-1}
\qquad(\theta\in V_F,\ hq_{\max}\le1).
\end{equation}
Now let $hq_{\max}\le\frac12$. Since $\log(1+x)\le x$, we have
$(p_i+h\theta_i)^{n_i}=(1-h(q_i-\theta_i))^{n_i}\le e^{-hn_i(q_i-\theta_i)}$, and
$p_i+h\theta_i\ge p_i\ge\frac12$. So the lemma holds with
$C_\theta=2^kb(\theta')\prod_{i\ge2}\theta_i$. By Lemma~\ref{lem:torus}, $b$ is
continuous on the open set of $s'$ with $\varrho(A'-\operatorname{diag}s')<0$, which
contains $\theta'$ for every $\theta\in V_F$ by~\eqref{eq:A-radii}. So $C_\theta$ is
bounded on compact subsets of $V_F$. \qed

\subsection{Proof of \texorpdfstring{Theorem~\ref{thm:local}(a)}{the local law}}

The field cancels from the two sides of Theorem~\ref{thm:local}, so we take $v=0$.
For $k=1$ the formula is part of Lemma~\ref{lem:continuum}. Let $k\ge2$ and
$\alpha\in\mathcal K$, and write $\theta=\theta(\alpha)$, $r=r_\alpha$,
$l=l_\alpha$ and $\Theta'=\operatorname{diag}(\theta')$. By~\eqref{eq:A-radii}
we may use Lemma~\ref{lem:torus} with the radii $\theta'$ and $t=T\alpha$.
With Lemma~\ref{lem:continuum} this gives
\begin{equation}\label{eq:A-fT}
f_T(\alpha)=\frac{T^{k-1}e^{-T\langle q,\alpha\rangle}}{(2\pi)^{k-1}}
\int_{\mathbb T^{k-1}}e^{T\Psi(\alpha,\varphi)}\chi(\alpha,\varphi)\,d\varphi,
\end{equation}
where $z_i=\theta_ie^{\mathrm i\varphi_i}$,
$\Psi(\alpha,\varphi)=\alpha_1g(z')+\sum_{i\ge2}\alpha_iz_i$ and
$\chi(\alpha,\varphi)=b(z')\prod_{i\ge2}z_i$. We check the hypotheses of
Lemma~\ref{lem:laplace} with $m=k-1$.

\emph{Hypothesis \textup{(iii)}.} By~\eqref{eq:A-radii},
$\Psi(\alpha,0)=\langle\alpha,\theta\rangle$ is real. By
Lemma~\ref{lem:torus}, $|g(z')|\le g(\theta')=\theta_1$ on the
torus, so
\[
\operatorname{Re}\Psi(\alpha,\varphi)\le\Psi(\alpha,0)
-\sum_{i\ge2}\alpha_i\theta_i(1-\cos\varphi_i).
\]
All $\alpha_i$ and $\theta_i$ are positive, so the sum is positive unless
$\varphi=0$.

\emph{Hypotheses \textup{(i)} and \textup{(ii)}.} We need the first two
derivatives of $g$ at $\theta'$. We get them by differentiating a relation
$\lambda(g(z'),z')=0$ for a Perron root $\lambda$, whose derivatives are
known from Lemma~\ref{lem:perron}. For $z\in\R^F$ near $\theta$ let
$\lambda(z)=\varrho(A_F-\operatorname{diag}z)$. With $R=\operatorname{diag}r$
we have $A_F-\operatorname{diag}(\theta-\eta)=R(\widehat
Q_\alpha+\operatorname{diag}\eta)R^{-1}$, and $\widehat Q_\alpha$ is an
irreducible generator with stationary law $\alpha$. So
$\lambda(\theta-\eta)=\varrho(\widehat Q_\alpha+\operatorname{diag}\eta)$, and
Lemma~\ref{lem:perron} gives $\nabla\lambda(\theta)=-\alpha$ and
$\nabla^2\lambda(\theta)=\Gamma_\alpha$. For real $z'$ near $\theta'$ we have
$\varrho(A'-\operatorname{diag}z')<0$, so $(1,X(z')b_1)\ge0$ by the first
fact. This vector lies in the kernel of $A_F-\operatorname{diag}(g(z'),z')$, in
the rows $i\ge2$ by the definition of $X$ and in row $1$ by the definition of
$g$. By the second fact a nonnegative eigenvector belongs to the Perron
root, so $\lambda(g(z'),z')=0$.

Write $g_j$, $g_{ij}$ for the derivatives of $g$ at $\theta'$ and
$\lambda_j$, $\lambda_{ij}$ for those of $\lambda$ at $\theta$.
Differentiating $\lambda(g(z'),z')=0$ in $z_j$ gives
$\lambda_1g_j+\lambda_j=0$, and differentiating again in $z_i$ gives
$\lambda_1g_{ij}+\lambda_{11}g_ig_j+\lambda_{1j}g_i+\lambda_{i1}g_j
+\lambda_{ij}=0$. Put $w_j=e_j-(\alpha_j/\alpha_1)e_1\in\R^F$ for $j\ge2$.
Then the first equation gives $g_j=-\alpha_j/\alpha_1$, so $w_j=e_j+g_je_1$,
and the last four terms of the second equation are
$w_i^\top\Gamma_\alpha w_j$. So
\[
g_j=-\frac{\alpha_j}{\alpha_1},\qquad \alpha_1g_{ij}=w_i^\top\Gamma_\alpha w_j .
\]
Hence $\partial\Psi/\partial z_j=\alpha_1g_j+\alpha_j=0$ at $\theta'$, and the
Hessian $\mathcal H$ of $\Psi$ in $z'$ at $\theta'$ is $V^\top\Gamma_\alpha V$,
where $V$ has columns $w_2,\dots,w_k$. Since
$\partial/\partial\varphi_j=\mathrm iz_j\,\partial/\partial z_j$ and the first
derivatives vanish, $\nabla_\varphi\Psi(\alpha,0)=0$. We saw above that
$\Psi(\alpha,0)$ is real, so this is hypothesis (i). Also
$-\nabla^2_\varphi\Psi(\alpha,0)=\Theta'\mathcal H\Theta'$, which is real and
symmetric, as (ii) requires.

For the rest of (ii), namely that this matrix is positive definite, let $B=I+\ones\alpha'^\top/\alpha_1$. For $y\in\R^{k-1}$ the vector
$Vy$ has the entry $-\alpha'^\top y/\alpha_1$ at state $1$ and the entries
$y$ elsewhere, so $Vy+(\alpha'^\top y/\alpha_1)\ones=(0,By)$. Since
$\Gamma_\alpha\ones=0$, we may replace $Vy$ by $(0,By)$ in
$y^\top V^\top\Gamma_\alpha Vy$, which gives
$\mathcal H=B^\top\Sigma_\alpha B$. The matrix determinant lemma gives
$\det B=1+\alpha'^\top\ones/\alpha_1=1/\alpha_1$. By Lemma~\ref{lem:perron},
$\Sigma_\alpha$ is positive definite, and hence so are $\mathcal H$ and
$\Theta'\mathcal H\Theta'$. Also
$\det\mathcal H=\det\Sigma_\alpha/\alpha_1^2$.

\emph{Regularity.} By Proposition~\ref{prop:constant}(a), $\theta(\alpha)$ is
real-analytic. Fix $\alpha_0\in\mathcal K$. By~\eqref{eq:A-radii} and
continuity there is $\delta\in(0,\min_i\theta_i(\alpha_0))$ with
$\varrho(A'-\operatorname{diag}(\theta'(\alpha_0)-\delta\ones))<0$. By
Lemma~\ref{lem:torus}, $X$, $g$ and $b$ are holomorphic where
$|z_i|>\theta_i(\alpha_0)-\delta$ for every $i\ge2$, and this holds
on the torus of radii $\theta'(\alpha)$ for $\alpha$ near $\alpha_0$. So $\Psi$
and $\chi$ are real-analytic near $\{\alpha_0\}\times\mathbb T^{k-1}$.
So their $\varphi$-derivatives are jointly continuous on
$\mathcal K\times\mathbb T^{k-1}$, and covering $\mathcal K$ by finitely many
such sets bounds them there.

\emph{The constant.} The matrix $\mathcal J=\Theta'\mathcal H\Theta'$ of
Lemma~\ref{lem:laplace} has
$\det\mathcal J=\prod_{i\ge2}\theta_i^2\,\det\Sigma_\alpha/\alpha_1^2$. So
the lemma gives, uniformly for $\alpha\in\mathcal K$ and $T\ge1$,
\[
\int_{\mathbb T^{k-1}}e^{T\Psi}\chi\,d\varphi
=e^{T\langle\alpha,\theta\rangle}\Bigl(\frac{2\pi}T\Bigr)^{\frac{k-1}2}
\frac{\alpha_1}{\prod_{i\ge2}\theta_i\sqrt{\det\Sigma_\alpha}}
\Bigl(\chi(\alpha,0)+O\Bigl(\frac1T\Bigr)\Bigr).
\]
Here $\chi(\alpha,0)=b(\theta')\prod_{i\ge2}\theta_i$. As in the derivation
of~\eqref{eq:A-radii}, $r'=r_1X(\theta')b_1$. In the same way the columns
$j\ge2$ of $l^\top(A_F-\operatorname{diag}\theta)=0$, which is
\eqref{eq:stationary}, give $l'=l_1X(\theta')^\top a_1$. Hence
$\mu\cdot r=r_1(\mu_1+\mu'^\top X(\theta')b_1)$ and
$l\cdot\ones=l_1(1+a_1^\top X(\theta')\ones)$. Since $l_1r_1=\alpha_1$,
\[
b(\theta')=\frac{(\mu\cdot r_\alpha)(l_\alpha\cdot\ones)}{\alpha_1}.
\]
As $r_\alpha,l_\alpha>0$ and $\mu(F)>0$ by (H), $\chi(\alpha,0)$ is positive
and continuous, hence bounded below on $\mathcal K$. So the last bracket is
$\chi(\alpha,0)(1+O(1/T))$. Putting this into~\eqref{eq:A-fT} and using
$J_F(\alpha)=\langle q-\theta,\alpha\rangle$ gives
$f_T(\alpha)=T^{(k-1)/2}C_F(\alpha)e^{-TJ_F(\alpha)}(1+O(1/T))$, uniformly on
$\mathcal K$. \qed

The proof uses $A_F$, $\mu$ and the end vector $\ones$ only through their signs and the graph of
$A_F$.

\begin{lemma}[general weights]\label{lem:local-general}
Let $A\ge0$ be a matrix on $F$ with zero diagonal and a strongly connected graph, let
$\nu,e\in[0,\infty)^F$ be nonzero, and let $\theta(\alpha)$, $r_\alpha$, $l_\alpha$ and
$\Sigma_\alpha$ be the objects of Definition~\ref{def:tilt} for $A$ in place of $A_F$. Let $g$
and $b$ be the functions of Lemma~\ref{lem:torus-complex} for $(A,\nu,e)$. Then the identities
\eqref{eq:A-radii}, $g_j=-\alpha_j/\alpha_1$, $\mathcal H=B^\top\Sigma_\alpha B$,
$\det\mathcal H=\det\Sigma_\alpha/\alpha_1^2$ and
$\alpha_1b(\theta')=(\nu\cdot r_\alpha)(l_\alpha\cdot e)$ of the proof above hold, and uniformly
for $\alpha$ in a compact $\mathcal K\subset\Delta_F^\circ$ and $T\ge1$
\[
H_{A,\nu,e}(T\alpha)=(2\pi T)^{-(k-1)/2}e^{T\langle\alpha,\theta(\alpha)\rangle}\,
\frac{(\nu\cdot r_\alpha)(l_\alpha\cdot e)}{\sqrt{\det\Sigma_\alpha}}\,\bigl(1+O(1/T)\bigr).
\]
\end{lemma}

\begin{proof}
Definition~\ref{def:tilt}, Proposition~\ref{prop:constant}(a) and Lemma~\ref{lem:perron} involve
$A_F$ only as a nonnegative matrix with zero diagonal and a strongly connected graph, so they
hold for $A$. We repeat the proof above with Lemma~\ref{lem:torus-complex}, for
$(\bar A,\bar\nu,\bar e)=(A,\nu,e)$, in place of Lemma~\ref{lem:torus}. The paragraphs on the
hypotheses and on the regularity do not change. In the paragraph on the constant the identities
$r'=r_1X(\theta')b_1$ and $l'=l_1X(\theta')^\top a_1$ give $\nu\cdot r_\alpha=r_1(\nu_1+\nu'^\top
X(\theta')b_1)$ and $l_\alpha\cdot e=l_1(e_1+a_1^\top X(\theta')e')$, and $l_1r_1=\alpha_1$ gives
the formula for $b(\theta')$. It is positive since $r_\alpha,l_\alpha>0$ and $\nu,e\ge0$ are
nonzero. Then~\eqref{eq:A-fT} without the factor $T^{k-1}e^{-T\langle q,\alpha\rangle}$ and the
Laplace lemma give the display.
\end{proof}

\subsection{Proof of \texorpdfstring{Lemma~\ref{lem:discrete}}{the discrete
law}}

The discrete law of Lemma~\ref{lem:runs} and the continuum density of
Lemma~\ref{lem:continuum} are sums over the same vectors $m$ of run counts, and
we compare them term by term.

\emph{Step 1. The ratio of the terms.} Let $t=hn$ and
$\alpha=n/N\in\mathcal K$. Lemma~\ref{lem:continuum} gives the identity
$N^{-(k-1)}f_T(n/N)=h^{k-1}e^{-\langle q,t\rangle}H(t)=\sum_mv_m$ with
\[
v_m=W(m)\,h^{k-1}\prod_{i\in F}e^{-q_it_i}\frac{t_i^{m_i-1}}{(m_i-1)!},
\qquad m\in\Z_{\ge1}^F .
\]
Since $h^{m_i-1}\binom{n_i-1}{m_i-1}=\frac{t_i^{m_i-1}}{(m_i-1)!}
\prod_{1\le j<m_i}(1-j/n_i)$, the $m$-th term of Lemma~\ref{lem:runs} is
$\kappa_mv_m$ with
\[
\kappa_m=\prod_{i\in F}\Bigl[\prod_{1\le j<m_i}\Bigl(1-\frac j{n_i}\Bigr)\Bigr]
(1-hq_i)^{n_i-m_i}e^{hq_in_i}.
\]
Let $w_m=v_m/\sum_{m'}v_{m'}$, a probability on the vectors $m$ since
$H(t)>0$ under (H). We write $\E_w$ and $\Prob_w$ for it and regard
$M=\sum_im_i$ as a random variable. By Lemma~\ref{lem:runs} the ratio in the
lemma is $\sum_m\kappa_mv_m/\sum_mv_m=\E_w\kappa_m$.

\emph{Step 2. Bounds for $\kappa_m$.} Let $n_{\min}=\min_in_i$, and let $N$ be so large that $hq_{\max}\le\frac12$, which is possible since $h\le N^{-1/2}$. We claim that
for every $m$
\begin{equation}\label{eq:A-kappa}
\Bigl(1-\frac{M^2}{2n_{\min}}\Bigr)e^{-q_{\max}^2T^2/N}\le\kappa_m\le
e^{hq_{\max}M}.
\end{equation}
First let $m_i\le n_i$ for every $i$. The factors $1-j/n_i$ are at most $1$,
and $(1-hq_i)^{n_i-m_i}\le e^{-hq_i(n_i-m_i)}$, so
$\kappa_m\le e^{hq_{\max}M}$. For the lower bound,
$\log(1-x)\ge-x-x^2$ for $0\le x\le\frac12$ gives
$(1-hq_i)^{n_i-m_i}e^{hq_in_i}\ge e^{-n_ih^2q_i^2}$, and the product of these
over $i$ is at least $e^{-q_{\max}^2T^2/N}$. The factors $1-j/n_i$ lie in
$[0,1]$, so their product is at least $1$ minus the sum of the $j/n_i$, and
hence at least $1-\sum_im_i^2/(2n_i)\ge1-M^2/(2n_{\min})$. Now let
$m_i>n_i$ for some $i$. Then the $m$-th term of Lemma~\ref{lem:runs} is $0$, so $\kappa_m=0$, while $M\ge n_{\min}+1$ gives
$M^2/(2n_{\min})\ge2$, so the left side of~\eqref{eq:A-kappa} is negative.

By~\eqref{eq:A-kappa} it remains to bound $\E_we^{hq_{\max}M}$ and
$\E_wM^2$ from above.

\emph{Step 3. The number of runs under $w$.} Since $W$ for $\sigma A_F$ is
$\sigma^{M-1}W(m)$, we have $\E_w\sigma^{M-1}=H_\sigma(t)/H(t)$. So the
function $\Upsilon(c)=\log\E_we^{c(M-1)}=\log H_{e^c}(t)-\log H(t)$ is finite. It
is convex with $\Upsilon(0)=0$, as the logarithm of a moment generating function,
and $\Upsilon(1)\ge0$ since $M\ge1$. So $\Upsilon(c)\le c\,\Upsilon(1)$ for $0\le c\le1$.
We show that $\Upsilon(1)\le C_{\mathcal K}T$, with an upper bound for $H_e$ and
a lower bound for $H$ at $t=T\alpha$. By~\eqref{eq:A-Hsigma} with $\sigma=e$
and $\theta=\theta(\alpha)$,
$H_e(T\alpha)\le e^{k-1}b(\theta')\prod_{i\ge2}\theta_i\,
e^{eT\langle\alpha,\theta\rangle}$. For the lower bound,
$H(T\alpha)e^{-T\langle\alpha,\theta\rangle}=T^{-(k-1)}e^{TJ_F(\alpha)}f_T(\alpha)$,
by the definition of $f_T$ and $J_F(\alpha)=\langle q-\theta,\alpha\rangle$.
By Theorem~\ref{thm:local}(a) this is at least
$\frac12T^{-(k-1)/2}C_F(\alpha)$ for $T\ge T_1$, and for $1\le T\le T_1$ it is
continuous and positive on $[1,T_1]\times\mathcal K$. So in both cases
$H(T\alpha)\ge c_{\mathcal K}T^{-(k-1)/2}e^{T\langle\alpha,\theta\rangle}$.
Since $\theta(\alpha)$ and $b(\theta'(\alpha))$ are bounded on $\mathcal K$,
dividing the two bounds gives $\Upsilon(1)\le C_{\mathcal K}T$.

\emph{Step 4. The two averages.} Let $x=hq_{\max}\le\frac12$. Then
$\E_we^{xM}=e^{x+\Upsilon(x)}\le e^{x(1+C_{\mathcal K}T)}=1+O(T^2/N)$, because
$x\le q_{\max}T/N$ and $T^2/N\le1$. For the second moment we use the Chernoff
bound $\Prob_w(M-1\ge y)\le e^{\Upsilon(1)-y}$ for $y\ge\Upsilon(1)$ and the bound $1$
for $y<\Upsilon(1)$. This gives
\begin{align*}
\E_w(M-1)^2&=\int_0^\infty2y\,\Prob_w(M-1\ge y)\,dy\\
&\le\Upsilon(1)^2+\int_{\Upsilon(1)}^\infty2ye^{\Upsilon(1)-y}\,dy
=\Upsilon(1)^2+2\Upsilon(1)+2,
\end{align*}
so $\E_wM^2=O(T^2)$. Finally $n_{\min}\ge\varepsilon_{\mathcal K}N$ with
$\varepsilon_{\mathcal K}=\min_{\alpha\in\mathcal K}\min_i\alpha_i>0$. So
by~\eqref{eq:A-kappa}, $\E_w\kappa_m$ lies between
$e^{-q_{\max}^2T^2/N}(1-\E_wM^2/(2\varepsilon_{\mathcal K}N))$ and
$\E_we^{hq_{\max}M}$, and both are $1+O(T^2/N)$, uniformly for
$\alpha\in\mathcal K$ and $1\le T\le N^{1/2}$. \qed

\subsection{The derivative in \texorpdfstring{$T$}{T}}

Theorem~\ref{thm:crossing}(iv) needs the derivative of the local law in $T$. For fixed $N$ and
$n$, $P_N(n)$ is a polynomial in $T$, and we differentiate it at fixed $N$ and $n$.

\begin{lemma}[the derivative in $T$]\label{lem:derivative}
Let $F$ satisfy \textup{(H)} with $k\ge2$, and let $\mathcal K\subset\Delta_F^\circ$ be compact.
There are $T_1$, $N_0$ and $C$ such that for $N\ge N_0$, $\supp n=F$, $n/N\in\mathcal K$ and
$T_1\le T\le N^{1/3}$,
\[
\Bigl|\partial_T\log P_N(n)+J_F(n/N)\Bigr|\le\frac CT .
\]
\end{lemma}

\begin{proof}
Put $\alpha=n/N$ and $t=hn$. Let $\tilde w$ be the probability on the vectors $m$ proportional to
the terms of Lemma~\ref{lem:runs}. Since $T\partial_T=h\partial_h$ and the $m$-th term has
the logarithmic derivative $(M-1)-\sum_i(n_i-m_i)hq_i/(1-hq_i)$ in $\log h$,
\[
T\partial_T\log P_N(n)=\E_{\tilde w}(M-1)-\sum_{i\in F}\frac{hq_i}{1-hq_i}\,\E_{\tilde w}(n_i-m_i).
\]
For $c\in\R$ let $P^{(c)}$ be the sum of Lemma~\ref{lem:runs} with $e^cA_F$ in place of $A_F$,
that is with $e^{c(M-1)}W(m)$ in place of $W(m)$ and the same factors $(1-hq_i)^{n_i-m_i}$. Then
$\Upsilon(c)=\log P^{(c)}-\log P^{(0)}=\log\E_{\tilde w}e^{c(M-1)}$ is convex, and
$\Upsilon'(0)=\E_{\tilde w}(M-1)$.

\emph{Step 1. The asymptotics of $\Upsilon$.} The proof of Lemma~\ref{lem:discrete} uses $A_F$
only through the nonnegative weights $W(m)$ and through bounds on its continuum sum $H$. For
$e^cA_F$ this sum is $H_{e^c}(t)=e^{c(k-1)}H(e^ct)$, since a word with $M$ letters has $M-1$
switches and $M-k$ is the total degree of $\prod_it_i^{m_i-1}$. Let $|c|\le1$ and $T\ge e$. Step~3
of that proof needs the ratio $H_{e^{c+1}}(t)/H_{e^c}(t)=e^{k-1}H(e^{c+1}t)/H(e^ct)$ to be at most
$e^{C_{\mathcal K}T}$. This follows from~\eqref{eq:A-Hsigma} and from the lower bound for $H$ at
time $e^cT\ge1$ in that step. So Lemma~\ref{lem:discrete} holds for the weights $e^cA_F$,
uniformly for $|c|\le1$, and gives
$P^{(c)}=h^{k-1}e^{-\langle q,t\rangle}e^{c(k-1)}H(e^ct)(1+O(T^2/N))$. By
Lemma~\ref{lem:continuum} and $J_F(\alpha)=\langle q-\theta(\alpha),\alpha\rangle$,
Theorem~\ref{thm:local}(a) at time $e^cT$ reads
$H(e^cT\alpha)=(e^cT)^{-(k-1)/2}C_F(\alpha)e^{e^cT\langle\theta(\alpha),\alpha\rangle}(1+O(1/T))$.
Hence, uniformly for $|c|\le1$, $\alpha\in\mathcal K$ and $e\le T\le N^{1/3}$,
\[
\Upsilon(c)=a(c)+\varepsilon(c),\qquad a(c)=\frac{k-1}2\,c+(e^c-1)\,T\langle\theta(\alpha),\alpha\rangle,
\qquad|\varepsilon(c)|\le\frac{C_1}T,
\]
where we used $T^2/N\le1/T$.

\emph{Step 2. Convexity.} We have $a'(0)=\frac{k-1}2+T\langle\theta(\alpha),\alpha\rangle$ and
$0\le a''\le C_2T$ on $[-1,1]$, since $\theta$ is bounded on $\mathcal K$. Let $\eta=1/T$. By
convexity, $(\Upsilon(0)-\Upsilon(-\eta))/\eta\le\Upsilon'(0)\le(\Upsilon(\eta)-\Upsilon(0))/\eta$.
By Taylor's theorem both quotients differ from $a'(0)$ by at most $C_2T\eta/2+2C_1/(T\eta)$,
which is $C_2/2+2C_1$. So $\E_{\tilde w}(M-1)=T\langle\theta(\alpha),\alpha\rangle+O(1)$, and in
particular $\E_{\tilde w}M=O(T)$.

\emph{Step 3. The other term.} Since $hn_i=T\alpha_i$ and $hq_{\max}\le\frac12$,
$\sum_ihq_in_i/(1-hq_i)=T\langle q,\alpha\rangle+O(Th)$ and
$0\le\sum_ihq_i\E_{\tilde w}m_i/(1-hq_i)\le2hq_{\max}\E_{\tilde w}M=O(T^2/N)$. Altogether
$T\partial_T\log P_N(n)=T\langle\theta(\alpha),\alpha\rangle-T\langle q,\alpha\rangle+O(1)
=-TJ_F(\alpha)+O(1)$, and we divide by $T$.
\end{proof}

\subsection{Proof of \texorpdfstring{Theorem~\ref{thm:first}(ii)}{the
first-order bounds}}\label{app:first}

Let $F$ be admissible, with components $C_1,\dots,C_r$. If $F$ is irreducible,
then $r=1$, and the argument below uses only Lemma~\ref{lem:upper} and an explicit
family of paths, not the local law.

For the upper bound we apply Lemma~\ref{lem:upper} to every block of a path.
For the lower bound a component of least exit rate gets almost all the time,
and every other component $C$ gets only about $N/T$ steps, the time scale of a
single switch. Such a short block has probability bounded below and only about
$(N/T)^{|C|-1}$ compositions instead of $N^{|C|-1}$.

If $Q=0$, only singletons with positive start mass are admissible, and the bounds
follow from $M_{\{i\}}=\mu_i$. Otherwise let $a$ be the least positive rate $q_{ij}$.
Since $T\le N^{1/3}$, we have
$h\le N^{-2/3}$ and $T^2/N\le1$, and we take $N_0$ so large that $h$ is as
small as we need. As in the
proof of Lemma~\ref{lem:admissible}, a path of positive probability with
support $F$ visits each component in one block of consecutive times, and it
passes from one block to the next by one switch. We call the order of the
blocks, the first state and these switches the pattern of the path. There are
finitely many patterns.

\emph{Upper bound.} Fix a pattern whose $j$-th block lies in the component
$D_j$ and starts at $y_j$, and whose switch out of block $j$ is
$x_j\to y_{j+1}$ for $j<r$. So block $j$ ends at $x_j$ for $j<r$, and the
pattern does not fix the end of the last block. Let $\supp n=F$. Block $j$ has
length $N_j=\sum_{i\in D_j}n_i$. By the Markov property the paths with this
pattern and composition $n$ have total probability
$\mu_{y_1}\prod_{j<r}hq_{x_jy_{j+1}}\prod_{j\le r}\varpi_j$. Here $\varpi_j$ is the
probability that the chain started at $y_j$ stays in $D_j$ for $N_j$ visits,
with the composition that $n$ prescribes on $D_j$, and for $j<r$ also ends
at $x_j$. We drop the condition on the end. Then $\varpi_j$ is at most the probability that
$X_1,\dots,X_{N_j}$ have that composition for the chain started at $y_j$. If
$|D_j|\ge2$, Lemma~\ref{lem:upper} with start $y_j$ and
$\theta=q-\lambda_{D_j}\ones$ gives
$\varpi_j\le\kappa_1h^{|D_j|-1}e^{-h\lambda_{D_j}N_j}$, where
$\kappa_1$ is the largest of these finitely many constants. If $D_j=\{i\}$,
then $\lambda_{D_j}=q_i$ and
$\varpi_j=(1-hq_i)^{N_j-1}\le e^{1/2-h\lambda_{D_j}N_j}$. We multiply over
$j$. The powers of $h$ add up to $\sum_j(|D_j|-1)+r-1=k_F-1$. In the
exponent, $\lambda_{D_j}\ge\lambda_F$ and $\sum_jN_j=N$ give
$h\sum_j\lambda_{D_j}N_j\ge T\lambda_F$. Summing over the patterns
gives $P_N(n)\le\kappa_2(T/N)^{k_F-1}e^{-T\lambda_F}$.

\emph{Lower bound.} We write down a family of paths with support $F$, bound
its probability from below and count its compositions. By
Lemma~\ref{lem:admissible}(a) we can number the components so that there are
$y_1\in C_1$ with $\mu_{y_1}>0$ and $x_j\in C_j$, $y_{j+1}\in C_{j+1}$ with
$q_{x_jy_{j+1}}>0$ for $j<r$. Take any $x_r\in C_r$, choose $j^*$ with
$\lambda_{C_{j^*}}=\lambda_F$ and $|C_{j^*}|=k^*$, and put
$m=\lceil N/T\rceil$, so that $1\le hm\le2$.

\emph{Claim A.} There is $c_A>0$ such that for every component $C$ and all
$y,x\in C$, the chain started at $y$ stays in $C$ for its first $m$ visits,
visits every state of $C$ and has $X_m=x$ with probability at least $c_A$.
Take a walk $y=v_0,\dots,v_J=x$ through every state of $C$ with
$q_{v_sv_{s+1}}>0$ and $J\le|C|^2$. Consider the paths $X_1,\dots,X_m$ that
make exactly these $J$ switches, in this order, at any $J$ of the $m-1$
transition times. There are $\binom{m-1}J$ of them, and each has probability
at least $(ha)^J(1-hq_{\max})^m$. For large $N$ we have $m>|C|^2\ge J$. Then
$\binom{m-1}Jh^J\ge(hm-Jh)^J/J!\ge(1-Jh)^J/J!$ and
$(1-hq_{\max})^m\ge e^{-2hq_{\max}m}\ge e^{-4q_{\max}}$, since
$1\le hm\le2$. So their total probability is at least
$(1-Jh)^Ja^Je^{-4q_{\max}}/J!$, and $(1-Jh)^J\ge2^{-J}$ for small $h$.

\emph{Claim B.} There is $c_B>0$ such that for all $y,x\in C_{j^*}$ and
$\ell\ge2m$, the event of Claim~A, with $C_{j^*}$ in place of $C$ and $\ell$
visits in place of $m$, has probability at least $c_Be^{-h\lambda_F\ell-h^2\lambda_F^2\ell}$. Let $w>0$
satisfy $Q_{C_{j^*}}w=-\lambda_Fw$ and let $B=I+hQ_{C_{j^*}}\ge0$. Then
$(B^s)_{zz'}$ is the probability of going from $z$ to $z'$ in $s$ steps
inside $C_{j^*}$. Since $B^sw=(1-h\lambda_F)^sw$, the probability of staying
in $C_{j^*}$ for $s$ steps from $z$ is
$\sum_{z'}(B^s)_{zz'}\ge(B^sw)_z/\max w\ge(1-h\lambda_F)^s\min w/\max w$. We
use Claim A for the first $m$ visits, from $y$ back to $y$, then
$s=\ell-2m+1$ steps inside $C_{j^*}$ to any state $z'$, then Claim A for the
last $m$ visits, from $z'$ to $x$. By the Markov property this gives at least
$c_A^2(\min w/\max w)(1-h\lambda_F)^\ell$. Also $\lambda_F\le q_{\max}$,
because the Perron root of $Q_{C_{j^*}}$ is at least each diagonal entry. So
$h\lambda_F\le\frac12$ and $\log(1-h\lambda_F)\ge-h\lambda_F-h^2\lambda_F^2$.

\emph{The family of paths.} Give block $j\ne j^*$ the length $m$ and block
$j^*$ the length $\ell=N-(r-1)m$, which is at least $2m$ for $T\ge T_0=r+2$
and $N\ge N_0$. Consider the paths that go from $y_j$ to $x_j$ in block $j$
as in the claims and pass from block $j$ to block $j+1$ by the switch
$x_j\to y_{j+1}$. They all have support $F$. By the Markov property, and
since $h\ell\le T$ and $h^2\ell\le T^2/N\le1$, their total probability is at
least $\mu_{y_1}(ha)^{r-1}c_A^{r-1}c_Be^{-T\lambda_F-q_{\max}^2}$. The
composition of such a path is fixed by its restrictions to the blocks. On
the block in $C_j$ the restriction is a vector of $|C_j|$ positive integers
whose sum is $m$ for $j\ne j^*$ and $\ell$ for $j=j^*$. A vector of
positive integers with a given sum is fixed by all but one of its entries,
and each entry is at most the sum. Also $\ell\le N$. So these paths
have at most $(N+1)^{k^*-1}\prod_{j\ne j^*}(m+1)^{|C_j|-1}$ compositions,
and $m+1\le3N/T$. Since $\sum_{j\ne j^*}(|C_j|-1)=k_F-k^*-(r-1)$ and $h=T/N$,
dividing the total probability by the number of compositions gives
$M_F\ge c\,N^{-(k_F-1)}T^{k_F-k^*}e^{-T\lambda_F}$. \qed

\subsection{The local law over the whole switching range}\label{app:full-range}

Here $F$ is irreducible with $k\ge2$ states, $h_*$ is fixed with $h_*q_{\max}<1$,
$p_i=1-hq_i$ and $p_{\min}=1-h_*q_{\max}>0$. Let $\Omega$ be the open set of
$s'\in(0,\infty)^{k-1}$ with $\varrho(A'-\operatorname{diag}s')<0$, and for
$s'\in\Omega$ put $s_1=g(s')$ and
\[
\Psi_h(\alpha;s')=-\sum_{i\in F}\alpha_iL_h(q_i-s_i),
\]
with $L_h$ from Section~\ref{sec:discrete-rate}. For $h>0$ this is
$\frac1h\sum_i\alpha_i\log(p_i+hs_i)$.

\begin{lemma}[the discrete Hessian]\label{lem:discrete-hessian}
Let $(h,\alpha)\in[0,h_*]\times\Delta_F^\circ$, and let $\theta=\theta^{(h)}(\alpha)$,
$r$, $l$ and $c_h=c_h(\alpha)$ be as in the definition of $C^{(h)}_F$ in
Section~\ref{sec:locallaw}. Put $\pi=l\circ r$.
\begin{enumerate}[(a)]
\item The point $\theta'$ is the only minimizer of $\Psi_h(\alpha;\cdot)$ on $\Omega$,
$\Psi_h(\alpha;\theta')=-J^{(h)}_F(\alpha)$, and the Hessian $H_h$ of
$\Psi_h(\alpha;\cdot)$ at $\theta'$ is positive definite.
\item There is a $(k-1)\times(k-1)$ matrix $K$ with
$\det K=1/(c_h\pi_1\prod_{i\in F}(p_i+h\theta_i))$ and
$\nabla^2J^{(h)}_F(\alpha)=K^\top H_h^{-1}K$. So $\nabla^2J^{(h)}_F(\alpha)$ is positive
definite.
\item $C^{(h)}_F(\alpha)=b(\theta')/\bigl((2\pi)^{(k-1)/2}\prod_{i\in F}(p_i+h\theta_i)\sqrt{\det H_h}\bigr)$,
and $C^{(h)}_F$ is positive and real-analytic on $[0,h_*]\times\Delta_F^\circ$ up to
$h=0$.
\end{enumerate}
\end{lemma}

\begin{proof}
(a) The map $\psi\mapsto\theta(\psi)'$ is a real-analytic diffeomorphism from
$\{\psi\in\R^F:\psi_1=0\}$ onto $\Omega$. Indeed $\theta(\psi)=A_Fe^\psi/e^\psi$ lies
in $V_F$, so $\theta(\psi)'\in\Omega$ and $\theta_1(\psi)=g(\theta(\psi)')$
by~\eqref{eq:A-radii}. Conversely, for $s'\in\Omega$ the vector $u=(1,X(s')b_1)$ is
nonnegative by the first fact, and it is a null vector of
$A_F-\operatorname{diag}(g(s'),s')$, as in the proof of Theorem~\ref{thm:local}(a).
By the second fact it is positive, and $\psi=\log u$ has $\psi_1=0$ and
$\theta(\psi)=(g(s'),s')$. Both maps are real-analytic. Since
$\Psi_h(\alpha;\theta(\psi)')=-G_h(\alpha,\psi)$, part (a) follows from
Proposition~\ref{prop:discrete-rate}(b). The first derivatives vanish at the critical
point, so there the two Hessians are congruent through the Jacobian of the
diffeomorphism, which is invertible.

(b) Write $y_i(s')=L_h(q_i-s_i)$ with $s_1=g(s')$, so that, in the coordinates
$\alpha'$, $-\Psi_h=y_1+\sum_{j\ge2}\alpha_j(y_j-y_1)$. By (a),
$J^{(h)}_F(\alpha)=-\Psi_h(\alpha;\theta^{(h)}(\alpha)')$, and $\theta^{(h)}(\alpha)'$ is
real-analytic in $\alpha$. So the envelope theorem gives
$\partial J^{(h)}_F/\partial\alpha_j=y_j(\theta')-y_1(\theta')$ for $j\ge2$. Let
$K_{ij}=-\partial(y_j-y_1)/\partial s_i$ at $\theta'$, for $i,j\ge2$. Differentiating
$\nabla_{s'}\Psi_h(\alpha;\theta^{(h)}(\alpha)')=0$ in $\alpha'$ gives
$H_h\,\partial\theta'/\partial\alpha'=-K$, and hence
$\nabla^2J^{(h)}_F=-K^\top\partial\theta'/\partial\alpha'=K^\top H_h^{-1}K$.

To compute $K$ put $e_i=L_h'(q_i-\theta_i)=1/(p_i+h\theta_i)$, and let $g_i$ be the
derivatives of $g$ at $\theta'$. The vector $\psi=\log r$ attains the supremum
defining $J_F(\pi)$ by the criterion in the proof of Proposition~\ref{prop:rate},
since $\pi$ is the stationary law of $R^{-1}(A_F-\operatorname{diag}\theta)R$ with
$R=\operatorname{diag}r$. So $\theta=\theta(\pi)$, and the proof of
Theorem~\ref{thm:local}(a) at the point $\pi$ gives $g_i=-\pi_i/\pi_1$. Hence
$K_{ij}=\delta_{ij}e_j+e_1\pi_i/\pi_1$, that is
$K=\operatorname{diag}(e')+(e_1/\pi_1)\pi'\ones^\top$, and the matrix determinant lemma
gives $\det K=\prod_{i\in F}e_i\cdot\sum_{i\in F}\pi_i(p_i+h\theta_i)/\pi_1$. By
Proposition~\ref{prop:discrete-rate}(c) with $\psi=\log r$, $\alpha_i=c\,\pi_i(p_i+h\theta_i)$
for a constant $c$. Summing $\alpha_i/(p_i+h\theta_i)$ over $i$ gives $c=c_h$, and
summing $\alpha_i$ gives $\sum_i\pi_i(p_i+h\theta_i)=1/c_h$. This is the value of
$\det K$, which is not $0$, so $\nabla^2J^{(h)}_F(\alpha)$ is positive definite by (a).

(c) The identity $\alpha_1b(\theta')=(\mu\cdot r_\alpha)(l_\alpha\cdot\ones)$ of the
proof of Theorem~\ref{thm:local}(a), at the point $\pi$, gives
$b(\theta')=(\mu\cdot r)(l\cdot\ones)/\pi_1$, since $l\cdot r=1$. By (b),
$\sqrt{\det\nabla^2J^{(h)}_F}=|\det K|/\sqrt{\det H_h}$. Putting both into the
definition of $C^{(h)}_F$ gives the formula, which is positive. By
Proposition~\ref{prop:discrete-rate}(b), $\theta^{(h)}(\alpha)$ and $J^{(h)}_F$ are
real-analytic up to $h=0$, and so are $c_h$ and the normalized null vectors $r$ and $l$,
since $0$ is a simple eigenvalue of $A_F-\operatorname{diag}\theta$. As
$\det\nabla^2J^{(h)}_F>0$, $C^{(h)}_F$ is real-analytic up to $h=0$.
\end{proof}

\begin{proof}[Proof of Theorem~\ref{thm:local}(b)]
The field cancels from the two sides, so we take $v=0$. Let $\mathcal K$ be as in the
theorem and put $\mathcal P=[0,h_*]\times\mathcal K$. For $(h,\alpha)\in\mathcal P$ let
$\theta=\theta^{(h)}(\alpha)$ and $\Theta'=\operatorname{diag}(\theta')$.

\emph{The integral.} Let $1\le T\le h_*N$, $\supp n=F$ and $\alpha=n/N\in\mathcal K$, so
that $h=T/N\in(0,h_*]$ and $n_i=T\alpha_i/h$. By~\eqref{eq:A-radii} we may use the
discrete formula of Lemma~\ref{lem:torus} with the radii $\theta'$. It gives
$P_N(n)=h^{k-1}(2\pi)^{-(k-1)}\int_{\mathbb T^{k-1}}F_n\,d\varphi$, where $F_n$ is the
integrand. With $z_i=\theta_ie^{\mathrm i\varphi_i}$ and $z_1=g(z')$ put
\[
\Psi_h(\alpha,\varphi)=-\sum_{i\in F}\alpha_iL_h(q_i-z_i),\qquad
a_h(\alpha,\varphi)=\frac{b(z')\prod_{i\ge2}z_i}{\prod_{i\in F}(p_i+hz_i)},
\]
where $\operatorname{Re}(p_i+hz_i)>0$ for every $i$. There $F_n=e^{T\Psi_h}a_h$, since
$(p_i+hz_i)^{n_i}=\exp(n_i\log(p_i+hz_i))$ with the principal logarithm and
$n_i\log(p_i+hz_i)=-T\alpha_iL_h(q_i-z_i)$. The logarithm has no continuous branch on the whole
torus when $h\theta_i>p_i$, so we use these formulas only near $\varphi=0$.

\emph{Regularity on a cube.} The map $(h,\alpha)\mapsto\theta^{(h)}(\alpha)$ is
continuous on the compact set $\mathcal P$, so $\theta_i\in[\theta_{\min},\theta_{\max}]$
with $\theta_{\min}>0$, and $\varrho(A'-\operatorname{diag}\theta')$ is bounded above by
a negative constant. As in the regularity paragraph of the proof of (a), $X$, $g$ and
$b$ are holomorphic near the tori of radii $\theta'$, uniformly on $\mathcal P$, and
$g(z')$ is uniformly continuous in $(h,\alpha,\varphi)$. So there is
$\varepsilon_1\in(0,1)$ such that $|g(z')-\theta_1|\le\theta_{\min}$ when
$\lVert\varphi\rVert\le\varepsilon_1$. Then $\operatorname{Re}g(z')\ge0$ and
$\operatorname{Re}z_i>0$ for $i\ge2$, so $\operatorname{Re}(p_i+hz_i)\ge p_{\min}$ for
every $i$. Since $1-shy=(1-s)+s(1-hy)$, the function $L_h(y)=\int_0^1y\,ds/(1-shy)$ is
holomorphic in $y$ and real-analytic in $h$ where $\operatorname{Re}(1-hy)>0$, including
$h=0$. So $\Psi_h$ and $a_h$ are defined and real-analytic on
$\mathcal P\times\{\lVert\varphi\rVert\le\varepsilon_1\}$, and their
$\varphi$-derivatives up to orders $4$ and $2$ are bounded there by constants $B_\Psi$
and $B_a$. At $h=0$ they are $\Psi-\langle q,\alpha\rangle$ and $\chi$
of~\eqref{eq:A-fT}.

\emph{Hypotheses \textup{(i)} and \textup{(ii)}.} At $\varphi=0$,
$\Psi_h(\alpha,0)=\Psi_h(\alpha;\theta')=-J^{(h)}_F(\alpha)$ by
Lemma~\ref{lem:discrete-hessian}(a), which is real. Since
$\partial/\partial\varphi_j=\mathrm iz_j\,\partial/\partial z_j$ and $\Psi_h$ is
holomorphic in $z'$ with the derivatives of the real function $\Psi_h(\alpha;\cdot)$ at
$\theta'$, Lemma~\ref{lem:discrete-hessian}(a) gives $\nabla_\varphi\Psi_h(\alpha,0)=0$
and $-\nabla^2_\varphi\Psi_h(\alpha,0)=\Theta'H_h\Theta'$. This matrix is real,
symmetric and positive definite, and it is continuous on $\mathcal P$, so its least
eigenvalue is at least some $\sigma_{\min}>0$ on $\mathcal P$.

\emph{The gap.} Choose $\varepsilon\in(0,\varepsilon_1]$ with
$(k-1)^3B_\Psi\varepsilon\le\frac32\sigma_{\min}$. For $i\ge2$,
\[
|p_i+h\theta_ie^{\mathrm i\varphi_i}|^2=(p_i+h\theta_i)^2-2p_ih\theta_i(1-\cos\varphi_i).
\]
With $\log(1-x)\le-x$, and with $n_i-1\ge n_i/2=T\alpha_i/(2h)$, which holds for
$N\ge N_0$ since then $n_i\ge2$, this gives
$|p_i+hz_i|^{n_i-1}\le(p_i+h\theta_i)^{n_i-1}e^{-c_*T(1-\cos\varphi_i)}$ with
$c_*=\min_{\mathcal K}\min_i\alpha_i\,p_{\min}\theta_{\min}/(2(1+h_*\theta_{\max})^2)>0$.
The $h$ in $n_i$ cancels the $h$ in $p_ih\theta_i$, so $c_*$ does not depend on $h$.
Also $|p_1+hg(z')|\le p_1+h\theta_1$ and $|b(z')|\le b(\theta')$ by
Lemma~\ref{lem:torus}. If $\lVert\varphi\rVert\ge\varepsilon$, then
$1-\cos\varphi_i\ge1-\cos\varepsilon$ for some $i\ge2$, so
\[
|F_n(\varphi)|\le e^{-\eta T}F_n(0)=e^{-\eta T}e^{T\Psi_h(\alpha,0)}a_h(\alpha,0),
\qquad\eta=c_*(1-\cos\varepsilon),
\]
and $a_h(\alpha,0)$ is bounded on $\mathcal P$.

\emph{The expansion.} By Remark~\ref{rem:laplace-uniform} in its localized form, with
$m=k-1$, uniformly in $(h,\alpha)$ and $T\ge1$,
\[
\int_{\mathbb T^{k-1}}F_n\,d\varphi=e^{-TJ^{(h)}_F(\alpha)}\Bigl(\frac{2\pi}T\Bigr)^{\frac{k-1}2}
\det(\Theta'H_h\Theta')^{-1/2}\Bigl(a_h(\alpha,0)+O\Bigl(\frac1T\Bigr)\Bigr).
\]
Here $a_h(\alpha,0)=b(\theta')\prod_{i\ge2}\theta_i/\prod_i(p_i+h\theta_i)$ is positive
and continuous on $\mathcal P$, so the bracket is $a_h(\alpha,0)(1+O(1/T))$. Since
$\det(\Theta'H_h\Theta')=\prod_{i\ge2}\theta_i^2\det H_h$ and
$h^{k-1}T^{-(k-1)/2}=N^{-(k-1)}T^{(k-1)/2}$, Lemma~\ref{lem:discrete-hessian}(c) gives
the display of (b) with a factor $1+E$, $|E|\le C/T$.

\emph{Bounded $T$.} It remains to bound $1+E$ below for $1\le T\le T_1=2C$. Let
$N\ge T_1^2$. Then $T\le N^{1/2}$, and Lemma~\ref{lem:discrete} and part (a) give
$P_N(n)=N^{-(k-1)}f_T(\alpha)(1+O(T_1^2/N))$. The ratio of $f_T(\alpha)$ to
$T^{(k-1)/2}C_F(\alpha)e^{-TJ_F(\alpha)}$ is continuous and positive on
$[1,T_1]\times\mathcal K$, and the ratio of
$C^{(h)}_F(\alpha)e^{-TJ^{(h)}_F(\alpha)}$ to $C_F(\alpha)e^{-TJ_F(\alpha)}$ is
$1+O(T_1h)$ by Proposition~\ref{prop:discrete-rate}(b) and
Lemma~\ref{lem:discrete-hessian}(c). So for large $N$, $1+E$ lies between two positive
constants for $1\le T\le T_1$, and for $T\ge T_1$ it lies in $[\frac12,\frac32]$.
\end{proof}

\section{A Laplace lemma with a parameter}\label{app:laplace}

We prove the Laplace lemma that Appendix~\ref{app:torus} uses for
Theorem~\ref{thm:local}, and its form with a complex Hessian that
Section~\ref{sec:boundary-layer} uses for Theorem~\ref{thm:L-facet}. The compact
parameter space $\mathcal K$ can include the rates and the start when constants are
needed uniformly over a class of chains. On $\mathbb T^m=(\R/2\pi\Z)^m$ we use
$\lVert\varphi\rVert=\max_i\operatorname{dist}(\varphi_i,2\pi\Z)$. For
$x\in\R^m$ we write $\lVert x\rVert=\max_i|x_i|$ and $|x|$ for the Euclidean
norm. Repeated indices are summed from $1$ to $m$.

\begin{lemma}[Laplace integrals with a parameter]\label{lem:laplace}
Let $\mathcal K$ be a compact metric space and $m\ge1$. Let
$\Psi,a:\mathcal K\times\mathbb T^m\to\mathbb C$ be such that the
$\varphi$-derivatives of $\Psi$ up to order $4$ and of $a$ up to order $2$
exist and are jointly continuous on $\mathcal K\times\mathbb T^m$. Assume that
for every $\alpha\in\mathcal K$
\begin{enumerate}[(i)]
\item the value $\Psi(\alpha,0)$ is real and $\nabla_\varphi\Psi(\alpha,0)=0$,
\item the matrix $\mathcal J(\alpha)=-\nabla^2_\varphi\Psi(\alpha,0)$ is real, symmetric
and positive definite, and
\item the inequality $\operatorname{Re}\Psi(\alpha,\varphi)<\Psi(\alpha,0)$ holds for every
$\varphi\ne0$ in $\mathbb T^m$.
\end{enumerate}
Then, uniformly for $\alpha\in\mathcal K$ and $T\ge1$,
\[
\int_{\mathbb T^m}e^{T\Psi(\alpha,\varphi)}a(\alpha,\varphi)\,d\varphi
=e^{T\Psi(\alpha,0)}\Bigl(\frac{2\pi}T\Bigr)^{m/2}\det\mathcal J(\alpha)^{-1/2}
\Bigl(a(\alpha,0)+O\Bigl(\frac1T\Bigr)\Bigr).
\]
\end{lemma}

\begin{proof}
By compactness
$B_\Psi=\max_{|\beta|\le4}\sup|\partial^\beta_\varphi\Psi|$ and
$B_a=\max_{|\beta|\le2}\sup|\partial^\beta_\varphi a|$ are finite, the least
eigenvalue of $\mathcal J(\alpha)$ has a positive minimum $\sigma_{\min}$ over
$\mathcal K$, and $\det\mathcal J(\alpha)\le(mB_\Psi)^m$. Below, $C$ depends
only on $m$, $\sigma_{\min}$, $B_\Psi$, $B_a$ and the numbers $\varepsilon$ and
$\eta$ chosen below, and it may change from line to line. Hence every bound
below is uniform on $\mathcal K$.

\emph{Localization.} Choose $\varepsilon\in(0,1)$ with
$m^3B_\Psi\varepsilon\le\frac32\sigma_{\min}$. The set of $(\alpha,\varphi)$ with
$\lVert\varphi\rVert\ge\varepsilon$ is compact, and by (iii) the continuous
function $\operatorname{Re}\Psi(\alpha,\varphi)-\Psi(\alpha,0)$ is negative on
it. Hence it is at most $-\eta$ there for some $\eta>0$, and the integral over
this set has modulus at most
$(2\pi)^mB_ae^{T(\Psi(\alpha,0)-\eta)}\le CT^{-1-m/2}e^{T\Psi(\alpha,0)}$.

\emph{Scaling.} Reduction modulo $2\pi$ identifies the cube
$\lVert x\rVert\le\varepsilon$ in $\R^m$ with the set
$\lVert\varphi\rVert\le\varepsilon$, and it preserves measure. Let
$R(\alpha,\varphi)=\Psi(\alpha,\varphi)-\Psi(\alpha,0)
+\frac12\varphi^\top\mathcal J(\alpha)\varphi$. By (i) and (ii), $R$ and its
derivatives of order $1$ and $2$ vanish at $\varphi=0$. Taylor's theorem with
integral remainder, which holds for complex values, gives on the cube
$R=\frac16\Psi_{ijk}\varphi^i\varphi^j\varphi^k+\omega_0$ with
$|\omega_0|\le C\lVert\varphi\rVert^4$, where $\Psi_{ijk}$ are the third
derivatives at $(\alpha,0)$. It also gives
$|R|\le\frac{m^3}6B_\Psi\lVert\varphi\rVert^3\le\frac14\sigma_{\min}|\varphi|^2$,
since $\sum_{|\beta|=3}1/\beta!=m^3/6$ and
$\lVert\varphi\rVert^3\le\varepsilon|\varphi|^2$. We substitute
$\varphi=x/\sqrt T$ and put $w(x)=TR(\alpha,x/\sqrt T)$. The integral over the
cube is
\[
T^{-m/2}e^{T\Psi(\alpha,0)}\int_{\lVert x\rVert\le\varepsilon\sqrt T}
e^{-\frac12x^\top\mathcal Jx}\,e^{w(x)}a(\alpha,x/\sqrt T)\,dx,
\]
and $|w(x)|\le\frac14\sigma_{\min}|x|^2$.

\emph{Taylor expansion.} Let $P(x)=\frac16\Psi_{ijk}x^ix^jx^k$,
$a_0=a(\alpha,0)$ and $a_i=\partial_{\varphi_i}a(\alpha,0)$. Then
$w=T^{-1/2}P+\omega_1$ with $|\omega_1|\le CT^{-1}|x|^4$, and
$|w|\le CT^{-1/2}|x|^3$. Since $|e^w-1-w|\le\frac12|w|^2e^{|w|}$, this gives
$|e^w-1-T^{-1/2}P|\le CT^{-1}(|x|^4+|x|^6)e^{\sigma_{\min}|x|^2/4}$. Taylor's
theorem for $a$ to orders $1$ and $2$ gives
$|a(\alpha,x/\sqrt T)-a_0|\le CT^{-1/2}|x|$ and
$|a(\alpha,x/\sqrt T)-a_0-T^{-1/2}a_ix^i|\le CT^{-1}|x|^2$. Multiplying out,
\[
e^{w}a(\alpha,x/\sqrt T)=a_0+T^{-1/2}\bigl(a_0P(x)+a_ix^i\bigr)+\omega(x),
\qquad|\omega(x)|\le CT^{-1}(1+|x|)^6e^{\sigma_{\min}|x|^2/4}.
\]

\emph{Integration.} Since $x^\top\mathcal Jx\ge \sigma_{\min}|x|^2$, the term
$\omega$ contributes at most $CT^{-1}$ to the integral, because
$(1+|x|)^6e^{-\sigma_{\min}|x|^2/4}$ is integrable. Outside the cube
$|x|>\varepsilon\sqrt T$, so extending the integrals of the other two terms to
$\R^m$ changes them by at most $Ce^{-\sigma_{\min}\varepsilon^2T/4}\le CT^{-1}$. On
$\R^m$ we have
$\int e^{-\frac12x^\top\mathcal Jx}dx=(2\pi)^{m/2}\det\mathcal J^{-1/2}$, and
the odd polynomial $a_0P(x)+a_ix^i$ integrates to $0$ against this weight.
This is why the error
is $O(1/T)$ and not $O(T^{-1/2})$. So the integral over the cube is
$e^{T\Psi(\alpha,0)}T^{-m/2}\bigl[(2\pi)^{m/2}\det\mathcal J^{-1/2}a_0
+O(T^{-1})\bigr]$. Since $\det\mathcal J\le(mB_\Psi)^m$, this error and the
one from the localization are $O(1/T)$ relative to
$e^{T\Psi(\alpha,0)}(2\pi/T)^{m/2}\det\mathcal J^{-1/2}$.
\end{proof}

\begin{remark}[the uniform form]\label{rem:laplace-uniform}
The proof of Lemma~\ref{lem:laplace} uses compactness and joint continuity only to obtain the bounds $B_\Psi$ and $B_a$ on
the derivatives, the lower bound $\sigma_{\min}$ on the least eigenvalue of $\mathcal J$, and a
gap $\operatorname{Re}\Psi(\alpha,\varphi)-\Psi(\alpha,0)\le-\eta$ for
$\lVert\varphi\rVert\ge\varepsilon$, where $\varepsilon\in(0,1)$ satisfies
$m^3B_\Psi\varepsilon\le\frac32\sigma_{\min}$. So the following holds with the same proof. Let
$\mathcal P$ be any set, and let $\Psi,a:\mathcal P\times\mathbb T^m\to\mathbb C$ satisfy (i) and
(ii) for every $\alpha\in\mathcal P$, with $\varphi$-derivatives of $\Psi$ up to order $4$ and
of $a$ up to order $2$ bounded by $B_\Psi$ and $B_a$ on $\mathcal P\times\mathbb T^m$, with
$\mathcal J(\alpha)\ge\sigma_{\min}I$ for all $\alpha$, and with the gap above for all $\alpha$.
Then the expansion of Lemma~\ref{lem:laplace} holds uniformly for $\alpha\in\mathcal P$ and
$T\ge1$, with a constant that depends only on $m$, $B_\Psi$, $B_a$, $\sigma_{\min}$,
$\varepsilon$ and $\eta$.

The proof uses the integrand off the cube $\lVert\varphi\rVert\le\varepsilon$ only in the
localization, through a bound on its modulus, and it uses the bounds on the derivatives only on
the cube. So the expansion also holds in the following localized form, with a constant that
depends in addition on a number $B$. Let $\Psi$ and $a$ be defined only on
$\mathcal P\times\{\lVert\varphi\rVert\le\varepsilon\}$, satisfy (i) and (ii), have
$\varphi$-derivatives bounded by $B_\Psi$ and $B_a$ there, with $\mathcal J(\alpha)\ge\sigma_{\min}I$
and $m^3B_\Psi\varepsilon\le\frac32\sigma_{\min}$, and let the integrand be
$e^{T\Psi}a$ on the cube and any measurable function of modulus at most
$Be^{T(\Psi(\alpha,0)-\eta)}$ off it. Appendix~\ref{app:full-range} needs this form, since
there the logarithm in the exponent has no branch on the whole torus.
\end{remark}

Section~\ref{sec:boundary-layer} needs the same expansion when $\Psi(\alpha,0)$ and the Hessian
are complex. Let $\mathcal S_m$ be the set of complex symmetric $m\times m$ matrices $M$ whose
real part is positive definite. For $M\in\mathcal S_m$ and $x\in\mathbb C^m$ with $x\ne0$ we have
$\operatorname{Re}(\bar x^\top Mx)=\bar x^\top(\operatorname{Re}M)x>0$, because $\bar x^\top Px$ is
real for a real symmetric $P$. So $\det M\ne0$. The set $\mathcal S_m$ is convex, so $1/\det M$ has
exactly two continuous square roots on it. We write $\det M^{-1/2}$ for the one that is positive
on the real matrices of $\mathcal S_m$. Then
\[
\int_{\R^m}e^{-\frac12x^\top Mx}\,dx=(2\pi)^{m/2}\det M^{-1/2}\qquad(M\in\mathcal S_m).
\]
Indeed both sides are holomorphic functions of the entries $M_{jk}$, $j\le k$, on the open set
$\mathcal S_m$. For the left side we may differentiate under the integral, since
$|e^{-x^\top Mx/2}|=e^{-x^\top(\operatorname{Re}M)x/2}$ is bounded by a fixed Gaussian for $M$
near a point of $\mathcal S_m$. The right side is a continuous square root of a holomorphic
function without zeros. The two sides agree on the real positive definite matrices, which form a
nonempty open subset of the real points of $\mathcal S_m$. A holomorphic function on a connected
open set that vanishes on such a subset vanishes everywhere, since its Taylor coefficients at a
real point are real derivatives. So the two sides agree on $\mathcal S_m$.

\begin{lemma}[a complex Hessian]\label{lem:laplace-complex}
Let $\mathcal P$ be a set and $m\ge1$. Let $\Psi,a:\mathcal P\times\mathbb T^m\to\mathbb C$ have
$\varphi$-derivatives up to order $4$ and $2$ that are continuous in $\varphi$ and bounded in
modulus by $B_\Psi$ and $B_a$ on $\mathcal P\times\mathbb T^m$. Let $\sigma_{\min}>0$, $\eta>0$ and
$\varepsilon\in(0,1)$ with $m^3B_\Psi\varepsilon\le\frac32\sigma_{\min}$. Assume that for every
$p\in\mathcal P$
\begin{enumerate}[(i)]
\item $\nabla_\varphi\Psi(p,0)=0$,
\item the complex symmetric matrix $\mathcal J(p)=-\nabla^2_\varphi\Psi(p,0)$ satisfies
$\operatorname{Re}\mathcal J(p)\ge\sigma_{\min}I$, and
\item $\operatorname{Re}\Psi(p,\varphi)-\operatorname{Re}\Psi(p,0)\le-\eta$ for
$\lVert\varphi\rVert\ge\varepsilon$.
\end{enumerate}
Then the expansion of Lemma~\ref{lem:laplace} holds uniformly for $p\in\mathcal P$ and $T\ge1$,
with the branch of $\det\mathcal J^{-1/2}$ defined above and a constant that depends only on $m$,
$B_\Psi$, $B_a$, $\sigma_{\min}$, $\varepsilon$ and $\eta$.
\end{lemma}

\begin{proof}
The proof of Lemma~\ref{lem:laplace} in the form of Remark~\ref{rem:laplace-uniform} goes through
with three changes. First, the factor $e^{T\Psi(p,0)}$ is now complex, and the localization uses
only its modulus $e^{T\operatorname{Re}\Psi(p,0)}$. By (iii) the integral over
$\lVert\varphi\rVert\ge\varepsilon$ has modulus at most
$(2\pi)^mB_ae^{-\eta T}|e^{T\Psi(p,0)}|$. Second, for real $x$ the number $x^\top(\operatorname{Im}\mathcal J)x$
is real, so $|e^{-x^\top\mathcal Jx/2}|=e^{-x^\top(\operatorname{Re}\mathcal J)x/2}\le e^{-\sigma_{\min}|x|^2/2}$.
This is the only bound on the Gaussian weight that the integration step uses, and the weight is
still even and integrable, so the odd polynomial still integrates to $0$. Third, the Gaussian
integral over $\R^m$ is $(2\pi)^{m/2}\det\mathcal J^{-1/2}$ by the identity above, and Hadamard's
inequality gives $|\det\mathcal J|\le(mB_\Psi)^m$, so the errors are $O(1/T)$ relative to the main
term. The scaling and the Taylor expansion use only the bounds on the derivatives, Taylor's
theorem for complex values and $|w(x)|\le\frac14\sigma_{\min}|x|^2$, and they do not change.
\end{proof}

\section{Proofs for the complete graph}\label{app:complete}

We prove Theorems~\ref{thm:simplex-crossing} and~\ref{thm:simplex-balance} and
Corollary~\ref{cor:simplex-modes}, with the notation of Section~\ref{sec:completegraph}. The
first lemma writes the law with blocks between refreshes in place of runs. Here
$I+hQ=\sigma I+hJ$, and a refresh is a step $s<N$ at which the term $h$ is chosen when each factor
$\sigma\mathbf 1\{x_s=x_{s+1}\}+h$ of the probability of a word is expanded.

\begin{lemma}[the refresh form]\label{lem:refresh}
Let $N\ge1$ and let $n$ be a composition with support $F$. Sums over $m$ run over the integer vectors
with $1\le m_i\le n_i$ for $i\in F$, and $M=\sum_im_i$.
\begin{enumerate}[(a)]
\item With $0^0=1$,
\[
P_N(n)=\sum_m\Bigl(\sum_{i\in F}\mu_im_i\Bigr)\frac{(M-1)!}{\prod_im_i!}\,
h^{M-1}\sigma^{N-M}\prod_{i\in F}\binom{n_i-1}{m_i-1}.
\]
\item Let $1/d<p<1$, so that $0<u<1$ and $v>0$. Then
\begin{equation}\label{eq:simplex-positive-sum}
S_n(u)=(1-u)^{N-1}\sum_mM!\,v^{M-1}\prod_{i\in F}\frac1{m_i!}\binom{n_i-1}{m_i-1},
\end{equation}
which is a sum of positive terms, and
\begin{equation}\label{eq:simplex-balance-integral}
S_n(u)=\frac{(1-u)^{N-1}}{v}\int_0^\infty e^{-t}\prod_{i\in F}B_{n_i}(vt)\,dt,
\end{equation}
where $B_n$ is defined in~\eqref{eq:refresh-polynomial}.
\end{enumerate}
\end{lemma}

\begin{proof}
(a) The transition matrix is $I+hQ=\sigma I+hJ$, so a word $x_1\cdots x_N$ has probability
$\mu_{x_1}\prod_{s<N}(\sigma\mathbf 1\{x_s=x_{s+1}\}+h)$. We expand each factor into its two
terms and call the steps where the term $h$ is chosen the refreshes. If there are $M-1$
refreshes, they cut the word into $M$ blocks, and a term of the expansion is not zero only if the
word is constant on each block. So $P_N(n)$ is the sum, over the sequences of $M$ block labels
and over the block lengths, of $\mu_{\ell_1}h^{M-1}\sigma^{N-M}$, where $\ell_1$ is the label of
the first block. If $m_i$ blocks carry the label $i$, there are $m_i(M-1)!/\prod_jm_j!$ sequences
of labels that start with $i$. Consecutive labels may be equal. Given the labels, the lengths of
the $m_i$ blocks of the label $i$ form an ordered sum of $n_i$ with $m_i$ positive parts, in
$\binom{n_i-1}{m_i-1}$ ways. Multiplying and summing gives (a).

(b) For the uniform start $\sum_i\mu_im_i=M/d$. By~\eqref{eq:SnVN} we divide (a) by
$V_N=p^{N-1}/d$, and $h^{M-1}\sigma^{N-M}/p^{N-1}=u^{M-1}(1-u)^{N-M}=(1-u)^{N-1}v^{M-1}$, since
$\sigma/p=1-u$. This gives~\eqref{eq:simplex-positive-sum}. We write
$M!=\int_0^\infty e^{-t}t^M\,dt$ and sum over each $m_i$ separately. All terms are nonnegative, so
the sum and the integral can be exchanged, and this gives~\eqref{eq:simplex-balance-integral}.
\end{proof}

When $\sigma\ge0$, part (a) says that the chain keeps its state with probability $\sigma$ and
otherwise draws a fresh state uniformly from all $d$ states, possibly the same one. The proof
does not need $\sigma\ge0$.

\subsection{Proof of Theorem~\ref{thm:simplex-crossing}}

Let $F$ be a face of $k$ states and $n$ a balanced composition of $F$, so that $P_N(n)/V_N=S_n(u)$
by~\eqref{eq:SnVN}, with $u=(1-p)/((d-1)p)$.

\emph{Part (i), one crossing.} A word with $j$ switches is a sequence of $j+1$ run colors in
which successive colors differ, together with run lengths. If color $i$ appears $m_i$ times,
its run lengths can be chosen in $\binom{n_i-1}{m_i-1}$ ways. So every coefficient of $S_n$ is
nonnegative, and the first nonzero one is $W_n(k-1)=k!$, from the $k!$ orders of $k$ single
runs. Hence $S_n$ increases strictly from $0$ to $\infty$ on $u>0$ and equals $1$ exactly once.
As $p$ increases over $(0,1)$, $u$ decreases over $(0,\infty)$. This gives $p_{N,k}$ and the
two inequalities for the balanced composition.

\emph{Part (i), monotonicity in $N$.} Choose the balanced compositions for increasing
$N$ by adding $1$ to one of the smallest coordinates. Each coefficient $W_n(j)$
is then nondecreasing. Indeed, inserting an extra copy of the incremented color just before
its first occurrence keeps the number of switches, and deleting the first letter of the first
run of that color inverts it. The coefficient of $u^k$ is
\[
W_n(k)=\frac{k!(k-1)}2\sum_i(n_i-1)=\frac{k!(k-1)}2(N-k),
\]
since exactly one run color appears twice. For each choice of it there are
$(k+1)!/2-k!=k!(k-1)/2$ admissible sequences of run colors, all orders minus those with the
two copies adjacent, and $n_i-1$ ways to split its $n_i$ visits into two runs. This coefficient
increases strictly with $N$. So at every $u>0$ the value $S_n(u)$ increases strictly with $N$,
the root of $S_n(u)=1$ decreases strictly, and $p=1/(1+(d-1)u)$ increases.

\emph{Part (i), the first crossing.} If $N=k$, every coordinate is $1$ and
$S_n(u)=k!u^{k-1}$, which gives $p_{k,k}$. Since the root decreases with $N$,
every root has $u\le(k!)^{-1/(k-1)}<1$, and $u<1$ means $h<p$, that is,
$p>1/d$.

\emph{Part (i), the face maximum.} For $p\ge1/d$, Theorem~\ref{thm:simplex-balance} gives
$M_F=P_N(n)$. For $p<1/d$ we have $u>1$, which exceeds the root of $S_n(u)=1$, so
$M_F\ge P_N(n)>V_N$. Hence $M_F=V_N$ exactly at $p_{N,k}$, with the same two inequalities.

\emph{Part (ii).} By~\eqref{eq:logmax} with $K_F=D_k/d$ and $\lambda_F=d-k$
(Corollary~\ref{cor:paperB}), uniformly for $1\le T\le N^{1/3}$,
\[
\log M_F=-(k-1)L+\tfrac{k-1}2\log T+\log(D_k/d)-(d-k)T+O(1/T+T^2/N).
\]
By~\eqref{eq:vertexmax}, $\log V_N=-\log d-(d-1)T+O(T^2/N)$. Subtracting, and using
$\log D_k=-(k-1)a_k$, gives the display. With $T=L-\frac12\log L+t$ and $t$ in a compact set,
$T-L+\frac12\log T=t+\frac12\log(T/L)=t+O(\log L/L)$, and the error tends to $0$.

\emph{Part (iii).} Let $G$ be a vertex, so that $M_G=V_N$, $\Delta=\delta=k-1$ and
$\log(K_F/K_G)=\log D_k=-(k-1)a_k$. By Theorem~\ref{thm:crossing}(i) the function
$\log(M_F/M_G)$ has a zero in $[L/2,2L]$, and by (i) this zero is $T_{N,k}$.
Dividing~\eqref{eq:crossing} by $k-1$ gives the implicit equation. For the explicit form put
$z_*=L-\frac12\log L+a_k+e_L/L$ with $e_L=\frac14\log L-\frac12a_k$. Since
$\log z_*=\log L-\frac{\log L}{2L}+\frac{a_k}L+O(\frac{(\log L)^2}{L^2})$, we get
$z_*+\frac12\log z_*=L+a_k+O((\log L)^2/L^2)$. The derivative of $z+\frac12\log z$ is at
least $1$, so $T_{N,k}-z_*=O(1/L)$, which is~\eqref{eq:simplex-root-expansion}.

\emph{Part (iv).} We have $a_k=\frac12\log(4\pi)-f(k)$ with
$f(x)=(x+\frac12)\log x/(x-1)$. For $x\ge2$ the sign of $f'(x)$ is the sign of
$g(x)=2x^2-x-1-3x\log x$. Here $g(2)=5-6\log2>0$, and
$g'(x)=4x-4-3\log x>0$ for $x\ge2$, since $g'(2)>0$ and
$g''(x)=4-3/x>0$. So $f$ increases strictly and $a_k$ decreases strictly. This order is
checked in Lean. \qed

\subsection{Proof of Theorem~\ref{thm:simplex-balance}}

Let $\mu_a=\mu_0$ for $a\in F$. By~\eqref{eq:runform}, $P_N(n)=\mu_0p^{N-1}S_n(u)$ for every
composition with support $F$, so it is enough to show that the move strictly increases $S_n(u)$.
The rest follows, since $S_n$ is symmetric in the coordinates of $n$ and repeated moves reach a
balanced composition. At $p=1/d$ we have $u=1$, $S_n(1)=N!/\prod_in_i!$, and the move multiplies
it by $n_b/(n_a+1)>1$. Let $1/d<p<1$, so that $v>0$.

\emph{Log-concavity.} We show that for each $w>0$ the sequence $B_n(w)$,
$n\ge1$, is strictly log-concave. Write
\[
\frac{B_n(w)}w=b_{n-1},\qquad
b_j=\sum_{i=0}^j\binom ji\rho_i,\quad
\rho_i=\frac{w^i}{(i+1)!},\qquad
c_j=\sum_{i=0}^j\binom ji\rho_{i+1}.
\]
Pascal's identity gives $b_{j+1}=b_j+c_j$. The ratio $c_j/b_j$ is the mean of the
strictly decreasing sequence $f_i=\rho_{i+1}/\rho_i=w/(i+2)$ under the weights
$\omega_i=\binom ji\rho_i$, $0\le i\le j$, normalized. Let $\omega'_i=\binom{j+1}i\rho_i$,
$0\le i\le j+1$, be the weights for $j+1$, and put $\omega_{j+1}=0$. On
$0\le i\le j$ the ratio $\omega'_i/\omega_i=(j+1)/(j+1-i)$ increases in $i$. So for
$i<l$ we have $\omega'_i\omega_l-\omega'_l\omega_i\le0$, with strict inequality when
$l=j+1$. Hence
\[
\sum_i\omega'_if_i\sum_l\omega_l-\sum_i\omega'_i\sum_l\omega_lf_l
=\sum_{i<l}(\omega'_i\omega_l-\omega'_l\omega_i)(f_i-f_l)<0,
\]
since $f_i>f_l$ for $i<l$. So $c_{j+1}/b_{j+1}<c_j/b_j$, and $b_{j+1}/b_j=1+c_j/b_j$
decreases strictly with $j$. By the Laguerre form in~\eqref{eq:refresh-polynomial}, this is a
Tur\'an-type inequality for the Laguerre polynomials.

\emph{The transfer.} Let $n_b\ge n_a+2$. By log-concavity,
$B_{n_a+1}(w)B_{n_b-1}(w)>B_{n_a}(w)B_{n_b}(w)$. We apply this
for $w=vt>0$ in~\eqref{eq:simplex-balance-integral} and integrate. \qed

\subsection{Proof of Corollary~\ref{cor:simplex-modes}}

For large $N$, $T\le CL$ forces $h\to0$ and so $p>1/d$. By
Theorem~\ref{thm:simplex-balance} every global mode is a vertex or a balanced
composition. Hence we compare the ratios $M_F/V_N$, which equal $S_n(u)$ at the balanced
compositions $n$ of $F$ and depend on $F$ only through $k=\lvert F\rvert$, with $1$ and with each
other. Fix $0<c<1$.

\emph{Case $T\le cL$.} The ratio $S_n(u)$ increases with $u$, and so with $T$. So it is at most
its value at $T=cL$, which is $\exp(-(k-1)(1-c)L+O(\log L))\to0$ by
Theorem~\ref{thm:simplex-crossing}(ii). So the vertices are the only modes, and $p>p_{N,d}$.

\emph{Case $cL\le T\le CL$.} Theorem~\ref{thm:simplex-crossing}(ii) gives
$\log(M_F/V_N)=(k-1)(y-a_k)+o(1)$ uniformly, with $y=T-L+\frac12\log T$ the same for every $k$.
Fix $0<\eta<(a_{d-1}-a_d)/2$ when $d\ge3$, and any $\eta>0$ when $d=2$.
If $y\le a_d-\eta$, every face ratio is at most
$e^{-(k-1)\eta+o(1)}<1$, so only the vertices are modes. If
$y\ge a_d+\eta$, the full center beats the vertices, and it beats a
balanced composition with $2\le k<d$ positive coordinates by
\[
(d-k)(y-a_d)+(k-1)(a_k-a_d)+o(1)>0 .
\]
If $|y-a_d|<\eta$, a proper face with $k\ge2$ states has
$(k-1)(y-a_k)+o(1)\le-(k-1)\eta+o(1)<0$, since $a_k-a_d>2\eta$. So it
loses to the vertices, and the comparison of the vertices with the full centers
is decided by the unique crossing $p_{N,d}$ of Theorem~\ref{thm:simplex-crossing}(i).
In each subcase the conclusion agrees with (i) to (iii).

The last claim follows from~\eqref{eq:simplex-root-expansion} and
$p_{N,k}=1-(d-1)T_{N,k}/N$, since $a_2>\cdots>a_d$. \qed

\section{Proofs near a vertex}\label{app:vertex}

We prove Lemma~\ref{lem:uniform-laguerre-bessel} and Theorem~\ref{thm:vertex-uniform-root},
with the notation of Section~\ref{sec:vertex}.

\begin{proof}[Proof of Lemma~\ref{lem:uniform-laguerre-bessel}]
By definition, $B_a(w)/w=\sum_{j=0}^{a-1}\binom{a-1}j w^j/(j+1)!$. So the
coefficient of $(aw)^j/[j!(j+1)!]$ in $B_a(w)/w$ is
$\binom{a-1}jj!/a^j=\prod_{l=1}^j(1-l/a)_+$, also for $j\ge a$, where it is $0$. By the
clipped products of~\cite[Lemma~3.2(b)]{paperA}, its deficit from $1$ lies between $0$ and
$j(j+1)/(2a)$. Note that $\sum_{j\ge1}j(j+1)x^j/[j!(j+1)!]=\sum_{j\ge1}x^j/[(j-1)!\,j!]=x\,\mathcal I(x)$.
So
\[
0\le\mathcal I(aw)-\frac{B_a(w)}w
\le\frac1{2a}\sum_{j\ge0}\frac{j(j+1)(aw)^j}{j!(j+1)!}
=\frac w2\,\mathcal I(aw).
\]
\end{proof}

\begin{lemma}[elasticity and concavity of the rare-count polynomial]
\label{lem:vertex-elasticity}
On $(0,\infty)$ the polynomial $\mathcal B$ is convex, $\log\mathcal B$ is concave
and $\kappa$ is nondecreasing. Also
\begin{equation}\label{eq:vertex-elasticity}
 \frac1{|a|+1}\le y_a\le1,\qquad
 \kappa(y)\le \ell+\sqrt{\ell|a|y}\quad(y>0).
\end{equation}
For $0<y\le1$ we also have $\kappa(y)\ge\max\{\ell,\tfrac12\sqrt{|a|y}\}$.
\end{lemma}

\begin{proof}
\emph{Step 1. One factor as a law.} The polynomial $\mathcal B$ has
nonnegative coefficients, so it is convex. For one factor $B_a$ and $y>0$, let
$J$ take values in $\{0,\ldots,a-1\}$ with $\Prob(J=j)$ proportional to
$\binom{a-1}{j}y^j/(j+1)!$. Since $B_a(y)=\sum_jy\binom{a-1}jy^j/(j+1)!$, the elasticity of
$B_a$ is $\kappa_a=yB_a'/B_a=\E[J+1]$. The ratio
\[
\mathcal R(j)=\frac{(j+1)\Prob(J=j+1)}{\Prob(J=j)}=\frac{y(a-1-j)}{j+2},
\qquad \mathcal R(a-1)=0,
\]
is nonincreasing in $j$. Summing the ratio identity gives
$\E J=\E \mathcal R(J)$ and $\E J^2=\E[(J+1)\mathcal R(J)]$. So
$\Var J=\E J+\Cov(J,\mathcal R(J))\le \E J$, since for an independent copy $J'$ twice the
covariance is $\E[(J-J')(\mathcal R(J)-\mathcal R(J'))]\le0$.

\emph{Step 2. Concavity and monotonicity.} Since
$y^2B_a''/B_a=\E[J(J+1)]$, we get
\[
 (\log B_a)''=\frac{\Var J-\E[J+1]}{y^2}<0,
 \qquad y\kappa_a'=\Var J\ge0.
\]
Summing over the factors proves the concavity of $\log\mathcal B$ and the monotonicity
of $\kappa=\sum_i\kappa_{a_i}$.

\emph{Step 3. The elasticity bounds.} The differential equation
$yB_a''+yB_a'=aB_a$, which we check on the coefficients, gives
\[
 ay=\kappa_a^2+\Var J+(y-1)\kappa_a.
\]
Using $0\le\Var J\le\kappa_a-1$ we get
\[
 \kappa_a\le1+\sqrt{ay},\qquad \kappa_a^2+y\kappa_a\ge ay+1.
\]
For $0<y\le1$ the second inequality gives $(\kappa_a+\frac12)^2>ay$, and
$\kappa_a\ge1$. So $\kappa_a\ge\max\{1,\sqrt{ay}-1/2\}\ge\tfrac12\sqrt{ay}$.
Summing over the factors proves the elasticity bounds, using the
Cauchy--Schwarz inequality $\sum_i\sqrt{a_i}\le\sqrt{\ell|a|}$ for the upper bound and
$\sqrt{\sum_ix_i}\le\sum_i\sqrt{x_i}$ for the lower one.

\emph{Step 4. The root.} We have $B_a(1)\ge1$, from its term $m=1$, and
$B_a(y)\le y\,\mathcal I(ay)\le y e^{ay/2}$ by Lemma~\ref{lem:uniform-laguerre-bessel} and
$(j+1)!\ge2^j$. So $\mathcal B(1)\ge1$, while, since $|a|+1\ge2$,
$\mathcal B(1/(|a|+1))\le(|a|+1)^{-\ell}e^{|a|/(2(|a|+1))}\le\frac12e^{1/2}<1$.
\end{proof}

\begin{lemma}[an exact compound-binomial representation]
\label{lem:vertex-randomization}
For $0<u<1$, put $v=u/(1-u)$. Let $\Gamma_2$ have the gamma distribution
with shape $2$ and rate $1$, let $E_1,\ldots,E_{b-1}$ be independent
rate-one exponential variables, and let $Z_1,\ldots,Z_{b-1}$ be
independent Bernoulli variables with parameter $u$. Take all these
variables independent and set
\[
 \mathcal X=\Gamma_2+\sum_{j=1}^{b-1}Z_jE_j.
\]
Then
\begin{equation}\label{eq:vertex-randomization}
 S_{(b,a)}(u)=(1-u)^{|a|}\E[\mathcal B(v\mathcal X)],\qquad
 \E e^{s\mathcal X}=(1-s)^{-2}
 \left(1+\frac{us}{1-s}\right)^{b-1}\quad(s<1).
\end{equation}
In particular, set $y=bu^2$ and $\kappa=\kappa(y)$. For every $0<u<1$,
\begin{equation}\label{eq:vertex-jensen-lower}
 (1-u)^{|a|} \mathcal B(y)\le S_{(b,a)}(u).
\end{equation}
If also $s=\kappa/[bu(1-u)]<1$, then
\begin{equation}\label{eq:vertex-tangent-upper}
 S_{(b,a)}(u)\le (1-u)^{|a|} \mathcal B(y)e^{-\kappa}(1-s)^{-2}
 \left(1+\frac{us}{1-s}\right)^{b-1}.
\end{equation}
\end{lemma}

\begin{proof}
For $0<u<1$ the stay probability $p=1/(1+(d-1)u)$ lies in $(1/d,1)$, so
by~\eqref{eq:simplex-balance-integral} with $N=b+|a|$,
\[
 S_{(b,a)}(u)=\frac{(1-u)^{b+|a|-1}}v
 \int_0^\infty e^{-t}B_b(vt)\,\mathcal B(vt)\,dt.
\]
The integral of $e^{-t}B_b(vt)$ is $\sum_m\binom{b-1}{m-1}v^m=v(1+v)^{b-1}$, and
$(1-u)^{b-1}(1+v)^{b-1}=1$. Hence $e^{-t}B_b(vt)$, normalized, is the density
$\sum_{j=0}^{b-1}\binom{b-1}ju^j(1-u)^{b-1-j}\,t^{j+1}e^{-t}/(j+1)!$, a mixture of
gamma densities with shape $J+2$, where $J$ has the binomial law with parameters $b-1$ and $u$.
This is exactly the law of $\mathcal X$, which proves the identity and its
moment generating function. Since $\E\mathcal X=2+(b-1)u$, we have
$v\E\mathcal X-bu^2=u(2-u+bu^2)/(1-u)>0$, so $v\E\mathcal X\ge y$.
Jensen's inequality for the convex function $\mathcal B$, and the monotonicity of $\mathcal B$,
prove~\eqref{eq:vertex-jensen-lower}. Concavity of $\log\mathcal B$ gives, for $w>0$,
\[
 \mathcal B(w)\le \mathcal B(y)\exp\!\left(\frac{\kappa}{y}(w-y)\right).
\]
Insert $w=v\mathcal X$, use $\kappa v/y=s$, and use the moment
generating function to prove~\eqref{eq:vertex-tangent-upper}.
\end{proof}

\begin{proof}[Proof of Theorem~\ref{thm:vertex-uniform-root}]
Write $S=S_{(b,a)}$, $u_0=\sqrt{y_a/b}$ and $\kappa_0=\kappa(y_a)$. Since
$S$ increases on $u>0$, it suffices to show $\log S(u_+)>0$ and
$\log S(u_-)<0$ at $u_\pm=(1\pm\delta)u_0$. Here $\delta\le1/2$ forces
$\varepsilon\le1/(32c_\ell)\le1/96$, and $u_0\le\sqrt{1/b}\le\varepsilon$.

\emph{Step 1. The upper point.} We have $u_+\le\tfrac32\varepsilon<1/4$. Since
$\log\mathcal B(y_a)=0$ and $(\log\mathcal B)'=\kappa/y$, the monotonicity of $\kappa$ gives
$\log\mathcal B(bu_+^2)\ge\kappa_0\log(bu_+^2/y_a)=2\kappa_0\log(1+\delta)$. The lower
bound in Lemma~\ref{lem:vertex-elasticity}, applied at $y_a\le1$, gives
$|a|u_0/\kappa_0\le2\sqrt{|a|/b}\le2\varepsilon$. Hence,
by~\eqref{eq:vertex-jensen-lower} and $-\log(1-x)\le x/(1-x)$,
\[
 \log S(u_+)\ge
 2\kappa_0\log(1+\delta)-\frac{|a|u_+}{1-u_+}
 \ge\kappa_0\left(\frac43\delta-4\varepsilon\right)>0.
\]
Here we used $2\log(1+\delta)\ge\frac43\delta$ for $\delta\le\frac12$, and
$|a|u_+/(1-u_+)\le\frac43\cdot\frac32|a|u_0\le4\kappa_0\varepsilon$. The final
inequality of the display holds since $\frac43\delta=\frac{64}3c_\ell\varepsilon>4\varepsilon$.

\emph{Step 2. The lower point.} Write $y=bu_-^2$ and $\kappa=\kappa(y)$. The
inequalities $y_a\ge1/(|a|+1)$, $\delta\le1/2$ and the upper bound
in~\eqref{eq:vertex-elasticity} imply
\[
 u_-\le\varepsilon,\qquad bu_-\ge\frac1{2\varepsilon},\qquad
 s=\frac{\kappa}{bu_-(1-u_-)}
 \le 2(2\ell+\sqrt\ell)\varepsilon\le4c_\ell\varepsilon\le\frac18.
\]
For the bound on $s$ we have
$\kappa/(bu_-)\le\ell/(bu_-)+\sqrt{\ell|a|y}/(bu_-)\le2\ell\varepsilon+\sqrt{\ell|a|/b}
\le(2\ell+\sqrt\ell)\varepsilon$, and $1-u_-\ge\frac12$. We use
$\log(1+x)\le x$, $-\log(1-s)\le s/(1-s)$ and $u_-,s\le1/4$
in~\eqref{eq:vertex-tangent-upper}, and we drop the nonpositive term
$|a|\log(1-u_-)$. This gives
\begin{align*}
 \log\frac{S(u_-)}{\mathcal B(y)}
 &\le -\kappa+\frac{2s}{1-s}+\frac{bu_-s}{1-s}\\
 &=\kappa\left(\frac1{(1-u_-)(1-s)}-1\right)
       +\frac{2s}{1-s}\\
 &\le 5\kappa(u_-+s)\le25c_\ell\kappa\varepsilon.
\end{align*}
Here the equality uses $bu_-s=\kappa/(1-u_-)$. The next inequality uses
$(1-u_-)(1-s)\ge\frac{21}{32}$, so the first term is at most
$\frac{32}{21}\kappa(u_-+s)$, and $\kappa\ge\ell\ge1$, so the second is at most
$\frac{16}7\kappa s$. The last one uses $u_-+s\le5c_\ell\varepsilon$. Again by the
monotonicity of $\kappa$,
\[
 \log\mathcal B(y)=-\int_y^{y_a}\frac{\kappa(x)}x\,dx\le2\kappa\log(1-\delta)\le-2\kappa\delta
 =-32c_\ell\kappa\varepsilon.
\]
Hence $\log S(u_-)\le\kappa(-32c_\ell+25c_\ell)\varepsilon<0$.
\end{proof}

\section{Proofs for reducible faces and for the counts near a facet}\label{app:counts}

\subsection{Blocks of a reducible face}\label{app:reducible-blocks}

In this subsection and the next, $F$ is admissible, the notation is that of
Section~\ref{sec:reducible-max}, and $hq_{\max}\le\frac12$. For an irreducible $C$, states
$y,x\in C$ and $m\in\Z_{\ge1}^C$ with $|m|=l$, let $G_C(m)_{yx}$ be the probability that the chain
with $X_1=y$ has composition $m$ in $X_1,\dots,X_l$ and $X_l=x$. So $G_{\{i\}}(m)=(1-hq_i)^{m-1}$.
For $n$ with support $F$ let $n^{(j)}$ be $n$ restricted to $C_j$, $N_j=|n^{(j)}|$, $s_j=hN_j$ and
$G_j=G_{C_j}(n^{(j)})$. The blocks of a path with composition $n$ occupy fixed times, a path that
starts outside $C_1$ never visits $C_1$, and the Markov property at the $r-1$ switches gives
\[
P_N(n)=\mu_{C_1}^\top G_1\,(hR_1)\,G_2\,(hR_2)\cdots(hR_{r-1})\,G_r\,\ones .
\]

\emph{The bridge run expansion.} Let $|C|=k\ge2$ and write $W_{yx}=W_{A_C,e_y,e_x}$. The proof of
Lemma~\ref{lem:runs}, with the first and the last letter of the words fixed, gives
$G_C(m)_{yx}=\sum_{m'}W_{yx}(m')h^{M'-1}\prod_i\binom{m_i-1}{m'_i-1}(1-hq_i)^{m_i-m'_i}$. As in
Step~1 of the proof of Lemma~\ref{lem:discrete}, with $t=hm$ the $m'$-th term is $\kappa_{m'}v_{m'}$,
where $v_{m'}=W_{yx}(m')h^{k-1}\prod_ie^{-q_it_i}t_i^{m'_i-1}/(m'_i-1)!$, so that
$\sum_{m'}v_{m'}=h^{k-1}\Phi_C(t)_{yx}$, and $\kappa_{m'}$ is the factor of that proof. Its
argument for~\eqref{eq:A-kappa} gives
\[
1-\sum_i\frac{m'_i(m'_i-1)}{2m_i}-q_{\max}^2h|t|\le\kappa_{m'}\le e^{hq_{\max}M'} .
\]
Since $W_{yx}(m')\le k^{M'}q_{\max}^{M'-1}$, the series of $\Phi_C$ converges for every
$t\in\R^C$, and so do the series with the extra factors $M'e^{M'/2}$ or $m'_i(m'_i-1)t_i^{-1}$ in their
terms. So $\Phi_C$ is real-analytic on $\R^C$, and these series are bounded on bounded sets.
Also $\Phi_C(t)_{yx}>0$ for $t\in(0,\infty)^C$, since some word from $y$ to $x$ uses every state.

\begin{lemma}[one bound for every block]\label{lem:block-bound}
There are $N_0$ and $C_0$ such that for $N\ge N_0$, $T\le N^{1/3}$, every $j$, all $y,x\in C_j$ and every
$m$ with $\supp m=C_j$ and $s=h|m|\le T$, we have $G_{C_j}(m)_{yx}\le C_0h^{k_j-1}w_j(s)$, where
$w_j(s)=(1\vee s)^{-(k_j-1)/2}e^{-\lambda_Fs}$ if $\lambda_j=\lambda_F$ and $k_j\ge2$, and
$w_j(s)=e^{-\lambda_js}$ otherwise.
\end{lemma}

\begin{proof}
We prove this in two cases. For $k_j=1$, $(1-hq_i)^{|m|-1}\le e^{hq_i}e^{-q_is}$. Let $k_j\ge2$. The bridge is at most the
probability that the chain started at $y$ has composition $m$ in its first $|m|$ visits.
Lemma~\ref{lem:upper} with the start $e_y$ and $\theta=q-\lambda_j\ones\in V_{C_j}$ bounds it by
$Ch^{k_j-1}e^{-\lambda_js}$. If $\lambda_j=\lambda_F$ and $s\ge1$, we apply
Theorem~\ref{thm:facemax}(b) with $v=0$ to the chain started at $y$ on the face $C_j$, which
satisfies (H), with $l=|m|$ steps and time $s$ in place of $N$ and $T$, and the same $h$. Its
range holds, since $l=sN/T\ge N^{2/3}$ and $s^3\le l$ is $s^2T\le N$. It gives
$G_{C_j}(m)_{yx}\le C'l^{-(k_j-1)}s^{(k_j-1)/2}e^{-\lambda_Fs}=C'h^{k_j-1}s^{-(k_j-1)/2}e^{-\lambda_Fs}$
with a constant $C'$ for the start $e_y$. There are finitely many starts $y$.
\end{proof}

\begin{lemma}[blocks of bounded time]\label{lem:block-short}
For every $R$ there is $C_R$ such that $|G_C(m)_{yx}/h^{|C|-1}-\Phi_C(hm)_{yx}|\le C_Rh$ for all
$y,x\in C$ and $\supp m=C$ with $h|m|\le R$. Hence, if $D=\nu^\top\Psi_1R'_1\Psi_2\cdots R'_{p-1}\Psi_pw$
with fixed nonnegative $\nu$, $R'_c$ and $w$ and $\Psi_c=G_{C_c}(m^{(c)})/h^{|C_c|-1}$, and $D^{\rm c}$
is the same product with $\Phi_{C_c}(hm^{(c)})$, then $|D-D^{\rm c}|\le C'_Rh$ when
$\sum_ch|m^{(c)}|\le R$.
\end{lemma}

\begin{proof}
For a singleton, $e^{hq}e^{-qt}\ge(1-hq)^{m-1}\ge e^{-qt-q^2ht}$ with $t=hm$. Let $|C|=k\ge2$.
By the bound on $\kappa_{m'}$, $\kappa_{m'}-1\le hq_{\max}M'e^{M'/2}$, which gives
$G/h^{k-1}-\Phi_C\le C_Rh$. Also $1-\kappa_{m'}\le\sum_im'_i(m'_i-1)/(2m_i)+q_{\max}^2hR$, and
since $1/m_i=h/t_i$, the sum of $v_{m'}m'_i(m'_i-1)/(2m_i)$ over $m'$ is $h^k$ times a series with
the extra factor $m'_i(m'_i-1)/(2t_i)$, which is bounded on $|t|\le R$. This gives
$\Phi_C-G/h^{k-1}\le C_Rh$. For the product we telescope, since every factor is bounded
on $|t|\le R$.
\end{proof}

\begin{lemma}[the bridge law]\label{lem:block-local}
Let $|C|=k\ge2$ and let $\mathcal K\subset\Delta_C^\circ$ be compact. Uniformly for $y,x\in C$,
$m$ with $m/l\in\mathcal K$ and $1\le s=hl\le l^{1/2}$,
$G_C(m)_{yx}=h^{k-1}\Phi_C(hm)_{yx}(1+O(s^2/l))$.
\end{lemma}

\begin{proof}
We prove this as we prove Lemma~\ref{lem:discrete} in Appendix~\ref{app:torus}, with the bridge run
expansion in place of Lemma~\ref{lem:runs}, and $l$ and $s$ in place of $N$ and $T$. Its Step~3
uses the weights only through an upper bound on $H$ for $eA_C$ and a lower bound on $H$ at
$t=s\alpha$. Step~2 of the proof of Lemma~\ref{lem:upper}, with Lemma~\ref{lem:torus-complex}
for $(A_C,e_y,e_x)$ in place of Lemma~\ref{lem:torus}, gives~\eqref{eq:A-Hsigma} for $H_{A_C,e_y,e_x}$
with $b_{yx}$ in place of $b$, and $b_{yx}(\theta'(\alpha))$ is bounded on $\mathcal K$.
Lemma~\ref{lem:local-general} with $(\nu,e)=(e_y,e_x)$ gives
$H_{A_C,e_y,e_x}(s\alpha)e^{-s\langle\alpha,\theta(\alpha)\rangle}\ge cs^{-(k-1)/2}$ for $s\ge s_1$,
and for $1\le s\le s_1$ the left side is continuous and positive on $[1,s_1]\times\mathcal K$.
\end{proof}

\begin{lemma}[bridge maxima]\label{lem:block-max}
Let $|C|=k\ge2$, $\lambda=\lambda_C$ and $\Sigma=\Sigma_C$, and put
$Y_{yx}=h^{k-1}s^{-(k-1)/2}r_{C,y}l_{C,x}e^{-\lambda s}/((2\pi)^{(k-1)/2}\sqrt{\det\Sigma})$. There
are $s_2$, $l_2$ and $C_3$ such that the following hold for $y,x\in C$, $l\ge l_2$ and
$s=hl\in[s_2,l^{1/3}]$. \textup{(i)} $G_C(m)_{yx}\le Y_{yx}(1+C_3(1/s+s^2/l))$ for every $m$ with
$|m|=l$ and $\supp m=C$. \textup{(ii)} $G_C(m^0)_{yx}\ge Y_{yx}(1-C_3(1/s+s^2/l))$ if $|m^0|=l$ and
$|m^0_i-l\alpha^*_{C,i}|<1$ for all $i$.
\end{lemma}

\begin{proof}
Put $c_{yx}(\alpha)=r_{\alpha,y}l_{\alpha,x}/((2\pi)^{(k-1)/2}\sqrt{\det\Sigma_\alpha})$, which is
real-analytic and positive on $\Delta_C^\circ$ by Proposition~\ref{prop:constant}(a), with
$c_{yx}(\alpha^*_C)=r_{C,y}l_{C,x}/((2\pi)^{(k-1)/2}\sqrt{\det\Sigma})$. Let $J^0=J_C-\lambda$, which is
nonnegative and vanishes only at $\alpha^*_C$ (Proposition~\ref{prop:rate}(a)), and let $\xi$ be
the displacement $\alpha'-\alpha^{*\prime}_C$. Since $\nabla^2J_C=\Sigma_\alpha^{-1}$, Taylor's
theorem gives $\rho_0,c_U,C_U,C_1>0$ such that $U=\{|\xi|\le\rho_0\}\subset\Delta_C^\circ$ and on
$U$, for all $y,x$, $c_U|\xi|^2\le J^0\le C_U|\xi|^2$,
$|\log(c_{yx}(\alpha)/c_{yx}(\alpha^*_C))-g_{yx}\cdot\xi|\le C_1|\xi|^2$ with $|g_{yx}|\le C_1$,
and $\log c_{yx}$ is $C_1$-Lipschitz.

\emph{Outside $U$.} The function $J^0$ is lower semicontinuous and positive on the compact set
$\mathcal E=\Delta_C\setminus\{|\xi|<\rho_0\}$, so $J^0\ge3\eta$ there for some $\eta>0$. Since
$J_C=\sup_{\theta\in V_C}\langle q-\theta,\cdot\rangle$ (Section~\ref{sec:firstorder}), each
point of $\mathcal E$ has a neighborhood on which one $\theta\in V_C$ gives
$\langle q-\theta,\cdot\rangle>\lambda+2\eta$, and finitely many $\theta_1,\dots,\theta_p$ cover
$\mathcal E$. For $m/l\in\mathcal E$, the bridge is at most the probability that the chain started at $y$ has
composition $m$ in its first $l$ visits, and Lemma~\ref{lem:upper} bounds it by
$Ch^{k-1}e^{-s(\lambda+2\eta)}=Y_{yx}\,O(s^{(k-1)/2}e^{-2\eta s})=Y_{yx}\,O(1/s)$.

\emph{Inside $U$.} Lemma~\ref{lem:block-local} and Lemma~\ref{lem:local-general} with
$(\nu,e)=(e_y,e_x)$ at $t=s\alpha$ give
$G_C(m)_{yx}=Y_{yx}(c_{yx}(\alpha)/c_{yx}(\alpha^*_C))e^{-sJ^0(\alpha)}(1+E)$ with
$|E|\le C_2(1/s+s^2/l)$, where $\alpha=m/l$. For $s\ge2C_1/c_U$ the logarithm of the two middle
factors is at most $g_{yx}\cdot\xi-\frac12c_Us|\xi|^2\le|g_{yx}|^2/(2c_Us)$. With the bound outside
$U$ this proves (i). For (ii), $|\xi|\le k/l$ at $m^0$, so $c_{yx}(m^0/l)\ge c_{yx}(\alpha^*_C)e^{-kC_1/l}$
and $sJ^0(m^0/l)\le C_Uk^2s/l^2$.
\end{proof}

\subsection{Proofs of Theorem~\ref{thm:reducible} and Corollary~\ref{cor:reducible-crossing}}

\begin{proof}[Proof of Theorem~\ref{thm:reducible}]
Write $a=a_F$, $K=K^{\rm red}_F$ and $X=N^{-(k_F-1)}T^{p_F}e^{-T\lambda_F}=h^{k_F-1}T^{-a}e^{-T\lambda_F}$.
For $j\in S^*$ let $F^<_j(t^<)$ and $F^>_j(t^>)$ be the functions inside the suprema that define
$A_j$ and $B_j$. Call a block flat if it is a singleton $\{i\}$ with $q_i=\lambda_F$. Flat
blocks exist only if $k^*=1$, and then they are the blocks in $S^*$. Put
$\tilde w_m(u)=e^{\lambda_Fu}w_m(u)\le1$, with $w_m$ from Lemma~\ref{lem:block-bound}, let
$\delta_0>0$ be the least positive $\lambda_m-\lambda_F$ ($\delta_0=1$ if there is none), and
$\eta(u)=\max(e^{-\delta_0u},u^{-1/2})$ for $u\ge1$. A block that is not flat and has time
$u\ge1$ has $\tilde w_m(u)\le\eta(u)$.

\emph{The constant.} Since $\Phi_C$ is real-analytic on $\R^C$, $F^<_j$ and $F^>_j$ are Lipschitz
on bounded parts of the closed orthants, and the suprema over the open and the closed orthants
agree. Fix $t\in(0,\infty)^C$, let $h\to0$ and $m=\lceil t/h\rceil$. By
Lemma~\ref{lem:block-short}, $G_C(m)/h^{|C|-1}\to\Phi_C(t)$, and Lemma~\ref{lem:block-bound} with
$T=|t|+1$ and $N=T/h$ gives $e^{\lambda_F|t|}\Phi_{C_m}(t)\le C_0\tilde w_m(|t|)$ entrywise. So
every factor in $F^<_j$ and $F^>_j$ is bounded, and $A_j,B_j<\infty$. They are positive, since
$\mu_{C_1}\ge0$ and every $R_m\ge0$ are nonzero, $r,l>0$ and every $\Phi_{C_m}>0$ on the open
orthant. If $k^*\ge2$, no factor is flat, so $F^<_j$ and $F^>_j$ tend to $0$ at infinity, and the
suprema are attained at points $t^<_*$ and $t^>_*$ of the closed orthants.

\emph{Upper bound.} Let $\supp n=F$, so $\sum_ms_m=T$. By the factorization and
Lemma~\ref{lem:block-bound}, summed over the finitely many switch states, and since
$\sum_m(k_m-1)+r-1=k_F-1$,
\[
P_N(n)\le C_1h^{k_F-1}e^{-\lambda_FT}\prod_m\tilde w_m(s_m).
\]
Let $k^*\ge2$ and take $m_0$ with $s_{m_0}\ge T/r$. (A) If $m_0\notin S^*$, then
$\lambda_{m_0}>\lambda_F$, or $\lambda_{m_0}=\lambda_F$ and $k_{m_0}\ge k^*+1$, so
$P_N(n)\le C_1XT^a\max(e^{-\delta_0T/r},(T/r)^{-a-1/2})\le KX$ for $T\ge T_1$. (B1) If $m_0=j\in S^*$
and $T-s_j\ge R_0$, another block has time at least $R_0/r$, and $\tilde w_j(s_j)\le(T/r)^{-a}$, so
$P_N(n)\le C_1r^a\eta(R_0/r)X\le KX$ once $R_0$ is fixed large. (B2) If $m_0=j\in S^*$ and
$T-s_j<R_0$, write $P_N(n)=\mathbf a^\top G_j\mathbf b$ with
$\mathbf a^\top=\mu_{C_1}^\top G_1hR_1\cdots G_{j-1}hR_{j-1}$ and $\mathbf b=hR_jG_{j+1}\cdots G_r\ones$,
both nonnegative. Lemma~\ref{lem:block-max}(i) applies with $l=N_j\ge N/2$ and $s=s_j\ge T-R_0$,
since $s_j^3\le N_j$ means $s_j^2T\le N$, and $s_j^2/N_j=s_jT/N$. It gives
\[
P_N(n)\le h^{k^*-1}s_j^{-a}e^{-\lambda_Fs_j}\,\frac{(\mathbf a\cdot\mathbf r)(l'\cdot\mathbf b)}{(2\pi)^a\sqrt{\det\Sigma_{C_j}}}\,
\Bigl(1+C\Bigl(\frac1T+\frac{T^2}N\Bigr)\Bigr),
\]
where $\mathbf r=r_{C_j}$ and $l'=l_{C_j}$. Let $\varkappa_<=\sum_{m<j}(k_m-1)+j-1$ be the number of factors $h$ in $\mathbf a$. Then
$\mathbf a\cdot\mathbf r/h^{\varkappa_<}$ is a product as in Lemma~\ref{lem:block-short} with total time
$\rho_<\,<\,R_0$, whose continuum version is at most $e^{-\lambda_F\rho_<}A_j$. So
$e^{\lambda_F\rho_<}\mathbf a\cdot\mathbf r\le h^{\varkappa_<}(A_j+Ch)$, and likewise
$e^{\lambda_F\rho_>}l'\cdot\mathbf b\le h^{\varkappa_>}(B_j+Ch)$. Since
$\varkappa_<+\varkappa_>+k^*-1=k_F-1$, $\rho_<+\rho_>+s_j=T$, $s_j^{-a}\le T^{-a}(1+CR_0/T)$ and
$h\le T^2/N$, we get $P_N(n)\le X(K+C(1/T+T^2/N))$.

Let $k^*=1$, and let $\rho$ be the total time of the blocks that are not flat. If $\rho\ge R_0$,
one of them has time at least $R_0/r$, and $P_N(n)\le C_1\eta(R_0/r)X\le KX$. If $\rho<R_0$, each
flat block has $(1-h\lambda_F)^{N_m-1}\le e^{h\lambda_F}e^{-\lambda_Fs_m}$, and the other blocks
form products as in Lemma~\ref{lem:block-short} with total time below $R_0$. So
$P_N(n)\le h^{k_F-1}e^{-\lambda_FT}e^{rhq_{\max}}(\Psi^*+Ch)$, where $\Psi^*$ is the supremum of
$e^{\lambda_F|t|}\mu_{C_1}^\top\Phi_{C_1}R_1\cdots\Phi_{C_r}\ones$, in which every flat factor
equals $1$. At a flat $j\in S^*$ this product splits into the products of $A_j$ and $B_j$, in
disjoint variables, so $\Psi^*=A_jB_j=K$.

\emph{Lower bound.} Let $k^*\ge2$, and take $j\in S^*$ that attains the maximum in $K$ and the
points $t^<_*$ and $t^>_*$, with total time $\tau_*$. For $m\ne j$ put
$n^{(m)}=\max(1,\lceil t^{(m)}_*/h\rceil)$ entrywise. Every entry of $hn^{(m)}$ is within $h$ of
that of $t^{(m)}_*$, so $F^<_j(hn^<)\ge A_j-Ch$ and $F^>_j(hn^>)\ge B_j-Ch$. Give block $j$ the
length $l=N-\sum_{m\ne j}N_m$ and a composition $m^0$ with $|m^0_y-l\alpha^*_{C_j,y}|<1$, which has
full support for large $N$, so that $s=hl\in[T-\tau_*-Ch,T]$. By Lemma~\ref{lem:block-short},
$\mathbf a\ge h^{\varkappa_<}(\mathbf a^{\rm c}-Ch)$ entrywise, with
$e^{\lambda_F|hn^<|}\mathbf a^{\rm c}\cdot\mathbf r=F^<_j(hn^<)$, and likewise for $\mathbf b$.
Lemma~\ref{lem:block-max}(ii) and nonnegativity give
$P_N(n)\ge h^{k_F-1}s^{-a}e^{-\lambda_FT}(A_j-Ch)(B_j-Ch)((2\pi)^a\sqrt{\det\Sigma_{C_j}})^{-1}(1-C(1/T+T^2/N))$,
and $s^{-a}\ge T^{-a}$. So $M_F\ge X(K-C(1/T+T^2/N))$. Let $k^*=1$. The supremum $\Psi^*$ is
attained at a point of the closed orthants. Indeed, with the flat factors equal to $1$, every
other factor is at most $C_0\tilde w_m\le C_0$, and a factor whose time grows tends to $0$, since
it is not flat. So the function tends to $0$ at infinity. We take a flat $j\in S^*$, treat the
blocks that are not flat in the same way at such a point, give every other flat block the length $1$,
and use $(1-h\lambda_F)^{l-1}\ge e^{-\lambda_Fs}(1-CT^2/N)$ for block $j$.
\end{proof}

\begin{proof}[Proof of Corollary~\ref{cor:reducible-crossing}]
Theorem~\ref{thm:reducible} for $F$ and $G$ gives, uniformly on $[T_1,N^{1/3}]$,
\[
D_N(T)=-\Delta L+\Xi(T)+\kappa+O\Bigl(\frac1T+\frac{T^2}N\Bigr),\qquad\Xi(T)=\delta T+\pi\log T .
\]
(i) For fixed $N$, $M_F$ and $M_G$ are positive maxima of finitely many polynomials in $h$, so
$D_N$ is continuous. At $T=\frac\Delta{2\delta}L$ the display gives $D_N=-\frac\Delta2L+O(\log L)<0$,
and at $T=\frac{2\Delta}\delta L$ it gives $D_N=\Delta L+O(\log L)>0$.

(ii) A zero has $\delta T_N+\pi\log T_N=\Delta L+O(1)$. If $\pi\ge0$, then $\delta T_N\le\Delta L+O(1)$.
If $\pi<0$, then $\pi\log T_N\ge\pi L/3$, since $T_N\le N^{1/3}$, and
$\delta T_N\le(\Delta+|\pi|/3)L+O(1)$. So $T_N=O(L)$, then $T_N=\frac\Delta\delta L+O(\log L)$, and
the error term is $O(1/L)$.

(iii) Let $\Xi(T^\circ)=\Delta L-\kappa$. As in (ii), $T^\circ=\frac\Delta\delta L+O(\log L)$, and
$\Xi'=\delta+\pi/T\ge\delta/2$ on $[\frac\Delta{2\delta}L,\frac{2\Delta}\delta L]$ for large $N$.
So every zero is within $O(1/L)$ of $T^\circ$. Putting
$\log T^\circ=\log L+\log\frac\Delta\delta+O(\log L/L)$ into $\delta T^\circ=\Delta L-\pi\log T^\circ-\kappa$
gives the formula.
\end{proof}

\subsection{Proof of Theorem~\ref{thm:pair-counts}}

(a) For $G=\{j\}$ the only word in Lemma~\ref{lem:layer-exact-discrete} is $j^{N'}$, so
$P^{(z)}_{N'}(\hat n)=(\mu_j+zh\mu_iq_{ij})(1+zhq_{ji})(p_j+zh^2\omega)^{N'-1}$ and
$P^{(0)}_{N'}(\hat n)=\mu_jp_j^{N'-1}$. We divide by it, expand
$(1+zh\frac{\mu_i}{\mu_j}q_{ij})(1+zhq_{ji})=\sum_ec_eh^ez^e$ and $(1+zh^2\omega/p_j)^{N'-1}$ by the
binomial theorem, and use the lemma.

(b) By Lemma~\ref{lem:L-two}, $\widetilde\Lambda(z)=(1+c_1z+c_2z^2)e^{S\omega z}$. So
$\xi_e=c_e$ in the Laguerre form of Section~\ref{sec:layer-counts}, which follows from
$\sum_K\binom{n-1}{K-1}p^{n-K}h^K(S\omega)^{K-e}/(K-e)!=p^n(h/p)^eB^{(e)}_n(u)$.

(c) Write $n=n_i$, $p=p_i$ and $m=K-e$. The $(K,e)$ term of $R$ is $\rho_m$ times the $(K,e)$ term
$t_{K,e}=\binom{n-1}{K-1}p^{n-K}h^Kc_e(S\omega)^m/m!$ of $\mathbf K^h$, where
$\rho_m=\prod_{l=1}^m(1-l/N')\,p_j^{-m}$, which is $0$ for $m\ge N'$. All $t_{K,e}\ge0$, so
$R/\mathbf K^h=\E_\nu\rho_m$ for the probability $\nu$ proportional to $t_{K,e}$. Since
$-\log(1-y)\le2y$ for $y\le\frac12$, $1-m(m+1)/(2N')\le\rho_m\le e^{2hq_jm}$. Fix $e$. Then
$t_{K,e}$ is proportional to $\beta_K=\binom{n-1}{K-1}u^m/m!$, and
$\beta_{K+1}/\beta_K=\frac{n-K}K\cdot\frac u{K+1-e}\le\frac U{K(K+1-e)}$. Multiplying out and
summing over $K$ gives $\E[(K-1)(K-e)]\le U$ under the law proportional to $\beta_K$, that is
$\E[(m-1)m]\le U$ for $e=0$, $\E[m^2]\le U$ for $e=1$ and $\E[m(m+1)]\le U$ for $e=2$. So
$\E m\le1+\sqrt U$ and $\E m(m+1)\le U+2\sqrt U+2$, which gives the lower bound. Since
$\beta_K\propto u^m$, the logarithmic derivative of $\mathcal S_e(u)=\sum_K\beta_K$ is
$\E^{(u)}[m]/u$. So with $c=2hq_j\le1$ and $\sigma=e^c$,
$\log\E e^{cm}=\int_u^{\sigma u}\E^{(v)}[m]\,v^{-1}dv\le c(1+\sqrt{\sigma U})\le c(1+2\sqrt U)$,
which gives the upper bound. Mixing over $e$ keeps both bounds. For the last claim,
$U/N'\le2\omega h^2n$, $\sqrt U/N'\le h\sqrt{2\omega}$, $1/N'\le2h$ since $hN'\ge T/2$, and
$hq_j\sqrt U=O(h^2\sqrt{N'n})$. For $T\le N^{1/2}$ we have $h^2\sqrt{N'n}\le h^2N=T^2/N\le1$, so
the exponent $2hq_j(1+2\sqrt U)$ is bounded and the upper bound is $1+O(h+h^2\sqrt{N'n})$.

(d) Let $a_K=[z^K]\widetilde\Lambda=\sum_ec_e(S\omega)^{K-e}/(K-e)!$. By~\eqref{eq:layer-series},
$\mathbf K(x,S)=e^{-q_ix}\sum_{K\ge1}x^{K-1}a_K/(K-1)!$, and the $(K,e)$ terms of $\mathbf K^h/h$ and
$\mathbf K$ have the ratio $\prod_{l<K}(1-l/n)\,p^{n-K}e^{hq_in}$, which lies between
$(1-K(K-1)/(2n))e^{-hq_i^2x}$ and $e^{hq_iK}$, and is $0$ for $K>n$. So
$\mathbf K^h/(h\mathbf K)$ is the mean of this ratio under the probability $\nu'$ proportional to
the terms $t'_{K,e}=c_ex^{K-1}(S\omega)^{K-e}/((K-1)!(K-e)!)$ of $\mathbf K$. For fixed $e$,
$t'_{K+1,e}/t'_{K,e}=U'/(K(K+1-e))$ with $U'=\omega Sx$, and as in (c) this gives
$\E_e[(K-1)(K-e)]=U'$. Since $m=K-e\ge0$ is an integer, this gives $\E_e[K(K-1)]\le3U'$ for $e\le1$ and
$\E_2[K(K-1)]\le3U'+2$. The total weights of $e=0,1,2$ are $S\omega\mathcal I(U')$, at least $c_1\mathcal I(U')$
and $c_2x\,\mathcal I(U')$, with
\ifarxivbuild $\mathcal I$ of Section~\ref{sec:vertex}\fi, so $\nu'(e=2)\le c_2x/(c_1+S\omega)$.
Hence $\E_{\nu'}[K(K-1)]/(2n)\le\frac32\omega hS+c_2h/c_1$. As in (c),
$\E_{\nu'}e^{hq_iK}\le e^{hq_i(3+2\sqrt{U'})}=1+O(h+hS)$, since $\sqrt{U'}\le\sqrt\omega\,S$ for
$n\le N'$. Also $e^{-hq_i^2x}\ge1-q_i^2hS$. This proves (d). \qed

\subsection{Proof of Theorem~\ref{thm:layer-counts}}

Constants depend only on $Q$, $\mu$, $i$ and $\mathcal K$. Write $n=n_i$, $p=p_i$, $s=h\hat n=S\beta$,
$\tilde x=(n-1)h$ and
\[
w_d(z)=\Bigl(p+\frac hz\Bigr)^{n-1}=\sum_{K=1}^nb_Kz^{-(K-1)},\qquad b_K=\binom{n-1}{K-1}p^{n-K}h^{K-1},
\]
so that $\mathbf K^h_\beta(n,S)=h\sum_Kb_K[z^K]\widetilde\Lambda_\beta$. Let $\Lambda$,
$\varepsilon$, $\Xi$ and $R=\log\Xi$ be as in the proof of Theorem~\ref{thm:layer}(a). So
$|\varepsilon(z)-\varepsilon(0)|\le C|z|/S$ for $|z|\le z_0$, and $R$ is real on the real axis with
$|R(z)|\le C(|z|+S|z|^2)$, $|R'(z)|\le C(1+S|z|)$ and $|R''|\le C(1+S)$ on $|z|\le z_0$. Throughout,
$hS\le T^2/N\le1$, $\tilde x\le x\le\delta T$ and $S\ge(1-\delta)T$.

\emph{The coefficient identity.} The proof of Lemma~\ref{lem:layer-exact} uses the rates only as
weights of words. So for every matrix $q'$ of nonnegative rates on $F$ in place of $Q$, with the
same $q_G$ in the exponent, the coefficient of $x^{K-1}/(K-1)!$ in~\eqref{eq:layer-series} gives
\[
e^{-\langle q_G,s\rangle}\sum_{m:\,m_i=K}W'(m)\prod_{j\in G}\frac{s_j^{m_j-1}}{(m_j-1)!}=[z^K]f'^{(z)}_G(s),
\]
where $W'$ and $f'^{(z)}_G$ are $W$ and $f^{(z)}_G$ built from $q'$, with
$d'(z)=q_G-zc'^-\circ c'^+$.

\emph{The radius.} Put $\tilde z_*=\sqrt{\tilde x/(\omega_\beta S)}$ and
$\tilde\zeta=2\sqrt{\omega_\beta S\tilde x}$. If $\tilde\zeta\le Z_0$, let $\rho=\max(\tilde z_*,1/S)$.
If $\tilde\zeta>Z_0$, let $\rho$ be the critical point on $I=[\tilde z_*/2,2\tilde z_*]$ of
$\lambda_d(z)=(n-1)\log(p+h/z)+S\omega_\beta z+R(z)$. Let
$\mathcal A_d=\frac1{2\pi\mathrm i}\oint_{|z|=\rho}w_d\Lambda\,dz/z^2$, let $\mathcal A_d^{\rm abs}$ be
the integral of $|w_d\Lambda||z|^{-2}$ against $|dz|/(2\pi)$, and $\zeta_d=\max(1,S\omega_\beta\rho)$.

\begin{lemma}[the discrete contour]\label{lem:counts-contour}
For $Z_0$ large, then $\delta$ small and $T_0$ large, $\rho$ is well defined, $\rho\le z_0/2$,
$\mathcal A_d>0$, $\mathcal A_d^{\rm abs}\le C\mathcal A_d$ and
$\mathcal A_d\ge c\rho^{-1}w_d(\rho)\Lambda(\rho)\zeta_d^{-1/2}$. With
$\mathbf D=\sum_Kb_K[z^K]f^{(z)}_G(s)=\frac1{2\pi\mathrm i}\oint_{|z|=\rho}w_df^{(z)}_G(s)\,dz/z^2$,
\[
\mathbf D=(1+\varepsilon(0))\mathcal A_d(1+O(\rho/S)),\qquad\rho/S\le C(T^{-2}+x^{1/2}T^{-3/2}).
\]
\end{lemma}

\begin{proof}
\emph{Few visits.} Here $\rho\le C/S$. On the circle $|w_d|\le e^{\tilde x/\rho}\le e^{\tilde\zeta/2}\le e^{Z_0/2}$,
$|e^{S\omega_\beta z}|\le C$ and $|\Xi-1|\le C/S$. With $\Xi=1$ the integral is
$\Lambda(0)\sum_Kb_K(S\omega_\beta)^K/K!\ge\Lambda(0)p^{n-1}S\omega_\beta\ge cS\Lambda(0)$, since
$p^{n-1}\ge e^{-2q_i\tilde x}$ and $\tilde x\le Z_0^2/(4\omega_{\min}S)$. The change from $\Xi=1$ is at
most $C\Lambda(0)\rho^{-1}/S\le C\Lambda(0)$, and $\mathcal A_d^{\rm abs}\le C\Lambda(0)S$ and
$\rho^{-1}w_d(\rho)\Lambda(\rho)\le CS\Lambda(0)$. This gives the claims.

\emph{Many visits.} Put $a=h/\tilde z_*=2h\omega_\beta S/\tilde\zeta\le2\omega_{\max}/Z_0\le\frac18$, where
$\omega_{\max}=\max_{\mathcal K}\omega_\beta$.
Here the lattice enters, and $hS\le1$ makes it harmless. On $I$,
$\lambda_d'(z)=-\tilde x/(z^2(p+h/z))+S\omega_\beta+R'(z)$. The first term is at most
$-4\omega_\beta S/(1+2a)$ at $\tilde z_*/2$ and at least $-\omega_\beta S/2$ at $2\tilde z_*$,
since $p\ge\frac12$, and $|R'|\le C(1+S\tilde z_*)$ with $\tilde z_*\le\sqrt{2\delta/\omega_{\min}}$.
So $\lambda_d'$ changes sign on $I$. Also
$\lambda_d''(z)=\tilde x(2p+h/z)z^{-3}(p+h/z)^{-2}+R''(z)\ge\omega_\beta S/(18\tilde z_*)-C(1+S)>0$
on $I$, since $(2p+y)/(p+y)^2\ge\frac49$ for $p\in[\frac12,1]$ and $0\le y\le\frac14$. So $\rho$ is
unique, and $\rho\le2\tilde z_*\le z_0/2$ for $\delta$ small. On the circle $h/\rho\le\frac14$, so
$\operatorname{Re}(p+h/z)\ge\frac14$, $\lambda_d$ is holomorphic near the circle, real on
$(0,z_0]$, and $e^{\lambda_d}=w_d\Lambda/\Lambda(0)$. Since
$|p+(h/\rho)e^{-\mathrm i\phi}|$ is largest at $\phi=0$, and
$\operatorname{Re}[R(z)-R(\rho)]\le C(\rho+S\rho^2)(1-\cos\phi)$ as in the proof of
Theorem~\ref{thm:layer}(a), we get
$\operatorname{Re}\lambda_d(\rho e^{\mathrm i\phi})-\lambda_d(\rho)\le-\frac{\zeta_d}2(1-\cos\phi)$,
with $\zeta_d\ge\tilde\zeta/4\ge Z_0/4$. Let $\ell(\phi)=\lambda_d(\rho e^{\mathrm i\phi})/\zeta_d$. The
critical equation reads $\tilde x/(\rho(p\rho+h))=S\omega_\beta(1+\eta)$ with $|\eta|\le\frac1{10}$,
so $(n-1)h/(p\rho)\le3\zeta_d$, and the $\phi$-derivatives of $\ell$ up to order four are bounded.
Also $-\ell''(0)=\rho^2\lambda_d''(\rho)/\zeta_d=(1+\eta)(2p\rho+h)/(p\rho+h)+O(\rho+1/S)\in[0.4,3]$.
Remark~\ref{rem:laplace-uniform} with $m=1$, the amplitude $e^{-\mathrm i\phi}$ and $\zeta_d$ in place of
$T$ gives
$\mathcal A_d=\Lambda(0)\rho^{-1}e^{\lambda_d(\rho)}(2\pi\zeta_d(-\ell''(0)))^{-1/2}(1+O(1/\zeta_d))$,
which proves the claims, since $\Lambda(0)e^{\lambda_d(\rho)}=w_d(\rho)\Lambda(\rho)$.

\emph{Consequence.} $\mathbf D-(1+\varepsilon(0))\mathcal A_d$ is the integral of
$w_d\Lambda(\varepsilon(z)-\varepsilon(0))$, which is at most $C(\rho/S)\mathcal A_d^{\rm abs}$ in
modulus. Finally $\rho\le C/S$ with few visits and $\rho\le2\sqrt{\tilde x/(\omega_\beta S)}$ with
many.
\end{proof}

\begin{lemma}[the facet states discretized]\label{lem:counts-half}
Uniformly as in Theorem~\ref{thm:layer-counts}, $P_N(n)=h^{k-1}\mathbf D(1+O(T^2/N))$.
\end{lemma}

\begin{proof}
\emph{Step 1.} By Lemma~\ref{lem:runs}, $P_N(n)=\sum_mu_m$ with
$u_m=W(m)h^{M-1}\prod_{l\in F}\binom{n_l-1}{m_l-1}(1-hq_l)^{n_l-m_l}$. Since
$M-1=m_i+(k-2)+\sum_{j\in G}(m_j-1)$, as in Step~1 of the proof of Lemma~\ref{lem:discrete} we get
$u_m=\kappa_mv_m$ with
\[
v_m=W(m)h^{k-2}h^{m_i}\binom{n-1}{m_i-1}p^{n-m_i}\prod_{j\in G}e^{-q_js_j}\frac{s_j^{m_j-1}}{(m_j-1)!}\ge0,
\]
and with $\kappa_m$ the factor of that proof with the product over $G$ only. The state $i$ is kept
exact. By the coefficient identity with $q'=Q$, $\sum_mv_m=h^{k-1}\mathbf D$, which is positive by
Lemma~\ref{lem:counts-contour}. So $P_N(n)/(h^{k-1}\mathbf D)=\E_w\kappa_m$ for the probability $w$
proportional to $v_m$. With $M_G=\sum_{j\in G}m_j$ and $n^G_{\min}=\min_{j\in G}n_j\ge\varepsilon_{\mathcal K}N/2$,
the argument for~\eqref{eq:A-kappa} gives
$(1-M_G^2/(2n^G_{\min}))e^{-q_{\max}^2T^2/N}\le\kappa_m\le e^{hq_{\max}M_G}$.

\emph{Step 2.} Let $\sigma=\mathrm e$, and let $q'$ be $Q$ with the rates out of the states of $G$
multiplied by $\sigma$. Every letter of $G$ in a word, except possibly the last, is followed by
a jump out of it, so $W'(m)\ge\sigma^{M_G-1}W(m)$, and $\E_w\sigma^{M_G-1}\le\mathbf D'/\mathbf D$
with $\mathbf D'=\sum_Kb_K[z^K]f'^{(z)}_G(s)$ by the coefficient identity. Here $f'^{(z)}_G$ is built
from $\sigma A(z)$, $\mu^{(z)}$, $\ones+z\sigma c^-$ and $q_G-z\sigma c^-\circ c^+$. All its coefficients
in $z$ are nonnegative, so $\mathbf D'\le\rho^{-1}w_d(\rho)f'^{(\rho)}_G(s)$. We claim that
$f'^{(r)}_G(S\beta)\le CS^{(k-2)/2}e^{CS}\Lambda(r)$ for real $0\le r\le z_0$. For $|G|=1$ this is
$f'^{(r)}_G\le\sigma e^{(\sigma-1)r\omega S}\Lambda(r)$, from the closed forms. For $|G|\ge2$ we use
Lemma~\ref{lem:torus-complex} for the weights $(\sigma A(r),\mu^{(r)},\ones+r\sigma c^-)$ at the
radii $\sigma\theta'(\beta,r)$, which are admissible by Lemma~\ref{lem:critical}(a), (e). As in
Step~2 of the proof of Lemma~\ref{lem:upper}, $g$ scales by $\sigma$, so $\bar g=\sigma\theta_1(\beta,r)$
by~\eqref{eq:A-radii} for $A(r)$, and $\bar b\le\sigma b_r(\theta'(\beta,r))\le\sigma B_0$ by
Lemma~\ref{lem:critical}(b), since $\ones+r\sigma c^-\le\sigma e^{(r)}$ and $X_{A(r)}(\theta'(\beta,r))\ge0$.
So $f'^{(r)}_G(S\beta)\le C\exp(S[\sigma\langle\beta,\theta(\beta,r)\rangle-\langle\beta,d(r)\rangle+(\sigma-1)r\langle\beta,c^-\circ c^+\rangle])$,
while $\Lambda(r)=S^{-(k-2)/2}C_G(\beta,r)e^{-S\langle\beta,d(r)-\theta(\beta,r)\rangle}$ by Step~4 of the
proof of Theorem~\ref{thm:L-facet}, with $C_G\ge c>0$ on $\mathcal K\times[0,z_0]$. This proves the
claim. With Lemma~\ref{lem:counts-contour} and $\mathbf D\ge\frac12\mathcal A_d$,
$\mathbf D'/\mathbf D\le C\zeta_d^{1/2}S^{(k-2)/2}e^{CS}\le e^{C'T}$. So the convex function
$\Upsilon(c)=\log\E_we^{c(M_G-1)}$ has $\Upsilon(0)=0$ and $\Upsilon(1)\le C'T$.

\emph{Step 3.} As in Step~4 of the proof of Lemma~\ref{lem:discrete},
$\E_we^{hq_{\max}M_G}=1+O(T^2/N)$ and $\E_wM_G^2=O(T^2)$, so $\E_w\kappa_m=1+O(T^2/N)$.
\end{proof}

\begin{proof}[Proof of Theorem~\ref{thm:layer-counts}]
We prove (a), (b) and (c) in turn. (a) By Lemmas~\ref{lem:counts-half} and~\ref{lem:counts-contour},
$P_N(n)=h^{k-1}(1+\varepsilon(0))\mathcal A_d(1+O(\rho/S+T^2/N))$. Lemma~\ref{lem:discrete} for the
face $G$, with $N'$ steps, the same $h$ and the time $S$, applies since $S^2/N'\le T^2/N$, and
gives $P_{N'}(\hat n)=h^{k-2}f_G(s)(1+O(T^2/N))$. For $|G|=1$ this is the formula for the path
$j^{N'}$. Also $f_G(s)=\Lambda(0)(1+\varepsilon(0))$ and $\mathbf K^h_\beta(n,S)=h\mathcal A_d/\Lambda(0)$.
Dividing, $1+\varepsilon(0)$ cancels, and $E=O(\rho/S+T^2/N)$. Taking $T_0$ large and
$\varepsilon_1$ small makes $|E|\le\frac12$.

(b) Fix $x'\in\{hn,h(n-1)\}$ and let $\rho$ and $\mathcal A$ be the radius and the integral of the
proof of Theorem~\ref{thm:layer}(a) for $x'$, so that $\mathcal A^{\rm abs}\le C\mathcal A$ and
$\mathbf K_\beta(x',S)=e^{-q_ix'}\mathcal A/\Lambda(0)$. For $n=1$ and $x'=0$ both sides are
$h[z]\widetilde\Lambda_\beta$. On the circle put $y=h/z-hq_i$. Since $h/\rho\le\max(hS,4h\omega_\beta S/Z_0)$,
we have $|y|\le CT^2/N\le\frac12$, and
$w_d=e^{(n-1)y}e^{(n-1)(\log(1+y)-y)}$ with $(n-1)|y|^2\le2h\tilde x/\rho^2+2hq_i^2\tilde x\le CT^2/N$,
since $\tilde x/\rho^2\le4\omega_\beta S$. Also $e^{(n-1)y}=e^{x'/z-q_ix'}$ if $x'=\tilde x$, with one
more factor $e^{-y}=1+O(T^2/N)$ if $x'=hn$. So $w_d=e^{x'/z-q_ix'}(1+\eta)$ with
$|\eta|\le CT^2/N$, and $|\mathcal A_d-e^{-q_ix'}\mathcal A|\le C(T^2/N)e^{-q_ix'}\mathcal A$. This is (b),
and with (a) it gives the last display.

(c) The proofs of Theorem~\ref{thm:layer}(b) and (c) compare $\mathbf K_\beta$ with $\mathbf B_\beta$
at the same $(x,S)$. They show $\mathbf K_\beta=\mathbf B_\beta(1+O(1/T+\sqrt{x/T}))$ for $x\le1$,
$\mathbf K_\beta=\mathbf B_\beta e^{O(\sqrt{x/T}(1+x))}$ for $1\le x\le\delta T$, and
$\mathbf K_\beta=\mathbf B_\beta(1+c_1(\beta)z_*I_0(\zeta)/I_1(\zeta)+O(x/T))$ when
$\pi^{\rm in}_\beta=\pi^{\rm out}_\beta$, where $c_1(\beta)z_*I_0/I_1$ is bounded. We combine these
with (b).
\end{proof}


{\small\begin{thebibliography}{99}

\bibitem{amos1974}
D.~E.~Amos, \emph{Computation of modified Bessel functions and their ratios},
Math.~Comp.~\textbf{28} (1974), 239--251.
\url{https://doi.org/10.1090/S0025-5718-1974-0333287-7}.

\bibitem{bbf2026}
M.~Bena\"im, M.~Bravo and M.~Faure, \emph{On the occupation measure of evolution
models with vanishing mutations}, preprint, 2026.
\href{https://arxiv.org/abs/2602.06829}{arXiv:2602.06829}.

\bibitem{bermanplemmons1994}
A.~Berman and R.~J.~Plemmons, \emph{Nonnegative Matrices in the
Mathematical Sciences}, SIAM, Philadelphia, 1994.
\url{https://doi.org/10.1137/1.9781611971262}.

\bibitem{bfg2014}
L.~Bertini, A.~Faggionato and D.~Gabrielli, \emph{From level 2.5 to level 2
large deviations for continuous time Markov chains}, Markov Process.~Related
Fields \textbf{20} (2014), 545--562.
\href{https://arxiv.org/abs/1212.6908}{arXiv:1212.6908}.

\bibitem{bfg2015}
L.~Bertini, A.~Faggionato and D.~Gabrielli, \emph{Large deviations of the
empirical flow for continuous time Markov chains}, Ann.~Inst.~H.~Poincar\'e
Probab.~Stat.~\textbf{51} (2015), 867--900.
\url{https://doi.org/10.1214/14-AIHP601}.

\bibitem{bgl2024}
L.~Bertini, D.~Gabrielli and C.~Landim, \emph{Metastable $\Gamma$-expansion
of finite state Markov chains level two large deviations rate functions},
Ann.~Appl.~Probab.~\textbf{34} (2024), 3820--3869.
\url{https://doi.org/10.1214/24-AAP2051}.

\bibitem{bollobas1998}
B.~Bollob\'as, \emph{Modern Graph Theory}, Springer, New York, 1998.
\url{https://doi.org/10.1007/978-1-4612-0619-4}.

\bibitem{bouguetcloez2018}
F.~Bouguet and B.~Cloez, \emph{Fluctuations of the empirical measure of freezing
Markov chains}, Electron.~J.~Probab.~\textbf{23} (2018), no.~2, 1--31.
\url{https://doi.org/10.1214/17-EJP130}.

\bibitem{bhk2007}
D.~Brydges, R.~van~der~Hofstad and W.~K\"onig, \emph{Joint density for the
local times of continuous-time Markov chains}, Ann.~Probab.~\textbf{35}
(2007), 1307--1332. \url{https://doi.org/10.1214/009171906000001024}.

\bibitem{cvc2022}
G.~Carugno, P.~Vivo and F.~Coghi, \emph{Graph-combinatorial approach for
large deviations of Markov chains}, J.~Phys.~A \textbf{55} (2022), 295001.
\url{https://doi.org/10.1088/1751-8121/ac79e6}.

\bibitem{cv2010}
V.~\v{C}ekanavi\v{c}ius and P.~Vellaisamy, \emph{Compound Poisson and signed
compound Poisson approximations to the Markov binomial law}, Bernoulli
\textbf{16} (2010), 1114--1136. \url{https://doi.org/10.3150/09-BEJ246}.

\bibitem{chendeng2013}
J.~Chen and X.~Deng, \emph{Large deviations and related problems for absorbing
Markov chains}, Stochastic Process.~Appl.~\textbf{123} (2013), 2398--2418.
\url{https://doi.org/10.1016/j.spa.2013.02.014}.

\bibitem{cmsm2013}
P.~Collet, S.~Mart\'inez and J.~San Mart\'in, \emph{Quasi-Stationary
Distributions: Markov Chains, Diffusions and Dynamical Systems}, Springer,
Heidelberg, 2013. \url{https://doi.org/10.1007/978-3-642-33131-2}.

\bibitem{darrochseneta1965}
J.~N.~Darroch and E.~Seneta, \emph{On quasi-stationary distributions in
absorbing discrete-time finite Markov chains}, J.~Appl.~Probab.~\textbf{2}
(1965), 88--100. \url{https://doi.org/10.2307/3211876}.

\bibitem{darrochseneta1967}
J.~N.~Darroch and E.~Seneta, \emph{On quasi-stationary distributions in
absorbing continuous-time finite Markov chains}, J.~Appl.~Probab.~\textbf{4}
(1967), 192--196. \url{https://doi.org/10.2307/3212311}.

\bibitem{dawsongood1957}
R.~Dawson and I.~J.~Good, \emph{Exact Markov probabilities from oriented
linear graphs}, Ann.~Math.~Statist.~\textbf{28} (1957), 946--956.
\url{https://doi.org/10.1214/aoms/1177706795}.

\bibitem{dietz2005}
Z.~Dietz and S.~Sethuraman, \emph{Large deviations for a class of
nonhomogeneous Markov chains}, Ann.~Appl.~Probab.~\textbf{15} (2005),
421--486. \url{https://doi.org/10.1214/105051604000000990}.

\bibitem{dietz2007}
Z.~Dietz and S.~Sethuraman, \emph{Occupation laws for some
time-nonhomogeneous Markov chains}, Electron.~J.~Probab.~\textbf{12} (2007),
no.~23, 661--683. \url{https://doi.org/10.1214/EJP.v12-413}.

\bibitem{dobrushin1953}
R.~L.~Dobrushin, \emph{Limit theorems for a Markov chain of two states},
Izv.~Akad.~Nauk SSSR Ser.~Mat.~\textbf{17} (1953), 291--330 (in Russian). English translation in Select.~Transl.~Math.~Statist.~Probab.~\textbf{1} (1961), 97--134.

\bibitem{dv1975}
M.~D.~Donsker and S.~R.~S.~Varadhan, \emph{On a variational formula for the
principal eigenvalue for operators with maximum principle}, Proc.~Nat.~Acad.~Sci.~U.S.A.
\textbf{72} (1975), 780--783. \url{https://doi.org/10.1073/pnas.72.3.780}.

\bibitem{dupuisliu2015}
P.~Dupuis and Y.~Liu, \emph{On the large deviation rate function for the
empirical measures of reversible jump Markov processes}, Ann.~Probab.~\textbf{43}
(2015), 1121--1156. \url{https://doi.org/10.1214/13-AOP883}.

\bibitem{gabriel1959}
K.~R.~Gabriel, \emph{The distribution of the number of successes in a sequence
of dependent trials}, Biometrika \textbf{46} (1959), 454--460.
\url{https://doi.org/10.1093/biomet/46.3-4.454}.

\bibitem{good1961}
I.~J.~Good, \emph{The frequency count of a Markov chain and the transition to
continuous time}, Ann.~Math.~Statist.~\textbf{32} (1961), 41--48.
\url{https://doi.org/10.1214/aoms/1177705138}.

\bibitem{goodman1958}
L.~A.~Goodman, \emph{Exact probabilities and asymptotic relationships for
some statistics from $m$-th order Markov chains}, Ann.~Math.~Statist.~\textbf{29}
(1958), 476--490. \url{https://doi.org/10.1214/aoms/1177706623}.

\bibitem{gnw2024}
A.~Guillin, B.~Nectoux and L.~Wu, \emph{Large deviations of the empirical
measures of a strong-Feller Markov process inside a subset and quasi-ergodic
distribution}, preprint, 2024.
\href{https://arxiv.org/abs/2411.17216}{arXiv:2411.17216}.

\bibitem{herveledoux2013}
L.~Herv\'e and J.~Ledoux, \emph{A local limit theorem for densities of the additive
component of a finite Markov additive process}, Statist.~Probab.~Lett.~\textbf{83}
(2013), 2119--2128. \url{https://doi.org/10.1016/j.spl.2013.05.032}.

\bibitem{hkst2018}
R.~Huang, D.~Kious, V.~Sidoravicius and P.~Tarr\`es, \emph{Explicit formula for
the density of local times of Markov jump processes}, Electron.~Commun.~Probab.~\textbf{23} (2018), no.~90, 1--7. \url{https://doi.org/10.1214/18-ECP194}.

\bibitem{kagey131}
P.~Kagey, \emph{Problem 131}, Open Problem Collection,
\url{https://peterkagey.com/problems/131/} (accessed October 2026).

\bibitem{kl2024}
S.~Kim and C.~Landim, \emph{$\Gamma$-expansion of the measure-current large
deviations rate functional of non-reversible finite-state Markov chains},
Stochastic Process.~Appl.~\textbf{189} (2025), 104733.
\url{https://doi.org/10.1016/j.spa.2025.104733}.

\bibitem{km2003}
I.~Kontoyiannis and S.~P.~Meyn, \emph{Spectral theory and limit theorems for
geometrically ergodic Markov processes}, Ann.~Appl.~Probab.~\textbf{13} (2003),
304--362. \url{https://doi.org/10.1214/aoap/1042765670}.

\bibitem{kmn2010}
Y.~Kovchegov, N.~Meredith and E.~Nir, \emph{Occupation times and Bessel
densities}, Statist.~Probab.~Lett.~\textbf{80} (2010), 104--110.
\url{https://doi.org/10.1016/j.spl.2009.09.018}.

\bibitem{landim2023}
C.~Landim, \emph{Metastability from the large deviations point of view: a
$\Gamma$-expansion of the level two large deviations rate functional of
non-reversible finite-state Markov chains}, Stochastic Process.~Appl.~\textbf{165}
(2023), 275--315. \url{https://doi.org/10.1016/j.spa.2023.09.001}.

\bibitem{miller1961}
H.~D.~Miller, \emph{A convexity property in the theory of random variables
defined on a finite Markov chain}, Ann.~Math.~Statist.~\textbf{32} (1961),
1260--1270. \url{https://doi.org/10.1214/aoms/1177704865}.

\bibitem{nn1987a}
P.~Ney and E.~Nummelin, \emph{Markov additive processes I. Eigenvalue
properties and limit theorems}, Ann.~Probab.~\textbf{15} (1987), 561--592.
\url{https://doi.org/10.1214/aop/1176992159}.

\bibitem{nn1987b}
P.~Ney and E.~Nummelin, \emph{Markov additive processes II. Large deviations},
Ann.~Probab.~\textbf{15} (1987), 593--609.
\url{https://doi.org/10.1214/aop/1176992160}.

\bibitem{dlmf}
F.~W.~J.~Olver, A.~B.~Olde Daalhuis, D.~W.~Lozier, B.~I.~Schneider,
R.~F.~Boisvert, C.~W.~Clark, B.~R.~Miller, B.~V.~Saunders, H.~S.~Cohl and
M.~A.~McClain (eds.), \emph{NIST Digital Library of Mathematical Functions},
Release 1.2.8 of 2026-09-15, \url{https://dlmf.nist.gov/}.

\bibitem{pedler1971}
P.~J.~Pedler, \emph{Occupation times for two state Markov chains},
J.~Appl.~Probab.~\textbf{8} (1971), 381--390.
\url{https://doi.org/10.2307/3211908}.

\bibitem{paperA}
A.~Pemmasani, \emph{Uniform Bessel bounds for magnetization crossings in finite
Ising chains}, preprint, 2026.
\href{https://arxiv.org/abs/2610.06935}{arXiv:2610.06935}.

\bibitem{paperB}
A.~Pemmasani, \emph{Vertex crossings in a symmetric Markov multinomial
model}, preprint, 2026. \href{https://arxiv.org/abs/2610.09258v1}{arXiv:2610.09258v1}\ifarxivbuild, the
first version of the present preprint\else, version~1 of the arXiv entry of~\cite{arxivC1}\fi.

\ifarxivbuild\else
\bibitem{arxivC1}
A.~Pemmasani, \emph{Occupation counts of rarely switching Markov chains: stratified local
laws, support selection and the boundary layer}, the arXiv version, with Sections~7 and~8 on the
complete graph and near a vertex, preprint, 2026.
\href{https://arxiv.org/abs/2610.09258v2}{arXiv:2610.09258v2}, version~2.
\fi

\ifarxivbuild
\bibitem{paperC2}
A.~Pemmasani, \emph{What the occupation law of a finite Markov chain
determines: time reversal, principal minors and cycle trees}, in
preparation, 2026.
\fi

\bibitem{polettini2015}
M.~Polettini, \emph{BEST statistics of Markovian fluxes: a tale of Eulerian
tours and Fermionic ghosts}, J.~Phys.~A \textbf{48} (2015), 365005.
\url{https://doi.org/10.1088/1751-8113/48/36/365005}.

\bibitem{stadje1999}
W.~Stadje, \emph{Joint distributions of the numbers of visits for finite-state
Markov chains}, J.~Multivariate Anal.~\textbf{70} (1999), 157--176.
\url{https://doi.org/10.1006/jmva.1999.1814}.

\bibitem{whittle1955}
P.~Whittle, \emph{Some distribution and moment formulae for the Markov
chain}, J.~Roy.~Statist.~Soc.~Ser.~B~\textbf{17} (1955), 235--242.
\url{https://doi.org/10.1111/j.2517-6161.1955.tb00197.x}.

\bibitem{wu1982}
F.~Y.~Wu, \emph{The Potts model}, Rev.~Modern Phys.~\textbf{54} (1982), 235--268.
\url{https://doi.org/10.1103/RevModPhys.54.235}.

\end{thebibliography}
}
\end{document}